\documentclass[10pt]{amsart}
\usepackage{amssymb, amsfonts, lacromay}
\usepackage{mathrsfs,comment}
\usepackage[usenames,dvipsnames]{color}
\usepackage[normalem]{ulem}
\usepackage{url}

\usepackage[all,arc,2cell]{xy}
\UseAllTwocells
\usepackage{enumerate}

\usepackage{todonotes}
\usepackage{hyperref}  
\hypersetup{%
  bookmarksnumbered=true,%
  bookmarks=true,%
  colorlinks=true,%
  linkcolor=blue,%
  citecolor=blue,%
  filecolor=blue,%
  menucolor=blue,%
  pagecolor=blue,%
  urlcolor=blue,%
  pdfnewwindow=true,%
  pdfstartview=FitBH}   

\def\makeautorefname#1#2{\expandafter\def\csname#1autorefname\endcsname{#2}}
\def\equationautorefname~#1\null{(#1)\null}
\makeautorefname{footnote}{footnote}%
\makeautorefname{item}{item}%
\makeautorefname{figure}{Figure}%
\makeautorefname{table}{Table}%
\makeautorefname{part}{Part}%
\makeautorefname{appendix}{Appendix}%
\makeautorefname{chapter}{Chapter}%
\makeautorefname{section}{Section}%
\makeautorefname{subsection}{Section}%
\makeautorefname{subsubsection}{Section}%
\makeautorefname{theorem}{Theorem}%
\makeautorefname{thm}{Theorem}%
\makeautorefname{fthm}{Folklore Theorem}%
\makeautorefname{mthm}{Main Theorem}%
\makeautorefname{cor}{Corollary}%
\makeautorefname{lem}{Lemma}%
\makeautorefname{prop}{Proposition}%
\makeautorefname{pro}{Property}
\makeautorefname{conj}{Conjecture}%
\makeautorefname{defn}{Definition}%
\makeautorefname{notn}{Notation}
\makeautorefname{notns}{Notations}
\makeautorefname{rem}{Remark}%
\makeautorefname{quest}{Question}%
\makeautorefname{exmp}{Example}%
\makeautorefname{ax}{Axiom}%
\makeautorefname{claim}{Claim}%
\makeautorefname{ass}{Assumption}%
\makeautorefname{asss}{Assumptions}%
\makeautorefname{con}{Construction}%
\makeautorefname{prob}{Problem}%
\makeautorefname{warn}{Warning}%
\makeautorefname{obs}{Observation}%
\makeautorefname{conv}{Convention}%
\makeautorefname{adden}{Addendum}

\newtheorem{thm}{Theorem}[section]
\newtheorem{fthm}{Folklore Theorem}[section]
\newtheorem{mthm}{Main Theorem}[section]
\newtheorem{cor}{Corollary}[section]
\newtheorem{prop}{Proposition}[section]
\newtheorem{lem}{Lemma}[section]

\theoremstyle{definition}
\newtheorem{defn}{Definition}[section]
\newtheorem{ass}{Assumption}[section]

\newtheorem{con}{Construction}[section]
\newtheorem{exmp}{Example}[section]
\newtheorem{notn}{Notation}[section]
\newtheorem{notns}{Notations}[section]

\newtheorem{rem}{Remark}[section]
\newtheorem{warn}{Warning}[section]

\newtheorem{adden}{Addendum}[section]
\newtheorem{obs}{Observation}[section]
\newtheorem{conv}{Convention}[section]

\makeatletter
\let\c@conv=\c@thm
\let\c@obs=\c@thm
\let\c@cor=\c@thm
\let\c@prop=\c@thm
\let\c@lem=\c@thm
\let\c@prob=\c@thm
\let\c@mthm=\c@thm
\let\c@con=\c@thm
\let\c@conj=\c@thm
\let\c@defn=\c@thm
\let\c@notn=\c@thm
\let\c@notns=\c@thm
\let\c@exmp=\c@thm
\let\c@ax=\c@thm
\let\c@pro=\c@thm
\let\c@ass=\c@thm
\let\c@warn=\c@thm
\let\c@rem=\c@thm
\let\c@sch=\c@thm
\let\c@adden=\c@thm
\let\c@equation\c@thm
\numberwithin{equation}{section}
\makeatother

\definecolor{orange}{rgb}{1,0.5,0}

\newcommand{\mb}[1]{\mathbf{#1}}

\newcommand{\gen}{$\bF_\bullet$}
\newcommand{\cof}{$\bF_\bullet$}

\newcommand{\lbull}{^{\bullet}S^{\star}}
\newcommand{\lbullv}{^{\bullet}S^{V}}
\newcommand{\lbullov}{^{\bullet}S_0^{V}}
\newcommand{\lbullo}{^{\bullet}S_0^{\star}}

\newcommand{\rbull}{S^{\star}{^{\bullet}}}
\newcommand{\rbullo}{S_0^{\star}{^{\bullet}}}
\newcommand{\rbulll}{S_1^{\star}{^{\bullet}}}
\newcommand{\rbullov}{{S_0^V}{^{\bullet}}}
\newcommand{\rbulllv}{{S_1^V}{^{\bullet}}}

\newcommand{\rbullv}{S^{V}{^{\bullet}}}
\newcommand{\rbullpv}{S^{V'}{^{\bullet}}}
\newcommand{\rbullw}{S^{W}{^{\bullet}}}
\newcommand{\rbullvw}{S^{V\oplus W}{^{\bullet}}}

\newcommand{\nSv}{^n\!(S^V)}
\newcommand{\nSvw}{^n\!(S^{V\oplus W})}
\newcommand{\mSv}{^m\!(S^V)}
\newcommand{\sSv}{^s\!(S^V)}
\newcommand{\phSv}{^{\ph_j}\!(S^V)}

\newcommand{\Svn}{(S^{V})^n}
\newcommand{\Svm}{(S^{V})^m}
\newcommand{\Svs}{(S^{V})^s}

\newcommand{\inj}{\Lambda}

\newcommand{\bi}{\mathbf{i}}
\newcommand{\bj}{\mathbf{j}}
\newcommand{\bk}{\mathbf{k}}
\newcommand{\bm}{\mathbf{m}}
\newcommand{\bn}{\mathbf{n}}
\newcommand{\bp}{\mathbf{p}}
\newcommand{\bq}{\mathbf{q}}

\newcommand{\TG}{\sT_G}
\newcommand{\UG}{\ul{G\sU_{\ast}}}

\newcommand{\FdashG}{\sF[G\sT]}
\newcommand{\PIdashG}{\PI[G\sT]}

\newcommand{\FsubG}{\sF_G[\sT_G]}
\newcommand{\PIsubG}{\PI_G[\sT_G]}

\newcommand{\DdashG}{\sD[\sT_G]}
\newcommand{\DsubG}{\sD_G[\sT_G]}

\newcommand{\EdashG}{\sE[\sT_G]}
\newcommand{\EsubG}{\sE_G[\sT_G]}

\newcommand{\Dalg}{\bD[\PI[G\sT]]}
\newcommand{\DGalg}{\bD_G[\PI_G[\sT_G]]}

\title{Equivariant infinite loop space theory, 
The space level story}

\author{J. Peter May}
\address{Department of Mathematics, The University of Chicago, Chicago, IL 60637}
\email{may@math.uchicago.edu}
\author{Mona Merling}
\address{department of Mathematics, Johns Hopkins University, Baltimore, MD 21218}
\email{mmerling@math.jhu.edu}
\author{Ang\'elica M. Osorno}
\address{Department of Mathematics, Reed College, Portland, OR 97202}
\email{aosorno@reed.edu}


\begin{document}

\begin{abstract}
We rework and generalize equivariant infinite loop space theory, which shows
how to construct $G$-spectra from $G$-spaces with suitable structure.  There is a 
classical version which gives classical $\OM$-$G$-spectra for any topological group $G$, but our 
focus is on the construction of genuine $\OM$-$G$-spectra when $G$ is finite. We also show
what is and is not true when $G$ is a compact Lie group.

We give new information about the Segal and operadic equivariant infinite loop space
machines, supplying many details that are missing from the literature, and we prove by direct 
comparison that the two machines give equivalent output when fed equivalent input. The proof of the 
corresponding nonequivariant uniqueness theorem, due to May and Thomason, works for classical 
$G$-spectra for general $G$ but fails for genuine $G$-spectra.  Even in the nonequivariant case, our comparison theorem is considerably more precise, giving an illuminating direct point-set level comparison.  

We have taken the opportunity to update this general area, equivariant and nonequivariant,
giving many new proofs, filling in some gaps, and giving a number of corrections to results and proofs in the literature.

\end{abstract}

\maketitle

\tableofcontents


\section*{Introduction}

Equivariant homotopy theory is much richer than nonequivariant homotopy theory.  Equivariant generalizations
of nonequivariant theory are often non-trivial and often admit several variants.  Nonequivariantly,
symmetric monoidal (or equivalently permutative) categories and $E_{\infty}$ spaces give rise to spectra.  
There are several ``machines'' that take such categorical or space level input and deliver spectra as
output \cite{BVbook, MayGeo, Seg} and there are comparison theorems showing that all such machines are 
equivalent \cite{MayIMon, MayPerm2, MT}.

Equivariantly, there are different choices of $G$-spectra that can be taken as output of infinite loop space machines. One choice is classical $G$-spectra, which are simply spectra with a $G$-action.  That is, in the simplest formulation, they are sequences of based $G$-spaces $E_n$ and based $G$-maps $\SI E_n\rtarr E_{n+1}$; they are $\OM$-$G$-spectra if the adjoint maps $E_n\rtarr \OM E_{n+1}$ are weak $G$-equivalences.  They have $H$-fixed point spectra for $H\subset G$, and a map of classical $G$-spectra is a weak equivalence if it induces a nonequivariant weak equivalence on each fixed point  spectrum.  For finite groups $G$, or compact Lie groups $G$, these are often called ``naive'' $G$-spectra in the literature.  We use the word ``classical" rather than ``naive'' to avoid prejudice.   

Classical $G$-spectra really are ``naive'', in the sense that they do not know about crucial structure.  For example, one cannot prove any version of Poincar\'e duality in the cohomology theories they represent.  They are indexed on the natural numbers, and $n$ should be thought of geometrically as a stand-in for $\mathbb{R}^n$ or $S^n$, with trivial $G$-action.  When $G$ is a compact Lie group,  we also have genuine $G$-spectra, which are indexed on representations or, more precisely, real vector
spaces (better, inner product spaces) $V$ with an action of $G$.  In the simplest formulation, they are systems of based $G$-spaces $E_V$ and  based $G$-maps $\SI^W E_V\rtarr E_{V\oplus W}$; they are $\OM$-$G$-spectra if the adjoint maps $E_V\rtarr \OM^{W}\OM E_{V\oplus W}$ are weak $G$-equivalences.  They have $H$-fixed point spectra for $H\subset G$, and a map of genuine $G$-spectra is a weak equivalence if it induces a nonequivariant weak equivalence on each fixed point  spectrum.  

Since suspending and looping with respect to the spheres $S^V$ does not commute with passage to $H$-fixed points when the action of $G$ on $V$ is nontrivial, the $H$-fixed spectra retain homotopical information about the representations of $G$. That information is encoded in the notion of equivalence just given.  One can also restrict attention to subclasses of representations and obtain a plethora of kinds of $G$-spectra intermediate between the classical ones (indexed only on trivial representations) and the genuine ones (indexed on all finite dimensional representations).

We shall refer to ``genuine" structures when dealing with structures that relate to deloopings with respect to all finite-dimensional representations, not just the trivial ones.  To simplify the exposition, we shall assume that $G$ is finite  until \autoref{general}.  There we shall make clear that everything that comes earlier works just as well for compact Lie groups, provided that we restrict attention everywhere to representations whose isotropy groups $H$ are of finite index in $G$.  

Thus, with $G$ finite, our main theme is equivariant infinite loop space machines whose inputs are $G$-spaces $X$ with extra structure and whose outputs are genuine $\OM$-$G$-spectra whose zeroth spaces are equivariant ``group completions" of $X$.  We recall the homological definition of a group completion of a (homotopy associative and commutative) Hopf space in \autoref{gpcomp}, and we explain how unique such a group completion is in \autoref{CCMT}. We recall the corresponding homological definition of a group completion of a Hopf $G$-space in \autoref{Defngpcomp}.

Nonequivariantly, infinite loop space machines give a ``recognition principle'', the point of which is to recognize what extra structure on a Hopf space $X$ is sufficient to ensure that its group completion is an infinite loop space and to do so in a way that facilitates concrete calculations. One machine is the operadic one developed by the first author in \cite{MayGeo} and another is the approach using $\GA$-spaces introduced by Segal in \cite{Seg}.  The opposite category of Segal's $\GA$ is the category $\sF$ (denoted elsewhere as $\bf{Fin}_*$)  of finite based sets, and we shall call $\GA$-spaces $\sF$-spaces. 
We require them to take values in the category $\sT$ of nondegenerately based spaces when applying infinite loop space machines.\footnote{We warn the reader that when revising this paper in 2021, we found several previously unnoticed subtleties in the verification that basepoints are nondegenerate, hence we have been much more careful about basepoints in the current version.}

We emphasize the group completion property. We say that a Hopf space $X$ is grouplike if $\pi_0(X)$ is a group, and then $X$ is its own group completion.  It is emphatically {\em not} the case in the most interesting examples that $X$ itself is grouplike, and the process of group completion has a dramatic effect on homotopy groups.  The book \cite{CLM} gives reams of concrete calculations based on Dyer-Lashof operations visible on the input and output of the operadic machine and visibly preserved by group completion. For example, the homologies of free $E_{\infty}$-algebras  
$\bC X$ and their group completions $QX = \OM^{\infty}\SI^{\infty}X$ are computed with all structure in sight (Hopf algebra structure, Dyer Lashof and Steenrod operations) as explicit functors of the homology of $X$. In particular, $\bC S^0$ is the disjoint union of the classifying spaces of the symmetric groups  $\SI_n$, whereas the homotopy groups of $Q S^0$ are the stable homotopy groups of spheres.  The group completion property here is the Barratt-Priddy-Quillen theorem, a dramatic example discovered in the 1970's. Many other applications have followed since.  

The paper \cite{FM} shows how the Dyer-Lashof  operations can also be constructed, albeit considerably less conveniently, using the Segal machine.  A key point is that the higher homotopies implicit in the Segal machine are given by {\em properties} (certain maps are homotopy equivalences with unprescribed homotopy inverses) whereas the higher homotopies implicit in the operadic machine are given by {\em structure}.  It is structure that most easily leads to explicit formulas.  It pays to have both machines, and thus it becomes necessary to know how they compare.

Categories of operators were introduced in \cite{MT}\footnote{Modulo multicategorical generalization, Lurie's $\infty$-operads \cite{lurieHA} are essentially defined as categories of operators.  The equivariant generalization of categories of operators is central to our work here and should lead to an equivariant version of $\infty$-operads.} 
to give a home for a common generalization of operadic input, namely  algebras over an $E_{\infty}$ operad, and Segalic input, namely special $\sF$-spaces.   To give context, we recall from \cite{MT} the nonequivariant definition of an infinite loop space machine from and the uniqueness theorem that characterizes them.  

\begin{defn}[{\cite[Definition 2.1]{MT}}]\label{machine}  Let $\sD[\sT]$ be the category of special $\sD$-spaces, where $\sD$ is an $E_{\infty}$ category of operators; an object $X\in \sD[\sT]$ has a first space $X_1$. Let $\bf{Sp}_{c}$ be the category of connective $\OM$-spectra.  An infinite loop space machine is a functor 
$$\bE \colon \sD[\sT] \rtarr \bf{Sp}_c$$
together with a natural group completion $\io\colon X_1\rtarr \bE_0X$. 
\end{defn}

\begin{rem} Machines appear most naturally in an equivalent form that more readily generalizes equivariantly.   The target category can be {\em positive} connective $\OM$-spectra, for which the adjoint structure maps  $E_n \rtarr \OM E_{n+1}$ are weak equivalences for $n\geq 1$.  We then require a natural map  $\nu\colon X_1 \rtarr E_0X$ such that the composite  of $\nu$ and $E_0X \rtarr \OM E_1X$ is a group completion.
\end{rem}

A category of operators $\sD$ comes with a map $\xi\colon \sD\rtarr \sF$ of categories of operators, and there is a Segal machine $\bS$ defined on special $\sF$-spaces. There is also a (derived) functor $\xi_*\colon \sD[\sT] \rtarr \sF[\sT]$. The following result compares machines, showing that  any machine is equivalent to the Segal machine.

\begin{thm}[{\cite[Corollary 2.6]{MT}}]\label{MT} For any infinite loop space machine $\bE$, there is a natural equivalence of spectra between $\bE X$ and  $\bS\xi_*X$.
\end{thm}

\begin{rem}\label{char}  On homotopy categories, any such machine $\bE$ restricts to an equivalence from the full subcategory of grouplike special $\sD$-spaces to the category of connective spectra.   To see that, it suffices to prove it for any one machine, and it has been proven for both the operadic and Segal machines. 
We call such a result a ``characterization principle'' rather than a recognition principle.  One can prove a characterization principle without any consideration of group completion (in the homological sense) of  non-grouplike objects.  Quite a few papers, starting with Bousfield-Friedlander \cite{BF}, consider characterization theorems without ever considering group completion.  Others consider group completion in senses alternative to our homological version.
\end{rem}

The proof of \autoref{MT} depends heavily on the combinatorial properties of the original Segal machine \cite{Seg} and does not work in the genuine context, where one must implicitly keep track of deloopings with respect to the representation spheres $S^V$ for finite dimensional $G$-representations $V$.  Our main result is a similar  comparison theorem in the equivariant context. It can be stated roughly as follows.  A precise statement is given in \autoref{WOW}. 

\begin{mthm}\label{MAIN0}
Given equivalent input, the equivariant versions of the Segal and operadic infinite loop space machines give equivalent $\OM$-$G$-spectra as output.
\end{mthm}

The genuine equivariant generalization of the operadic approach was worked out by Costenoble and Waner \cite{CW}, following an earlier but unpublished treatment of Hauschild, May, and Waner.  Their treatment was quite brief since the point of their paper was a new machine defined on appropriate presheaves of spaces on the orbit category of $G$ rather than on $G$-spaces and since they had the unpublished treatment on hand.  The operadic approach has been worked out more fully and quite recently in \cite{GM3}, which can be viewed in part as a prequel to this paper.  It includes discussion of the recognition principle for $V$-fold loop spaces for representations $V$ of $G$, as we will recall briefly in \autoref{operadicmachine}.

The genuine equivariant version of Segal's approach was first studied by Segal \cite{Seg2} in an influential, albeit flawed,\footnote{The definition in \cite{Seg2} of the specialness condition on $\sF$-$G$-spaces needed to get  $\OM$-$G$-spectra as output is incorrect, and the mistake was followed in several later papers published by others.} unpublished preprint.   The first published version is due to Shimakawa \cite{Shim}.  He followed a modification due to Woolfson \cite{Woolf} of Segal's original nonequivariant machine \cite{Seg} and he corrected \cite{Seg2} to give the first proof that the construction gives genuine $\OM$-$G$-spectra when fed appropriate input. In order to give a self-contained treatment,  we give the details of the proof, following \cite{Seg2, Shim}, in \autoref{heredity}.  As discussed in \autoref{Shimgp}, Shimakawa outlined one proof of the group completion property \cite[p.357]{Shim1}.  We give details of a quite different proof, and  the length of our treatment of the Segal machine is largely due to subtleties in proving that property. The work of Segal and Shimakawa was later revisited by Blumberg \cite{Blum}, whose focus was on axiomatizing what can be said in the case of compact Lie groups $G$.  As said before, we defer discussion of the compact Lie case to \autoref{general}.

We shall introduce $G$-categories of operators in Sections \ref{GCF}. They come in two flavors, one with a map 
$\xi\colon \sD \rtarr \sF$ and the other with a map $\xi_G\colon \sD_G\rtarr \sF_G$, where $\sF_G$ is the category of finite based $G$-sets.  We shall develop equivalent infinite loop $G$-space machines starting from either flavor.  With definitions of as yet undefined terms to be supplied later, here is the equivariant analogue of \autoref{machine}.  We say that a $G$-spectrum $E$ is a positive $\Omega$-$G$-spectrum if its adjoint structure maps $E(V) \rtarr \Omega^WE(V\oplus W)$ are weak $G$-equivalences when $V^G\neq 0$ and is connective if its fixed point spectra are all connective.

\begin{defn}\label{Gmachine}  Let $G$ be a finite group.  Let $\sD[G\sT]$ be the category of \gen-special $\sD$-spaces (as defined in \autoref{weakFn}), where $\sD$ is an $E_{\infty}$ $G$-category of operators over $\sF$ (as defined in \autoref{Einf}); an object $X\in \sD[G\sT]$ has a first $G$-space $X({\bf 1})$. Let $G\bf{Sp}_{cp}$ be the category of connective positive $\OM$-$G$-spectra.
An infinite loop $G$-space machine is a functor 
$$\bE\colon \sD[G\sT] \rtarr G\bf{Sp}_{cp}$$
 together with a natural map of $G$-spaces 
$$\nu\colon X({\bf 1}) \rtarr \bE X(S^0)$$
such that the composite of $\nu$ with the adjoint structure map
\[  \bE (S^0) \rtarr  \Omega^{V}\bE X(S^V)\]  
is a group completion for all $V$ such that  $V^G\neq 0$.
\end{defn}

Replacing $\sD$ and $\sF$ by $\sD_G$ and $\sF_G$ and replacing \gen-special by special (as defined in \autoref{PIGSpec}) gives the definition of an infinite loop $G$-space machine starting from an $E_{\infty}$ $G$-category of operators over $\sF_G$. Incorporating a characterization principle, we require that, on homotopy categories,  $\bE$ restricts to an equivalence from the full subcategory of grouplike special $\sD_G$-spaces to the category of connective positive $\OM$-$G$-spectra.  Either way, equivariant infinite loop space machines restrict to give underlying classical machines, and the group completion property is seen there.  

\begin{rem}\label{omega}   Machines appear most naturally as prescribed in \autoref{Gmachine}. However, if we set 
$\bE'X(V) = \OM \bE X(V\oplus \bR)$, then the adjoint structure maps $\bE X(V) \rtarr \bE'X(V)$ specify 
an equivalence from $\bE X$ to an $\OM$-$G$-spectrum whose zeroth space is a group completion of $X({\bf 1})$.
\end{rem}

There are implications common to any equivariant infinite loop space machine, and we refer the reader to \cite[\S2.3]{GM3} for a discussion.  For example, the group completion property directly implies that any machine commutes with products and with passage to fixed points.  

Our goal here is to give a self-contained development of the Segal and operadic machines, both generalized to categories of operators, and to prove that they are equivalent.  
Due to their very different constructions, these machines have different advantages and disadvantages, as already indicated in the nonequivariant case.  The Segal machine works simplicially and seems to be the machine of choice in proving a motivic recognition principle.  In the framework of $\infty$-categories, Elmanto, Hoyois, Khan, Sosnilo, and Yakerson \cite{motivic} have adapted it to construct a motivic infinite loop space machine.
The operadic machine  generalizes directly to give machines that manufacture intermediate types of $G$-spectra from intermediate types of input data, but we have not worked out a Segal type analogue.  Due to its more topological flavor, the operadic machine was used to produce genuine $G$-spectra from elementary categorical data in \cite{GM3}, where the machine was used to give categorical proofs of topological results, especially a strong categorical version of the Barratt-Priddy-Quillen theorem. 

\subsection*{Outline and commentary}
Aside from \autoref{prelim}, which gives preliminaries relevant to both machines, all sections have their own introductions.  We therefore limit the rest of the introduction to a brief overview.  

We have several ways to generalize the Segal machine equivariantly, and we have comparisons among them.  We develop the one closest to Segal's original version in \autoref{SegSec},  highlighting the role of its inductive simplicial definition in proving the group completion property.  This version of the machine has not previously been developed equivariantly, and we shall use it to prove the equivariant group completion property for our other versions.   Following Segal \cite{Seg} nonequivariantly, we compare that machine to a conceptual version that is defined by categorical prolongation of functors defined on $\sF$ first to functors defined on $\sF_G$ and from there to functors defined on the category $\sW_G$ of finite  based $G$-CW complexes.  Reproving a result of Shimakawa \cite{Shim2}, we show in \autoref{compFFG} that the input categories of $\sF$-$G$-spaces and of $\sF_G$-$G$-spaces are equivalent.  The comparison of the relevant ``specialness'' conditions and of equivalences sheds considerable light on the underlying homotopy theory and is revisited model categorically in \autoref{epilogue}.

The homotopical conditions needed to make the conceptual machine useful are seldom satisfied by the examples that arise in nature.  We recall the homotopically well-behaved Segal machine in \autoref{SegHom}.  It is based on the categorical two-sided bar construction.  This version  was first defined nonequivariantly by Woolfson \cite{Woolf}, before the relevant bar constructions had been formalized.  However, the obvious equivariant generalization of his definition starting from $\sF$-$G$-spaces fails to be well-behaved homotopically.  Again following Shimakawa \cite{Shim, Shim2, Shim1}, we instead focus on an equivariant generalization that starts from the category of $\sF_G$-$G$-spaces.  The definition of our preferred homotopical Segal machines starting from $\sF$ or from $\sF_G$ is given in \autoref{idstar}.  The statement that they give equivalent infinite loop space machines in the sense of \autoref{Gmachine} is made more precise and proven in \autoref{bigSegal}.  We summarize the conclusions in the following rough statement.  

\begin{thm}   There are equivalent Segal infinite loop $G$-space machines that take \gen-special $\sF$-$G$-spaces or special $\sF_G$-$G$-spaces to (genuine)  $\OM$-$G$-spectra.
\end{thm}

To explain ideas without undue technical clutter, we defer the longer proofs about the Segal machine from Sections \ref{SegSec} and \ref{SegHom} to \autoref{SEGALPf}.

To prove that the Segal and operadic machines are equivalent, we must of course first redevelop and generalize both so that they do in fact  accept the same input.  The generalizations follow the corresponding nonequivariant theory in May and Thomason \cite{MT}. We have $E_{\infty}$ $G$-categories of operators $\sD$ over $\sF$ and $E_{\infty}$ $G$-categories of operators $\sD_G$ over $\sF_G$, and  we have algebras over each.  Again, by \autoref{compGGG}, these categories of input data are equivalent. Still ignoring operads, we develop and compare Segal machines that produce genuine $G$-spectra from such algebras in \autoref{SegGen}.   For an $E_{\infty}$ $G$-category of operators $\sD$ over $\sF$ with prolonged $E_{\infty}$ $G$-category of operators $\sD_G$ over $\sF_G$, we show that the input categories of $\sF$, $\sF_G$, $\sD$, and $\sD_G$ $G$-spaces are 
all equivalent in Theorems \ref{Segalin2} and \ref{Segalin4}.  We show that the Segal machines starting from these four input categories are all equivalent in \autoref{Segalouttoo}.  Again, we give a rough statement of the conclusions.

\begin{thm}   For a category of operators $\sD$ over $\sF$ with associated category of operators $\sD_G$ over 
$\sF_G$, there are equivalent Segal infinite loop space machines that take \gen-special $\sD$-$G$-spaces or special $\sD_G$-$G$-spaces to (genuine)  $\OM$-$G$-spectra.
\end{thm}

We turn to the operadic machine in \autoref{OperadGen}.  We first recall the equivariant Steiner operads and the classical operadic machine whose input is $\sC_G$-algebras for an $E_{\infty}$ operad $\sC_G$.  We show how to interpret it as an infinite loop space machine in the sense of \autoref{Gmachine} in \autoref{classic} and \autoref{omega}.  This machine is defined in terms of a monadic $2$-sided bar construction,  starting from the monad $\bC_G$ associated to $\sC_G$ whose algebras are the $\sC_G$-algebras.  We show how to construct $G$-categories of operators $\sD$ over $\sF$ and $\sD_G$ over $\sF_G$ from $\sC_G$, and we show how to construct monads $\bD$ and $\bD_G$ from $\sD$ and $\sD_G$ whose algebras are the $\sD$-algebras and $\sD_G$-algebras.    That allows us to define infinite loop space machines on $\sD$-algebras and $\sD_G$-algebras  in \autoref{DGmachine}.  In marked contrast to the Segalic context, we show in \autoref{DDG} that the categories of $\bD$-algebras and of $\bD_G$-algebras are equivalent.  Thus there is no particular need to consider $\sD_G$ rather than $\sD$ when developing the operadic machine, although use of $\sD_G$ is more convenient for purposes of comparison. 

We give formal comparisons of $\sC_G$-algebras, $\sD$-algebras, and $\sD_G$-algebras in Propositions \ref{formalMT} and \ref{formalMT2} and homotopical comparisons in \autoref{homotopMT}.  These allow us to conclude in \autoref{Mayouttoo} that the machines defined on $\sC_G$-spaces and on special $\sD_G$-$G$-spaces are equivalent.  This is conceptually the same as in May and Thomason \cite{MT}, but the key equivariant proof is more intricate and is deferred to \autoref{SegOpPf}.   The following rough statement summarizes the conclusions.

\begin{thm}   For an $E_{\infty}$ $G$-operad $\sC_G$, there is an infinite loop $G$-space machine that takes $\sC_G$-algebras to (genuine) $\OM$-$G$-spectra.  If $\sD$ and $\sD_G$ are the categories of operators over $\sF$ and $\sF_G$ constructed from $\sC_G$, this machine extends to equivalent infinite loop $G$-space machines defined on  \gen-special $\sD$-$G$-spaces or on special $\sD_G$-$G$-spaces.
\end{thm}

The comparison of the Segal and operadic machines starting from the same input is given in \autoref{Equivalence}, where \autoref{MAIN0} is proven.  The comparison seems quite amazing to us.   Even nonequivariantly, it is far more precise than the comparison given in \cite{MT}.  The Steiner operads are variants of the little cubes and little discs operads that share the good properties
and lack the bad properties of each of those, as explained in \cite[\S3]{Rant1}.  Nonequivariantly, they have played an important role in infinite loop space theory ever since their introduction in 1979.  We see here that they mediate between the Segal and operadic  infinite loop space machines just as if they had been invented for that purpose.  That is wholly unexpected and truly uncanny.

A little more background may help explain why we find such a precise point-set level comparison so surprising.  Expanding on the property versus structure dichotomy,
in the Segal machine, higher homotopies are encoded in the specialness {\em property}  of the structure 
map $\de\colon X_n\rtarr X_1^n$ relating the $n$th space to the $n$th power of the first space of an $\sF$-$G$-space.  In the operadic machine, they are encoded in the {\em structure} given by the action maps $\sC_G(n)\times X^n \rtarr X$ of a
$\sC_G$-algebra, which is why the operadic machine is more useful for purposes of computation.  Actions by the $G$-category of  operators  $\sD_G$  associated to $\sC_G$ encode both  sources of higher homotopies in the 
same structure and yet, up to equivalence, carry no more information than either does alone.

The difference is perhaps illuminated by thinking about the commutativity operad $\sN$ with $\sN(j) = \ast$.  
The $G$-category $\sD(\sN)$ over $\sF$ is just $\sF$ itself.   An $\sN$-$G$-space $X = X_1$ is the same thing as an $\sF$-$G$-space with $X_n = X_1^n$.   Both just give $X$ the structure of a commutative monoid in $G\sT$, and that is far too restricted to give the domain of  an infinite loop space machine: the fixed point spaces of the infinite loop $G$-spaces resulting from such input are equivalent to products of Eilenberg-MacLane $G$-spaces.  From the point of view of the Segal machine, we are replacing $\sN$-$G$-spaces with homotopically well-behaved $\sF$-$G$-spaces as input.  From the point of view of the operadic machine, we are replacing $\sN$-$G$-spaces by $\sC_G$-spaces for any chosen $E_{\infty}$ operad $\sC_G$.

In \autoref{general}, we explain what is true for general topological groups $G$ and for compact Lie groups $G$. For general topological groups $G$, there are classical Segal  and operadic infinite loop space machines. They construct classical $\OM$-$G$-spectra from special $\sF$-$G$-spaces or from algebras over classical $E_{\infty}$ operads. When generalized to equivariant categories of operators, they are equivalent.

Restricting to compact Lie groups and considering genuine $\OM$-$G$-spectra, we reach the following conclusion.  See \autoref{WOWW} for a precise statement of the extended version of \autoref{MAIN0}.
\begin{thm}   Without exception, all results in this paper generalize from finite groups to compact Lie groups, provided that we restrict attention to representations whose isotropy groups have finite index in $G$.
\end{thm}
That is a severe restriction, but it is a very natural one here: those are the subgroups $H$ such that $G/H$ is a finite $G$-set.   With this restriction, all of our details when $G$ is finite make sense and apply equally well when $G$ is compact Lie.

As noted by the referee, model categorical interpretations were originally conspicuous by their absence. We have sketchily rectified that in  \autoref{epilogue}, which is an epilogue  by the senior author.  As we explain there, it is implicit or explicit in the literature, for example in \cite{MM, MMSS, Sant, Oster}, that  chains of Quillen equivalences connect all of the input categories of our machines and that the machines themselves construct connective Quillen equivalences, meaning that they are Quillen adjunctions that induce equivalences of homotopy categories between the full subcategories of grouplike objects on the input side and the full subcategory of connective $G$-spectra on the output side.  The references just cited proceed by analogy with the classical nonequivariant treatment of Bousfield and Friedlander \cite{BF}.  As there, the model category theory of the cited references makes no mention of the group completion property.  That property is the main point  of the theory and its applications, but  it is irrelevant model categorically.  

Even aside from the missing group completion property, the model theoretic discussion of the Segal and operadic machines will expose some limitations of thinking model categorically in the present context; see Remarks \ref{quibble1} and \ref{quibble2}.   As we will explain, it is folklore that, with appropriate mixed model structures, the bar constructions that we use to approximate the input data of both machines can be interpreted as highly structured model categorical cofibrant approximations.   That structure is central to our work, but the model theoretic interpretation is not.

Philosophically, in this paper we are interested in understanding and preserving  as much point-set level structure as possible, which we find illuminating and useful.  While of course more abstract frameworks have their advantages, such point-set level precision is precisely what they are designed to avoid.

\subsection*{Omissions}

We do not consider categorical input in this paper, but that feeds in effortlessly.  It is shown in \cite{GM3} that permutative $G$-categories give natural input to our operadic machine, and we showed how to strictify symmetric monoidal $G$-categories to equivalent permutative $G$-categories in \cite{GMMOAddCat1}.    

More deeply, we do not considerable {\em multiplicative} infinite loop space theory here.  That is needed to give ring and module $G$-spectra, rather than just the additive  $G$-spectra we obtain here.  Starting from $G$-space level input,  we solve that problem in \cite{GMMO}, again with Guillou. There we develop a variant of the Segal machine that starts from and is equivalent to the version given here,  but that gives a symmetric monoidal functor from the category of $\sF$-$G$-spaces to the category of $G$-spectra.  Therefore any multiplicative structure present on the level of $\sF$-$G$-spaces gives rise to corresponding structure on the level of $G$-spectra.   

However, that multiplicative machine has a serious defect: it does not accept categorical input crucial to the applications to equivariant algebraic $K$-theory that we have in mind.  For that, it seems essential to work $2$-categorically with lax functors, or at least pseudofunctors, rather than just with categories and  functors, That requires quite different categorical underpinnings than are discussed in this paper or in \cite{GM3, GMMO}.  It is treated in \cite{GMMO3}.  A $2$-monadic alternative treatment will be given in  \cite{MayMOM}. 

\subsection*{Acknowledgements}  This  work  was  partially  supported  by Simons  Foundation grant No. 359449, a Woodrow Wilson Career Enhancement Fellowship, and NSF grant DMS-1709302,  held by Osorno, and by an AMS Simons Travel grant, NSF grant DMS-1709461/1850644, NSF CAREER grant DMS-1943925, and NSF FRG grant DMS-2052988 held by Merling. An NSF RTG grant supported a month long visit of Osorno and a week long visit of Merling to Chicago. Part of the writing took place while Merling was in residence at the  Hausdorff Research Institute for Mathematics, and she is thankful for their hospitality. 

We would like to thank the following people for many insightful discussions that helped shape ideas during the long course of this project: Clark Barwick, Andrew Blumberg, Anna Marie Bohmann, Emanuele Dotto, Bertrand Guillou, Nick Gurski, Lars Hesselholt, Mike Hill, Niles Johnson, Cary Malkiewich, Kristian Moi, Irakli Patchkoria, Emily Riehl, Jonathan Rubin, and Stefan Schwede. 

We are especially grateful to the anonymous referee for very helpful comments and questions. 
Answering some of them led us to give a self contained treatment of some of the proofs in the literature and also led us to some new results. In particular, May thanks Stefan Waner for discussion of a partial generalization of the approximation theorem to compact Lie groups (see \autoref{opercom})  and he thanks Mike Shulman for help with thinking through the relationship between bar constructions and cofibrant approximations (see \autoref{Shul}).

\section{Preliminaries}\label{prelim}
  
\subsection{Preliminaries about $G$-spaces and Hopf $G$-spaces}\label{prelimspace}  We repeat that we assume that $G$ is finite until \autoref{general}.  We let $\sU$ be the category of compactly generated weak Hausdorf spaces, let $\sU_{\ast}$ be the category of  based spaces in $\sU$, and let $\sT$ be its subcategory of nondegenerately based spaces.  We let $G\sU$, $G\sU_{\ast}$ and $G\sT$ be the categories of $G$-spaces, based $G$-spaces, and nondegenerately  based $G$-spaces, with left action by $G$; basepoints of $G$-spaces are fixed by $G$.  Maps in these categories are $G$-maps.  While $G\sU_{\ast}$ is convenient for formal arguments,  $G\sT$ is required when questions of homotopy type are at issue, and we shall focus on that category. 
In our treatment of the Segal machine, we will also need to consider the full subcategory $G\sW \subset G\sT$ of based finite $G$-CW complexes.  

Properties of $G$-spaces are very often defined by passage to fixed point spaces. 
For example, a $G$-space $X$ is said to be $G$-connected if $X^H$ is (path) connected for all $H\subset G$.
 
\begin{defn} Let $f\colon K\rtarr L$ be a map of $G$-spaces. We say that $f$ is a {\em weak $G$-equivalence} if 
$f^H\colon K^H\rtarr L^H$ is a weak equivalence for all $H\subset G$.  A family of subgroups 
of $G$ is a set of subgroups closed under subconjugacy. For a family $\bF$ of subgroups of $G$, we say that
$f$ is a {\em weak $\bF$-equivalence} if $f$ is a weak $H$-equivalence for all $H\in \bF$.  
We often omit the word weak, taking it to be understood throughout.
\end{defn}

The following families are central to equivariant bundle theory 
and to the analysis of equivariant infinite loop space machines. 
They will be used ubiquitously. Let
$\SI_n$ denote the $n$th symmetric group.

\begin{defn}\label{famFn} For a subgroup $H$ of $G$ and a homomorphism $\al\colon H\rtarr \SI_n$, 
let $\LA_{\al}$ be the subgroup $\{(h,\al(h))\,|\, h\in H\}$ of $G\times \SI_n$; thus $\LA_{\al}$ is the graph of $\al$.  
Note that projection to the first coordinate gives an isomorphism  $\pi\colon \LA_{\al}\rtarr H$.
All subgroups $\LA$ 
of $G\times \SI_n$ such that $\LA\cap \SI_n =\{e\}$ are of this form. Let $\bF_n$ denote the family of 
all such subgroups. Taking $\al$ to be trivial, we see that $H\in\bF_n$ for all $n$ and all $H\subset G$.
\end{defn}  

\begin{defn}\label{Xal} Let $Y$ be a $(G\times \Sigma_n)$-space and $\al\colon G\rtarr \Sigma_n$ be a homomorphism. We define $Y^\al$ to be the $G$-space with underlying space $Y$ and with a new $G$-action $\cdot_\al$ given by
$g\cdot_\al y= (g,\al(g))\cdot y$.  Thus $Y^{\al}$ is $Y$ with $G$-action twisted by $\al$.  Said another way $Y^{\al}$ is just notation for $Y$ with its action by the subgroup $\LA_{\al}$ of $G\times \SI_n$, pulled back along the isomorphism  $\pi^{-1}\colon G\rtarr \LA_{\al}$.   Thus, on passage to fixed point spaces,
$$   (Y^{\al})^G \iso Y^{\LA_{\al}}.$$
\end{defn} 

Of course, we can replace  $G$ by $H\subset G$ in \autoref{Xal}.   We use this definition to fix conventions about finite $G$-sets.

\begin{conv}\label{finiteGset}
Let $\mathbf{n}$ denote the based set $\{0, 1, \dots, n\}$ with 
basepoint $0$. For a group $G$ and a homomorphism $\al\colon G\rtarr \SI_n$,  note that $\bn^{\al}$ is the based $G$-set specified by letting $G$ act on $\mathbf{n}$ by $g\cdot i = \al(g)(i)$  for $1\leq i\leq n$.  Conversely, a based $G$-action on $\mathbf{n}$ determines a $G$-homomorphism 
$\al$ by the same formula. Every based  finite $G$-set with $n$ non-basepoint elements is  isomorphic to one of the form $\bn^{\al}$ for some $\al$. We understand based finite $G$-sets to be of this form throughout.
\end{conv}

We highlight the following special case of \autoref{Xal}.  

\begin{exmp}\label{prodal}
Let $X$ be a based $G$-space and consider the based $(G\times \Sigma_n)$-space $X^n$.  Then we have the identification $(X^n)^\al=X^{\mathbf{n}^\al}$, where  $X^{\mathbf{n}^\al}$ denotes the set of based maps $\bn^{\al}\rtarr X$ with $G$ acting by conjugation.   Explicitly, this is  the cartesian power $X^n$ with twisted action by $G$ given by
\[  g(x_1,\dots,x_n) = (gx_{\alpha(g^{-1})(1)},\dots, gx_{\alpha(g^{-1})(n)}). \]
\end{exmp}

We also need some preliminaries about $H$-spaces, which we call Hopf spaces
to avoid confusion with subgroups of $G$. Recall that a Hopf space is a based space $X$ with a product such 
that the basepoint is a two-sided unit up to homotopy.   For simplicity, we assume once and for all that our
Hopf spaces are homotopy associative and homotopy commutative\footnote{It would suffice to assume that 
left and right translation by any element are homotopic.}, since that holds in our examples.
We say that a Hopf space is grouplike if, in addition, $\pi_0(X)$ is a group, necessarily abelian.

\begin{defn}\label{gpcomp} A Hopf map $f\colon X\rtarr Y$ is a {\em group completion} if
$Y$ is grouplike, $f_{*}\colon \pi_0(X) \rtarr \pi_0(Y)$ is the Grothendieck group of the commutative
monoid $\pi_0(X)$, and for every field of coefficients,  $f_{*}\colon H_{*}(X)\rtarr H_{*}(Y)$
is the algebraic localization obtained by inverting the elements of the submonoid $\pi_0(X)$ 
of $H_{*}(X)$.\footnote{Segal \cite[\S4]{Seg} describes the notion of group completion a bit differently, in a form less amenable 
to equivariant generalization, and he makes several reasonable restrictive hypotheses in 
his proof of the group completion property. In particular, he assumes that $X$ is a
topological monoid and that $\pi_0(X)$ contains a cofinal free abelian monoid.}
\end{defn}

\begin{rem}\label{CCMT} Universal properties of group completions are studied in \cite[\S 1]{CCMT}.  The discussion there implies 
that if $Y$ and $Y'$ are group completions of $X$, then  $Y$ and $Y'$ are weakly homotopy equivalent.
\end{rem}

\begin{defn}\label{Defngpcomp} A Hopf $G$-space is a based $G$-space $X$ with a product $G$-map such that its basepoint $e$ is a 
two-sided unit element, in the sense that left or right multiplication by $e$ is a weak $G$-equivalence $X\rtarr X$.  Then each $X^H$ is
a Hopf space, and we assume as before that each $X^H$ is homotopy associative and commutative. A Hopf $G$-space $X$ is grouplike 
if each $X^H$ is grouplike.  A Hopf $G$-map $f\colon X\rtarr Y$ is a group completion if $Y$ is grouplike and the fixed point maps 
$f^H$ are all nonequivariant group completions. Clearly a group completion of a $G$-connected Hopf $G$-space is a weak $G$-equivalence.
\end{defn}

\subsection{Preliminaries about $G$-cofibrations and simplicial $G$-spaces}\label{simpre}

Since $G$-cofibrations play an important role in our work, we insert some standard remarks about them.\footnote{They are part of the 
$h$-model structure on $G\sU$, as in \cite{morecon} nonequivariantly.} 

\begin{rem} A map is a $G$-cofibration if it satisfies the $G$-homotopy extension property and a basepoint $\ast\in X$ is nondegenerate
if the inclusion $\ast\rtarr X$ is a $G$-cofibration. Since we are working in $G\sU$, a 
$G$-cofibration is an inclusion with closed image \cite[Problem 5.1]{Concise}.  By \cite[Proposition A.2.1]{BV} (or \cite[p. 43]{Concise}),  
if $i\colon A\rtarr B$ is a closed inclusion
of $G$-spaces,  then $i$ is a $G$-cofibration if and only if $(B,A)$ is a $G$-NDR pair.   Using this criterion, we see that
$i$ is then also an $H$-cofibration for any  subgroup $H$ of $G$ and that passage to orbits or to fixed points over $H$ 
gives a cofibration.  Moreover, just as nonequivariantly, a pushout of a map of $G$-spaces along a $G$-cofibration is a $G$-cofibration. 
\end{rem}

\begin{defn}\label{latch} Let $X_{*}$ be a simplicial $G$-space with $G$-space of $n$-simplices $X_n$.  The $n$th latching space of $X$ is 
given by
\[L_nX = \bigcup _{i=0}^{n-1} s_i(X_{n-1}).\]
It is a $G$-space, and the  inclusion $L_nX\rtarr X_n$ is a $G$-map. We say that $X_{*}$ is Reedy cofibrant if this map 
is a $G$-cofibration for each $n$.
\end{defn}

With different nomenclature, this concept was studied nonequivariantly in the early 1970's 
(e.g. \cite[\S11]{MayGeo}, \cite[Appendix]{MayPerm}, \cite[Appendix A]{Seg}). We will use the following standard results. 

\begin{lem}\label{PlenzReedy}  
A simplicial $G$-space $X_{*}$ is Reedy cofibrant if all degeneracy operators $s_i$ are $G$-cofibrations.
\end{lem}
\begin{proof}  The nonequivariant statement is proven by an inductive application of Lillig's union theorem stating that the union of cofibrations is a cofibration \cite{Lillig}
(or \cite[Lemma A.6]{MayGeo}).  The proof can be found in \cite[proof of A.5]{Seg} or \cite[proof of 2.4.(b)]{LewisOMSIX}. The equivariant proof is the same, using the equivariant version of Lillig's theorem, which is a particular case of \cite[Theorem A.2.7]{BVbook}. 
\end{proof}

The converse to \autoref{PlenzReedy} is proved in \cite[Proposition 4.11]{Plenz}, but we shall not use it.   

\begin{thm}\label{PlenzWeak}
Let $f_{*}\colon X_{*}\rtarr Y_{*}$ be a map of Reedy cofibrant simplicial $G$-spaces such that each $f_n$ is a weak $G$-equivalence. 
Then the realization $ |f_{*}|\colon |X_{*}|\rtarr  |Y_{*}|$ is a weak $G$-equivalence.
\end{thm}
\begin{proof} Nonequivariantly this is in \cite[Theorem A.4]{MayPerm}.   Since passage to fixed points is a finite limit, it commutes with geometric realization. Therefore  the equivariant version follows by application of the  nonequivariant version to fixed point spaces.
\end{proof}

The following companion result is well-known, but since we could not find a proof in the published literature, we provide one in \autoref{COF}.\footnote{We were inspired by \cite{Plenz}, an unpublished masters thesis, which gives a detailed exposition of simplicial spaces. However, its statement of \autoref{PlenzCof} is missing a necessary hypothesis.}

\begin{thm}\label{PlenzCof} 
Let $f_{*}\colon X_{*}\rtarr Y_{*}$ be a map of Reedy cofibrant simplicial $G$-spaces such that each $f_n$ is a $G$-cofibration. 
Then the realization $|f_{*}|\colon |X_{*}|\rtarr |Y_{*}|$ is a $G$-cofibration.
\end{thm}

\begin{rem}\label{blanket}  
Without exception, every simplicial $G$-space used in this paper is Reedy cofibrant.  In each case, we can check from the definitions and 
the fact that we are working with nondegenerately based $G$-spaces that all $s_i$ are $G$-cofibrations.  For the examples appearing in the
Segal machine, the verifications are straightforward.  For the examples appearing in the operadic machine, the verifications follow those in
\cite[Proposition A.10]{MayGeo} and elaborations of the arguments there.
\end{rem}

\subsection{Categorical preliminaries and basepoints}\label{catpre}
Some familiarity with enriched category theory, especially in equivariant contexts, may be helpful.   A more thorough
treatment of the double enrichment present here is given in \cite{GMR}.  We first recall some general definitions that start with a closed symmetric monoidal category $\sV$ with unit object $U$ and product denoted by $\otimes$. 
Closed means that we have internal function objects $\ul{\sV}(V,W)$ in $\sV$ giving an adjunction
$$ \ul{\sV} (X\otimes Y, Z) \iso \ul{\sV}(X,\ul{\sV}(Y,Z))$$
in $\sV$.  We write $\ul{\sV}$ instead of $\sV$ when we wish to emphasize the closed structure.

A $\sV$-category $\sE$ is a category enriched in $\sV$.
This means that for each pair $(m,n)$ of objects of $\sE$ there is an object $\sE(m,n)$ of $\sV$ and there are unit and composition maps
$I\colon U\rtarr \sE(m,m)$ and $C\colon \sE(n,p)\otimes \sE(m,n)\rtarr \sE(m,p)$ satisfying the identity and associativity axioms.  
A $\sV$-functor 
$F\colon \sE \rtarr \sE'$  between $\sV$-categories is a functor enriched in $\sV$.  This means that for each pair $(m,n)$, there is a map 
$$F\colon \sE(m,n) \rtarr \sE' (F(m), F(n)) $$
in $\sV$, and these maps are compatible with the unit and composition of $\sE$ and $\sE'$.   A $\sV$-transformation $\et\colon F \rtarr F'$ 
is given by maps $\et_m\colon U \rtarr \sE'(F(m),F'(m))$ in $\sV$ such that the evident naturality diagram commutes in $\sV$.
\[ \xymatrix{ \sE(m,n) \ar[r]^-F \ar[d]_{F'}  & \sE'(F(m), F(n) )  \ar[d]^{(\et_n)_{*} }\\
\sE'(F'(m),F'(n) ) \ar[r]_{(\et_m)^*} & \sE'(F(m),F'(n)) \\}  \]

As is standard, we use the same name for morphism spaces (or $G$-spaces) as for the underlying  morphism sets (or $G$-sets).   In view of the double enrichment, in spaces and in $G$-spaces, that is present equivariantly (and is discussed in general in \cite{GMR}), we carefully establish notations for the ambient categories in which we work.

\begin{notns}\label{ambient}
Nonequivariantly, $\sU(X,Y)$ denotes the space of maps $X\rtarr Y$.  Equivariantly, for $G$-spaces $X$ and $Y$,  $G\sU(X,Y)$ denotes the space of  $G$-maps $X\rtarr Y$ of the category $G\sU$. In conformity with the categorical notation  $\ul{\sV}$, we let $\ul{G\sU}(X,Y)$ denote the $G$-space of maps $X\rtarr Y$, with $G$ acting by conjugation: for a map $f\colon X\rtarr Y$, $(gf)(x)=g\cdot f(g^{-1}\cdot x)$.  For based spaces or based $G$-spaces, we have the analogous based spaces $\sU_{\ast}(X,Y)$ and $G\sU_{\ast}(X,Y)$ and the analogous based $G$-space 
$\ul{G\sU_{\ast}}(X,Y)$; here the basepoints are the trivial based maps.

We are especially interested in the categories $\sT$ and $G\sT$, but they are not closed since their hom spaces need not be nondegenerately based in general.   Therefore, for nondegenerately based spaces $X$ and $Y$, we agree to write $\sT(X,Y) = \sU_{\ast}(X,Y)$ for the based space of maps $X\rtarr Y$.  That expresses the enrichment of $\sT$ in $\sU_{\ast}$, but not necessarily in $\sT$.   For nondegenerately based $G$-spaces $X$ and $Y$, we write $G\sT(X,Y) = G\sU_{\ast}(X,Y)$ for the based space of based $G$-maps $X\rtarr Y$, expressing the enrichment of $G\sT$ in $G\sU_{\ast}$.  We write $\sT_G(X,Y) = \ul{G\sU_{\ast}}(X,Y)$ for the $G$-space of based maps  $X\rtarr Y$, with $G$-acting by conjugation, expressing the enrichment of $G\sT$ in $G\sU_*$.\footnote{We avoid the notation ${\ul{G\sT}}$ since we reserve the underline notation for closed categories.}   The notation is meant to emphasize the implied assumption that the object spaces $X$ and $Y$ are nondegenerately based.  Thus $\sT_G(X,Y)$ is a based $G$-space with fixed point space $G\sT(X,Y)$.  Viewing  $\sT_G$ as a $G$-category such that $G$ acts trivially on objects, we can view $G\sT$ as the fixed point category of  $\sT_G$.

We have the full subcategory $\sW_G\subset \sT_G$ of based finite $G$-CW complexes. It is enriched in $\ul{G\sU_{\ast}}$, and we can view $G\sW$ as the fixed point category of $\sW_G$.
\end{notns}

Of course, the category $\sU_\ast$ has two symmetric monoidal products, $\sma$ with unit $S^0$ and $\times$ with unit $\ast$.  When enriching in $\sU_{\ast}$, we must use $\sma$ since we must use the closed structure given by the spaces of based maps.  However, we occasionally use the forgetful functors $\sU_{\ast}\rtarr \sU$ and $G\sU_{\ast} \rtarr G\sU$ to forget basepoints in our enrichments, and implicitly we are then thinking about $\times$. 

\begin{rem}\label{UTbit}  For a $G\sU$-category $\sE$ and a $G\sU_\ast$-category $\sE'$, we can add disjoint basepoints to the hom objects of $\sE$ to form a $G\sU_{\ast}$-category $\sE_+$, and we can forget basepoints to regard $\sE'$ as a $G\sU$-category $\bU \sE'$.  Via the adjunction between $(-)_+$ and $\bU$, $G\sU$-functors $\sE\rtarr \bU\sE'$ can be identified with $G\sU_{\ast}$-functors $\sE_+ \rtarr \sE'$.
\end{rem} 

In all variants and  generalizations of the Segal machine, we start with a category $\sE$ a priori enriched over $G\sU$ (although sometimes with $G$ acting trivially on morphisms in $\sE$) such that  $\sE$ has a zero object $0$,  so that there are unique maps $0 \to n$ and $n \to 0$ for every object $n\in \sE$.  Then assigning the map  $m \to 0 \to n$ as the basepoint of $\sE(m,n)$ makes $\sE$ into a category enriched over $G\sU_\ast$, as composition in $\sE$ descends to a map out of the smash product.
We are concerned with $G\sU_{\ast}$-functors defined on $\sE$. The following trivial observation has been overlooked since the start of this subject.  It depends only on basepoint information, not on the topology or the $G$-action.

\begin{lem}\label{duh}  Let $\sE$ be a $G\sU$-category with a zero object, considered as a category enriched over $G\sU_\ast$ as above. Then, 
\begin{enumerate}[(i)]
\item for any $G\sU_{\ast}$-functor $X\colon \sE \rtarr \UG$, $X(0) = \ast$;
\item conversely, any $G\sU$-functor $Y\colon \bU\sE \rtarr \ul{G\sU}$ such that $Y(0)=\ast$ can be considered as a $G\sU_\ast$-functor $\sE \rtarr \UG$.
\end{enumerate}
\end{lem}
\begin{proof} The unique map $0\to 0$ in $\sE(0,0)$ is both the basepoint and the identity; $X\colon \sE(0,0) \rtarr \sU_\ast(X(0),X(0))$ must send it to 
both the trivial map $X(0) \rtarr X(0)$ that sends all points to the basepoint and to the identity map. This can only happen if $X(0)$ is a point, proving (i).

For (ii), if $Y\colon \bU\sE \rtarr \ul{G\sU}$ is such that $Y(0)=\ast$,  the map $Y(0)\rtarr Y(n)$ induced by $0\rtarr n$ gives each $Y(n)$ a basepoint, and these are preserved by the maps induced by maps in $\sE$. Moreover, $Y$ gives a $G\sU_{\ast}$-functor $\sE\rtarr \UG$ since $Y$ then sends the zero map $m\rtarr 0 \rtarr n$  to the trivial map 
$$Y(m)\rtarr \ast=X(0) \rtarr X(n). \qedhere$$
\end{proof}

\begin{rem}\label{UTU}
Let $\sE$ be a $G\sU_{\ast}$-category with a zero object. A $G\sU$-functor $Y\colon \bU \sE \rtarr  \ul{G\sU}$ is said to be \emph{reduced} if $Y(0)$ is a point.  Thus the previous lemma can be interpreted as saying that there is a one-to-one correspondence between reduced $G\sU$-functors $\bU\sE \rtarr \ul{G\sU}$ and $G\sU_\ast$-functors $\sE \rtarr \UG$. If moreover the maps $ \ast=Y(0)\rtarr Y(n)$ are $G$-cofibrations,
$Y$ will restrict to give a $G\sU_\ast$-functor $\sE\rtarr \TG$.
 \end{rem}

Let $\sE$ be a $G\sU_{\ast}$-category with a zero object $0$, and let $X$ and $Y$ be covariant and contravariant $G\sU_{\ast}$-functors $\sE \rtarr \UG$, respectively.
Then the tensor product of functors $Y\otimes _\sE X$ is defined as the coequalizer of the diagram
\[\xymatrix{
\displaystyle{\bigvee _{m,n} Y_n \sma \sE(m,n) \sma X_m}  \ar@<.7ex>[r]  \ar@<-.7ex>[r] & \displaystyle{\bigvee_n Y_n \sma X_n,}
} \] 
where the arrows are given by the action of $\sE$ on $X$ and on $Y$, respectively.

We obtain $G\sU$-functors $\bU X$, $\bU Y\colon \bU\sE \rtarr \ul{G\sU}$ by forgetting basepoints. 
The tensor product of functors $\bU Y\otimes_{\bU \sE}\bU X$ is the coequalizer of the diagram
\[\xymatrix{
\displaystyle{\coprod _{m,n} Y_n \times \sE(m,n) \times X_m}  \ar@<.7ex>[r]  \ar@<-.7ex>[r] & \displaystyle{\coprod_n Y_n \times X_n.}
} \]

The following result shows that the difference between these two constructions is only apparent.   Later on, we shall 
sometimes use wedges and smash products and sometimes instead use disjoint unions and products, whichever seems convenient.

\begin{lem}\label{OK} With $\sE$, $X$, and $Y$ as above, there is a natural isomorphism
\[\bU(Y \otimes_\sE X) \iso \bU Y \otimes_{\bU \sE} \bU X.\]

\end{lem}
\begin{proof}
We show that the quotient of $\coprod _n Y_n \times X_n$ given by the coequalizer encodes the required basepoint identifications. Let $(y,\ast_n)\in Y_n\times X_n$. As noted above, the basepoint  in $X_n$ is given by $\ast_n=(0_{0,n})_{*}(\ast_0)$, where $\ast_0$ is the unique point in $X_0$ and $0_{0,n}$ is the unique map $0 \rtarr n$ in $\sE$. Then 
\[(y,\ast_n) =(y,(0_{0,n})_{*}(\ast_0))\sim (0_{0,n}^*(y),\ast_0) = (\ast_0,\ast_0),\]
the last equation following from the fact that $Y_0$ is a singleton. A similar argument shows that $(\ast_n,x)\sim (\ast_0,\ast_0)$.
\end{proof}
 
\subsection{Spectrum level preliminaries}\label{specpre}

\begin{defn}  A classical $G$-spectrum,\footnote{We remind the reader that ``classical'' is often called ``naive''.} which we prefer to call a $G$-prespectrum, is a sequence of based $G$-spaces $\{T_n\}_{n\geq 0}$ and based $G$-maps $\si_n\colon \SI T_n\rtarr T_{n+1}$. It is a classical $\OM$-$G$-spectrum if the adjoint maps $\tilde{\si}_n\colon T_n \rtarr \OM T_{n+1}$ 
are weak $G$-equivalences.  It is a positive $\OM$-$G$-spectrum if the $\tilde{\si}_n$ are weak $G$-equivalences 
for $n\geq 1$. We let $G\sP$ denote the category of $G$-prespectra.  We call zeroth spaces $T_0$ of classical $\OM$-$G$-spectra
classical infinite loop $G$-spaces.
\end{defn} 
In this paper, our preferred category of (genuine) $G$-spectra is the
category $G\sS$ of orthogonal $G$-spectra. Orthogonal $G$-spectra and their 
model structures are studied in \cite{MM}, to which we refer the reader for details and discussion of 
the following definition.  We use \autoref{UTbit}.

\begin{defn}\label{OrthSpec}  Let $\sI_G$ be the $G\sU$-category
of finite dimensional
real $G$-inner product spaces and linear isometric isomorphisms, with $G$ acting on morphism spaces by
conjugation.  Note that $\sI_G$ is symmetric monoidal under $\oplus$.  An $\sI_G$-$G$-space is a $G\sU$-functor
${\sI_G} \rtarr  \UG$ \footnote{To be consistent with the notation of \autoref{UTbit} we should write $\bU(\UG)$ for the target here, but we have not done so to avoid cluttering the notation.} or, equivalently by \autoref{UTbit}, a $G\sU_{\ast}$-functor  ${\sI_G}_+\rtarr  \ul{G\sU_{\ast}}$.
The sphere $\sI_G$-$G$-space $S$ is given by $S(V) = S^V$.  The external smash
product 
$$X\barwedge Y\colon \sI_G\times \sI_G\rtarr \sT_G$$ 
of $\sI_G$-$G$-spaces $X$ and $Y$ is 
the $G\sU$-functor given by 
$$(X\barwedge Y)(V,W) = X(V)\sma Y(W).$$
A (genuine orthogonal) $G$-spectrum is an $\sI_G$-$G$-space $E\colon \sI_{G} \rtarr \sT_G$ together 
with a  $G\sU$-transformation $E\barwedge S\rtarr E\com \oplus$ between $G\sU$-functors $\sI_G\times \sI_G \rtarr \sT_G$.  Thus we have $G$-spaces $E(V)$,
morphism $G$-maps 
$$\sI_{G}(V,V') \rtarr \sU_G(E(V), E(V'))$$  
and structure $G$-maps
$$\si\colon E(V) \sma S^W \rtarr E(V\oplus W).$$ 
natural in $V$ and $W$. Note in particular that $\sI_G(V,V)$ is the orthogonal group $O(V)$, with $G$
acting by conjugation, so that $E(V)$ is both a $G$-space and an $O(V)$-space and $\si$ is a map of both $G$-spaces and
$O(V)\times O(W)$-spaces.  A $G$-spectrum $E$ is an $\OM$-$G$-spectrum if the adjoint
maps 
$$\tilde{\si}\colon E(V)\rtarr \OM^W E(V\oplus W)$$ 
are weak $G$-equivalences.  It is a positive $\OM$-$G$-spectrum
if these maps are weak $G$-equivalences when $V^G\neq 0$.  We let $G\sS$ denote the category of $G$-spectra.
We call zeroth spaces $E(0)$ of $\OM$-$G$-spectra genuine infinite loop $G$-spaces, or simply infinite loop $G$-spaces.
\end{defn}

\begin{defn} The forgetful functor $i^*\colon G\sS\rtarr G\sP$ sends a $G$-spectrum $X$ to the (classical) $G$-prespectrum
with $n$th space $X_n = X(\bR^n)$.
\end{defn}

\begin{rem}\label{Vgpcomp}  If $V^G\neq 0$, we can write $V \cong \bR\oplus W$ and thus $S^V \cong S^1\sma S^W$ and
$\OM^V \cong \OM\OM^W$.  Then 
$\tilde{\si}\colon X_0\rtarr \OM^V X(V)$ factors as the composite
\[ \xymatrix@1{ X_0 \ar[r]^-{\tilde{\si}} &  \OM X_1 \ar[r]^-{\OM\tilde{\si}} & \OM \OM^W X(\bR\oplus W) \cong \OM^V X(V).\\} \]
If $X$ is a positive $\OM$-$G$-spectrum, then the second arrow is a weak $G$-equivalence.  Therefore, if $X_0\rtarr \OM X_1$  is a group completion, then so is $X_0 \rtarr \OM^V X(V)$ for all $V$ such that  $V^G\neq 0$. 
\end{rem}

\begin{rem} We take for granted the basic homotopy theory of orthogonal $G$-spectra.  In particular, the homotopy groups $\pi_q^H$ of $G$-spectra are defined in \cite[Definition 3.2]{MM}, and a map of $G$-spectra is a stable equivalence if it induces an isomorphism of homotopy groups. 
\end{rem}

Geometric realization of simplicial objects can be carried out in any bicomplete category that is tensored over spaces, using the usual coend definition, and we shall make use of the geometric realization of simplicial orthogonal  $G$-spectra.  See \cite[Definition X.1.1]{EKMM} for a discussion in the case of EKMM spectra.  In the case of orthogonal $G$-spectra, since colimits and tensors with spaces are constructed levelwise, we see that for a simplicial orthogonal $G$-spectrum $T_{*}$, the $V$-th level of the geometric realization is given by
\[|T_{*} |(V)=|T_{*}(V)|.\] 

We shall need the spectrum level analogue of \autoref{PlenzWeak}. That requires a notion of Reedy cofibrancy for simplicial orthogonal $G$-spectra, which we now explain.  We can define latching objects $L_n T$ for simplicial objects $T$ in any cocomplete category as the evident colimit (see, for example, \cite[Definition 15.2.5]{Hirsc} or \cite{reedypractice}); \autoref{latch} is a specialization. Since colimits of orthogonal spectra are computed levelwise, we have that $(L_n T)(V) = L_n (T(V))$ for a simplicial orthogonal $G$-spectrum $T$.
 
\begin{defn}\label{ReedySpec}  
A map of orthogonal $G$-spectra $A \rtarr X$ is an \emph{$h$-cofibration} if it satisfies the homotopy extension property (see \cite[\S A.5]{HHR}, \cite[\S I.4]{MM}, \cite[\S 5]{MMSS}).   A simplicial orthogonal $G$-spectrum is \emph{Reedy $h$-cofibrant} if the latching map $L_n T \rtarr T_n$ is an $h$-cofibration for each $n$.
\end{defn}

\begin{prop}\label{reedyforspectra}
Let $f_{*} \colon T_{*} \rtarr T'_{*}$ be a map of Reedy $h$-cofibrant simplicial orthogonal $G$-spectra that is a stable equivalence at each simplicial level. Then the map of orthogonal $G$-spectra $|f_{*}|$ obtained by geometric realization is a stable equivalence.
\end{prop}
\begin{proof}
The proof is the same as for the space level analogue, using the construction of the filtration on geometric realization via pushouts (see \cite[Theorem 4.15]{Plenz}  and \cite[Theorem X.2.4]{EKMM}). The  key facts we need about $h$-cofibrations of orthogonal $G$-spectra are that they are stable under pushouts  \cite[Proposition A.62]{HHR},  that they satisfy the analogue of \cite[Lemma X.2.3]{EKMM} so that the relevant realizations are constructed by pushouts along $h$-cofibrations, and that the gluing lemma for $h$-cofibrations and stable equivalences holds (see \cite[Corollary B.21]{HHR}, \cite[Theorem I.4.10 (iv)]{MM}), and that the filtered colimit along $h$-cofibrations of a sequence of stable equivalences is a stable equivalence \cite[Proposition B.17]{HHR}. 
\end{proof}

We note that the above result holds generally in any good model category tensored over spaces. A recent treatment is offered in \cite[see Corollary 10.6.]{reedypractice}.  We have not quoted that result because there is no published proof that there is a (mixed) model structure on the category of orthogonal $G$-spectra in which the cofibrations are the $h$-cofibrations.  However, the methods of \cite{BR} (and \cite[\S6.4]{BMR}) can be applied to construct one.

\section{The simplicial and conceptual versions of the Segal machine}\label{SegSec} 

There are several variants of the Segal infinite loop space machine, as
originally developed by Segal \cite{Seg} and Woolfson \cite{Woolf}.
Later sources include Bousfield and Friedlander \cite{BF}, working simplicially,
and, much later, Mandell, May, Schwede, and Shipley \cite{MMSS}.  Equivariant versions appear in
Shimada and Shimakawa \cite{ShSh, Shim, Shim2} and, later, Blumberg \cite{Blum}.  

We here give a simplicial variant\footnote{We are referring to simplicial spaces, not simplicial sets, here.} and two equivalent conceptual variants,  one starting from finite sets and the other starting from finite $G$-sets.  We
defer consideration of our preferred homotopical variant to the next section.

The simplicial variant is the equivariant version of Segal's original definition \cite{Seg}.
As far as we know, his paper is the only source in the literature that actually proves the 
crucial group completion property, and his proof makes essential use of his original 
simplicial definition.\footnote{His proof imposes some unnecessary restrictive hypotheses that generally hold in practice.}
This version does not directly generalize to give genuine $G$-spectra, and it does not appear in the equivariant literature.  
Therefore, even at this late date, there is no published account of the equivariant Segal machine that 
proves the group completion property. Just as nonequivariantly, this property is central to the 
applications, especially to algebraic $K$-theory.

In fact, we do not know a direct proof of the group completion property starting from the conceptual or
homotopical variants treated in \cite{BF, MMSS, ShSh, Woolf} and, equivariantly, \cite{Blum, Shim}. 
Rather, we derive it for the conceptual variants from their equivalence with the simplicial variant.  
To give the group completion property equivariantly, to fill in other gaps in the pubished proofs, and to prepare for a comparison with the operadic machine, we give a fully detailed exposition of the Segal machine in all of its forms. This may also be helpful to the modern reader since, even nonequivariantly, the original sources make for hard reading and are sketchy. 

\subsection{Definitions: the input of the Segal machine}\label{InSegal} 

We begin with some standard nonequivariant definitions and observations.

\begin{defn}
Let $\sF$ be the opposite of Segal's category $\GA$.\footnote{As in \cite{MMSS} and elsewhere, we use the notation $\sF$ to avoid confusion between $\GA$ 
and $\GA^{op} = \sF$.}    It is the category of finite based
sets $\mathbf{n} = \{0,1, \dots,n\}$ with $0$ as basepoint. The morphisms are the based maps, and the unique morphism
that factors through $\mathbf{0}$ is a basepoint for $\sF(\mathbf m,\mathbf n)$. 
Let $\PI\subset \sF$ be the subcategory with the same objects and those morphisms $\ph\colon \mathbf{m}\rtarr \mathbf{n}$ 
such that $\ph^{-1}(j)$ has at most one element for $1\leq j\leq n$; these are composites of projections, 
injections, and permutations.  Let $\SI\subset \PI$ be the subgroupoid with the same objects and the elements
of the symmetric groups $\SI_n$, regarded as based isomorphisms $\mathbf{n}\rtarr \mathbf{n}$, as morphisms. 
\end{defn}

The composition in $\sF$ factors through the smash product, and we view $\sF$ as a category enriched in $\sT\subset \sU_*$, with the discrete topology on the based hom sets. We first review the nonequivariant definitions, following \cite{MT} for technical details that are needed for homotopical control.

\begin{notn}\label{Sigph} Given an injection $\phi \colon \mb{m} \rtarr \mb{n}$ in $\PI$, let $\SI_\phi$ denote the subgroup of $\SI_n$ consisting of those permutations  $\tau$ such that $\tau(\im \phi)=\im \phi$. Then $\SI_{\ph}$ acts on the set $\mathbf{m}$ and $\ph$ is a $\SI_{\ph}$-map.  
\end{notn}

\begin{defn}  A {\em $\PI$-space} is a $\sU_{\ast}$-functor $X\colon \sF\rtarr \sU_*$, written $\mathbf{n} \mapsto X_n$, such that the map $\ph_*\colon  X_m \rtarr X_n$ is a
$\SI_{\ph}$-cofibration for every injection $\ph\colon \bm\rtarr \bn$.  Taking $m=0$, it follows that $X$ takes values in $\sT$.  Similarly, an {\em $\sF$-space} is a $\sU_*$-functor whose restriction to $\PI$ is a $\PI$-space.   A map of $\PI$-spaces or of $\sF$-spaces is a $\sU_\ast$-natural transformation. We denote the categories of $\PI$-spaces and of $\sF$-spaces by $\PI[\sT]$ and $\sF[\sT]$.
\end{defn}

\begin{rem}  The definition implicitly  requires care to check that constructions on $\sF$-spaces that make obvious sense for functors $\sF\rtarr \sU_*$  do in fact take values in $\sF$-spaces.  The equivariant situation will be similar. 
\end{rem}

\begin{rem}\label{onlyordered}We note that it is enough to impose that $\phi_\ast$ is a $\SI_\ph$-cofibration only for ordered injections $\ph$, where $\phi$ is ordered if $\phi(i)<\ph(i')$ whenever $i<i'$. Indeed, if $\phi$ is an arbitrary injection, it can be written as the composite $\psi\circ \rho$, where $\psi$ is an ordered injection and $\rho \in \SI_m$. One can check that $\ph_\ast$ is a $\SI_\ph$-cofibration if and only $\psi_\ast$ is a $\SI_\psi$-cofibration by noting that $\SI_\ph=\SI_\psi$ and, while the actions on $X_m$ are not the same, $\rho$ gives a homeomorphism that intermediates between the two.
\end{rem}

\begin{defn}\label{Rdefn}   For a $\PI$-space  $X$, let $\bL X = X_1$.  For a based space $Y\in \sT$, let $\bR Y$ denote the $\PI$-space with $n$th space $Y^n$.  The $\PI$-space structure is given by basepoint inclusions, projections, and permutations. \end{defn}

\begin{defn} The {\em Segal maps}
$\de_i\colon \mathbf{n}\rtarr \mathbf{1}$ in $\PI$ send $i$ to $1$ and $j$ to $0$ for $j\neq i$.  For a $\PI$-space $X$, the {\em Segal map} $\de\colon X_n\rtarr X_1^n$ has coordinates induced by the $\de_i$.  If $n=0$, we interpret $\de$ as the terminal map $X_0\rtarr \ast$.   

The $n$th  {\em multiplication map} $\varphi_n\colon \mathbf{n}\rtarr \mathbf{1}$ in $\sF$ sends $j$ to $1$ for $1\leq j\leq n$.
It induces an ``$n$-fold multiplication'' $X_n\rtarr X_1$ on an $\sF$-space $X$. 
\end{defn}

\begin{lem}\label{Rlem} The pair  $(\bL,\bR)$ is an adjunction.  On a based space $Y\in \sT$, the counit $\epz\colon \bL \bR Y\rtarr Y$ is the identity map.  On a $\PI$-space $X$, the unit $\de\colon X \rtarr \bR\bL X$ is given levelwise by the Segal maps.
\end{lem}

\begin{rem} It is traditional, starting in \cite[Definition 1.2]{Seg}, to define a $\PI$-space to be a \em{not necessarily enriched}  functor $\PI \rtarr \sT$, requiring  $X_0$ to be contractible, and to say that $X$ is reduced if $X_0$ is a point.  That led to mistakes and confusion, as explained in \cite{Rant2}.   The essential point is that, without enrichment,   $\bL$  is not left adjoint to $\bR$.  As we observed in \autoref{duh}, our requirement that $X$ be a $\sU_{\ast}$-functor forces $X_0$ to be a point  for trivial reasons. 
\end{rem} 

Now return to the equivariant context.  

\begin{defn}\label{Fspace}  A {\em $\PI$-$G$-space}  $X$ is a $\sU_{\ast}$-functor $X\colon \PI\rtarr G\sU_\ast$ such that the map $\ph_*\colon X_m \rtarr X_n$ is a $(G\times \SI_\ph)$-cofibration for every injection $\ph\colon \bm\rtarr \bn$ in $\PI$.  It follows that $X$ takes values in $G\sT$.  Similarly, an {\em $\sF$-$G$-space} is a 
$\sU_{\ast}$-functor $\sF\rtarr G\sU_\ast$ whose restriction to $\PI$ is a $\PI$-$G$-space.   A map of $\PI$-$G$-spaces or of $\sF$-$G$-spaces is a $\sU_{\ast}$-natural transformation. We denote the categories of $\PI$-$G$-spaces and of $\sF$-$G$-spaces by $\PIdashG$ and $\FdashG$.  
\end{defn}

As in \autoref{onlyordered}, it is enough to check the condition on ordered injections.

\begin{rem} Since $G$ acts trivially on $\PI$ and $\sF$, the notations  $\PI[\sT_G]$ and $\sF[\sT_G]$ carry the same meaning: $\sU_{\ast}$-functors  
$\PI\rtarr \sT_G$ necessarily land in $G\sT$. 
\end{rem}

While the cofibration condition is always assumed when using the terminology above, it will not be used in this section.
The following equivariant generalization of \cite[Proposition 1.6]{MT} implies that assuming it does not result in any loss of generality.  The cited result shows how to  ``beard'' functors $\sF\rtarr  \sU_*$ to obtain levelwise equivalent functors in $\sF[\sT]$, and we shall see that equivariance requires only minor changes.    We defer the proof to \autoref{beard}, where we generalize the result from $\sF$ to $G$-categories of operators over $\sF$, again following \cite{MT}.

\begin{prop}\label{MTbeard}
There is a functor $W\colon \Fun(\sF, G\sU_*)\rtarr   \Fun(\sF, G\sU_*)$  together with a natural transformation $\pi\colon W\rtarr \Id$ such that, for any functor
$X\colon \sF\rtarr G\sU_*$, $WX$ is in $\FdashG$ and $\pi\colon WX\rtarr X$ is a levelwise $G$-homotopy equivalence. 
\end{prop} 

\begin{rem}\label{Requiv}
The adjunction $(\bL,\bR)$ of \autoref{Rdefn} and \autoref{Rlem} extends to the equivariant case to give an adjunction between $G\sT$ and $\PIdashG$. Indeed, if $Y$ is a nondegenerately based $G$-space, the functor $\bR Y$ that sends $n$ to $Y^n$ satisfies the cofibration condition since, for an injection 
$\ph\colon \bm\rtarr \bn$, $\ph_*\colon Y^m\rtarr Y^n$ just inserts the basepoint in the coordinates indexed by those $j$ not in the image of $\ph$.
\end{rem}

Recall from \autoref{famFn} that, for a homomorphism $\al\colon G\rtarr \SI_n$, $\LA_{\al}$ is the subgroup $\{ (g,\al(g))|g\in G \}$ of $G\times \SI_n$.  Observe that since $\SI\subset \PI\subset \sF$,  $X_n$ and $X_1^n$ are $(G\times \SI_n)$-spaces and $\de\colon X_n\rtarr X_1^n$ is a map  of $(G\times \SI_n)$-spaces. 

\begin{defn}\label{weakFn} Let $X,Y$ be $\PI$-$G$-spaces, and $f\colon X\rtarr Y$ be a map of $\PI$-$G$-spaces.
\begin{enumerate}[(i)]
\item The $\PI$-space $X$ is \gen-{\em special} if 
$\de\colon X_n\rtarr X_1^n$ is an $\bF_n$-equivalence for all $n\geq 0$.
\item  Less restrictively, $X$ is {\em special} if each $\de\colon X_n\rtarr X_1^n$ is a weak $G$-equivalence.
\item  The map $f$ is an {\em \gen-level equivalence} if each $f_n\colon X_n\rtarr Y_n$ is an $\bF_n$-equivalence. 
\item  Less restrictively,  $f$ is a {\em level $G$-equivalence} if each $f_n\colon X_n\rtarr Y_n$ is a weak 
$G$-equivalence. 
\end{enumerate}
An $\sF$-$G$-space $X$ is \gen-special or special if its underlying $\PI$-$G$-space is so.
A special $\sF$-$G$-space $X$ is \emph{grouplike} or, synonymously, \emph{very special} if $\pi_0(X_1^H)$ is a group 
(necessarily abelian) under the induced product for each $H\subset G$.
A map $f\colon X\rtarr Y$ of $\sF$-$G$-spaces is an \gen-level equivalence or level equivalence 
if it is so as a map of $\PI$-$G$-spaces.
\end{defn} 

We need several technical results about these notions, the first of which is the key to the following two.   

\begin{lem}\label{lem:lafixedpoints}
Let $X$ be a $G$-space, and let $\LA_{\al} = \{(h,\al(h)) \mid h \in H\}\subset G\times \SI_n$, where $H\subset G$ and $\al \colon H \rtarr \SI_n$ is a homomorphism. 
Then there is a natural homeomorphism
\[ (X^n) ^{\LA_{\al}} \cong \prod X^{K_i},\]
where the product is taken over the orbits of the $H$-set $\bn^{\al}$ and the $K_i\subset H$ are the stabilizers of chosen elements in the corresponding orbit.
\end{lem}

\begin{proof}
The $\LA_{\al}$-action on $X^n$ is given by
$$(h,\al(h))(x_1, \dots, x_n) = (hx_{\al(h^{-1})(1)}, \dots, hx_{\al(h^{-1})(n)}).$$
The partition of $\bf{n}^\al$ into $H$-orbits decomposes $\bf{n}$ as the wedge of finite subsets, each 
with a transitive set of shuffled indices, so it is enough to consider each $H$-orbit separately. Thus we may as well 
assume that the $H$-action on $\bn^{\al}$ is transitive. Note that this reduction depends only on $\al$ and is natural in $X$. 

Let $K\subset H$ be the stabilizer of $1\in \mb{n}$.  We claim that projection onto the first coordinate induces the required natural homeomorphism
$\pi\colon (X^n)^{\LA_{\al}} \rtarr X^K$. 
Let $(x_1,\dots,x_n) \in X^n$ be a $\LA_{\al}$-fixed point.   Then $x_1$ is a $K$-fixed point of $X$ since 
$$kx_1 = kx_{\al(k^{-1})(1)}=x_1$$
for $k\in K$,  the second equality holding because $(x_1,\dots, x_n)$ is fixed by $\LA_{\al}$.

To construct $\pi^{-1} \colon X^K\rtarr (X^n) ^{\LA_{\al}}$, for $1\leq j\leq n$ choose $h_j\in H$ such that $\al(h_j)(1)=j$.
This choice amounts to choosing a system of coset representatives for $H/K$, and the map $j\mapsto [h_j]$ 
gives a bijection of $H$-sets between $\bn^{\al}$ and $H/K$. We claim that the map $X \rtarr X^n$ 
that sends $x$ to the $n$-tuple $(h_1x, \dots, h_nx)$ restricts to the required inverse  $\pi^{-1}$.  This map
is clearly continuous. We first
show that if $x\in X^K$, then $(h_1x,\dots,h_nx)$ is fixed by $\LA_{\al}$. Let $h\in H$ and note that
\[\al(hh_{\al(h^{-1})(j)})(1)=\al(h)(\al(h_{\al(h^{-1})(j)})(1))=\al(h)(\al(h^{-1})(j))=j.\] 
In view of our bijection between $\bn^{\al}$ and $H/K$, there exists $k\in K$ such that
$hh_{\al(h^{-1})(j)}=h_jk$. The $j$th coordinate of $(h,\al(h))\cdot (h_1x,\dots, h_nx)$ is given by
$$hh_{\al(h^{-1})(j)}x=h_jkx=h_jx,$$
the second equality holding because $x\in X^K$.  Thus $(h_1x,\dots,h_nx)$ is fixed by $\LA_{\al}$. 

Since $K$ is the stabilizer of 1, $h_1\in K$. Thus the first coordinate of $(h_1x,\dots,h_nx)$ is $x$ itself, and $\pi\com \pi^{-1} =\id$.
If $(x_1,\dots, x_n)$ is fixed by $\LA_{\al}$, then $x_j=h_jx_1$ for all $j$.  By  the definition of $h_j$, the $j$th 
coordinate of $(h_j,\al(h_j))\cdot (x_1,\dots,x_n)$ is $h_jx_{\al(h_j^{-1})(j)}=h_jx_1$.
Since$(x_1,\dots, x_n)$ is a $\LA_{\al}$-fixed point, this shows that $x_j=h_jx_1$, hence $\pi^{-1}\com \pi = \id$.
\end{proof}

\begin{lem}\label{Rgen} If $f\colon X\rtarr Y$ is a weak equivalence of based $G$-spaces, then the induced map
$\bR f\colon \bR X\rtarr \bR Y$
is an \gen-level equivalence of $\PI$-$G$-spaces.
\end{lem}
\begin{proof} This is immediate from \autoref{lem:lafixedpoints}.
\end{proof}

\begin{lem}\label{iff}  Let $f\colon X\rtarr Y$ be an \gen-level equivalence of $\PI$-$G$-spaces.  Then $X$ is \gen-special if and only if $Y$ is \gen-special.
Similarly, if $f$ is a level $G$-equivalence, then $X$ is special if and only if $Y$ is special.
\end{lem}
\begin{proof}  Consider the commutative diagram
\[\xymatrix{
X_n \ar[r]^{f_n} \ar[d]_{\de} & Y_n \ar[d]^{\de}\\
X_1^n \ar[r]_{f^n_1} & Y_1^n.\\} \]
For the first statement, the horizontal arrows are $\bF_n$-equivalences by assumption and \autoref{Rgen}, so one of the vertical arrows is a 
$\bF_n$-equivalence if and only if the other one is.  The proof of the second statement is similar but simpler.
\end{proof} 

We defined \gen-special $\PI$-$G$-spaces in terms of those subgroups
$\LA$ in the family $\bF_n$ (see Definition \ref{famFn}) which are
defined by homomorphisms $\al\colon G\rtarr \SI_n$, ignoring those which
are defined by homomorphisms $\be\colon H\rtarr \SI_n$ for proper
subgroups $H$ of $G$. The following result shows that when $G$ is finite
we obtain the same notion if we instead use all of the groups in
$\bF_n$.  The result implicitly relates \gen-special $\PI$-$G$-spaces to
equivariant covering space theory and relates \gen-special
$\sF$-$G$-spaces to the operadic approach to equivariant infinite loop
space theory. It therefore explains and justifies the terms 
\gen-special and
  \gen-level equivalence.

\begin{lem}\label{speciallevel} Assume that $G$ is finite.  Then a
$\PI$-$G$-space $X$ is \gen-special if and only if the Segal maps
$\de\colon X_n\rtarr X_1^n$ are weak $\bF_n$-equivalences for all $n\geq
0$.  Similarly, a map $f\colon X\rtarr Y$ of $\PI$-$G$-spaces is an
\gen-level equivalence if and only if $f_n$ is a weak
$\bF_n$-equivalence for all $n\geq 0$.
\end{lem}

\begin{proof}  Given the hypotheses, we must prove that each $\de$ and each $f_n$ is a weak $\bF_n$-equivalence.
Thus consider a subgroup $\LA_\be\subset G\times \SI_n$, $\be\colon H\rtarr \SI_n$. We will show that $\de\colon X_n\rtarr X_1^n$ 
is a weak $\LA_\be$-equivalence by displaying it as a retract of a suitable weak equivalence. The homomorphism $\be$ gives 
an $H$-set $B=\bn^{\be}$. Embed $B$ as a subset of the $G$-set $A = G_+\sma_H B$  and observe that, as an $H$-set, $A$ splits
as $B\wed C$, where $C=(A \setminus B)_+$.  Let $p = |A \setminus B|$ and $q = n+p$.  Use the given ordering of $B$ and an ordering of $C$
to identify $A$ with some $\bq^{\al}$.  Here $\mathbf{q} = \mathbf{n}\wed\mathbf{p}$ and $\al$ is a homomorphism $G\rtarr \SI_q$ which when restricted
to $H$ is of the form $\be\wed \ga$.  That is, $\bq^{\al|_H} = \bn^{\be}\wed \bp^{\ga|_H}$.  Let
\[ \io\colon \bn^{\be}\rtarr \bq^{\al|_H} \ \ \text{and} \ \  \pi\colon \bq^{\al|_H}\rtarr \bn^{\be} \]
be the inclusion that sends $i$ to $i$ for $0\leq i\leq n$ and the projection that sends $i$ to $i$ for $0\leq i\leq n$ and $i$
to $0$ for $i> n$. Then the following diagram displays a retraction. Its bottom arrows are the evident inclusion and projection.
\[ \xymatrix{
X_n \ar[d]_{\de} \ar[r]^-{\io_{*}} & X_q \ar[d]^{\de} \ar[r]^-{\pi_{*}} & X_n \ar[d]^{\de} \\
X_1^n \ar[r] & X_1^q \ar[r] & X_1^n \\ } \]
Since $X$ is \gen-special, the middle vertical arrow $\de$ is a weak $\LA_{\al}$-equivalence and thus 
a weak $\LA_{\al|H}$-equivalence. Therefore the left arrow $\de$ is a weak $\LA_{\be}$-equivalence.

Similarly, for $f_*$, we have a retract diagram
\[ \xymatrix{
X_n \ar[d]_{f_n} \ar[r]^-{\io_{*}} & X_q \ar[d]^{f_q} \ar[r]^-{\pi_{*}} & X_n \ar[d]^{f_n} \\
Y_n \ar[r]_-{\io_{*}} & Y_q \ar[r]_{\pi_{*}} & Y_n \\ } \]
in which $f_q$ is a weak $\LA_{\al}$-equivalence and therefore $f_n$ is a weak $\LA_{\be}$-equivalence.
\end{proof}

\subsection{The simplicial version of the Segal machine}\label{SEGALsimp}

Let $\DE$ be the usual simplicial category. A simplicial object is a contravariant
functor defined on $\DE$, and a cosimplicial object is a covariant functor.  Regarding
$\sF$ as a full subcategory of the category of based sets, we may regard the simplicial circle 
$S^1_s = \Delta[1]/\partial \Delta[1]$ as a contravariant functor $F\colon \DE^{op}\rtarr \sF$. 
By pullback along $F$, an $\sF$-$G$-space $X$ can be viewed as a simplicial $G$-space, and it has a 
geometric realization $|X| = |X\com F|$; we use the standard realization, taking degeneracies into account.
The evident $G$-map $X_1\times I\rtarr |X|$ factors through a natural $G$-map 
$\SI X_1\rtarr |X|$ with adjoint $\et\colon X_1\rtarr \OM |X_1|$. We shall sketch the proof of
the following result in \autoref{groupcomp}. It is implicit in many early sources; we will follow
\cite[\S15]{MayClass}.  Here we need our $\sF$-$G$-spaces to be levelwise nondegenerately based since 
the argument involves realizations of Reedy cofibrant simplicial $G$-spaces, but we do not require the stronger cofibration condition 
of \autoref{Fspace}.

\begin{prop}\label{Ggpcomp}  If $X$ is a special $\sF$-$G$-space, then the $G$-map 
$\et\colon X_1\rtarr \OM |X|$ is a group completion of Hopf $G$-spaces.
\end{prop} 

From here, the Segal machine in its first avatar is constructed as follows \cite[\S 1]{Seg}.  
Working equivariantly, we use a slight reformulation that is given in \cite{MT}.\footnote{Perversely, \cite{MT}
takes $\DE$ to be the opposite of the category every other reference calls $\DE$.}  

\begin{rem}\label{opsonF}
We have the smash product $\sma\colon  \sF\times \sF\rtarr \sF$.  It sends $(\mathbf{m},\mathbf{n})$ to $\mathbf{mn}$ and 
is strictly associative and unital using lexicographic ordering.  The unit is $\mathbf{1}$.  We also have the wedge sum 
$\wed \colon  \sF\times \sF\rtarr \sF$ which sends $(\mathbf{m},\mathbf{n})$ to $\mathbf{m+n}$. It is also strictly associative 
and unital, with unit $\mathbf{0}$, and $\sF$ is bipermutative under this sum and product.  
\end{rem}

\begin{defn}\label{class} Let $X$ be a special $\sF$-$G$-space.  We have the functor 
$$X\com \sma\colon \sF\times \sF\rtarr G\sT.$$ 
For each $\mathbf{p}$, let $X[p]$ be the $\sF$-$G$-space that sends $\mathbf{q}$ to $X(\mathbf{p}\sma \mathbf{q})$;
thus $X[0] = \ast$ and  $X[1] = X$.  Following Segal, define the {\em classifying $\sF$-$G$-space} $\bB X$ to be the 
$\sF$-$G$-space whose $p$th $G$-space is the realization $|X[p]|$. Observe that  $\bB X$ is again special 
since factorizations of the Segal maps for $X$ give commutative diagrams
\begin{equation}\label{Bspecial}
\xymatrix{ X[p](\mathbf{q}) \ar[r]^-{\de} \ar@{=}[d]& X[p](\mathbf{1})^q \ar@{=}[r] & X(\mathbf{p})^q \ar[d]^{(\de)^q} \\
X(\mathbf{pq}) \ar[rr]_{\de}  & & X(\mathbf{1})^{pq}. } \\  
\end{equation}
Iterating, with $\bB^0X = X$, define $\bB^{n+1}X = \bB(\bB^{n}X)$ for $n\geq 0$.  The $\sF$-$G$-spaces 
$\bB^nX$ for $n\geq 1$ are special and, since $(\bB^nX)_1$ is $G$-connected, they 
are also grouplike.  
\end{defn} 

\begin{defn}\label{notone}
Let $\bS_G^C X$ denote the resulting classical $G$-prespectrum with $n$th $G$-space $(\bS_G^CX)_n = (\bB^nX)_1$ for $n\geq 0$. 
Thus its $0$th $G$-space is $X_1$ and, by \autoref{Ggpcomp}, its structure map $X_1 \rtarr \OM (\bS_{G}^C X)_1$ is a group completion
and the structure maps $(\bS_G^CX)_n \rtarr \OM (\bS_{G}^C X)_{n+1}$ for $n\geq 1$ are weak $G$-equivalences. If we redefine $(\bS_G^CX)_0$ to be $\OM(\bS_G^CX)_1$, then $\bS_G^CX$ is a classical $\OM$-$G$-spectrum with a group completion $\io\colon X_1\rtarr (\bS_G^CX)_0$.   This means that $\bS_G^C$ is a Segal machine as in \autoref{machine}, but with $G$-action.
\end{defn}

When $G$ is trivial, this machine plays a special role: any classical infinite loop space machine as in \autoref{machine} is equivalent to this Segal machine \cite{MT}.  The proof  in \cite{MT} directly generalizes equivariantly to classical $G$-spectra.   It makes essential use of the fact that the Segal machine produces $\sF\sF$-$G$-spaces, namely functors from $\sF$ to $\sF$-$G$-spaces.  However, this construction does {\em not} work to construct {\em genuine} $G$-spectra from $\sF$-$G$-spaces: there is no evident way to build in deloopings by non-trivial representations of $G$.  

\subsection{The conceptual version of the Segal machine}\label{Segalconc}

The more conceptual variants of the nonequivariant Segal machine do generalize to give 
genuine $G$-spectra. These variants do {\em not} make use of the 
functor $F\colon \DE^{op} \rtarr \sF$.  That is, underlying simplicial $G$-spaces play no role in their
construction.  We follow the nonequivariant exposition of \cite{MMSS}. 
We first introduce
notation for the categories of enriched functors that we shall be using.  Recall \autoref{ambient}.

\begin{notn}  For a $G\sU_{\ast}$-category $\sE$, let $\Fun(\sE,\UG)$ denote the category
of $G\sU_{\ast}$-functors $\sE\rtarr\UG$ and $G\sU_{\ast}$-natural transformations between them. 
When $G$ acts trivially on $\sE$, as is the case of $\sF$, a $G\sU_{\ast}$-functor defined on $\sE$
takes values on morphisms in the fixed point spaces $\UG(X,Y)^G = G\sU_\ast(X,Y)$ of based $G$-maps $X\rtarr Y$.  We often use the alternative notation $\Fun(\sE,G\sU_\ast)$ in that case. 
\end{notn} 

\begin{defn}\label{WGspace} A {\em $\sW_G$-$G$-space} $Y$ is a $G\sU_{\ast}$-functor $\sW_G \rtarr \UG$.
A map of $\sW_G$-$G$-spaces is a $G\sU_{\ast}$-natural transformation between them.
Regarding $\sF$ as a $G$-trivial $G$-category, it is both a full subcategory of $G\sW$ and a $G$-trivial
full $G$-subcategory of $\sW_G$. We have the functor categories 
\[\Fun(\sF,G\sU_\ast) = \Fun(\sF,\UG)\]
of $\sF$-$G$-spaces and $\Fun(\sW_G,\UG)$ of $\sW_G$-$G$-spaces.  The inclusion $\sF\subset\sW_G$ 
induces a forgetful functor
\[ \bU\colon \Fun(\sW_G,\UG) \rtarr \Fun(\sF,\UG). \]
\end{defn}

As a matter of elementary category theory (see e.g. \cite{MMSS}), the functor $\bU$ has a left adjoint 
prolongation functor 
\[  \bP\colon \Fun(\sF,\UG) \rtarr \Fun(\sW_G,\UG). \]

\begin{defn}\label{structuremaps} Let $Y$ be any $\sW_G$-$G$-space, such as $Y = \bP X$ for an $\sF$-$G$-space $X$.
For $G$-spaces $A, B\in G\sW$, the adjoint $B\rtarr \sW_G(A, A\sma B)$ of the identity map 
on $A\sma B$  can be composed with $Y$ to obtain a $G$-map 
\[ B\rtarr \UG(Y(A), Y(A\sma B)).\]
Its adjoint is a $G$-map
\begin{equation}\label{strucW}
  Y(A)\sma B \rtarr  Y(A\sma B).
 \end{equation}
For example, taking $B=I_+$, it follows that the functor $Y$ preserves based homotopies.
Letting $A=S^n$ and $B = S^1$ with trivial $G$-action, these maps give the structure maps  
\[ \SI Y(S^n) \rtarr Y(S^{n+1})\]
of a classical $G$-prespectrum $\bU_{G\sP}Y$. 

Going further, we define an orthogonal $G$-spectrum given at level $V$ by $Y(S^V)$. The composites
\[ \xymatrix@1{ \sI_{G}(V,V') \ar[r] & \sW_G(S^V, S^{V'}) \ar[r]^-{Y} & \UG(Y(S^V), Y(S^{V'}))\\} \]
of $Y$ and the map induced by one-point compactification of maps $V\rtarr V'$ give a $G\sU$-functor $\sI_{G}  \rtarr G\sU_\ast$ or, equivalently
and more sensibly here, a $G\sU_{\ast}$-functor ${\sI_{G}}_+ \rtarr \UG$. Just as in the nonequivariant case
\cite{MMSS}, letting $A= S^V$ and $B=S^W$ in \autoref{strucW} for representations $V$ and $W$, we obtain the structure $G$-maps 
\[ \SI^W Y(S^V) \rtarr Y(S^{V\oplus W})\]
of an orthogonal $G$-spectrum $\bU_{G\sS}Y$ such that $i^*\bU_{G\sS}Y = \bU_{G\sP}Y$.
\end{defn}

\begin{defn}\label{UPW} The conceptual Segal machine on $\sF$-$G$-spaces is the restriction of the composite
 \[ \bU_{G\sS}\com \bP\colon \Fun(\sF,G\sU) \rtarr G\sS  \]
 to $\FdashG$.
\end{defn}

The functor $\bP$ is a left Kan extension that is best viewed as a categorical tensor product of functors.
For $A\in G\sW$, we have the contravariant $\sU_{\ast}$-functor $A^{\bullet}\colon \sF\rtarr G\sU_\ast$.
Conceptually, $A^{\bullet}$  is the represented functor that sends $\mathbf{n}$ to the function
space $\sW_G(\mathbf{n},A)\iso A^n$ with its induced action by $G$.
By definition,
\begin{equation}
 (\bP X)(A) = A^{\bullet}\otimes_{\sF} X. 
 \end{equation}
Taking $A=\mathbf{n}$, the unit $\et\colon X\rtarr \bU\bP X$ of the adjunction sends 
$x\in X_n$  to  $(\id_{\mathbf{n}},x)$; by Yoneda, $\et$ is a natural isomorphism. 
For a $\sW_G$-$G$-space $Y\colon \sW_G\rtarr\UG$, the counit $\epz\colon \bP\bU Y\rtarr Y$ 
is given on $A\in G\sW$ by the composites
\[ \xymatrix@1{
\sW_G(\mathbf{n},A)\sma Y(\mathbf{n}) \ar[r]^-{Y\sma\id} 
&\UG(Y(\mathbf{n}),Y(A))\sma Y(\mathbf{n}) \ar[r]^-{eval} & Y(A).\\} \]

For an $\sF$-$G$-space $X\colon \sF\rtarr G\sT$, we write $X_n$ for $X(\mathbf{n})$ as before, but 
we follow the usual convention of abbreviating notation by writing $(\bP X)(A) = X(A)$ for general $A\in G\sW$. 
Note that we do not claim that $X(A)$ is nondegenerately based in general.
The following result is a variant of Segal's \cite[Proposition 3.2 and Lemma 3.7]{Seg}.  

\begin{prop}\label{consist} The classical $G$-prespectrum $\bU_{G\sP}\bP X$ is naturally
isomorphic to the classical $G$-prespectrum $\bS_G^C X$ (of \autoref{notone}). Therefore, 
if $X$ is special, then  $\bU_{G\sP}\bP X$ is a positive $\OM$-$G$-prespectrum whose bottom structure map
 is a group completion of $X_1$.
\end{prop}

Segal's nonequivariant proof is briefly sketched in \cite[\S3]{Seg}, in different language. For the reader's convenience, we give a more complete argument in \autoref{SecSegDet}. The following result is the key observation, and it is the crucial point for us.
Via \autoref{Ggpcomp}, it makes the group completion property for $\bU_{G\sP}\bP X$ transparent. 

\begin{prop}\label{cute} For $\sF$-$G$-spaces $X$, there is a natural $G$-homeomorphism 
$$ |X| \rtarr (S^1)^{\bullet}\otimes_{\sF} X = (\bP X)(S^1). $$
\end{prop} 

\subsection{$\sF_G$ and a factorization of the conceptual Segal machine}\label{sectionFFG}   
We now consider the $G$-category $\sF_G$ of finite based $G$-sets rather than just 
the category $\sF$ of finite based sets.  Use of $\sF_G$ in tandem with $\sF$ is essential to our work.  Input arises most often as $\sF$-$G$-spaces but, by a result of Shimakawa \cite{Shim2} that we shall reprove, these are interchangeable with $\sF_G$-$G$-spaces.  

\begin{defn} Let $\sF_G$ be the $G$-category of finite based $G$-sets
and all based functions, with $G$ acting by conjugation on function sets.  
For convenience and precision, we restrict the objects of $\sF_G$ to 
be the finite $G$-sets $\bn^{\al}$, as in \autoref{finiteGset}. 
Let $\PI_G$ be the $G$-subcategory with the same objects and those
morphisms $\ph\colon \bm^{\al} \rtarr \bn^{\be}$ such that 
$\ph^{-1}(j)$ has at most one element for $1\leq j\leq n$. We obtain 
inclusions $\sF\subset \sF_G$ and $\PI\subset \PI_G$ by restricting 
to the trivial homomorphisms $\epz_n\colon G\rtarr \SI_n$.
\end{defn}

As with $\sF$, we view $\sF_G$ as a category enriched in $G\sT\subset G\sU_*$, using the discrete topology on the based hom $G$-sets  $\sF_G(\bm^{\al},\bn^{\be})$.  The basepoint is the unique map that factors through $\mb{0}$, and $G$ acts by conjugation.  A finite $G$-set is evidently a $G$-CW complex, hence we have an inclusion $\sF_G \subset \sW_G$. 

\begin{defn}\label{FGspace}
A {\em $\PI_G$-$G$-space} $Y$ is a $G\sU_{\ast}$-functor $Y\colon \PI_G\rtarr \UG$ whose underlying functor $\PI\rtarr G\sU_{\ast}$ is a {$\PI$-$G$-space}.
We shall see in \autoref{nicebpt} below that $Y$ then  takes values in nondegenerately based $G$-spaces.  Similarly, an {\em  $\sF_G$-$G$-space} $Y$ is a $G\sU_{\ast}$-functor $Y\colon \PI_G\rtarr\TG$ whose restriction to $\PI_G$ is a $\PI_G$-$G$-space.  Morphisms  in both cases are $G\sU_{\ast}$-natural transformations, and we write $\PIsubG$ and $\FsubG$ for the categories of $\PI_G$-$G$-spaces and $\sF_G$-$G$-spaces.
\end{defn}

Write  $Y(A)$ for the value of $Y$ on $A = \bn^{\al}$ and write $Y_n$ for the value of $Y$ on $\bn^{\epz_n}$.  
Recall from \autoref{prodal} that $(Y_1^n)^{\al}$ denotes $Y_1^n$ with the $G$-action 
\[  g(y_1,\dots,y_n) = (gy_{\alpha(g^{-1})(1)},\dots, gy_{\alpha(g^{-1})(n)}) \]
and can be identified with the based $G$-space $\sT_G(\mathbf{n}^\al,Y_1)=Y_1^{\mathbf{n}^\al}$ when $Y_1$ is nondegenerately based.

\begin{defn}\label{Segalmaps2} For $A = \bn^{\al}$ and a $\PI_G$-$G$-space $Y$, define a 
based $G$-map 
$$\epz\colon A\sma A\rtarr \mathbf 1 = S^0$$ 
by the Kronecker $\delta$ function:  $\de(i,j) = 1$ if $i=j$ and
$\de(i,j)=0$ if $i\neq j$.  Its adjoint is a $G$-map $A\rtarr \sF_G(A,\mathbf 1)$.  Composing with 
$$Y\colon \sF_G(A,\mathbf 1) \rtarr\TG(Y(A),Y_1)$$ 
and taking the adjoint, we obtain a $G$-map $\pa_A \colon A\sma Y(A) \rtarr Y_1$.  Thus $\pa_A(j,y) = (\de_j)_{*}(y)$ for $1\leq j\leq n$,
where $\de_j$ is induced by the $j$th projection $\bn^{\al}\rtarr \mathbf{1}^{\epz_1}$. 
The {\em Segal map} 
\[ \de_A \colon Y(A) \rtarr\TG(A,Y_1) \iso (Y_1^n)^{\al} \]
is the adjoint of $\pa_A$; we usually abbreviate $\de_A$ to $\de$. Note that $\de$ is a $G$-map, although
it components $\de_j$ are usually not.
\end{defn}

\begin{defn}\label{PIGSpec} A $\PI_G$-$G$-space $Y$ is {\em special} 
 if the $\de_A$ are weak $G$-equivalences for all $A=\bn^{\al}$.  
A map $f\colon Y\rtarr Z$ of $\PI_G$-$G$-spaces is a {\em level $G$-equivalence} if each $f\colon Y(A)\rtarr Z(A)$ 
is a weak $G$-equivalence. 
We say that an $\sF_G$-$G$-space is special if its underlying $\PI_G$-$G$-space
is so and that a map of $\sF_G$-$G$-spaces is a level $G$-equivalence if its underlying map of
$\PI_G$-$G$-spaces is so.
\end{defn}

The inclusion $\sF \hookrightarrow \sF_G$ induces a restriction functor 
$$\bU \colon \Fun(\sF_G,\UG) \rtarr \Fun(\sF,\UG)$$
which has a left adjoint prolongation functor 
$$\bP\colon \Fun(\sF, \UG) \rtarr \Fun(\sF_G,\UG).$$ 
The adjunction $(\bP,\bU)$ of \autoref{WGspace} factors as the composite of the analogous adjunctions given by the functors
\[ \xymatrix@1{ \Fun(\sF, \UG) \ar[r]^-{\bP} & \Fun(\sF_G,\UG) \ar[r]^-{\bP} & \Fun(\sW_G,\UG),\\} \]
\[ \xymatrix@1{ \Fun(\sW_G,\UG) \ar[r]^-{\bU} & \Fun(\sF_G,\UG) \ar[r]^-{\bU} & \Fun(\sF, \UG).\\} \]
The units of these adjunctions are isomorphisms since the forgetful functors $\bU$ are induced by the full
and faithful inclusions $\sF \hookrightarrow \sF_G$ and $\sF_G\hookrightarrow \sW_G$.  

\begin{notn}\label{AA} For $A\in \sW_G$,  we now write $A^{\bullet}$ ambiguously for both the restriction to 
$\sF_G\subset\UG$ and the restriction to $\sF\subset \sF_G\subset\UG$ of the represented functor $\UG(-,A)\colon\UG^{op}\rtarr \UG$.
\end{notn}

Then the factorization of $\bP$ as $\bP\bP$ takes the explicit form
\begin{equation}\label{TensorEqual}
A^{\bullet}\otimes_{\sF} X \iso A^{\bullet}\otimes_{\sF_G}\! (\sF_G\otimes_{\sF} X)  =A^{\bullet}\otimes_{\sF_G}\!  \bP X .
\end{equation} 
While our main interest is in $\sF_G$-$G$-spaces and $\sF$-$G$-spaces, we will also use the analogous 
forgetful and prolongation functors relating $\PI_G$-$G$-spaces and $\PI$-$G$-spaces.

The following result shows that the notions defined in \autoref{PIGSpec} for 
$\PI_G$-$G$-spaces correspond to the \gen-\-notions for $\PI$-$G$-spaces.  That should help motivate the latter, which may have seemed unnatural at first sight. The analogue for $\PI_G$-$G$-spaces (or for $\sF_G$-$G$-spaces)  of the dichotomies between 
\gen-special and special  and between \gen-level equivalences and level $G$-equivalences that we have for $\PI$-$G$-spaces
(and thus for $\sF$-$G$-spaces) will be discussed in a model theoretical context in \autoref{Both}.

\begin{thm}\label{compFFG}
The adjoint pairs of functors 
\[ \xymatrix@1{\Fun(\PI,\UG) \ar@<0.5ex>[r]^{\bP}  &  \Fun(\PI_G,\UG) \ar@<0.5ex>[l]^{\bU} }\\ \]
and
\[ \xymatrix@1{\Fun(\sF,\UG) \ar@<0.5ex>[r]^{\bP}  &  \Fun(\sF_G,\UG) \ar@<0.5ex>[l]^{\bU} }\\ \]
specify equivalences of categories. Remembering the cofibration condition from \autoref{Fspace}, these equivalences of categories restrict to equivalences
\[ \xymatrix{ \PIdashG   \ar@<0.5ex>[r]^{\bP}  & \PIsubG \ar@<0.5ex>[l]^{\bU} }\\ \]
and \[ \xymatrix{ \FdashG   \ar@<0.5ex>[r]^{\bP}  & \FsubG. \ar@<0.5ex>[l]^{\bU} }\\ \]
Moreover, the following statements hold.
\vspace{1mm}
\begin{enumerate}[(i)]
\item A $\PI_G$-$G$-space $Y$ is special if and only if the $\PI$-$G$-space $\bU Y$ is \gen-special.
\vspace{1mm}
\item A map $f\colon Y\rtarr Z$ of $\PI_G$-$G$-spaces is a level $G$-equivalence if and only if the map 
$\bU f\colon \bU Y\rtarr \bU Z$ of $\PI$-$G$-spaces is an \gen-level equivalence.  
\vspace{1mm}
\item  A $\PI$-$G$-space $X$ is \gen-special if and only if the $\PI_G$-$G$-space $\bP X$ is special.
\vspace{1mm}
\item  A map $f$ of $\PI$-$G$-spaces is an \gen-level equivalence if and only if the map $\bP f$ of 
$\PI_G$-$G$-spaces is a level $G$-equivalence.
\end{enumerate}
All of these statements remain true with with $\PI$ and $\PI_G$ replaced by $\sF$ and $\sF_G$.
\end{thm}
\begin{proof}  
For a $G\sU_\ast$-functor $X\colon \Pi\rtarr \UG$ and a finite $G$-set $A=\bn^{\al}$,
$$(\bP X)(A)=A^\bullet \otimes_\PI X,$$ 
where $A^\bullet\colon \PI \rtarr  \UG$ is the functor that sends $\mathbf{m}$ to 
$\PI_G(\mathbf{m}, A)$. Recall that the underlying set of $\PI_G(\mathbf{m},A)$ is 
just $\PI(\mathbf{m},\mathbf{n})$ with $G$-action induced by the action 
of $G$ on $\mathbf{n}$ given by $\al$. The action of $G$ on 
$A^\bullet \otimes_\PI X$ is induced by the diagonal action.

Since the inclusion $\Pi \rtarr \Pi_G$ is full and faithful, the unit of the adjunction $\eta\colon X \rtarr \bU\bP X$ is an isomorphism.  We must show that the counit $\epz\colon \bP\bU Y\rtarr Y$ is an isomorphism for a $G\sU_\ast$-functor $Y\colon \PI_G\rtarr \UG$. 
Again let $A=\bn^{\al}\in \PI_G$. 
Then 
$$\epz\colon (\bP\bU Y)(A) = A^\bullet \otimes_\PI (\bU Y) \longrightarrow Y(A)$$ 
is the $G$-map given by $\epz(\mu, y) = \mu_{*}y$ for $\mu\colon \mathbf{m}\rtarr A$ and $y\in Y_m$, 
where $\mu_{*}\colon Y_m \rtarr Y(A)$.  It is a $G$-homeomorphism with inverse given by $\epz^{-1}(y) = (\iota^{-1}, \iota_{*} y)$
for  $y\in Y(A)$, where $\iota \in \PI_G(A,\mathbf{n})$ is the morphism whose underlying function on $\mathbf{n}$ is the identity. 
Clearly $\epz\epz^{-1} = \id$, and $\epz^{-1}\epz = \id$ since $(\mu,y) \sim (\io^{-1},\io_{*}\mu_{*} y)$. The identification uses the morphism $\io \circ \mu$ in $\PI$.
Since $\epz^{-1}$ is inverse to a $G$-map, it is a $G$-map.
The proof with $\PI$ and $\PI_G$ replaced by $\sF$ and $\sF_G$ is the same. The restriction of the equivalence to our subcategories of interest follows from the fact that a $\PI_G$-$G$-space is precisely a functor out of $\PI_G$ whose restriction to $\Pi$ is a $\PI$-$G$-space.

To prove the rest of the theorem, we use the following explicit description of a  $G\sU_\ast$-functor $Y\colon \Pi_G \rtarr \UG$ in terms of its restriction to  $\Pi$. Recall from \autoref{Xal} that,
for a homomorphism $\al \colon G \rtarr \SI_n$,  $Y_n^{\alpha}$ denotes $Y_n$ with a twisted action $\cdot_{\al}$ of $G$ specified in terms of $\alpha$ and the original action of $G$ by 
$g\cdot_\alpha y = \alpha(g)_{*}(g\cdot y)$.
For a finite $G$-set $\bn^\al$, $\epz^{-1}$ identifies $Y(\bn^\al)$ with the $G$-space $Y_n^{\alpha}$.   

To prove (i) and (ii),  consider a subgroup $\LA_\alpha=\{(h,\al(h))\}$ of $G\times \SI_n$ such that $\LA_{\al}\cap \PI = \{e\}$. Projection onto the first coordinate gives an isomorphism $\LA_\al \rtarr H$. 
The $\LA_\al$-action on $Y_n$ 
obtained by restriction of the action of $G\times \Sigma_n$ is given by 
$$(h, \al(h))\cdot y =\al(h)_{*}(h\cdot y).$$
Thus it coincides with the restriction to $H$ of the $G$-action used to define $Y_n^\al$. This immediately implies (ii). Similarly, the $\LA_\al$-action on $Y_1^n$ obtained by restriction of the action of $G\times \SI_n$ given by the diagonal action of $G$ and the permutation action of $\SI_n$ is 
\[  h(y_1,\dots,y_n) = (hy_{\al(h^{-1})(1)},\dots, hy_{\al(h^{-1})(n)}). \] 
Thus it coincides with the restriction to $H$ of the $G$-action that we used to define $(Y_1^n)^{\al}$. Therefore $\epz^{-1}$ identifies the restriction to $H$ of the Segal $G$-map $\delta\colon Y(\bn^{\al}) \rtarr (Y_1^n)^\al$ with the $\LA_\al$-map $\delta\colon Y_n\rtarr Y_1^n$.  This immediately implies (i), and (iii) and (iv) follow formally from (i) and (ii) since $\et\colon \mathrm{id} \rtarr \bU\bP$ is an isomorphism.

Since statements (i)-(iv) for $\sF$-$G$-spaces and $\sF_G$-$G$-spaces depend only on their underlying $\PI$-$G$-spaces and $\PI_G$-$G$-spaces, they follow immediately.   
\end{proof}

\begin{rem}\label{nicebpt}
 Let $Y$ be a $\PI_G$-$G$-space. As noted in the proof of \autoref{compFFG}, for a finite $G$-set $\bn^\al$, we have a $G$-homeomorphism $Y(\bn^\al)\cong Y_n^\al$. Since $\bU Y$ is a $\Pi$-$G$-space, the inclusion of the basepoint $\ast \rtarr Y_n$ is a $(G\times \SI_n)$-cofibration, and in particular $Y_n^\al$ is in $G\sT$. Thus $Y$ takes values in nondegenerately based $G$-spaces, as claimed in \autoref{FGspace}.
\end{rem}

We record analogues for $\sF_G$-$G$-spaces of Lemmas \ref{Rgen} and \ref{iff} for $\sF$-$G$-spaces.   
While they could be proven directly,
we just observe that the first follows immmediately from \autoref{compFFG}(iii), and the second follows as in the proof of \autoref{iff}.

\begin{defn}\label{Rdefn2}   For a based $G$-space $X$, let $\bR_GX$ denote the $\PI_G$-$G$-space with $\bn^{\al}$th $G$-space $(X^n)^{\al}$.  
Conceptually, it is obtained by prolonging the $\PI$-$G$-space $\bR X$ from \autoref{Requiv}  to a $\PI_G$-$G$-space.
\end{defn}

\begin{lem}\label{Rgen2} If $f\colon X\rtarr Y$ is a weak equivalence of based $G$-spaces, then the induced map
$\bR_Gf\colon  \bR_G X\rtarr \bR_GY$ is a level $G$-equivalence of $\PI_G$-$G$-spaces.
\end{lem}

\begin{lem}\label{iff2}  If $f\colon X\rtarr Y$ is a level $G$-equivalence of $\PI_G$-$G$-spaces, then $X$ is special if and only if $Y$ is special.
\end{lem}

\section{The homotopical version of the Segal machine}\label{SegHom}

While $\bU_{G\sS}\bP(X)$ gives the most conceptually natural equivariant version of the Segal 
machine on $\sF$-$G$-spaces $X$, it is by itself of negligible use since the 
functor $\bP$ does not generally enjoy good homotopical properties before some kind of homotopical 
approximation of $X$.  We define a classical homotopical Segal machine in \autoref{bar2} and show
its defects.   We define a genuine homotopical Segal machine
and summarize its homotopical properties in \autoref{barFG}, deferring the longer
proofs to \autoref{SEGALPf}.
 
 The homotopical version of 
 the Segal machine is defined in terms of an appropriate enriched version of the two-sided
 categorical bar construction.  We start in \autoref{bar} with a general discussion 
 of that construction. Since it is defined using geometric realization of simplicial $G$-spaces, it is only well-behaved homotopically when those simplicial $G$-spaces are Reedy cofibrant, as we will dissuss in \autoref{bartoo}.
  
\subsection{The categorical bar construction}\label{bar}  

We here define the variant of the bar construction used in the homotopical Segal machine.
We begin with some general definitions that start with a closed symmetric monoidal category $\sV$ 
and a $\sV$-category $\sE$, as in \autoref{catpre}.  We could start more generally with $\sV$ enriched over some related category rather than over itself, but we restrict to $\sV$ for notational simplicity.
Let $Y$ be a contravariant and $X$ be a covariant $\sV$-functor 
$\sE\rtarr \ul{\sV}$.   They are given by objects $Y_n$ and $X_n$ in $\sV$ and maps 
$$ Y\colon \sE(m,n) \rtarr \ul{\sV}(Y_n, Y_m) \ \ \text{and} \ \  X\colon \sE(m,n) \rtarr \ul{\sV}(X_m, X_n) $$
 in $\sV$ that are compatible with composition and identity. These have adjoint evaluation maps 
$$ E_Y\colon \sE(m,n) \otimes Y_n \rtarr Y_m \ \ \text{and}\ \ E_X\colon \sE(m,n)\otimes X_m\rtarr X_n.$$

Associated to the triple $(Y,\sE,X)$ we have a categorical two-sided bar construction $B_{*}(Y, \sE,X)$.  It is a simplicial object in $\sV$.   Its 
$q$-simplex object in $\sV$ is  the coproduct 
 \begin{equation}\label{BBQ}
 B_q(Y,\sE,X) = \coprod_{(n_0,\dots,n_q)} Y_{n_q} \otimes \sE(n_{q-1},n_q) \otimes \cdots \otimes \sE(n_0, n_1)\otimes X_{n_0}, 
 \end{equation}
 where $(n_0, \dots, n_q)$ runs over the $(q+1)$-tuples of objects of $\sE$.
 Its faces $d_i$ for $0\leq i \leq q$ are induced by the 
 evaluation maps of $Y$ and $X$ and by composition in $\sE$, and its degeneracies $s_i$ 
 for $0\leq i\leq q$  are induced by the unit maps of $\sE$.  
In more detail, the $d_i$ are induced by 
$$ \xymatrix@1{\sE(n_{0},n_{1})\otimes X_{n_{0}} \ar[r]^-{E_X} & X_{n_{1}} \ \ \text{if $i= 0$,} \\} $$
$$ \xymatrix@1{\sE(n_{i},n_{i+1})\otimes \sE(n_{i-1},n_{i}) \ar[r]^-{C} & \sE(n_{i-1},n_{i+1})  \ \ \text{if $ 0< i < q$,}\\} $$
$$\xymatrix@1{ Y_{n_q}\otimes \sE(n_{q-1},n_q) \ar[r]^-{E_Y} &  Y_{n_{q-1}} \ \ \text{if $i=q$}.\\} $$
The $s_i$ are induced by
$$\xymatrix@1{ U \ar[r]^-{I} & \sE(n_{i},n_{i})  \ \ \text{if $0\leq i\leq q$.}\\}$$
We shall use the following observation in \autoref{WEDGE}.

\begin{rem}\label{Groth} When $\sV$ is cartesian closed, so that $\otimes=\times$, $B_{*}(Y,\sE,X)$ is the nerve of an internal ``Grothendieck category of elements'' $\sC(Y,\sE,X)$. The objects and morphisms of this category are both objects of $\sV$.  The object of objects is the coproduct 
$$ \sC_0=\coprod_{n}   Y_n\times X_n,$$ 
where $n$ runs over the objects of $\sE$.   The object of morphisms is the coproduct
$$ \sC_1=\coprod_{m,n} Y_n\times \sE(m,n)\times X_m,$$
where $(m,n)$ runs over the pairs of objects of $\sE$.  We have source, target, and identity
maps $S$, $T$, and $I$ given as follows:  $S$ and $T$ are given on components by the evaluation maps of $Y$ and $X$.
$$   S=E_Y\times  \id \colon Y_n\times \sE(m,n)\times X_m\rtarr Y_m\times X_m$$
$$   T = \id\times E_X \colon Y_n\times \sE(m,n)\times X_m \rtarr Y_n\times X_n;$$
$I$ is induced  by the identity maps $U\rtarr \sE(n,n)$ of $\sE$. 
The composition
$$ C\colon \sC_1 \times _{\sC_0} \sC_1 \rtarr \sC_1$$
is induced by the composition in $\sE$.  The point is that when $\otimes = \times$ we have the identification
\[\sC_1 \times _{\sC_0} \sC_1 \iso \coprod_{(m,n,p)}Y_p \times \sE(n,p) \times \sE(m,n) \times X_m. \] 
where $(m,n,p)$ runs over the triples of objects of $\sE$.
\end{rem}

\begin{rem}  The two-sided bar construction goes back to \cite[\S12]{MayClass}. It is described in the categorical form just
given in Shulman \cite[Definition 12.1]{Shul}, where more details can be found.  The construction in this generality and 
in this form is central to the study of weighted colimits in enriched category theory.
\end{rem}

When $\sV$ has a suitable covariant simplex functor $\DE_{\ast}\colon \DE\rtarr \sV$ so that we have a 
``realization'' functor from simplicial objects in $\sV$ to $\sV$, we define $B(Y,\sE,X)$ to be 
the realization of  $B_{*}(Y,\sE,X)$.  Formally, it is the categorical tensor product
$$  B(Y,\sE,X) =  B_{\ast}(Y, \sE, X)\otimes_{\DE} \DE_{\ast}. $$

\begin{rem} The target $\sV$ could be varied here, as it was in the original homological use of the bar construction (in the 1940's).  
There, ``realization'' went from simplicial $R$-modules to chain complexes of $R$-modules.
\end{rem} 

\begin{exmp}\label{approximation}  One example of $Y$ is the represented contravariant $\sV$-functor\linebreak  
$\sE_n =\sE(-,n)$ for an object $n$ of $\sE$.  
For an object $K$ of $\sV$, we have the constant simplicial object $K_{\ast}$ with $q$-simplices $K$ for each $q$.   
Using composition in $\sE$ and the action of $\sE$ on $X$, we obtain a map
\[ \epz_n \colon  B_{\ast} (\sE_n,\sE,X)\rtarr  {(X_n)}_{\ast}  \]   
of simplicial objects in $\sV$ for each $n$.  At each simplicial level, these form a $\sV$-natural transformation of $\sV$-functors out of $\sE$.  Thus, upon realization, these give a $\sV$-natural transformation 
\begin{equation}\label{epz} 
  \epz\colon  B(\sE,\sE, X) \rtarr X. 
 \end{equation}
Using identity morphisms in $\sE$, we also obtain maps
\[  \et_n \colon {(X_n)}_{\ast} \rtarr B_{\ast} (\sE_n,\sE,X) \]
such that  $\epz_n \com \et_n = \id$.   These maps are not $\sV$-natural in $\sE$, as can be readily checked. The standard extra degeneracy argument gives a simplicial homotopy  $h\colon \id \htp \et_n \com \epz_n$.  Passing to realization, which in our applications takes simplicial homotopies to homotopies, these induce homotopies $\id\htp \et_n\epz_n$, giving homotopy equivalences   
\[   B(\sE_n,\sE,X) \htp X_n \]
We view $B(\sE,\sE,X)$ as a canonical approximation to $X$, just as in the original use in homological algebra.  
\end{exmp}

Via the following observation, the general construction gives a natural ``derived'' approximation to $Y\otimes_{\sE} X$. 

\begin{lem}\label{halfsmash}  There is a natural isomorphism
\[  Y\otimes_{\sE} B(\sE,\sE,X) \iso B(Y,\sE,X) \]
of objects in $\sV$ and a natural map $B(Y,\sE,X)\rtarr Y\otimes_{\sE} X$ in $\sV$. 
\end{lem}
\begin{proof}
One proof uses  a comparison of definitions on the level of $q$-simplices for each $q$, but 
the result is also an application of the categorical Fubini theorem, which implies that we can commute realization and $\otimes_{\sE}$.
Details are given in  \cite[Lemma 19.7]{Shul}.  The natural map is obtained by applying $Y \otimes _\sE -$ to the natural map $\epz$ of \autoref{epz}.
\end{proof}

\begin{rem} More generally, if $\Psi\colon \sD \rtarr \sE$ is a $\sV$-functor, $Y$ is a contravariant $\sV$-functor $\sE\rtarr \sV$, 
and $X$ is a covariant $\sV$-functor $\sD\rtarr \sV$, there is an isomorphism
\[
B(\Psi^*Y, \sD, X)\iso Y\otimes_{\sE} B(\sE,\sD,X),
\]
where $\Psi^*Y=Y\circ \Psi$ and each $\sE_n$ is viewed as a contravariant $\sV$-functor $\sD \rtarr \sV$ by precomposition with $\Psi$.
\end{rem}

\subsection{Specializations to spaces and $G$-spaces}\label{bartoo}

Now let us return to our space level context.  The discussion just given applies with $\sV = \sU$ or $\sV=G\sU$ for any $G$,
where we take $\otimes = \times$.  With no basepoint assumptions, this gives a bar construction  $B^{\times}(Y,\sE,X)$. The discussion
also applies with $\sV = \sU_{\ast}$ or $\sV=G\sU_{\ast}$, where we take $\otimes = \sma$.   This is our preferred choice and we write
\begin{equation}\label{Bsmash}
B(Y,\sE,X) = B^{\sma}(Y,\sE,X).
\end{equation}
Focusing on $G\sU_{\ast}$, we make some standing assumptions that will be satisfied in our examples.

\begin{ass}\label{Assbar} As in \autoref{catpre}, we assume that $\sE$ has a zero object $0$, so that each $\sE(0,n)$ and each $\sE(n,0)$ is a point. Then each 
$\sE(m,n)$ has the basepoint $m\to 0 \to n$.  We assume that the functors $Y$ and $X$ are $G\sU_\ast$-enriched so that they are  given by action $G$-maps
$$ Y_n \sma \sE(m,n) \rtarr Y_m \ \ \text{and} \ \  \sE(m,n)\sma X_m \rtarr X_n.$$ By \autoref{duh}, this is equivalent to requiring $G\sU$-enriched functors such that $Y(0)$ and $X(0)$ are a point.
Finally, we assume that the $G$-spaces  $Y_n$,  $\sE(m,n)$, and $X_m$ are all nondegenerately based and that the inclusion of the identity map
\[\ast \rtarr \sE(n,n)\]
is a $G$-cofibration. This will ensure that our bar constructions are given by
the geometric realizations of Reedy cofibrant simplicial $G$-spaces in $G\sT$.
\end{ass} 

\begin{rem}\label{between} 
When $\sV=G\sU_\ast$, we can forget basepoints and consider $\sE$ as  enriched in $G\sU$, and $Y$ and $X$ as landing in $\ul{G\sU}$. The variant $B^\times$ suffers from the defect that $B^\times(\sE_0,\sE,X)$ is not a point, and hence by \autoref{duh}, the functor $B^\times(\sE,\sE,X)$ is enriched over $\sU$ but not over $\sU_\ast$. Previous work, starting with \cite{Shim} and going through earlier versions of this paper, focused on an intermediate variant, namely 
\[  B^{\times}(Y,\sE,X)/B^{\times}(\ast,\sE,X).\]
There are quotient maps relating the three variants
\begin{equation}\label{bvariants}
B^\times (Y,\sE,X) \rtarr B^\times (Y,\sE,X)/B^\times (*,\sE,X) \rtarr B (Y,\sE,X).
\end{equation}
Under our assumption that $\sE$ has a zero object, the map $\epz$ of \autoref{epz} shows that $B^\times (*,\sE,X)$ is contractible, and the nondegeneracy of the basepoints implies that the inclusion
\[B^{\times}(\ast,\sE,X) \rtarr B^{\times}(Y,\sE,X)\] 
is a $G$-cofibration, showing that the first map is an equivalence. The second map is not in general an equivalence, but by \cite[Theorem 3.19]{GMMO}, it gives an equivalence of the associated Segal machines when $\sE=\sF_G$ and $X$ is special. The choice of  \autoref{Bsmash} has distinct technical advantages.   Detailed comparisons are given in \cite{GMMO}.\footnote{Although published earlier, \cite{GMMO} began as a sequel to this paper.}
\end{rem}

The following specialization of \autoref{approximation} makes precise the idea that $B(\sE,\sE,X)$ is an approximation of $X$.  We write it equivariantly for definiteness, under \autoref{Assbar}.

\begin{prop}\label{GOODbar}  For every object $n$ in $\sE$ there are maps
\[\epz_n\colon B(\sE_n,\sE,X) \rtarr X_n  \ \ \text{and}  \ \ \et_n \colon X_n \rtarr B(\sE_n,\sE,X).\]
 such that $\epz_n\et_n = \id$ and $\et_n\epz_n$ is levelwise $G$-homotopic to the identity. The maps $\epz_n$ assemble to give a $G\sU_\ast$-natural transformation between $G\sU_*$-functors $\sE \rtarr \TG$.
\end{prop}

\begin{rem}\label{UhOh}
The homotopical Segal machine is obtained by using examples of two-sided bar constructions to construct
$\sW_G$-$G$-spaces.  They give $G$-spectra by restricting the variable $Y$ to $A^{\ast}$ as in \autoref{AA} for spheres $A=S^V$.  The structure maps of these $G$-spectra come 
from comparison maps of the form
 \begin{equation}\label{barcommuteswithsmash} 
 B(Y,\sE, X) \sma C \rtarr B(Y\sma C,\sE, X). 
 \end{equation}
 These are immediate with our smash product definition of the bar construction, but we would not have them if we tried to use $B^{\times}$.
The point is that \autoref{strucW} in \autoref{structuremaps} does not apply if we use $B^{\times}$ due to the difference 
 between unbased and based enrichment.  To make this more precise, consider the relationship between smash products and products.
 For based spaces $A$, $B$, and $C$,  there is no natural map
 $$ (A\times B)\sma C \rtarr (A\sma C) \times B. $$
 Under the natural isomorphism $(A\times B)\times C \iso (A\times C) \times B$, we collapse out different subspaces to construct the source and target,
 and neither is contained in the other.  Explicitly, writing $a,b,c$ for the basepoints of $A,B,C$ and $x,y,z$ for general points of $A,B,C$, we identify 
 all points $(x,y,c)$ and $(a,b,z)$  with $(a,b,c)$ in the source, but we identify  all points $(a,z,y)$ and $(x,c,y)$ with the point $(a,c,y)$ in the target.
 \end{rem}
  
Our interest is in the case when $\sE$ is $\sF$ or $\sF_G$ or the more 
general categories of operators $\sD$ and $\sD_G$ to be introduced later. 
Note that although $\sF$ is topologically discrete and  
$G$-trivial, we still view it as a category enriched in $\sT_G$.  

 \subsection{The classical homotopical Segal machine}\label{bar2}
 
Specializing from the previous section, let
$Y\colon \sF \rtarr G\sT$ be a contravariant $G\sU_{\ast}$-functor  and  $X\colon \sF\rtarr G\sT$ be a covariant $G\sU_{\ast}$-functor. 
We then have the bar construction $B(Y,\sF,X)$.    
The action of $G$ on it is induced diagonally by the actions on the $Y_n$ and $X_n$.  It is important to note that given our assumptions, $B(Y,\sF,X)$ is a nondegenerately based $G$-space.

For $G\sU_{\ast}$-functors $Y\colon \sF_G^{op} \rtarr \sT_G$ and  $X\colon \sF_G \rtarr \sT_G$, we have the analogous two-sided bar construction $B(Y,\sF_G,X)$. Just as above, it is a nondegenerately based $G$-space.
For its construction, we must remember the action of $G$ on the finite $G$-sets $\sF_G(\bm^{\al},\bn^{\be})$.
While we are interested in general $X$, in both cases we are only interested in particular $Y$, namely those of the form 
$Y = A^{\bullet}$, as in \autoref{AA}. 

Nonequivariantly, Woolfson \cite{Woolf} constructed a homotopical Segal machine
by restricting $B(A^{\bullet}, \sF, X)$ to spheres $A= S^n$.\footnote{This is revisionist. He was writing before the two-sided bar construction was   
formally defined; technically, he worked with $B^\times$.} Equivariantly, we can apply the same construction,
taking  $X$ to be an $\sF$-$G$-space, and $A$ to be in $G\sW$.  For reasons
we now explain, this construction {\em fails} to lead to genuine $\OM$-$G$-spectra, even 
when $X$ is \gen-special. 

When $A = \mathbf{n}$,  $A^{\bullet}$ is the represented functor $\sF_n=\sF(-,\mathbf{n})$, and as $n$ varies we obtain the
$\sF$-$G$-space $B(\sF,\sF,X)$ whose $n$th $G$-space is $B(\sF_n,\sF,X)$.
We have an implicit and important action of $\SI_n$ on source and target; $\SI_n\subset \sF(\mathbf n, \mathbf n)$  
acts from the left on $\sF_n$ by 
postcomposition in $\sF$, and that induces the action on $B(\sF_n,\sF,X)$.  Observe that $B_{*}(\sF_n,\sF,X)$ 
is a simplicial $(G\times \SI_n)$-space and $\epz_n$ is the realization of a map  $B_{*}(\sF_n,\sF,X)\rtarr (X_n)_{*}$ of 
simplicial $(G\times \SI_n)$-spaces.  

\begin{lem}\label{weeny}  Let $X$ be an $\sF$-$G$-space.  Then  $B(\sF,\sF,X)$ is an $\sF$-$G$-space.
\end{lem}
\begin{proof}   For an injection $\ph\colon \bm\rtarr \bn$ in $\PI$,  the induced function
$$   \ph_*\colon  \sF(\mathbf q, \bm) \rtarr \sF(\mathbf q, \bn) $$
is a $(G\times \SI_{\ph})$-cofibration. Therefore the cofibration condition of \autoref{Fspace} follows from \autoref{PlenzCof}, 
applied with $G$ replaced by $G\times\SI_\phi$.
\end{proof}  

Moreover, \autoref{GOODbar} specializes to give the following result.

\begin{prop}\label{crux} Let $X$ be an $\sF$-$G$-space.  Then
the map
$$ \epz\colon B(\sF,\sF,X)  \rtarr X  $$
of $\sF$-$G$-spaces is a level $G$-equivalence, hence $X$ is special if and only if $B(\sF,\sF,X)$ is special.
\end{prop}

\begin{warn}\label{warning} Recall from \autoref{approximation} that while the maps $\epz_n$ are natural with respect to maps in $\sF$, the maps $\eta_n$ are \emph{not}.   In particular, while the map $\epz_n$ is a map of $(G\times \SI_n)$-spaces, the map
$$\et_n\colon X_n\rtarr B(\sF_n, \sF, X)$$ 
is {\em not} $\SI_n$-equivariant.  Explicitly, the 
action of $\SI_n$ occurs on $\sF(-,\bf n)$ in the target and on $X_n$ in the source, so that 
$\et(\si x) = (\id_n,\si x)$ while $\si \et(x) = (\si,x)$.  Thus there is no reason to expect $\epz$ to be a level \gen-equivalence
and no reason to expect $B(\sF,\sF,X)$ to be \gen-special even if $X$ is so.  

We now have two very different ways to get around this problem.  The solution presented here is to exploit the equivalence between $\sF$-$G$-spaces and
$\sF_G$-$G$-spaces.  The problem does not appear when using $\sF_G$, and it is simple to translate along the adjoint equivalence $(\bP,\bU)$.  With Guillou, 
we later \cite{GMMO}\footnote{We repeat that, although published earlier, that paper began as a sequel to this one.} found an alternative solution that does 
not require  use of $\sF_G$  and which has the advantage of making the Segal machine multiplicative.  However,
that solution is less convenient for the comparison with the operadic machine that we develop here.
\end{warn} 

By Propositions \ref{consist} and \ref{crux}, $B(\sF,\sF,X)$ and $X$ can be used interchangeably when
passing to classical $G$-prespectra.   When passing to genuine $G$-spectra, we must consider the two-sided bar construction obtained by replacing $\sF$ with $\sF_G$.  Here, for an $\sF_G$-$G$-space $Y$,  $B(\sF_G,\sF_G,Y)$  is defined at level $\bn^{\al}$ by replacing the left variable $\sF_G$  by the functor $\sF_G(-,\bn^{\al}) \colon \sF_G^{op}\rtarr \sT_G$ represented by $\bn^{\al}$.   \autoref{PlenzCof} implies the analog  of \autoref{weeny}.

\begin{lem}\label{weeny2}  Let $Y$ be an $\sF_G$-$G$-space.  Then  $B(\sF_G,\sF_G,Y)$ is an $\sF_G$-$G$-space.
\end{lem}

Moreover, using \autoref{iff2}, we have  the following more powerful analog of \autoref{crux}.

\begin{prop}\label{crux2}  Let $Y$ be an $\sF_G$-$G$-space. Then the map
$$\epz\colon B(\sF_G,\sF_G, Y)\rtarr Y$$
of $\sF_G$-$G$-spaces is a level $G$-equivalence, so $Y$ is special if and only if $B(\sF_G,\sF_G, Y)$ is special. 
\end{prop} 

We view $\bU B(\sF_G,\sF_G, \bP X)$ as a genuine homotopical approximation to $X$ in view of the following corollary, 
which is immediate from \autoref{compFFG}.  Note that we can identify $X$ with $\bU\bP X$ via the unit of the adjunction.

\begin{cor}\label{crux3} For any $\sF$-$G$-space $X$,  the map $\bU\epz\colon \bU B(\sF_G,\sF_G, \bP X)\rtarr X$ is an \gen-level  equivalence, hence $X$ is \gen-special if and only if $\bU B(\sF_G,\sF_G, \bP X)$ is \gen-special. 
\end{cor}

\begin{rem}\label{pedant}  We compare bar constructions along the adjoint equivalence $(\bP,\bU)$ between
$\sF$-$G$-spaces and $\sF_G$-$G$-spaces.  For $\sF$-$G$-spaces $X$, we have
\[ \bP B(\sF,\sF,X) = \sF_G\otimes_{\sF} B(\sF,\sF,X) \iso B(\sF_G,\sF,X). \]
The inclusion $\io\colon \sF\rtarr \sF_G$ induces a natural map of $\sF_G$-$G$-spaces
\[\io_{*}\colon \bP B(\sF,\sF,X) \rtarr B(\sF_G,\sF_G, \bP X) \]
such that the following diagrams commute; the second is obtained from the first by applying $\bU$.
\[  \xymatrix{ 
\bP B(\sF,\sF,X) \ar[d]_{\io} \ar[r]^-{\bP\epz} & \bP X\\
B(\sF_G,\sF_G,\bP X) \ar[ur]_{\epz} }  
\xymatrix{ 
B(\sF,\sF,X) \ar[d]_{\bU \io} \ar[r]^-{\epz} & X\\
\bU B(\sF_G,\sF_G,\bP X) \ar[ur]_{\bU \epz} \\}   \]
In the first, the diagonal arrow $\epz$ is a level $G$-equivalence, but we cannot expect $\io$ and $\bP\epz$ 
to be level $G$-equivalences since that would imply that all three arrows in the second diagram are \gen-level
equivalences, contradicting \autoref{warning}.   In the second diagram, $\bU\epz$ is an \gen-level equivalence
and the other two arrows are level $G$-equivalences.  
\end{rem}  

Using \autoref{consist}, we see that this comparison implies the following comparison of classical $G$-prespectra.

\begin{prop}\label{itsago}   Let $X$ be a special $\sF$-$G$-space.  Then the positive classical $\OM$-$G$-prespectra obtained by prolonging 
$X$ or $B(\sF,\sF,X)$ or $\bU B(\sF_G,\sF_G, \bP X)$  to $\sW_G$-$G$-spaces and then restricting to spheres $S^n$ are 
level $G$-equivalent.  Their bottom structural maps are compatible group completions of $G$-spaces equivalent to $X_1$.
\end{prop}  
\begin{proof} 
For any $A$, functoriality of prolongation applied to the second diagram of \autoref{pedant} gives a commutative diagram
\[ \xymatrix{ 
B(A^{\bullet},\sF,X) \ar[d]  \ar[r]  & X(A). \\
B(A^{\bullet},\sF_G,\bP X) \ar[ur]. \\} \]
Here we have used the factorization of prolongation from $\sF$-$G$-spaces to $\sW_G$-$G$-spaces through $\sF_G$-$G$-spaces
and the isomorphism $\bP\bU\iso \Id$, where $\bP$ is prolongation from $\sF$-$G$-spaces to $\sF_G$-$G$-spaces.
By \autoref{iff} and \autoref{pedant}, all three of our $\sF$-$G$-spaces are special.  Restricting to spheres $S^n$, we can apply 
\autoref{consist} to each of them.  Taking $A= S^0$, we see compatible weak $G$-equivalences with $X_1$, and taking $A = S^1$, we see
that the bottom structure maps are compatible group completions.   That implies that we have weak $G$-equivalences
at level $1$. In turn, since we are comparing $\OM$-$G$-prespectra, that implies that we have weak $G$-equivalences at all levels $n$.
\end{proof}

\subsection{The genuine homotopical Segal machine}\label{barFG}

The following definition gives a modernized version of Shimakawa's equivariant Segal 
machine \cite{Shim}.\footnote{Orthogonal $G$-spectra had not been developed when \cite{Shim} was
written; he worked with Lewis-May $G$-spectra.}  Recall \autoref{AA}. 

\begin{defn}\label{idstar}  Write $\bI Y = B(\sF_G,\sF_G, Y)$ for an $\sF_G$-$G$-space $Y$;  thus $\bI $ is a functor $\FsubG\rtarr \FsubG$.
The Segal machine $\bS_G^{\sF_G}$ on $\sF_G$-$G$-spaces is the composite
\[  \xymatrix@1{  \FsubG \ar[r]^-{\bI }  & \FsubG \ar[r]^-{\bP} &  
\Fun(\sW_G,\UG) \ar[r]^-{\bU_{G\sS}} & G\sS.\\} \]
More explicitly, taking $A= S^V$, 
\[  \bS_G(Y)(V) = B(\rbullv , \sF_G, Y) =  \rbullv \otimes_{\sF_G} \bI  Y.\]
The Segal machine $\bS_G^{\sF}$ on $\sF$-$G$-spaces $X$ is defined by
\[  \bS_G^{\sF}X = \bS_G^{\sF_G} \bP X.\]
\end{defn} 

The definition makes sense for any $X$.  When $X$ is special, \autoref{itsago} shows that the 
underlying classical $G$-prespectrum of $\bS_G^{\sF} X$ is equivalent to $\bS_G^C X$ of \autoref{notone}, hence is a
positive $\OM$-$G$-prespectrum with bottom structural map a group completion of $X_1$.

\begin{rem} \label{Shimgp} The group completion property is not easy to see directly from the definition of $\bS_G^{\sF_G}$.  
Shimakawa's strategy \cite[p 357]{Shim1}, not carried out in detail, was to show that for $H\subset G$, Woolfson's 
version of the  nonequivariant Segal machine $\bS (Y^H)$ is equivalent to  $(\bS_G^{\sF_G} Y)^H$, where 
$Y^H$ is the composite of restriction to $\sF$ and the $H$-fixed point functor, so that $(Y^H)_n = Y(\mb{n})^H$, and then to quote the equivalence
of Woolfson's version with Segal's original version.  In contrast, with our proof, this equivalence on fixed points follows 
formally from the group completion property, as is shown quite generally in \cite[Theorem 2.20]{GM3}.
\end{rem}

We are primarily interested in understanding $\bS_G^{\sF}X$ when $X$ is \gen-special.
Recall  that $G\sW$ and $\sW_G$ are the categories of based $G$-CW 
complexes whose respective morphisms are based $G$-maps and all based maps, with $G$ acting by conjugation.
It is standard to say that a functor is linear if it converts homotopy cocartesian squares to homotopy cartesian squares.
We will use a  variant form of the definition that is convenient for displaying a helpful distinction that is needed in our work.

\begin{defn}\label{lineardefn}  A $\sW_G$-$G$-space $Z$ is {\em linear} if  for any $G$-map $f\colon A\rtarr B$ in $G\sW$ with cofiber $i\colon B\rtarr Cf$, 
\[ \xymatrix@1{ Z(A)\ar[r]^-{f_{*}}  & Z(B) \ar[r]^-{i_{*}}   & Z(Cf)\\} \]
is a fibration sequence of based $G$-spaces.
That is, the induced map from $Z(A)$ to the homotopy fiber of $i_{*}$ is a weak $G$-equivalence. 
We say that $Z$ is {\em positive linear}  if this condition holds when $A$ is $G$-connected, but not necessarily in general.
\end{defn} 

We shall need the following companion definitions. 

\begin{defn}\label{conn1}  A $\sW_G$-$G$-space $Z$ {\em preserves connectivity} if $Z(A)$ is $G$-connected 
when $A$ is $G$-connected.
\end{defn}

\begin{defn}\label{ZA}  For a $G$-space $A$ in $\sW_G$, define a  $\sW_G$-$G$-space  $Z[A]$ by  
$$Z[A](B) = Z(A\sma B).$$
\end{defn}

In \autoref{cof}, we adapt and extend nonequivariant arguments of Segal and Woolfson
\cite{Seg, Woolf} to prove the following result.  It is perhaps surprising that we only need 
$X$ to be special, not \gen-special, for the first statement and that we do not know how to
derive either statement from the other.  However, we will only make use of the second statement in this paper.

\begin{thm}\label{Omnibus}   If $X$ is a special
$\sF$-$G$-space, then the $\sW_G$-$G$-space that sends $A$ to $B(A^{\bullet},\sF,X)$
is positive linear and preserves connectivity.  If $Y$ is a special $\sF_G$-$G$-space, such as $\bP X$ for an \gen-special
$\sF$-$G$-space $X$, then $B(A^{\bullet},\sF_G,Y)$ is positive linear and preserves connectivity.
\end{thm} 

We need the following definition to go from this result to an understanding of the $G$-spectrum given by $Y$.

\begin{defn}\label{specialWG}  Let $Z$ be a $\sW_G$-$G$-space.   We say that $Z$ is special if it preserves connectivity and
 its restriction to  $\sF_G$ is a special $\sF_G$-$G$-space.
 \end{defn}
 
We defer the proof of the following key result to \autoref{heredity}.

\begin{lem}\label{part2}  Let $Z$ be a positively linear and special $\sW_G$-$G$-space.  If $A$ is $G$-connected, then $Z[A]$ is a linear and special $\sW_G$-$G$-space.
\end{lem}

The lemma gives input to the following result, which is the technical heart of the Segal machine.  For the sake of completeness, we present a proof in \autoref{heredity}. It is a result originally due to  Segal \cite{Seg2} and Shimakawa \cite{Shim}.  The proof we present follows the blueprint presented in \cite{Seg2,Shim}, while providing more details.

\begin{thm}\label{keystone}  If $Z$ is a linear and special $\sW_G$-$G$-space, then the structure map
$$\tilde{\si}\colon Z(S^0) \rtarr  \OM^VZ(S^V)$$ 
is a weak $G$-equivalence for all representations $V$ of $G$.
\end{thm}

The following consequence is originally due to Segal \cite{Seg2}, Shimakawa \cite{Shim}, and Blumberg \cite{Blum} and builds on nonequivariant results in \cite{BF, MMSS, Puppe, Seg, Woolf}.

\begin{thm}\label{Blumberg}  Let $Z$ be a positively linear and special  $\sW_G$-$G$-space.  Then $\bU_{G\sS}Z$ is a  positive  
$\OM$-$G$-spectrum.
\end{thm}
\begin{proof}  
Let $V$ and $W$ be representations of $G$ and assume that $V^G\neq 0$.   We must prove that the adjoint structure map
 \[ \tilde{\si} \colon Z(S^V) \rtarr \OM^W Z(S^{V\oplus W}) = \OM^W Z(S^V\sma S^W) \]
 is a weak $G$-equivalence.  Since $S^V$ is $G$-connected, $Z[S^V]$ is linear and special by \autoref{part2}.  The conclusion follows by applying \autoref{keystone} with $Z$ there replaced by $Z[S^V]$  here and $V$ there replaced by $W$ here.
 \end{proof}
 
Here now is the fundamental theorem about the Segal machine.

\begin{thm}\label{bigSegal}  Let $X$ be an \gen-special $\sF$-$G$-space. 
Then $\bS_G^{\sF} X$ is a positive  $\OM$-$G$-spectrum. Moreover,  if $S^V\in \sW_G$ and $V^G\neq 0$, then the composite 
$$ X_1 \rtarr B(\sF_G,\sF_G,\bP X)_1 = (\bS_G^{\sF} X)(S^0) \rtarr \OM^V(\bS_G^{\sF}X)(S^V)$$ 
of $\et_1$ and the structure $G$-map is a group completion. 
\end{thm}
\begin{proof}  Let $Z$ be the $\sW_G$-$G$-space $B((-)^\bullet,\sF_G,\bP X)$.  Then $Z$ is positive 
linear and preserves connectivity by \autoref{Omnibus} and its restriction to $\sF$ is an \gen-special $\sF$-$G$-space by 
\autoref{crux3}.   Therefore \autoref{Blumberg} implies the first statement, and the second  follows  from \autoref{itsago} and \autoref{Vgpcomp}. 
\end{proof}

\begin{rem} Clearly \autoref{Blumberg} applies to $Z = \bP X$ when $X$ is \gen-special. Despite \autoref{Omnibus}, we have no such conclusion when $X$ is only special.  In \autoref{Segcom}, we will come back to here to give a partial generalization to compact Lie groups of all results in this section.
\end{rem}

\section{The generalized Segal machine}\label{SegGen}

The input of the Segal infinite loop space machine looks nothing like the input of the operadic machine.
To compare them, we must generalize the natural input of both to obtain common input to which generalizations
of both machines apply.  We explain the generalized Segal machine in this section, postponing consideration
of operads to the next.  We define two equivariant versions of the categories of operators introduced in \cite{MT}.   
One version has finite sets as objects, the other finite $G$-sets, generalizing $\sF$ and $\sF_G$ respectively. 
As in \autoref{sectionFFG}, we show how to construct examples of the second kind from examples of the first kind.  

After defining what it means for a $G$-category of operators to be an $E_{\infty}$ $G$-category of operators, we generalize 
the homotopical version of the Segal machine by generalizing its input from $\sF$-$G$-spaces to $\sD$-$G$-spaces, where 
$\sD$ is any $E_{\infty}$ $G$-category of operators over $\sF$.  We compare the $\sD$ and $\sD_G$ machines to the 
$\sF$ and $\sF_G$ machines by proving that they have equivalent inputs and that they produce equivalent output when fed equivalent input.  Thus the increased generality is more apparent than real. The point of the generalization is that operadic data feed naturally into the  $\sD$-$G$-space rather than the $\sF$-$G$-space machine. We reiterate that the categorical input data of the sequels 
\cite{GMMO3, MayMOM} is intrinsically operadic.

\subsection{$G$-categories of operators $\sD$ over $\sF$ and $\sD_G$ over $\sF_G$}\label{GCF}

\begin{defn}\label{GCO/F} A $G$-category of operators $\sD$ over $\sF$, abbreviated $G$-$CO$ over $\sF$, is a $G\sU_*$-category with objects the based sets $\mathbf{n}$ for $n\geq 0$, together with $G\sU_\ast$-functors
\[\xymatrix@1{ \Pi \ar[r]^-{\io} & \sD \ar[r]^-{\xi} & \sF\\} \] 
such that $\io$ and $\xi$ are the identity on objects and $\xi\com \io$ is the inclusion. Here $G$ acts trivially on $\PI$ and $\sF$. 
We say that $\sD$ is reduced if $\sD(\mathbf{m},\mathbf{n})$ is a point if either $m=0$ or $n=0$, and we restrict attention to 
reduced $G$-$COs$ over $\sF$ henceforward.   A morphism $\nu\colon \sD\rtarr \sE$ of $G$-$CO$s over $\sF$
is a $G\sU_{\ast}$-functor over $\sF$ and under $\PI$. In particular, $\xi\colon \sD\rtarr \sF$ is a map of $G$-$COs$ over $\sF$ for any 
$G$-$CO$ $\sD$ over $\sF$.
\end{defn}

Note that we have maps
\[\PI(\mathbf{m},\mathbf{n})\rtarr \sD(\mathbf{m},\mathbf{n})\rtarr \sF(\mathbf{m},\mathbf{n}) \] 
whose composite is the inclusion. We have the following parallel analogue.

\begin{defn}\label{GCO/FG} A $G$-category of operators $\sD_G$ over $\sF_G$, abbreviated $G$-$CO$ over $\sF_G$, 
is a $G\sU_*$-category, with objects the based $G$-sets $\bn^{\al}$ for $n\geq 0$ and 
$\al\colon G\rtarr \SI_n$, together with $G\sU_\ast$-functors
\[\xymatrix@1{ \Pi_G \ar[r]^-{\io_G} & \sD_G \ar[r]^-{\xi_G} & \sF_G\\} \] 
such that $\io_G$ and $\xi_G$ are the identity on objects and $\xi_G\com \io_G$ is the inclusion.  
We say that $\sD_G$ is reduced if $\sD_G(\bm^{\al} , \bn^{\be})$ is a point if either 
$m=0$ or $n=0$, and we restrict attention to reduced $G$-$COs$ over $\sF_G$ henceforward.
A morphism $\nu_G\colon \sD_G\rtarr \sE_G$ of $G$-$CO$s 
over $\sF_G$ is a $G\sU_{\ast}$-functor over $\sF_G$ and under $\PI_G$.  In particular, $\xi_G\colon \sD_G\rtarr \sF_G$ is a map of $G$-$COs$ over $\sF_G$ for any  $G$-$CO$ $\sD_G$ over $\sF_G$.
\end{defn}

We similarly have $G$-maps 
\[ \Pi_G(\bm^{\al},\bn^{\be}) \rtarr \sD_G(\bm^{\al},\bn^{\be})\rtarr \sF_G(\bm^{\al},\bn^{\be}) \]
whose composite is the inclusion.

Following \cite[Addendum 1.7]{MT}, cofibration conditions will be added to the previous  two definitions in \autoref{adden}.

  The full subcategory $\sD\subset \sD_G$ with objects $\bn$ is a $G$-category of operators over $\sF$.  Conversely, just as we constructed $\PI_G$ and $\sF_G$ from $\PI$ and $\sF$, we can construct a $G$-$CO$  
$\sD_G$ over $\sF_G$ from any $G$-$CO$ $\sD$ over $\sF$.  As noted in \cite[Proposition 6.9]{GMMO3}, up to isomorphism
all $\sD_G$ can be constructed in this fashion.

\begin{con}\label{DtoDG}  Let $\sD$ be a $G$-category of operators over $\sF$.  We define a $G$-category of operators
$\sD_G$ over $\sF_G$ whose full subcategory of objects $\mathbf{n}$ is $\sD$.  The morphism $G$-space
$\sD_G(\bm^{\al},\bn^{\be})$ is the space $\sD(\mathbf{m},\mathbf{n})$, with $G$-action induced by conjugation
and the original $G$-action on $\sD(\mathbf{m}, \mathbf{n})$.  Explicitly, for $f\in \sD_G(\bm^{\al}, \bn^{\be})$, 
\[ g\cdot f= \beta(g) \circ (gf) \circ \al(g^{-1});\]
We check that $g\cdot (h\cdot f) = (gh)\cdot f$ using that $G$ acts trivially on permutations since they 
are in the image of $\PI$.  Composition and identity maps are inherited from $\sD$ and are appropriately equivariant.
\end{con}

A routine verification shows the following.

\begin{lem}\label{littlediagram}
The inclusion $\sD \hookrightarrow \sD_G$ makes the following diagram of $G\sU_\ast$-categories commute.
\begin{equation}\label{littlediag}
\xymatrix{
\Pi \ar[r] \ar@{^{(}->}[d] & \sD \ar[r] \ar@{^{(}->}[d] & \sF \ar@{^{(}->}[d]\\
\Pi_G \ar[r] & \sD_G \ar[r] & \sF_G.\\}
\end{equation}
Moreover, a map $\nu\colon \sD \rtarr \sE$ of $G$-COs over $\sF$ induces a map
\[\nu_G \colon \sD_G \rtarr \sE_G\]
of $G$-COs over $\sF_G$ that is compatible with the inclusions.
\end{lem} 

The following cofibration conditions will hold automatically for categories of operators constructed from operads. 

\begin{adden}\label{adden} We add the following cofibration conditions to the definition of a  $G$-CO $\sD_G$ over $\sF_G$. 
\begin{enumerate}[(i)]
\item  The morphism $G$-spaces $\sD_G(\bm^{\al},\bn^{\be})$ are nondegenerately based.
\item  The inclusions $*\rtarr \sD_G(\bn^{\al},\bn^{\al})$ of identity maps are $G$-cofibrations. 
\item  The map $\sD_G(\bq^{\al},\bm) \rtarr \sD_G(\bq^{\al},\bn)$ induced by an injection $\ph\colon \bm \rtarr \bn$ in $\PI\subset \PI_G$ is a $(G\times \SI_{\ph})$-cofibration, where $\SI_{\ph}$ is as defined in \autoref{Sigph}.  
\end{enumerate}
We add to the definition of a $G$-CO $\sD$ over $\sF$ the requirement that its prolonged category of operators $\sD_G$ satisfies the conditions just specified.   Note that these conditions imply in particular that $\sD$ satisfies the following.
\begin{enumerate}[(i)]
\item  The morphism $G$-spaces $\sD(\bm,\bn)$ are nondegenerately based.
\item  The inclusions $*\rtarr \sD(\bn,\bn)$ of identity maps are $G$-cofibrations. 
\item  The map $\sD(\mathbf{q},\mathbf{m}) \rtarr \sD(\mathbf{q},\mathbf{n})$ induced by an injection $\ph\colon \bm \rtarr \bn$ is a $(G\times \SI_{\ph})$-cofibration.  
\end{enumerate}
\end{adden}

We have the following generalizations of  Definitions \ref{Fspace},  \ref{weakFn}, \ref{FGspace} and \ref{PIGSpec} from $\sF$ and $\sF_G$ to $G$-categories of operators $\sD$ and $\sD_G$.

\begin{defn}\label{Dspace}
A $\sD$-$G$-space $X$ is a $G\sU_{\ast}$-functor $X\colon\sD\rtarr  \UG$ whose restriction to $\PI$ is a $\PI$-$G$-space.  
A map of $\sD$-$G$-spaces is a
$G\sU_{\ast}$-natural transformation. 
We say that $X$ is \gen-special if its underlying $\PI$-$G$-space is \gen-special.  We say that a map of $\sD$-$G$-spaces is an \gen-level equivalence if its underlying map of $\PI$-$G$-spaces is an \gen-level equivalence.  Let $\DdashG$ denote the
category of $\sD$-$G$-spaces. 
\end{defn}

\begin{defn}\label{DGspace}
A $\sD_G$-$G$-space $Y$ is a $G\sU_{\ast}$-functor $Y\colon\sD_G\rtarr \ul{G\sU_*}$ whose restriction to $\PI_G$ is a $\PI_G$-$G$-space. A map of $\sD_G$-$G$-spaces is a $G\sU_{\ast}$-natural transformation. 
We say that $Y$ is special if its underlying $\PI_G$-$G$-space is special.  We say that a map of $\sD_G$-$G$-spaces
is a level $G$-equivalence if its underlying map of $\PI_G$-$G$-spaces is a level $G$-equivalence.
Let $\DsubG$ denote the category of $\sD_G$-$G$-spaces. 
\end{defn}

Generalizing \autoref{MTbeard}, we show in \autoref{beard} that the implied cofibration assumptions result in no loss of generality.   Condition (iii) in \autoref{adden} ensures that the following analogues of \autoref{weeny} hold. In the context of maps  $\nu\colon \sD \rtarr \sF$ or $\nu_G\colon \sD_G\rtarr \sE_G$, they have evident generalizations to bar constructions $B(\sE,\sD,X)$ and $B(\sE_G,\sD_G, Y)$. 

\begin{lem}\label{weeny3}  Let $X$ be a $\sD$-$G$-space.  Then  $B(\sD,\sD,X)$ is a $\sD$-$G$-space.
\end{lem}

\begin{lem}\label{weeny4}  Let $Y$ be a $\sD_G$-$G$-space.  Then  $B(\sD_G,\sD_G,Y)$ is a $\sD_G$-$G$-space.
\end{lem}

\subsection{The equivalence between $\DdashG$ and $\DsubG$}\label{DDGSec} 
We can now generalize \autoref{sectionFFG} to a comparison between  $\sD$-$G$-spaces and $\sD_G$-$G$-spaces.
The forgetful functor 
\[ \bU\colon \mathrm{Fun}(\sD_G, \UG) \rtarr \mathrm{Fun}(\sD, \UG) \]
has a left adjoint prolongation functor 
\[\bP\colon \mathrm{Fun}(\sD, \UG) \rtarr \mathrm{Fun}(\sD_G, \UG).\]
Explicitly, 
\[ (\mathbb{P}X)(\bn^{\al})= \sD_G(-, \bn^{\al})\otimes_\sD X = \bigvee_{m}  \sD_G(\mathbf{m},\bn^{\al}) \sma X_m /\sim, \] 
where $(f, \phi_{*}x)\sim (f\phi, x)$ for a map $\phi\colon \mathbf{k} \rtarr \mathbf{m}$ in $\sD$, an element $x\in X_k$, and a map 
$f\colon \mathbf{m} \rtarr \bn^{\al}$ in $\sD_G(\mathbf{m}, \bn^{\al})$.  (We have written out
this coequalizer of $G$-spaces explicitly to facilitate checks of details).
The following result generalizes \autoref{compFFG} from $\sF$ to an arbitrary $G$-$CO$ over $\sF$.

\begin{thm}\label{compGGG}
The adjoint pair of functors 
\[ \xymatrix@1{\Fun(\sD,\UG) \ar@<0.5ex>[r]^{\bP}  &  \Fun(\sD_G,\UG) \ar@<0.5ex>[l]^{\bU} }\\ \]
specifies an equivalence of categories. Remembering the cofibration condition from Definitions \ref{Dspace} and \ref{DGspace},  this equivalence restricts to an equivalence
\[ \xymatrix@1{\DdashG \ar@<0.5ex>[r]^{\bP}  &  \DsubG. \ar@<0.5ex>[l]^{\bU} }\\ \]
\end{thm}

The proof is very similar to that of the special case $\sD=\sF$ dealt with in \autoref{compFFG}.
Only points of equivariance require comment, and the following lemma is the key to understanding
the relevant $G$-actions.  It identifies the $G$-space $(\bP X)(\bn^{\al})$ with the $G$-space $X_n^{\al}$ of \autoref{Xal}.  

\begin{lem}\label{actact}  For a $G\sU_\ast$-functor $X\colon \sD \rtarr \UG$, the $G$-space $(\bP X)(\bn^{\al})$ is $G$-homeomorphic to the 
$G$-space $X_n^{\al}$, namely $X_n$ with the $G$-action $\cdot_{\al}$ specified by  $g\cdot_\al x= \al(g)_{*}(gx)$. Via this 
homeomorphism, the evaluation maps 
\[\sD_G(\bn^{\al},\mb{p}^{\be})\sma (\bP X)(\bn^{\al}) \rtarr (\bP X)(\mb{p}^{\be})\]
are given on the underlying spaces by the corresponding maps for $X$,
\[\sD(\mb{n},\mb{p})\sma X_n \rtarr X_p.\]
\end{lem} 

\begin{proof}  Modulo equivariance, this is an application of the Yoneda lemma. Write $\id_{\al}\colon \mathbf{n}\rtarr \bn^{\al}$ for 
$\id\colon \mathbf{n}\rtarr \mathbf{n}$ regarded as an element of $\sD_G(\mathbf{n}, \bn^{\al})$.  Define
\[
F\colon X_n^\al \rtarr \sD_G(-,\bn^{\al})\otimes_\sD X 
\]
by sending $x\in X_n$ to the equivalence class of $(\id_{\al}, x)$. Then $F$ is a $G$-map since 
$$ F(g\cdot_\al x) = (\id_{\al},g\cdot_\al x) = (\id_{\al}, \al(g)_{*}(gx))\sim (\id_{\al} \circ \al(g),gx)=(g\id_{\al},gx).$$
Define an inverse map 
$$F^{-1}\colon \sD_G(-,\bn^{\al})\otimes_\sD X \rtarr X_n^\al$$ 
by sending the equivalence class of $(f,x)\in  \sD_G(\mathbf{m}, \bn^{\al})\sma X_m$ to $f_{*}(x)$, where we think of
$f$ as a map $\mathbf{m}\rtarr \mathbf{n}$ in $\sD$ and interpret $f_{*}(x)$ to mean $X(f)(x)\in X_n$. Note that $F^{-1}$ is well defined.
We have 
$$F^{-1}F(x)=F^{-1}(\id_{\al},x)=\mathrm{id}_{*}x=x$$ 
and 
$$FF^{-1}(f,x)=F(f_{*}x)=(\id_{\al}, f_{*}x)\sim (f,x),$$ 
hence $F$ and $F^{-1}$ are inverse homeomorphisms.  Note that $F^{-1}$ is automatically a $G$-map since it is inverse
to the $G$-map $F$. The compatibility with the action of $\sD_G$ is clear.
\end{proof}  
Using this, we mimic the proof of \autoref{compFFG} to prove the equivalence of the categories of $\sD$-$G$-spaces and $\sD_G$-$G$-spaces.
  
\begin{proof}[Proof of \autoref{compGGG}]
Since the inclusion $\sD \rtarr \sD_G$ is full and faithful, the unit $X\rtarr \bU\bP X$ of the adjunction is an isomorphism for any $G\sU_\ast$-functor $X\colon \sD\rtarr \UG$.
Let $Y$ be a $G\sU_\ast$-functor $\sD_G\rtarr \UG$.  We must show that the counit $\bP\bU Y\rtarr  Y$ of the adjunction is an isomorphism.
A check of definitions shows that the counit $G$-map $(\bP\bU Y)(\bn^{\al})\rtarr Y(\bn^{\al})$ agrees under the isomorphism 
of \autoref{actact} with the map, necessarily a $G$-map, 
$${\id_{\al}}_{*}\colon Y_n^\al \rtarr Y(\bn^{\al})$$ 
induced by the morphism $\id_{\al} \in \sD_G(\mathbf{n}, \bn^{\al})$. 
Writing $_{\al}\!\id\colon \bn^{\al}\rtarr \mathbf{n}$ for $\id\colon \mathbf{n}\rtarr \mathbf{n}$ regarded as an element of 
$\sD_G(\bn^{\al},\mathbf{n})$, we see that 
$_{\al}\!\id$ induces the inverse homeomorphism
$${_{\al}\!\id}_{*}\colon Y(\bn^{\al})\rtarr Y_n^{\al}$$ 
to ${\id_{\al}}_{*}$.  Again, ${_{\al}\!\id}_{*}$ is automatically a $G$-map since it is inverse to a $G$-map.

The restriction of the equivalence to the subcategories of $\sD$-$G$-spaces and $\sD_G$-$G$-spaces follows directly from the definitions of our cofibration conditions, since these depend only on the underlying $\PI$-$G$-spaces.
\end{proof}

Just as for $\sF$-$G$-spaces in \autoref{sectionFFG}, a $\sD$-$G$-space $X$ has two 
$\PI_G$-$G$-spaces associated to it.  We can either apply $\bP$ to its underlying $\PI$-$G$-space or we can apply 
$\bP$ to $X$ and take its underlying $\PI_G$-$G$-space.  \autoref{actact} implies that these two $\PI_G$-$G$-spaces 
coincide.  Therefore the four statements about $\PI$ and $\PI_G$ that are listed in \autoref{compFFG} also hold for 
$\sD$ and $\sD_G$.  We record them in the following two corollaries.

\begin{cor}\label{DFstarspec} A $\sD$-$G$-space $X$ is \gen-special if and only if the $\sD_G$-$G$-space $\mathbb{P}X$ is special.  
A $\sD_G$-$G$-space $Y$ is special if and only if the $\sD$-$G$-space $\mathbb{U}Y$ is \gen-special.  
\end{cor}

\begin{cor}\label{Dleveleq}  A map $f$ of $\sD$-$G$-spaces is an \gen-level equivalence if and only $\bP f$ is a level 
$G$-equivalence of $\sD_G$-$G$-spaces.  A map $f$ of $\sD_G$-$G$-spaces is a level $G$-equivalence if and 
only if $\bU f$ is an \gen-level equivalence of $\sD$-$G$-spaces.
\end{cor}

\subsection{Comparisons of $\sD$-$G$-spaces and $\sE$-$G$-spaces for $\nu\colon \sD\rtarr \sE$}
\label{compareinSeg} 

Precomposition with $\xi\colon \sD \rtarr \sF$ induces a functor $\xi^\ast \colon \Fun(\sF,\UG) \rtarr \Fun(\sD,\UG)$, and the parallel conditions in Definitions \ref{Fspace} and \ref{Dspace} imply that $\xi^\ast$ restricts to give a functor $\FdashG \rtarr \DdashG$. Moreover, an $\sF$-$G$-space $X$ and the $\sD$-$G$-space $\xi^\ast X = X \com \xi$ have the same underlying $\PI$-$G$-space, hence one is \gen-special or special if and only if the other is so. 
The functor $\xi^\ast$ has a left adjoint given by the left Kan extension
along $\xi$, but that is not well behaved homotopically.  Instead, following \cite[Theorem 1.8]{MT} nonequivariantly, we expect the bar construction to give a variant that is homotopically well-behaved. 
With $\sF$ replaced by $\sD$, the analogue of \autoref{crux} holds and
admits the same proof.

To implement this strategy, we start with an  \gen-special  $\sD$-$G$-space $Y$ and construct from it
an \gen-special  $\sF$-$G$-space  $\xi_{*}Y$ together with a zigzag of  \gen-equivalences between $Y$
and $\xi^*\xi_{*} Y$.  We shall use this to construct a Segal machine whose input is an
\gen-special $\sD$-$G$-space $Y$ and whose output is equivalent to $\bS_G^{\sF} \xi_{*} Y$. 

As in \cite{MT}, we work more generally here, starting from a map $\nu\colon \sD\rtarr \sE$ 
of $G$-COs over $\sF$ and comparing $\sD$-$G$-spaces and $\sE$-$G$-spaces.    
We are mainly interested in the case $\nu=\xi$.  We write
$\nu_G\colon \sD_G \rtarr \sE_G$ for the induced map of $G$-COs over $\sF_G$. 
Focus on $\nu_G$ rather than $\nu$ allows us to focus on $G$-equivalence rather
than \gen -equivalence.
For clarity, we sometimes write $\bU_{\sD}$ and $\bP_{\sD}$ instead of  
$\bU$ and $\bP$ for the adjunction between $\Fun(\sD,\UG)$ and $\Fun(\sD_G,\UG)$, 
and similarly for $\sE$.  

\begin{defn}\label{veestar} For $Z\in \Fun(\sD_G,\UG)$, 
define $\nu^G_{*} Z\in \Fun(\sE_G,\UG)$ 
by
\[\nu^G_{*} Z = B(\sE_G,\sD_G,Z).\]
Here the target is defined levelwise by replacing $\sE_G$ by the composite
$$\sE_G(-,\bn^{\al})\com \nu_G \colon \sD^{op}\rtarr \sT_G,$$
of $\nu_G$ and the $G\sU_{\ast}$-functor $\sE_G^{op}\rtarr \sT_G$ represented by $\bn^{\al}$. 
If $Z$ is moreover a $\sD_G$-$G$-space, so that each $Z(\bn^\al)$ is nondegenerately based, then all of the conditions in \autoref{Assbar} are satisfied, and the bar construction is the geometric realization of a Reedy cofibrant simplicial $G$-space. Moreover, as in \autoref{weeny}, condition (iii) of \autoref{adden} implies that $\nu^G_\ast Z$ is indeed a $\sE_G$-$G$-space. We thus obtain a functor
\[\nu_\ast^G \colon \DsubG \rtarr \EsubG.\]
For  $Y\in \Fun(\sD,\UG)$, define ${\nu}_{*} Y\in \Fun(\sE,\UG)$ by 
$$ \nu_{*}Y = \bU_{\sE} \nu^G_{*} \bP_{\sD} Y =  \bU_{\sE}  B(\sE_G,\sD_G,\bP_{\sD} Y).$$
If $Y$ is moreover a $\sD$-$G$-space, we are again in the situation of \autoref{Assbar}. We thus obtain a functor
\[\nu_*\colon \DdashG \rtarr \EdashG.\] 
\end{defn}

\begin{defn}\label{nuequiv}
A map $\nu_G\colon \sD_G\rtarr \sE_G$ of $G$-$CO$s over $\sF_G$ is a $G$-equivalence if each map
$$\nu_G\colon \sD_G(\bm^{\al},\bn^{\be}) \rtarr \sE_G(\bm^{\al},\bn^{\be})$$
is a weak $G$-equivalence.  A map $\nu\colon \sD\rtarr \sE$ of $G$-$CO$s over $\sF$ is 
an \cof -equivalence if the associated map $\nu_G\colon \sD_G\rtarr \sE_G$ of $G$-$CO$s over $\sF_G$
is a $G$-equivalence.  
\end{defn}

Recall the notion of an \gen-level equivalence of $\PI$-$G$-spaces from 
\autoref{weakFn}.
Recall too that a map of $\PI$-$G$-spaces
is an \gen-level equivalence if and only if its associated map of $\PI_G$-$G$-spaces is a level $G$-equivalence; see 
\autoref{PIGSpec} and \autoref{compFFG}(iv).  These definitions and results are inherited by 
$\sD$ and $\sD_G$ (as in \autoref{Dleveleq}).  The following definition recalls notation from \autoref{bar}.

\begin{defn}  For a $G$-CO $\sD$ over $\sF$ and a fixed $n$, let $\sD_n$ be the corepresented 
$\sD$-$G$-space specified by $\sD_n(\mathbf p) = \sD(\mathbf n,\mathbf p)$, with the action of $\sD$ 
given by composition.  Analogously, we have corepresented $\sD_G$-$G$-spaces ${(\sD_G)}_{{\bn^{\al}}}$.
\end{defn}

\begin{lem}\label{easily} If $\nu\colon \sD\rtarr \sE$ is an \cof-equivalence of $G$-COs over $\sF$, 
then for each $n$, $\nu$ restricts to an \gen-level equivalence of $\sD$-$G$-spaces 
$\sD_n \rtarr \nu^*\sE_n$. 
\end{lem}
\begin{proof} One can check that the cited restriction is a map of $\sD$-$G$-spaces.  Moreover,
an easy comparison of definitions shows that  $\bP\sD_n$ can be identified
with ${(\sD_G)}_{\mathbf{n}}$.  Therefore the conclusion follows from \autoref{Dleveleq} and 
our definition of an \cof -equivalence of $G$-COs over $\sF$, which of course was chosen in order to make
this and cognate results true.
\end{proof}

\begin{thm}\label{nuequiv2} Let $\nu_G\colon \sD_G\rtarr \sE_G$ be a $G$-equivalence of $G$-COs over $\sF_G$.  
\begin{enumerate}[(i)]
\item Let $X$ be an $\sE_G$-$G$-space and $Y$ a $\sD_G$-$G$-space. Then there are natural zigzags of
level $G$-equivalences between $\nu^G_{*}\nu_G^* X$ and $X$ and between $\nu_G^{\ast}\nu^G_{*} Y$ and $Y$.
\item $Y$ is a special $\sD_G$-$G$-space if and only if $\nu^G_{*} Y$ is a special $\sE_G$-$G$-space.
\item A map $f$ of $\sD_G$-$G$-spaces is a level $G$-equivalence if and only if $\nu^G_{*} f$ is a level $G$-equivalence
of $\sE_G$-$G$-spaces.
\end{enumerate}
\end{thm} 
\begin{proof} Abbreviate notation here by writing $\al$ for a finite $G$-set $\bn^{\al}$ and $X_{\al}$ for $X(\bn^{\al})$ when $X$ is an
$\sE_G$-space. Starting with $X$, we have the natural maps 
\begin{equation}
\xymatrix@1{ (\nu^G_{*}\nu_G^{\ast}X)_{\al} = B(\sE_G,\sD_G,\nu_G^{\ast}X)_{\al} \ar[rr]^-{B(\id,\nu_G,\id)}
& & B(\sE_G,\sE_G,X)_{\al} \ar[r]^-{\epz} & X_{\al}. \\}
\end{equation}
Starting with $Y$ we have the natural maps
\begin{equation}\label{Ycomp}
\xymatrix@1{ (\nu_G^{\ast}\nu^G_{*} Y)_{\al} =  \nu_G^{\ast}B(\sE_G,\sD_G,Y)_{\al} & &
 \ar[r]^-{\epz} B(\sD_G,\sD_G, Y)_{\al} \ar[ll]_-{B(\nu_G,\id,\id)} & Y_{\al} .\\}
\end{equation}
The maps $\epz$ with targets $X_{\al}$ and $Y_{\al}$ are $G$-equivalences, with the usual inverse equivalences $\et$,
as in the proof of \autoref{crux}.  At each level $\al$, the other two maps are realizations of levelwise simplicial $G$-equivalences, 
and the Reedy cofibrancy of the simplicial bar construction ensures that these realizations are themselves $G$-equivalences, 
by \autoref{PlenzWeak}.
Note that we do not need $X$ or $Y$ to be special to prove (i).  

By \autoref{iff2}, (i) implies (ii). 
Indeed, $Y$ is special if and only if $\nu_G^{\ast}\nu^G_{*}Y$ is special and, since $\nu_G^{\ast}\nu^G_{*}Y$ and $\nu^G_{*}Y$ 
have the same underlying $\PI$-$G$-space, one is special if and only if the other is so.  Part (iii) follows from the naturality of
the $G$-equivalences in \autoref{Ycomp}.
\end{proof}

As in \autoref{littlediagram}, the following diagram of $G\sU_\ast$-categories commutes.
\[ \label{littlediag2}
\xymatrix{
 \PI  \ar@{^{(}->}[d] \ar[r] & \sD \ar[r]^-{\nu} \ar@{^{(}->}[d] & \sE \ar@{^{(}->}[d] \ar[r]^{\xi} & \sF \ar@{^{(}->}[d]\\
\PI_G \ar[r]  &  \sD_G \ar[r]_{\nu_G} & \sE_G \ar[r]_{\xi_G} & \sF_G.\\}\]
Therefore  $\nu^*\bU_{\sE} = \bU_{\sD}\nu_G^*$, as we shall use in the proof of the following
analogue of \autoref{nuequiv2} for $G$-COs over $\sF$.

\begin{thm} \label{nuequiv4}  Let $\nu\colon \sD\rtarr \sE$ be an \gen-equivalence of $G$-COs over $\sF$. 
\begin{enumerate}[(i)]
\item Let $X$ be an $\sE$-$G$-space and $Y$ be a $\sD$-$G$-space. Then there are natural zigzags of
\gen-equivalences between ${\nu}_{*}\nu^* X$ and $X$ and between $\nu^*\nu_{*} Y$ and $Y$.
\item $Y$ is an \gen-special $\sD$-$G$-space if and only if $\nu_{*} Y$ is an \gen-special $\sE$-$G$-space.
\item A map $f$ of $\sD$-$G$-spaces is an \gen-level equivalence if and only if $\nu_{*} f$ is an \gen-level equivalence
of $\sE$-$G$-spaces.
\end{enumerate}
\end{thm}
\begin{proof} Recall from \autoref{compGGG} that $(\bP_{\sD},\bU_{\sD})$
is an adjoint equivalence of categories, and similarly for $\sE$.  
Note that we have the following sequence of natural isomorphisms of $\sD_G$-$G$-spaces
\[
\bP_{\sD} \nu^* X \cong \bP_{\sD} \nu^* \bU_{\sE}\bP_{\sE} X =\bP_{\sD}  \bU_{\sD}\nu_G^*\bP_{\sE} X\cong \nu_G^* \bP_{\sE} X.
\]

To prove (i), write ${\htp}$ to indicate a zigzag of \gen-level equivalences.  Recall from \autoref{Dleveleq} 
that $\bP$ takes \gen-level equivalences to level $G$-equivalences and $\bU$ takes level $G$-equivalences 
to \gen-level equivalences, while $\nu^G_{*}$ preserves level $G$-equivalences by \autoref{nuequiv2}(iii).
Therefore, by \autoref{nuequiv2}(i),  we have the zigzags
\[  \nu_{*}\nu^* X = \bU_{\sE} \nu^G_{*}\bP_{\sD} \nu^* X  \iso
 \bU_{\sE} \nu^G_{*}\nu_G^* \bP_{\sE} X \htp \bU_{\sE} \bP_{\sE} X \iso X  \] 
and 
\[   \nu^*\nu_{*} Y = \nu^*\bU_{\sE} \nu^G_{*}\bP_{\sD}Y  = \bU_{\sD}\nu_G^* \nu^G_{*}\bP_{\sD}Y \htp 
\bU_{\sD} \bP_{\sD}Y \iso Y \]
of level $G$-equivalences.  Using \autoref{Dleveleq}, (ii) and (iii) follow as in the proof of \autoref{nuequiv2}.
\end{proof} 

\subsection{Comparisons of inputs and outputs of the generalized Segal machine}\label{inandout}

\begin{defn}\label{Einf}  We say that a $G$-$CO$ $\sD_G$ over $\sF_G$ is an $E_{\infty}$ $G$-$CO$ 
if the map $\xi_G\colon \sD_G\rtarr \sF_G$ is a $G$-equivalence.
We say that a $G$-$CO$ $\sD$ over $\sF$ is an $E_{\infty}$ $G$-$CO$ if its associated $\sD_G$ is an $E_{\infty}$ $G$-$CO$ over $\sF_G$. 
\end{defn}

The term ``$E_{\infty}$'' is convenient, but it is a little misleading, as will become clear when we turn to operads. 

We assume throughout this section that $\sD$ is an $E_{\infty}$ $G$-$CO$ over $\sF$, and we specialize the results 
of the previous section to $\xi\colon \sD\rtarr \sF$ and $\xi_G\colon \sD_G\rtarr \sF_G$.  
The following results are just specializations of Theorems \ref{nuequiv2} and \ref{nuequiv4}.

\begin{thm}\label{Segalin2}  The following conclusions hold.
\begin{enumerate}[(i)]
\item Let $X$ be an $\sF_G$-$G$-space and $Y$ a $\sD_G$-$G$-space. Then there are natural zigzags of level
$G$-equivalences between $\xi^G_{*}\xi_G^* X$ and $X$ and between $\xi^*_G\xi^G_{*} Y$ and $Y$.
\item $Y$ is a special $\sD_G$-$G$-space if and only if $\xi^G_{*} Y$ is a special $\sF_G$-$G$-space.
\item A map $f$ of $\sD_G$-$G$-spaces is a level $G$-equivalence if and only if $\xi^G_{*} f$ is a level $G$-equivalence 
of $\sF_G$-$G$-spaces.
\end{enumerate}
\end{thm}

\begin{thm}\label{Segalin4} The following conclusions hold. 
\begin{enumerate}[(i)]
\item Let $X$ be an $\sF$-$G$-space and $Y$  a $\sD$-$G$-space. Then there are natural zigzags of
\gen-equivalences between ${\xi}_{*}\xi^* X$ and $X$ and between $\xi^*\xi_{*} Y$ and $Y$.
\item $Y$ is an \gen-special $\sD$-$G$-space if and only if $\xi_{*} Y$ is an \gen-special $\sF$-$G$-space.
\item A map $f$ of $\sD$-$G$-spaces is an \gen-level equivalence if and only if $\xi_{*} f$ is an \gen-level equivalence
of $\sF$-$G$-spaces.
\end{enumerate}
\end{thm}

Therefore the pair of functors 
\[ \xi_G^*\colon \FsubG \rtarr \DsubG \ \ \text{and}\ \  
\xi^G_{*}\colon \DsubG\rtarr \FsubG\] 
induce inverse equivalences between the homotopy categories 
obtained by inverting the respective level $G$-equivalences, and this remains
true if we restrict to special objects.  Similarly, the pair of functors
\[ \xi^*\colon \FdashG \rtarr \DdashG \ \ \text{and}\ \  
\xi_{*}\colon \DdashG\rtarr \FdashG\] 
induce inverse equivalences between the homotopy categories 
obtained by inverting the respective \gen-level equivalences, and this 
remains true if we restrict to \gen-special objects.  
We conclude that the four
input categories for Segal machines displayed in the square of the following 
diagram are essentially equivalent.  Formally, as we shall see in \autoref{inmodel}, we can reinterpret this comparison in terms of 
Quillen equivalences of model categories.

\begin{equation}\label{InputDiag}
\xymatrix{
\FdashG \ar[d]_{\xi^*}& \ar[l]_-{\bU}  \FsubG \ar[d]^{\xi^*_G} \\
\DdashG & \DsubG \ar[l]^-{\bU}  \\} 
\end{equation}
Here $\xi^*\bU =\bU\xi_G^*$.  We regard the four categories in the square as possible 
domain categories for generalized Segal infinite loop space machines.  

 By Theorems \ref{compFFG} and \ref{compGGG},
the inclusions of $\sF$ in $\sF_G$ and $\sD$ in $\sD_G$ induce equivalences of categories, denoted $\bU$ in the diagram.
Their left adjoints $\bP$ give inverses, and the functors $\bU$ and $\bP$ preserve the relevant special objects and levelwise equivalences. Theorems \ref{Segalin2} and \ref{Segalin4} show that the vertical arrows $\xi^*$ and $\xi_G^*$ become equivalences of homotopy categories with inverses $\xi_{*}$ and $\xi^{G}_{*}$ (not the left adjoints) after inverting the relevant equivalences.  Consider the following diagram of functors.
\begin{equation}\label{InputDiag2}
\xymatrix{
\FdashG  \ar[r]^-{\bP} & \FsubG \ar[r]^-{\bP}  & 
\Fun(\sW_G,\UG)   \ar[r]^-{\bU_{G\sS}} & G\sS \\
\DdashG \ar[u]^{{\xi}_{*}} \ar[r]_-{\bP} & \DsubG \ar[u]_{\xi^G_{*}} 
 }
\end{equation}

The isomorphism $\bP\bU\iso \id$ on $\sF_G$-$G$-spaces and the definitions of $\xi_{*}$ and $\xi^G_{*}$ 
imply that the square commutes up to natural isomorphism.  The composite in the top row is our original conceptual Segal 
machine. We can specialize by letting $\sD=\sF$,
$\xi=\id\colon \sF\rtarr \sF$ and $\xi_G=\id\colon \sF_G\rtarr \sF_G$.  According to \autoref{idstar}, writing $\bI$ for $\id_\ast$, the composite
$\bU_{G\sS}\bP\bI $  is the Segal machine $\bS_G^{\sF_G}$ on 
$\sF_G$-$G$-spaces $Y$ and the composite  $\bU_{G\sS}\bP\bI \bP$ is the Segal machine $\bS_G^{\sF}$ on $\sF$-$G$-spaces $X$. That is
\begin{equation}\label{Fmach}
\bS_G^{\sF_G} Y = \bU_{G\sS}\bP\bI  Y \quad \text{and} \quad \bS_G^{\sF} X =\bU_{G\sS}\bP \bI \bP X \iso \bU_{G\sS}\bP\bP \bI X.
\end{equation}

We define the composites obtained by replacing the functors $\bI=\id_{\ast}$ with  $\xi_{*}$ and $\xi^G_{*}$ as generalized Segal machines $\bS_G^{\sD}$ defined on 
$\sD$-$G$-spaces  and $\bS_G^{\sD_G}$ defined on $\sD_G$-$G$-spaces.  

\begin{defn}\label{Dmach}
Let $X$ be a $\sD$-$G$-space and $Y$ a $\sD_G$-$G$-space. The generalized Segal machine is defined as
\[
\bS_G^{\sD_G} Y = \bU_{G\sS}\bP\xi^G_{*} Y \quad \text{and} \quad \bS_G^{\sD} X =\bU_{G\sS}\bP\xi^G_{*}\bP X \iso \bU_{G\sS}\bP\bP \xi_{*}X.
\]
\end{defn}

Clearly the machines starting with $\sF$-$G$-spaces or $\sF_G$-$G$-spaces are equivalent
and similarly for $\sD$ and $\sD_G$.  \autoref{Segalouttoo} below will show that the machines starting with  $\sF$-$G$-spaces or $\sD$-$G$-spaces
and the machines starting with  $\sF_G$-$G$-spaces or $\sD_G$-$G$-spaces are equivalent.   Thus the four machines in sight
are equivalent under our equivalences of input data.  That is, we obtain equivalent output by starting at any of the
four vertices of the square, converting input data from the other three vertices to that one, and taking
the relevant machine $\bS_G$.  We use the following invariance principle in the proof of \autoref{Segalouttoo}. 

\begin{prop}\label{Reedyprop1} The following conclusions about homotopy invariance hold.
\begin{enumerate}[(i)]
\item  If $f\colon X\rtarr Y$ is a level $G$-equivalence of $\sD$-$G$-spaces, 
then 
$$f\colon B(A^{\bullet},\sD, X) \rtarr B(A^{\bullet},\sD, Y)$$
is a weak $G$-equivalence for all $A\in G\sW$. 
\item If $f\colon X\rtarr Y$ is a level $G$-equivalence of 
$\sD_G$-$G$-spaces, then the induced map
$$f\colon B(A^{\bullet},\sD_G, X) \rtarr B(A^{\bullet},\sD_G, Y)$$
is a weak $G$-equivalence for all $A\in G\sW$.
\item If $f\colon X\rtarr Y$ is an \gen-level equivalence of 
$\sD$-$G$-spaces, then the induced map
$$f\colon B(A^{\bullet},\sD_G,\bP X) \rtarr B(A^{\bullet},\sD_G,\bP Y)$$
is a weak $G$-equivalence for all $A\in G\sW$.
\end{enumerate}
\end{prop}
\begin{proof}  Our bar constructions are all geometric realizations of Reedy cofibrant 
simplicial $G$-spaces, hence \autoref{PlenzWeak} gives the first conclusion.  The second statement follows similarly. The third follows 
from the second using that $\bP f$ is a level $G$-equivalence of $\sD_G$-$G$-spaces by \autoref{Dleveleq}.
\end{proof}

The limitations of the first part and need for the second are clear from the fact that $B(\sF,\sD,X)$ 
is only level $G$-equivalent, not \gen-level equivalent, to $X$.
 
\begin{thm}\label{Segalouttoo}  The following four equivalences of outputs hold.
\begin{enumerate}[(i)]
\item If $X$ is a $\sD$-$G$-space, then the $G$-spectra $\bS_G^{\sD} X$ and $\bS_G^{\sF} \xi_{*}X$ are equivalent.  
\item If $Y$ is a $\sD_G$-$G$-space, then the $G$-spectra $\bS_G^{\sD_G}Y$ and $\bS_G^{\sF_G} \xi^G_{*}Y$ are equivalent.  
\item If $X$ is an $\sF$-$G$-space, then the $G$-spectra $\bS_G^{\sF}X$ and $\bS_G^{\sD} \xi^*X$ are equivalent.  
\item If $Y$ is an $\sF_G$-$G$-space, then the $G$-spectra $\bS_G^{\sF_G} Y$ and $\bS_G^{\sD_G} {\xi_G}^*Y$ are equivalent.  
\end{enumerate}
\end{thm}
\begin{proof}  We first prove (ii), which is the hardest part,  and then show the rest.  Thus let $Y$ be a $\sD_G$-$G$-space.
We claim that the $\sW_G$-$G$-spaces $\bP \xi^G_{*}Y$ and $\bP \bI  \xi^G_{*} Y$ are level equivalent.  Applying
$\bU_{G\sS}$, this will give (ii).  Thus let $A\in \sW_G$.  Then
\[ (\bP \xi^G_{*} Y)(A) = A^{\bullet}\otimes_{\sF_G}B(\sF_G,\sD_G,Y) \iso B(A^{\bullet},\sD_G,Y) \]
and 
\[  (\bP \bI  \xi^G_{*} Y)(A) = A^{\bullet}\otimes_{\sF_G}B(\sF_G,\sF_G, \xi^G_{*} Y)  \iso B(A^{\bullet},\sF_G, \xi^G_{*} Y). \]
In both cases, the isomorphism comes from \autoref{halfsmash}. By \autoref{Segalin2}(i)  there is a natural zigzag of level
$G$-equivalences between  $\xi^*_G\xi^G_{*} Y$ and $Y$. By \autoref{Reedyprop1}(ii), this gives a  zigzag of level $G$-equivalences of $\sD_G$-$G$-spaces 
\[B(A^\bullet, \sD_G, Y) \simeq B(A^\bullet,\sD_G, \xi_G^*\xi^G_{*}Y).\]  Since $\xi_G\colon \sD_G\rtarr \sF_G$ is a $G$-equivalence of $G$-categories of operators, $B(\id,\xi_G,\id)$ induces an equivalence at the level of $q$-simplices of the bar constructions
\[B(A^\bullet,\sD_G,\xi_G^*\xi^G_{*}Y)\rtarr B(A^\bullet,\sF_G,\xi^G_{*}Y).\] Again, since these bar constructions are geometric realizations of Reedy cofibrant 
simplicial $G$-spaces, we get a weak $G$-equivalence on geometric realizations by \autoref{PlenzWeak}. This proves the claim and thus proves (ii).   

To prove (i), let $X$ be a $\sD$-$G$-space. Applying (ii) to  $Y = \bP X$ and using \autoref{Fmach} and \autoref{Dmach}, we see that
\[ \bS_G^{\sF} X = \bS_G^{\sF_G} \bP X \htp \bS_G^{\sD_G} \xi^G_{*} \bP X \iso \bS_G^{\sD_G} \bP \xi_{*} X = \bS_G^{\sD} \xi_{*} X. \]
To prove (iv), let $Y$ be an $\sF_G$-$G$-space.  We claim that the $\sW_G$-$G$-spaces 
$\bP \bI  Y$ and $\bP \xi^G_{*}\xi_G^* Y$ are level equivalent.  Here
\[(\bP \bI  Y)(A)=A^\bullet \otimes_{\sF_G} B(\sF_G,\sF_G,Y)\cong B(A^\bullet, \sF_G,Y)\]
and 
\[(\bP \xi^G_{*}\xi_G^* Y)(A)=A^\bullet \otimes_ {\sF_G} B(\sF_G,\sD_G,\xi_G^*Y) \cong B(A^\bullet, \sD_G, \xi_G^*Y).\]
The isomorphisms again come from \autoref{halfsmash}.  Just as before, $B(\id,\xi_G,\id)$ induces a weak $G$-equivalence 
$$B(A^\bullet, \sD_G, \xi_G^*Y) \rtarr B(A^\bullet, \sF_G,Y).$$
Finally, (iii) follows by application of (iv) to $Y= \bP X$. 
\end{proof}

Use of the generalized homotopical Segal machine will be convenient when we compare the Segal and operadic machines, 
but it is logically unnecessary.  We could just as well replace $\sD$-$G$-spaces $Y$  by the $\sF$-$G$-spaces $\xi_{*}Y$ and 
apply the Segal machine $\bS_G^{\sF_G}$ on them.  We have just shown that we obtain equivalent outputs from these two homotopical 
variants of the Segal machine.   We conclude that all homotopical Segal machines in sight produce equivalent 
output when fed equivalent input.  The resulting $G$-spectra are equivalent via compatible natural zigzags.  
We conclude that all of our machines are essentially equivalent, and they are all equivalent to our preferred machine $\bS_G^{\sF}$
on \gen-special $\sF$-$G$-spaces. 

\section{The generalized operadic machine}\label{OperadGen}

Having redeveloped the Segal infinite loop space machine equivariantly, we
now review and generalize the operadic equivariant infinite loop space machine.  
We assume that the reader is familiar with operads, as originally defined in \cite{MayGeo}.
More recent expositions can be found in \cite{MayOp1, MayOp2}. Operads make sense in any symmetric monoidal category. 
 Ours will be in the cartesian monoidal category  $G\sU$, and we use the notation $\sC_G$ for an operad in $G\sU$, or $G$-operad for short.
\begin{ass}\label{opass} We  assume once and for all that our operads $\sC_G$ are reduced, meaning that $\sC_G(0)$ is a point, and that the inclusion 
 $*\rtarr \sC_G(1)$ of the identity element is a $G$-cofibration.
 \end{ass}
 This holds in all of our examples.

We recall the Steiner operads in \autoref{Steiner} and the classical operadic machine in \autoref{operadicmachine}.
Since the prequel \cite{GM3} of Guillou and the first author develops this machine in detail and the equivariant generalizations 
of the basic definitions of \cite{MayGeo} are straightforward, we shall be brief.  

Our focus is on material that is needed here and is not treated in \cite{GM3}.   We show how the operadic machine, 
like Segal machine, can be generalized to $G$-categories of operators.  This is again based on the nonequivariant 
theory developed in \cite{MT}, but considerations of equivariance require some work.  

In \autoref{sectionOpCats}, we construct a $G$-category of operators $\sD$ over $\sF$ and a $G$-category of 
operators $\sD_G$ over $\sF_G$ from a $G$-operad $\sC_G$.\footnote{To avoid a clash of notation, we choose not to use the 
original notation $\hat{\sC}$ from \cite{MT}.}  We show there that the construction takes $E_{\infty}$ 
$G$-operads to $E_{\infty}$ $G$-categories of operators,  which is  not obvious equivariantly. We construct monads
associated to $G$-categories of operators in \autoref{sectionmonad}.
Finally, in \autoref{inoutop}, we compare the inputs and outputs of the classical and generalized machines and show
 that they are equivalent.   When $\sC_G$ is an $E_{\infty}$  $G$-operad, the generalized input of $\sD$ and $\sD_G$-algebras 
 is the same as the generalized input to the Segal machine. 

\subsection{The Steiner operads}\label{Steiner}  The advantages of the Steiner operads over the little cubes
or little discs operads are explained in detail in \cite[\S3]{Rant1}.  The little cubes operads $\sC_n$ work well 
nonequivariantly, but are too square for multiplicative and equivariant purposes.  The little discs 
operads are too round to allow maps of operads $\sD_n\rtarr\sD_{n+1}$ that are 
compatible with the natural map $\OM^n X\rtarr \OM^{n+1}\SI X$.  The Steiner operads are more
complicated to define, but they enjoy all of the good properties of both the 
little cubes and little discs operads.  Some such family of operads must play a central role in any version 
of the operadic machine, but the special features of the Steiner operads play an entirely new and 
unexpected role in our comparison of the operadic and Segal machines.

We review the definition and salient features of the equivariant Steiner operads from \cite[\S1.1]{GM3} 
and \cite{St}, referring to those sources for more detailed treatments. 
Let $V$ be a finite dimensional real inner product space and let $G$ act on $V$ through linear isometries.
Define a {\em Steiner path} to be a continuous map $h$ from $I$ to the space of distance-reducing embeddings $V\rtarr V$
such that $h(1)$ is the identity map; distance reducing means that
$|h(t)(v)-h(t)(w)|\leq |v-w|$ for all $v,w\in V$ and $t\in I$.  Define 
$\pi(h)\colon V\rtarr V$ by $\pi(h) = h(0)$ and define the ``center point'' of $h$  to be the value of $0\in V$ under 
the embedding $h(0)$, that is $c(h) = \pi(h)(0)\in V$. Crossing embeddings 
$V\rtarr V$ with $\id_{W}$ sends Steiner paths in $V$ to Steiner paths in $V\oplus W$.

For $j\geq 0$, define $\sK_V(j)$ to be the $G$-space of $j$-tuples $(h_1,\dots, h_j)$ of Steiner paths such that the
embeddings $\pi(h_r) = h_r(0)$ for $1\leq r\leq j$ have disjoint images; $G$ acts by conjugation on embeddings and
thus on Steiner paths and on $j$-tuples thereof.  Pictorially, one can think of a point in the 
Steiner operad as a continuous deformation of $V$ into a point in the little disks operad. The symmetric group $\SI_j$ permutes $j$-tuples.   
We take $\sK_V(0) = \ast$ and let $\id\in \sK_V(1)$ be the constant path at the identity $V\rtarr V$.
Compose Steiner paths pointwise, $(\ell\com h)(t) = \ell(t)\com h(t)\colon V\rtarr V$.  Define the structure maps
\[ \ga\colon \sK_V(k)\times \sK_V(j_1)\times \cdots \times \sK_V(j_k) \rtarr \sK_V(j_1 +\cdots + j_k) \]
by sending
\[(\langle \ell_1,\dots,\ell_k\rangle; \langle h_{1,1},\dots,h_{1,j_1}\rangle,\dots, \langle h_{k,1},\dots,h_{k,j_k}\rangle) \]
to 
\[\langle \ell_1\com h_{1,1},\dots, \ell_1\com h_{1,j_1},\dots, \ell_k\com h_{k,1},\dots, \ell_k\com h_{k,j_k}\rangle. \]
Note that $\sK_0$ is the trivial operad, $\sK_0(0) = \ast$, $\sK_0(1) = \{\id\}$ and $\sK_0(j)= \emptyset$ for $j\geq 2$. 
Via $\pi$, the operad $\sK_V$ acts on $\OM^V X$ for any $G$-space $X$ in the same way that the little cubes operad acts on $n$-fold
loop spaces or the little discs operad acts on $V$-fold loop spaces.

Define $\ze\colon \sK_V(j) \rtarr \mb{Conf}(V,j)$, where $\mb{Conf}(V,j)$ is the configuration $G$-space of ordered $j$-tuples of distinct points of $V$, by sending $\langle h_1,\dots, h_j \rangle$ to $(c(h_1),\dots , c(h_j))$.  The original argument of Steiner \cite{St} 
generalizes without change equivariantly to prove that $\ze$ is a $(G\times \SI_j)$-deformation retraction.  We 
may take the colimit over $V$ in a complete $G$-universe $U$ to obtain an $E_{\infty}$ $G$-operad $\sK_U$. 

\subsection{The operadic machine}\label{operadicmachine}\label{opermach}
\begin{defn}\label{Einfop}  We say that a $G$-operad $\sC_G$ is an $E_{\infty}$ 
$G$-operad if  $\sC_G(n)$ is a universal principal $(G,\SI_n)$-bundle for each $n$.  This means  that $\sC_G(n)$ is a
$\SI_n$-free $(G\times\SI_n)$-space such that 
$$\sC_G(n)^\LA   \simeq \ast \ \  \text{ if $\LA\subset G\times \SI_n$ and $\LA\cap \SI_n= e$  (that is, if $\LA\in \bF_n$)}.  $$
Since $\sC_G(n)$ is $\SI_n$-free, $\sC_G(n)^\LA =\emptyset$ if $\LA\notin \bF_n$.
\end{defn}

A $\sC_G$-algebra $(X,\tha)$ is a $G$-space $X$ together with $(G\times \SI_j)$-maps 
$$\tha_j\colon \sC_G(j)\times X^j\rtarr X$$
for $j\geq 0$ such that the diagrams specified in \cite[\S1]{MayGeo} or \cite{MayOp1, MayOp2} commute.   We
call an algebra over $\sC_G$ a $\sC_G$-space.  Since $\sC_G(0)=\{\ast\}$, the action determines a 
basepoint in $X$, and we assume that it is nondegenerate.  Several examples of $E_{\infty}$
operads $\sC_G$ and $\sC_G$-spaces appear in \cite{GM3, GMM}.  

A $G$-space with an action by some $E_{\infty}$ operad is called an $E_{\infty}$ $G$-space.  The operadic infinite
loop space machine sends an $E_{\infty}$ $G$-space $X$ to a genuine positive $\OM$-$G$-spectrum $\bE_GX$ with a group completion $X\rtarr \OM (\bE_G X)_1$.   As explained in \cite[\S8]{Rant1}, the operadic machine is a 
homotopical adaptation of Beck's categorical monadicity theorem.   We summarize the construction of $\bE_G$, following \cite{GM3}.
Let $\sC_G$ be a fixed chosen $E_{\infty}$ $G$-operad throughout this section. 

The operad $\sC_G$ determines a monad $\bC_G$ on based $G$-spaces such that the category $\sC_G[G\sT]$ of $\sC_G$-algebras
is isomorphic to the category $\bC_G[G\sT]$ of $\bC_G$-algebras. Nonequivariantly, this motivated the definition of operads \cite{MayGeo, MayOp1, MayOp2}, and it is proven equivariantly in \cite{GM3}.   Intuitively, $\bC_GX$ is constructed as the quotient of  $\coprod_{j\geq 0}\sC_G(j)\times _{\SI_j} X^j$ by basepoint identifications.   The proof of \cite[Proposition 2.6]{MayGeo} directly generalizes equivariantly to describe a cofibered filtration of $\bC_G$ with easily understood subquotients and, in particular, to show that $\bC_GX$ is nondegenerately based since $X$ is so.

Formally, $\bC_GX$ is the categorical tensor product of functors $\sC_G\otimes_{\inj} X^{\bullet}$, where $\inj$ is the category of based finite sets $\bj=\{0, 1,\dots,j\}$ with basepoint $0$ and based injections.  To be explicit, for $1\leq i\leq j$, let $\si_i\colon \mathbf{j-1}\rtarr  \bj$ be the ordered based  injection that skips $i$ in its image.  Note that the morphisms in $\inj$ are generated by the maps $\si_i$ and the permutations. Then
$\sC_G$ is regarded as a functor $\inj^{op} \rtarr G\sU$ via the right action of $\SI_j$ and the ``degeneracy maps"
$\si_i\colon \sC_G(j) \rtarr \sC_G(j-1)$ specified in terms of the structure map $\ga$ of $\sC_G$ by
\begin{equation}\label{degensig}
  \si_i(c)= \ga(c;  \id^{i-1}\times{0}\times \id^{j-i}) 
 \end{equation}
where $0\in\sC_G(0)$ and $\id\in \sC_G(1)$; $X^{\bullet}$ is the covariant functor $\inj\rtarr G\sU$ that sends $\bj$ to $X^j$ 
and uses the left action of $\SI_j$ and the injections $\si_i\colon X^{j-1}\rtarr X^{j}$ given by insertion of the basepoint in the $i$th position.

\begin{rem} \autoref{opass}  implies that the monad unit $\et\colon X\rtarr \bC_G X$ is a $G$-cofibration for all  $X\in G\sT$.
This ensures that the bar constructions below are realizations of Reedy cofibrant simplicial $G$-spaces.
\end{rem}

There are several choices that can be made in the construction of the machine $\bE_G$, as discussed in \cite{GM3}.  We use
the construction landing in orthogonal $G$-spectra.  It is more natural topologically to land in Lewis-May or EKMM
$G$-spectra \cite{EKMM, LMS, MM} since all such $G$-spectra are fibrant and the relationship of $\bE_G$ to the
equivariant Barratt-Priddy-Quillen theorem is best explained using them. That variant of the operadic machine 
is discussed and  applied in \cite{GM3}, and a reinterpretation with better formal properties will be given in \cite{MayTwist}.

Regardless of such choices, the construction of $\bE_G$ is based on the two-sided monadic bar construction defined
in \cite[\S9]{MayGeo}.  In our context, that specializes to give a based $G$-space $B(F,\bC_G,X)$ for a monad $\bC_G$ in $G\sT$,
a $\bC_G$-algebra $X$, and a $\bC_G$-functor $F = (F,\la)$.  Here $F\colon G\sT\rtarr G\sT$ is a functor and $\la\colon F\bC_G\rtarr F$ 
is a natural transformation such that $\la \com F\et = \Id\colon F\rtarr F$ and 
$$\la\com F\mu = \la\com \la\bC_G \colon F\bC_G \bC_G \rtarr F.$$
There results a simplicial $G$-space $B_{*}(F,\bC_G,X)$ with $q$-simplices $F \bC_G^q X$. Its geometric realization is $B(F,\bC_G,X)$.
We emphasize that while the bar construction can be specified in
sufficiently all-embracing generality that both the categorical version used in the Segal machine and the 
monadic version used in the operadic machine are
special cases \cite{Meyer1, Meyer2, Meyer3, Shul}, these constructions look very different.  

To incorporate the relationship to loop spaces encoded in the Steiner operads,  we define $\sC_V = \sC_G \times \sK_V$.  We view $\sC_G$-spaces as $\sC_V$-spaces via the projection to $\sC_G$, and we view $G$-spaces $\OM^V X$ as $\sC_V$-spaces via the projection to $\sK_V$.   We write 
$\bC_V$ for the monad on $G\sT$ associated to $\sC_V$.

\begin{thm}\label{approx}  The composite of the map $\bC_V X\rtarr \bC_V \OM^V\SI^V X$ induced by the unit
of the adjunction $(\SI^V,\OM^V)$ and the action map  $\bC_V \OM^V\SI^V X \rtarr \OM^V\SI^V X$ specifies a 
natural map $\al\colon \bC_V X\rtarr \OM^V\SI^VX$ which is a weak group completion if $V$ contains $\bR$.  
These maps specify a map of monads $\bC_V\rtarr \OM^V\SI^V$.   
\end{thm}

The second statement is proven by the same formal argument as in \cite[Theorem 5.2]{MayGeo}. 
The first statement is part of \cite[Theorem 1.14]{GM3}.  Weak group completions of $G$-spaces are defined in \cite[Definitions 1.9 and 1.10]{GM3}, but all we need to know is that they are group completions if their sources and targets are homotopy commutative \cite[Theorem 1.8]{GM3}, as  holds in our applications. 
The adjoint $\tilde{\al}\colon \SI^V \bC_V \rtarr \SI^V$ of $\al$ is an
action of the monad $\bC_V$ on the functor $\SI^V$ and we have the monadic 
bar construction
\[  (\bE_GX)(V) = B(\SI^V,\bC_V,X) \]
for $\bC_G$-spaces $X$.  In later comparisons, we write  $\bE_GX = \bE_G^{\sC}X$.

An isometric isomorphism $V\rtarr V'$ in $\sI_G$ induces natural transformations $\SI^V \rtarr \SI^{V'}$ and $\bC_V \rtarr \bC_{V'}$, 
which in turn induce a map  
\[ B(\SI^V,\bC_V,X) \rtarr  B(\SI^{V'},\bC_{V'},X). \] These maps assemble to make $\bE_GX$ into an $\sI_G$-$G$-space.  
Smashing with $G$-spaces commutes with realization of based simplicial $G$-spaces, 
by the same proof as in the nonequivariant case \cite[Proposition 12.1]{MayGeo}, and inclusions 
$V\rtarr W$ induce maps of monads $\bC_V\rtarr \bC_W$.  This gives the structural maps
\[ \si\colon \SI^{W-V} \bE_GX(V)\iso B(\SI^{W},\bC_V,X) \rtarr B(\SI^{W},\bC_{W},X) = \bE_GY(W) \]
that make $\bE_GX$ into an orthogonal $G$-spectrum.
The structure maps of $\bE_G X$ and their adjoints are closed inclusions.
As explained in \cite{GM3}, and as goes back to \cite{MayGeo} nonequivariantly and to Costenoble and
Waner equivariantly \cite{CW}, we have the following theorem, which gives the basic homotopical property 
of the infinite loop space machine $\bE_G$. 

\begin{thm}\label{classic} There  are natural maps
\[\xymatrix@1{ X & \ar[l]_-{\epz} \ar[rr]^-{B(\al,\id,\id)} 
B(\bC _V,\bC _V,X) & &
B(\OM^V\SI^V,\bC _V,X) \ar[r]^-{\ze} & 
\OM^{V} B(\SI^{V},\bC _V,X)\\}\]
of $\sC_V$-spaces.  The map $\epz$ is a $G$-homotopy equivalence with a natural $G$-homotopy inverse $\nu$ 
(which is not a $\sC_V$-map), the map $B(\al,\id,\id)$ is a group completion when $V$ contains $\bR$, and the
map $\ze$ is a weak $G$-equivalence.
\end{thm}
\begin{proof}
The first statement is a standard property of the bar construction that works just as well equivariantly
as nonequivariantly \cite[Proposition 9.8]{MayGeo} or \cite[Lemma 9.9]{Shul}.   The second statement is deduced from 
\autoref{approx} by passage to fixed point spaces and use of the same argument as in the nonequivariant
case \cite[Theorem 2.3(ii)]{MayPerm}.  The last statement is an equivariant generalization of \cite[Theorem 12.7]{MayGeo} that is 
proven carefully in \cite[Lemmas 5.4, 5.5]{CW}.  See \cite{GM3, Rant1} for further discussion and variants of the construction.
\end{proof}

Define 
$$\xi = \ze \com B(\al,\id,\id)\com \nu\colon X\rtarr \OM^{V}\bE_GX(V).$$
Then $\xi$ is a natuaral group completion when $V\supset \bR$ and is thus a weak $G$-equivalence
when $X$ is grouplike. The following diagram commutes, where $\tilde{\si}$ is adjoint to $\si$. 
\[ \xymatrix{
& X \ar[dl]_{\xi}  \ar[dr]^{\xi} & \\
\OM^{V}\bE_GX(V)\ar[rr]_-{\OM^V\tilde{\si}}  & & \OM^{V\oplus W}\bE_GX(V \oplus W).\\} \] 
 Therefore $\OM^V\tilde{\si}$ is a weak equivalence if $V$ contains $\bR$.  

\subsection{$G$-categories of operators associated to a $G$-operad $\sC_G$}\label{sectionOpCats}
\begin{defn}\label{defGops} Let $\sC_G$ be an operad of $G$-spaces.  We construct a $G$-CO over $\sF$,
which we denote by $\sD(\sC_G)$, abbreviated $\sD$ when there is no risk of confusion.  Similarly, we write
$\sD_G = \sD_G(\sC_G)$ for the associated $G$-CO over $\sF_G$.  The morphism $G$-spaces of $\sD$ are 
\[ {\sD}(\mathbf{m}, \mathbf{n}) = \coprod_{\ph\in\sF(\mathbf{m},\mathbf{n})} \prod_{1\leq j\leq n} \sC_G(\ph_j)  \]
with $G$-action induced by the $G$-actions on the $\sC_G(n)$.   Here $\ph_j = |\ph^{-1}(j)|$.
Write elements in the form $(\ph,c)$, where $c = (c_1,\dots,c_n)$.  
For $(\ph,c)\colon \mathbf{m}\rtarr \mathbf{n}$ and $(\ps,d)\colon \mathbf{k}\rtarr \mathbf{m}$, define
\[ (\ph,c)\com (\ps,d) = (\ph\com \ps, 
\prod_{1\leq j\leq n}\ga(c_j;\prod_{\ph(i) = j} d_i)\si_j). \]
Here $\ga$ denotes the structural maps of the operad.  The $d_i$ with $\ph(i) =j$ 
are ordered by the natural order on their indices $i$
and $\si_j$ is that permutation of $(\phi\com\ps)_j$ letters which converts 
the natural ordering of $(\phi\com\ps)^{-1}(j)$ as a subset of $\{1,\dots,k\}$ to
its ordering obtained by regarding it as $\coprod_{\ph(i)=j}\ps^{-1}(i)$, so ordered
that elements of $\ps^{-1}(i)$ precede elements of $\ps^{-1}(i')$ if $i< i'$ and 
each $\ps^{-1}(i)$ has its natural ordering as a subset of $\{1,\dots,k\}$. 

The identity element in $\sD(\mathbf{n},\mathbf{n})$ is $(id, \id^n)$, where $\id$ on
the right is the unit element in $\sC_G(1)$. The map $\xi\colon \sD \rtarr \sF$
sends $(\ph,c)$ to $\ph$.  The inclusion $\io\colon \PI\rtarr \sD$ sends 
$\ph\colon \mathbf{m}\rtarr \mathbf{n}$ to $(\ph, c)$, where $c_i = \id \in \sC_G(1)$ if $\ph(i) = 1$ and
$c_i = \ast\in \sC_G(0)$ if $\ph(i) = 0$.  This makes sense since $\PI$ is the
subcategory of $\sF$ with morphisms $\ph$ such that $\ph_j\leq 1$ for 
$1\leq j\leq n$. 
\end{defn}

Observe that $\sD$ is reduced as a $G$-CO over $\sF$ since $\sC_G$ is reduced as an operad.
\autoref{opass} implies the following result, which says that $\sD_G$ satisfies \autoref{GCO/FG}(ii). It ensures that the bar constructions used in the Segal machine are Reedy cofibrant, and it will also ensure that the bar constructions used in the operadic machine are Reedy cofibrant.

\begin{lem}\label{unitlem} The map $\ast \rtarr \sD_G(\bn^{\al}, \bn^{\al})$ given by the inclusion of the identity is a $G$-cofibration for all finite $G$-sets $\bn^{\al}$.
\end{lem}

\begin{obs}\label{obsGops}  Since we will need it later and it illustrates the definition, we describe explicitly
how composition behaves when the point $c$ or $d$ in one of the maps is of the form 
$(\id,\dots,\id)\in \sC(1)\times \cdots\times \sC(1)$. 

For $\phi\colon \mathbf{m} \rtarr \mathbf{n}$ 
and a permutation $\tau\colon \mathbf{m}\rtarr \mathbf{m}$, 
\begin{eqnarray*}
(\phi, c_1, \dots, c_n)\circ (\tau, \id, \dots \id) &=& (\phi\circ \tau, \prod_j(\ga(c_j, \id,\dots, \id)\si_j)\\
&=& (\phi \circ \tau,  c_1 \si_1, \dots , c_n\si_n),
\end{eqnarray*}
where $c_j\in \sC(|\phi^{-1}(j)|)$, and $\si_j\in \Sigma_{(|(\phi\com\tau)^{-1}(j)|}= \Sigma_{|\phi^{-1}(j)|}$. 
Note that $\si_j$ depends only on $\phi$ and $\tau$ and not on $c$.

For $\ps\colon \mathbf{m}\rtarr \mathbf{n}$ and a permutation $\tau\colon \mathbf{n} \rtarr \mathbf{n}$,
\begin{eqnarray*}
(\tau, \id, \dots, \id)\circ (\psi, d_1, \dots, d_n) &=& (\tau \circ\psi, \prod_j (\ga(\id, d_{\tau^{-1}(j)})\si_j))\\
&=& (\tau \circ\psi,  d_{\tau^{-1}(1)}, \dots, d_{\tau^{-1}(n)})
\end{eqnarray*}
since each $\si_j$ is the identity (because there is only one $i$ such that $\ta(i) = j$).
\end{obs}

\begin{exmp} The commutativity operad $\sN$ has $n$th space a point for all $n$.
We think of it as a $G$-trivial $G$-operad. Then $\sF = \sD(\sN)$, again regarded 
as $G$-trivial. 
\end{exmp}

For any operad $\sC_G$, an $\sF$-$G$-space $Y$ can be viewed as the $\sD(\sC_G)$-$G$-space 
$\xi^*Y$.  Thus $\sD(\sC_G)$-$G$-spaces give a generalized choice of input to the Segal 
machine. As we shall discuss in \autoref{inoutop}, they also give generalized input to the 
operadic machine.  Indeed, an action of the operad $\sC_G$ on a $G$-space $X$
gives rise to an action of $\sD(\sC_G)$ on the $\PI$-$G$-space $\bR X$ with $(\bR X)_n = X^n$.

Now recall the definitions of $E_{\infty}$ operads and $E_{\infty}$ $G$-categories of operators from
Definitions \ref{Einfop} and \ref{Einf}.  We prove the following expected theorem.

\begin{thm}\label{EinfEinf} If $\sC_G$ is an $E_{\infty}$ $G$-operad, then $\sD = \sD(\sC_G)$ is an $E_{\infty}$
$G$-CO over $\sF$ or, equivalently, $\sD_G$ is an $E_{\infty}$ $G$-CO over $\sF_G$.
\end{thm} 

Consider the trivial map of $G$-operads $\xi\colon \sC_G\rtarr \sN$ that sends each $\sC_G(n)$ to the point $\sN(n)$.
Of course, $\sN$ is not an $E_{\infty}$ $G$-operad, but it is clear from the definitions that $\sF = \sD(\sN)$
is an $E_{\infty}$ $G$-CO over $\sF$.  The map $\xi\colon \sD\rtarr \sF$ can be thought of as  $\sD(\xi)$.  The following result 
with $\sC'=\sN$ has  \autoref{EinfEinf} as an immediate corollary.  The proof requires pedantic details of equivariance and is 
therefore deferred to  \autoref{levelwise}, but this is the crux of the comparison of inputs of the Segal and operadic machines.

\begin{thm}\label{lemmaspecial}
Let $\nu\colon \sC_G \rtarr \sC'_G$ be a map of $G$-operads such that the fixed point map 
$\nu^\LA\colon\sC_G(n)^\LA\rtarr \sC'_G(n)^\LA$ is 
a weak equivalence for all $n$ and all $\LA\in \bF_n$. Then the induced map 
$\nu_G\colon \sD_G\rtarr \sD'_G$ of $G$-COs over $\sF_G$ is a $G$-equivalence.
\end{thm}

\subsection{The monads $\bD$ and $\bD_G$ associated to the $G$-categories $\sD$ and $\sD_G$}\label{sectionmonad} 

Recall that the category $\sC_G[G\sT]$ of algebras over an operad $\sC_G$ is isomorphic to the category $\bC_G[G\sT]$ of 
algebras over the associated monad $\bC_G.$  Let $\sD = \sD(\sC_G)$ be the $G$-CO over $\sF$ associated to $\sC_G$ and 
let $\sD_G$ be the associated $G$-CO over $\sF_G$. 
As worked out nonequivariantly in \cite[\S5]{MT}, we define monads $\bD$ on the category of $\PI$-$G$-spaces 
and $\bD _G$ on the category of $\PI_G$-$G$-spaces. The commutative diagram of inclusions of categories  
\[ \xymatrix{
\PI \ar[r] \ar[d]_{i} & \PI_G \ar[d]^{i_G} \\
\sD  \ar[r] &  \sD_G}
\]
gives rise to a commutative diagram of forgetful functors 

\begin{equation}\label{Ui}
 \xymatrix{
\Fun(\PI, \UG) & \Fun(\PI_G,\UG) \ar[l]_-{\bU} \\
\Fun(\sD, \UG)  \ar[u]^{i^*}&  \Fun(\sD_G, \UG). \ar[u]_{i_G^*} \ar[l]^-{\bU}}
\end{equation}

Categorical tensor products then give left adjoints making the following diagram 
commute up to natural isomorphism.

\begin{equation}\label{PD}
 \xymatrix{
\Fun(\PI, \UG) \ar[r]^-{\bP} \ar[d]_{\bD} & \Fun(\PI_G,\UG) \ar[d]_{\bD_G} \\
\Fun(\sD, \UG)  \ar[r]_-{\bP} &  \Fun(\sD_G, \UG) \\}
\end{equation}

Here $\bD$ and $\bD_G$ are the left adjoints of $i^\ast$ and $i_G^\ast$, respectively. By a standard abuse of notation, we write $\bD$ and $\bD_G$ for the resulting endofunctors
$i^*\bD$ on  $\Fun(\PI_G,\UG)$ and  $i_G^* \bD_G$ on $\Fun(\PI_G,\UG)$. Explicitly, the monads $\bD$ and $\bD_G$ are defined as 
\begin{equation}\label{DX}
 (\bD X)(\mathbf{n}) =  \sD(-,\mathbf{n}) \otimes_{\PI} X 
 \end{equation}
for a $\PI$-$G$-space $X$, where $\sD(-,\mathbf{n})$ is the represented functor induced by $i$, and
\begin{equation}\label{DGY}
 (\bD_G Y)(\bn^{\al}) = \sD_G(-,\bn^{\al}) \otimes_{\PI_G} Y 
 \end{equation}
for a $\PI_G$-$G$-space $Y$, where $\sD_G(-,\bn^{\al})$ is the
represented functor induced by $i_G$.  

The units $\et$ of the adjunctions $(\bD,i^*)$ and $(\bD_G,i_G^*)$ give the unit maps of the 
monads, and the action maps of  the $\sD$-$G$-spaces $\bD X$ and $\sD_G$-$G$-spaces $\bD_GY$ give the
products $\mu$.  More concretely,  $\mu\colon \bD\bD\rtarr \bD$ and $\mu\colon \bD_G \bD_G \rtarr \bD_G$ 
are derived from the compositions in $\sD$ and $\sD_G$, respectively, and thus from the structure maps $\ga$ 
of $\sC_G$. The unit maps $\et$ are derived from the identity morphisms in these categories and thus from the unit element 
$\id\in \sC_G(1)$. 

Moreover, we have the following standard categorical result.

\begin{prop}\label{generalDDG}
 The category of $\bD$-algebras in $\Fun(\PI,\UG)$ is isomorphic to the category $\Fun(\sD,\UG)$. Similarly, the category of $\bD_G$-algebras in $\Fun(\PI_G,\UG)$ is isomorphic to the category $\Fun(\sD_G,\UG)$.
\end{prop}

The above discussion used the existence of left adjoints that are constructed formally. Our cofibrancy conditions in the definition of $\PI$-$G$-spaces ensure that these constructions are well behaved homotopically.

\begin{prop}\label{Dgood}
The functors $\bD$ and $\bD_G$ restrict to give functors
\[\bD\colon \PIdashG \rtarr \DdashG \ \ \text{and} \ \ \bD_G\colon \PIsubG \rtarr \DsubG.\]
Restricting appropriately, we obtain analogues of the diagrams \autoref{Ui} and \autoref{PD}.
\[ \xymatrix{
\PIdashG & \PIsubG \ar[l]_-{\bU} \\
\DdashG  \ar[u]^{i^*}&  \DsubG \ar[u]_{i_G^*} \ar[l]^-{\bU}}
\qquad
 \xymatrix{
\PIdashG \ar[r]^-{\bP} \ar[d]_{\bD} & \PIsubG \ar[d]_{\bD_G} \\
\DdashG  \ar[r]_-{\bP} &  \DsubG. \\} \]
The monads $\bD$ and $\bD_G$ restrict to give monads on the categories $\PIdashG$ and $\PIsubG$, respectively.
\end{prop}

\begin{proof}
By definition, the functors $\bU$ and $\bP$ restrict appropriately to  $\DdashG$ and $\DsubG$ (and their $\PI$ analogues), and they give equivalences of categories by  Theorems \ref{compFFG} and \ref{compGGG}. Thus, it is enough to prove that $i^\ast$ and $\bD$ preserve the cofibrancy conditions required of $\PI$ and $\sD$-$G$-spaces. This is certainly true for $i^\ast$, since the condition is specified precisely on the underlying $\PI$-$G$-space. The verification for $\bD$ requires detailed combinatorial analysis and is relegated to \autoref{DXReedy}. 
 \end{proof}
 
Denote by $\Dalg$ the category of algebras over the monad $\bD$ in $\PIdashG$, and similarly, denote by $\DGalg$ the category of algebras over the monad $\bD_G$ in $\PIsubG$.  The previous result implies the following refinement of \autoref{generalDDG}.

\begin{prop}\label{DDG} The categories $\DdashG$ and $\Dalg$ are isomorphic and
the categories $\DsubG$ and $\DGalg$ are isomorphic.
Therefore the categories 
\linebreak
$\Dalg$ and  $\DGalg$ are equivalent.
\end{prop}

Therefore  the monads $\bD$  and $\bD_G$ can be used interchangeably.   This contrasts markedly with the Segal machine,
where considerations of specialness led us to focus on $\sD_G$ rather than $\sD$.
We shall give a conceptual explanation of the difference in \autoref{inmodel}.

We need some homotopical and some formal properties of the monads $\bD$ and $\bD_G$, following \cite{MT}. 
We first establish the formal properties, whose proofs are identical to those in \cite{MT}.
We write the following results in terms of $\sD$  and $\bD$ for simplicity.
With attention to enrichment, the parallel results for $\sD_G$ work in the same way.  They can also be derived 
from the results for $\bD$ using $\bP\bD \iso \bD_G \bP$ and \autoref{DDG}.

Recall from \autoref{InSegal} that the evident functors $\bL \colon \PIdashG \rtarr  G\sT$ and $\bR \colon G\sT \rtarr \PIdashG$ give an adjunction such that  $\bL \bR  = \Id$ and the unit of the adjunction is given by the Segal maps. It induces an operadic variant.  Let $\sC_G[G\sT]$ denote the  category of $\sC_G$-spaces.

\begin{prop}\label{RLC}  A $\sD$-$G$-space with underlying $\PI$-$G$-space 
$\bR X$ determines and is determined by a $\sC_G$-space structure on $\bL \bR X = X$.
\end{prop}
\begin{proof} The nonequivariant proof of \cite[Lemma 4.2]{MT} applies verbatim.
\end{proof}

We require an analyis of the behavior of the monad $\bD$ with respect to the adjunction $(\bL ,\bR )$.  The following two results are equivariant generalizations of results in \cite[Section 6]{MT} and the equivariance adds no complications.   The proofs are inspections of definitions and straightforward diagram chases.\footnote{These results are specializations of a more general one with other applications \cite{MayTwist}.}

\begin{prop}\label{formalMT}  Let $X$ be a $G$-space and $Y$ be a $\PI$-$G$-space.  
\begin{enumerate}[(i)]  
\item The $G$-space $\bL \bD \bR X = (\bD \bR X)_1$ is naturally $G$-homeomorphic to $\bC_G X$.
\item  The $\PI$-$G$-space $\bD \bR  X$ is naturally isomorphic to the $\PI$-$G$-space $\bR \bC_G X$.
\item The following diagram is commutative for each n. 
\[ \xymatrix{ 
(\bD Y)_n \ar[rr]^-{(\bD \de)_n} \ar[d]_{\de} & & (\bD \bR \bL Y)_n\iso (\bC_G \bL Y)^n \ar[d]^{\iso}_{ \de}\\
(\bD Y)_1^n \ar[rr]_-{(\bD \de)_1^n} & & (\bD \bR \bL  Y)_1^n \iso (\bC_G\bL Y)^n\\}  \]
\item The functor $\bR \bC_G \bL $ on $\PI$-$G$-spaces is a monad with product and unit induced from those
of $\bC_G$ via the composites
\[ \xymatrix@1{ \bR \bC_G \bL \bR \bC_G \bL  = \bR \bC_G \bC_G \bL  \ar[r]^-{\bR \mu \bL } & \bR \bC_G \bL \\}  
\ \ \text{and} \ \ \xymatrix@1{ \Id \ar[r]^{\de} & \bR \bL  \ar[r]^-{\bR \et \bL } & \bR \bC_G \bL.  \\} \]
\item  The natural transformation 
$\bD \de \colon \bD  \rtarr \bD  \bR \bL  \iso \bR \bC_G \bL $
is a morphism of monads in the category $\PIdashG$. 
\item If $(F,\la)$ is a $\bC_G$-functor in $\sV$, then $F\bL \colon \Fun(\PI,G\sT) \rtarr \sV$ is an
$\bR \bC_G \bL $-functor in $\sV$ with action $\la \bL \colon F\bL \bR \bC_G \bL  = F\bC_G \bL  \rtarr F\bL $. Therefore, by
pullback, $F\bL $ is a $\bD$-functor in $\sV$ with action the composite
\[ \xymatrix@1{ F\bL \bD\ar[r]^-{F\bL \bD\de} 
&  F\bL \bD  \bR \bL  \iso F\bL \bR  \bC_G \bL  = F\bC_G \bL  \ar[r]^-{\la \bL } & F\bL .\\} \]
\end{enumerate}
\end{prop}
\begin{proof}
Nonequivariantly, these results are given in \cite[\S 6]{MT} and the equivariance adds no complications. 
The proofs are inspections of definitions and straightforward diagram chases.
\end{proof}

Using the adjunction $(\bP,\bU)$ and the isomorphism $\bD_G\bP \iso \bP\bD$, we derive the analogue for $\bD_G$.
Following \autoref{Rdefn2}, we write  $\bR_G = \bP \bR \colon G\sT \rtarr \PI_G[\sT_G]$ and we write $\bL_G =\bL\bU$ for its
left adjoint.   While  $\bR_G$ restricts to a functor  $\sC_G[G\sT]\rtarr \DsubG$, the analog for $\bL_G$ is false.

With $\bD$, $\bR$, and $\bL$ replaced by $\bD_G$, $\bR_G$, and $\bL_G$,
we then have the following analogue of \autoref{formalMT}. 

\begin{prop}\label{formalMT2}  Let $X$ be a $G$-space and $Y$ be a $\PI_G$-$G$-space.  
\begin{enumerate}[(i)]  
\item The $G$-space $\bL_G \bD_G\bR_G X$ is naturally $G$-homeomorphic to $\bC_G X$.
\item  The $\PI_G$-$G$-space $\bD_G \bR_G  X$ is naturally isomorphic to  $\bR_G \bC_G X$.
\item The following diagram is commutative for each $\bn^{\al}$. 
\[ \xymatrix{ 
(\bD_G Y)\bn^{\al} \ar[rr]^-{\bD_G \de} \ar[d]_{\de} & & (\bD_G\bR_G \bL_G Y)\bn^{\al}\iso (\bC_G \bL_{G} Y)^{\bn^{\al}}  \ar[d]^{\iso}_{ \de} \\
(\bD_G Y)_1^{\bn^{\al}} \ar[rr]_-{(\bD_G \de)_1^{\bn^{\al}}}& & (\bD_G \bR_G \bL_G  Y)_1^{\bn^{\al}} \iso (\bC_G\bL_G Y)^{\bn^{\al}}\\}  \]
\item The functor $\bR_G \bC_G \bL_G $ on $\PI_G$-$G$-spaces is a monad with product and unit induced from those
of $\bC_G$ via the composites
\[  \bR_G \bC_G \bL_G \bR_G \bC_G \bL_G  = \bR_G \bC_G \bC_G \bL_G  \xrightarrow{\bR _G\mu \bL_G }  \bR_G \bC_G \bL_G   \]
and
\[  \Id \xrightarrow{\quad \de \quad }  \bR_G \bL_G  \xrightarrow{\bR_G \et \bL_G }  \bR_G \bC_G \bL_G.   \]
\item  The natural transformation 
$\bD_G \de \colon \bD_G  \rtarr \bD_G  \bR_G \bL_G  \iso \bR_G \bC_G \bL_G $
is a morphism of monads in the category $\PIsubG$. 
\item If $(F,\la)$ is a $\bC_G$-functor in $\sV$, then $F\bL_G \colon \Fun(\PI_G,\sT_G) \rtarr \sV$ is an
$\bR_G \bC_G \bL _G$-functor in $\sV$ with action $\la \bL_G \colon F\bL_G \bR_G \bC_G \bL _G = F\bC_G \bL_G  \rtarr F\bL_G $. Therefore, by
pullback, $F\bL_G $ is a $\bD_G$-functor in $\sV$ with action the composite
\[ \xymatrix@1{ F\bL_G \bD_G\ar[r]^-{F\bL_G \bD_G\de} 
&  F\bL _G\bD_G  \bR_G \bL_G  \iso F\bL_G \bR_G  \bC_G \bL_G  = F\bC_G \bL_G  \ar[r]^-{\la \bL_G } & F\bL_G .\\} \]
\end{enumerate}
\end{prop}

We now turn to the homotopical properties of the monads $\bD$ and $\bD_G$ and their algebras. 
The proofs of the homotopical properties are similar to those in \cite{MT}, but considerably more difficult, 
so some will be deferred to \autoref{HARD}.    In contrast with the Segal 
machine, we start with $\bD$ rather than $\bD_G$.  Our interest is in $E_{\infty}$ operads, but we allow more 
general operads until otherwise indicated.

Following \cite{MayGeo} nonequivariantly, we say that a $G$-operad $\sC_G$ is $\SI$-free if the action
of $\SI_j$ on $\sC_G(j)$ is free for each $j$.   Surprisingly, we only need that much structure to prove the following
result. It is the equivariant generalization of \cite[Lemma 5.6]{MT}.

\begin{thm}\label{homotopMT}  Let $\sC_G$ be a $\SI$-free $G$-operad.
\begin{enumerate}[(i)]
\item If $f\colon X\rtarr Y$ is an \gen-level equivalence of $\PI$-$G$-spaces,  then 
$\bD  f\colon \bD  X \rtarr \bD Y$ is an \gen-level equivalence.
\item If $X$ is an \gen-special $\PI$-$G$-space, then $\bD X$ is \gen-special.
\end{enumerate}
\end{thm}

\begin{proof}The proof of (i) requires quite lengthy combinatorics about the structure of $\bD X$ and its fixed point subspaces,
hence we defer it to \autoref{HARD}.  By (i) applied to the Segal map $\de\colon X\rtarr \bR\bL X$, 
the map $\bD \de$ is an \gen-level equivalence. Since its target $\bD\bR\bL X \cong \bR\bC_G\bL X$ is \gen-special, 
\autoref{iff} implies that $\bD X$ is also \gen-special, giving (ii).
\end{proof}

By \autoref{compGGG} and Corollaries  \ref{DFstarspec} and \ref{Dleveleq}, the isomorphism  $\bP\bD\iso \bD_G\bP$
and \autoref{homotopMT} imply the following analogue of that result.

\begin{prop}\label{homotopMT2} Assume that each $\sC_G(j)$ is $\SI_j$-free. 
\begin{enumerate}[(i)] 
\item If $f\colon X\rtarr Y$ is an \gen-level equivalence of $\PI$-$G$-spaces,  then 
$\bD_G  \bP f\colon \bD_G  \bP X \rtarr \bD_G \bP Y$, is a level $G$-equivalence.
\item If $X$ is an \gen-special $\PI$-$G$-space, then $\bD_G \bP X$ is a special $\PI_G$-$G$-space.
\end{enumerate}
\end{prop}

\subsection{Comparisons of inputs and outputs of the operadic machine}\label{inoutop}

We give three ways to construct $G$-spectra from $\sD_G$-$G$-spaces.  We can convert 
$\sD_G$-$G$-spaces to $\sD$-$G$-spaces via \autoref{DDG}, and we can convert those to $\sC_G$-spaces by \autoref{Mayin} below.
We can then apply the original operadic machine, or we can generalize the machine to both $\sD$-$G$-spaces and $\sD_G$-$G$-spaces.  
All three make use of the two-sided monadic bar construction of \cite{MayGeo}, starting from the formalities of Propositions \ref{formalMT} 
and \ref{formalMT2}. We show that all three are equivalent. In fact, we really only have two machines in view of the following  result.  We emphasize how different 
this is from the Segal machine, where the $\sD$ and $\sD_G$ bar constructions are not even equivalent, let alone isomorphic.

\begin{prop}\label{OperDDG}  Let $\sC_G$ be a $G$-operad with category of operators $\sD$. For any $\sD$-space $X$ and 
any $\bC_G$-functor $F\colon G\sT \rtarr G\sT$,  there is a natural isomorphism of $G$-spaces
\[  B(F\bL ,\bD,X) \iso B(F\bL_G ,\bD_G,\bP X). \]
\end{prop}
\begin{proof}
Since $\Id\iso \bU\bP$, $\bP \bD\iso \bD_G\bP$ and $\bL_G=\bL\bU$, we have isomorphisms of 
$q$-simplices 
\[  F\bL  \bD^q X \iso F\bL_G \bD_G^q\bP X. \]
Formal checks show that these isomorphisms commute with the face and degeneracy operators.
The conclusion follows on passage to geometric realization.
\end{proof}

For variety, and because that is what we shall use in the next section, we focus on $\sD_G$ rather than $\sD$ in this section. 
The following result  compares the inputs to machines given by $\sD_G$-$G$-spaces and $\sC_G$-spaces. 

\begin{defn}\label{CDGMachine}  For a $\sD_G$-$G$-space $Y$, define a $\sC_G$-space $\bX(Y)$ by
\[ \bX(Y) = B(\bC_G\bL_G ,\bD_G,Y).  \]
Here $\bC_G$ is regarded as a functor $G\sT \rtarr \sC_G[G\sT]$, and the construction makes sense
since the realization of a simplicial $\sC_G$-space is a $\sC_G$-space, exactly as 
nonequivariantly \cite[Theorem 12.2]{MayGeo}. 
\end{defn}

\begin{prop}\label{Mayin}  For special $\sD_G$-$G$-spaces $Y$,
there is a zigzag of natural level 
$G$-equivalences of $\sD_G$-$G$-spaces between $Y$ and $\bR_G \bX(Y)$. For $\sC_G$-spaces $X$, there 
is a natural $G$-equivalence of  $\sC_G$-spaces from  $\bX(\bR_G X)$ to $X$.
\end{prop}
\begin{proof} \autoref{homotopMT2} implies that $\bD_G \de\colon \bD_G  Y \rtarr \bD_G  \bR_G \bL_G Y \iso \bR_G \bC_G \bL_G  Y$ is  a level 
$G$-equivalence of $\sD_G$-$G$-spaces.  The
realization of simplicial $\PI_G$-$G$-spaces is defined levelwise, and since realization commutes
with products of $G$-spaces, we have the indicated isomorphism in the diagram

{\small{
\[ \xymatrix@1{ Y & \ar[l]_-{\epz} B(\bD_G,\bD_G, Y) \ar[rr]^-{B(\bD_G \de,\id,\id)} 
& & B(\bR_G \bC_G \bL_G ,\bD_G, Y) \iso \bR_G B(\bC_G\bL_G ,\bD_G, Y) = \bR_G \bX(Y).\\} \]
}}

By standard properties of the bar construction, as in \cite{MayGeo} nonequivariantly, 
$\epz$ is a level $G$-equivalence
of $\sD_G$-$G$-spaces. Since the bar constructions are geometric realizations of Reedy cofibrant simplicial $G$-spaces (see \autoref{blanket}), 
it follows from \autoref{PlenzWeak}  that $B(\bD_G \de,\id,\id)$ is a level $G$-equivalence of $\sD_G$-$G$-spaces. For the second statement, we apply
$\bL_G $ to the level $G$-equivalence of $\sD_G$-$G$-spaces
\[ \xymatrix@1{   
\bR_G \bX(\bR_G X) \iso B(\bR_G \bC_G \bL_G , \bD_G, \bR_G X) \iso B(\bR_G \bC_G \bL_G , \bR_G \bC_G \bL_G , \bR_G X) \ar[r]^-{\epz} & \bR_G X,}  \] 
where the second isomorphism follows from \autoref{formalMT2} and inspection.
\end{proof}

Therefore,  after inverting the respective equivalences, the functors $\bR_G $ and $\bX$ induce an equivalence of categories between 
$\sC_G$-spaces and special $\sD_G$-$G$-spaces whose underlying $\PI_G$-$G$-spaces
are Reedy cofibrant.  We conclude that, for an
$E_{\infty}$-operad $\sC_G$, the input categories for operadic machines given by $\sC_G$-spaces and
by $\sD_G$-$G$-spaces are essentially equivalent. 

To generalize the machine from $\sC_G$-spaces to $\sD_G$-$G$-spaces, we again use the product
operads $\sC_V = \sC_G \times \sK_V$, where $\sK_V$ is the $V$th Steiner operad.
We write $\bD_{G,V}$ for the monad associated to the resulting category of operators $\sD_{G,V}=\sD_G(\sC_V)$ over $\sF_G$.
Then a $\sD_G$-space is a $\sD_{G,V}$-space
for any representation $V$ by pullback along the projection $\sD_{G,V}\rtarr \sD_G$. 

\begin{defn}\label{DGmachine} For a $\sD_G$-$G$-space $Y$, define the $V$th space of the orthogonal 
$G$-spectrum $\bE^{\sD_G}_GY$  to be the monadic two-sided bar construction
\begin{equation}
 \bE^{\sD_G}_G(Y)(V) = B(\SI^V\bL_G , \bD_{G,V}, Y). 
 \end{equation}
The right action of $\bD_{G,V}$ on $\SI^V\bL_G$ is obtained from the projection $\bD_{G,V}\rtarr \bK_{V}$ and the
action of $\bK_{V}$ on $\SI^V$, via \autoref{formalMT}(vi).
 The $\sI_G$-$G$-space structure is given as follows. For an isometric isomorphism $V\rtarr V'$ in $\sI_G$, the map 
 \[ B(\SI^V\bL_G , \bD_{G,V}, Y)\rtarr B(\SI^{V'}\bL_G , \bD_{G,{V'}}, Y)\]
  is the geometric realization of maps induced at all simplicial levels  by  
  $S^V \rtarr S^{V'}$ and $\sK_V\rtarr \sK_{V'}$. Similarly, since smashing commutes with geometric realization, the structure maps 
   \[ B(\SI^V\bL_G , \bD_{G,V}, Y)\sma S^W \rtarr B(\SI^{V\oplus W}\bL_G , \bD_{G,{V\oplus W}}, Y)\] are induced from the maps of monads 
   $\bD_{G,V}\rtarr \bD_{G,{V\oplus W}}$.
\end{defn}

The machine $\bE^{\sD_G}_G$ on $\sD_G$-$G$-spaces $Y$ generalizes the machine $\bE^{\sC}_G$ on
$\sC_G$-spaces $X$. To see that, observe that \autoref{formalMT2}(ii)  implies that we have a natural isomorphism 
$\bD_{G,V}\bR_G X \cong \bR_G \bC_V X$.  Since  $\bL_G \bR_G = \Id$,  that gives us a natural  isomorphism 
\begin{equation}\label{Mayout}
B(\SI^V\bL_G ,\bD_{G,V}, \bR_G X) \iso B(\SI^V, \bC_V, X),
\end{equation}
where we regard $\bR_G X$ as a $\sD_G$-$G$-space via \autoref{RLC}.\footnote{As in \cite[Remark 2.9]{GM3}, checks of definitions when $V=0$ give $\SI^0 = \Id$, $\bK_0 = \Id$, $\bC_0 X \iso \sC_G(1)_+\sma X$ for a $G$-space $X$, and 
$\bD_{G,0} Y \iso \bR_G\bC_0 \bL_G  Y$ for a $\PI_G$-$G$-space $Y$.}
for     Together with \autoref{Mayin},
this gives the following comparison of outputs of our machines.

\begin{cor}\label{Mayouttoo} For $\sC_G$-spaces $X$, $\bE^{\sC}_G X$ is naturally isomorphic to $\bE^{\sD_G}_G \bR_G X$.  For special
$\sD_G$-$G$-spaces $Y$ whose underlying $\PI_G$-$G$-space is Reedy cofibrant, there is a zigzag of natural equivalences connecting 
$\bE^{\sD_G}_GY$ to $\bE^{\sD_G}_G \bR_G \bX(Y) \iso \bE^{\sC}_G \bX(Y)$. 
\end{cor}

Thus the machines $\bE_G$ on $\sC_G$-spaces and on special $\sD_G$-$G$-spaces are essentially equivalent.
Properties of the machine on special $\sD_G$-$G$-spaces are essentially the same as properties of the
machine on $\sC_G$-spaces, as can either be proven directly or read off from the equivalence of machines.

\section{The equivalence between the Segal and operadic machines}\label{Equivalence}

We give an explicit comparison between the generalized Segal and generalized operadic infinite 
loop space machines. The comparison is needed for consistency and because each 
machine has significant advantages over the other. That was already clear nonequivariantly, 
and it seems even more true equivariantly.  As in \cite{GM3, Rant1}, in the previous section we used the Steiner 
operads rather than the little cubes operads that were used in \cite{MayGeo, MT}. That change made equivariant 
generalization easy, and \cite{Rant1} gave other good reasons for 
the change. However, nothing like the present comparison was envisioned in earlier work.  
As we have recalled, the Steiner operad is built from paths of embeddings.  We shall see that these
paths give rise to a homotopy that at one end relates to the generalized Segal machine and at the 
other end relates to the generalized operadic machine. That truly seems uncanny.

\subsection{The statement of the comparison theorem}

To set the stage, we recapitulate some of what we have done.  We fix an $E_{\infty}$ operad $\sC_G$
of $G$-spaces.  We then have an $E_{\infty}$ $G$-CO $\sD = \sD(\sC_G)$ over $\sF$ and an $E_{\infty}$ 
$G$-CO $\sD_G$ over $\sF_G$.  Our primary interest here 
is in infinite loop space machines defined either on special $\sF_G$-$G$-spaces or on $\sC_G$-spaces.  The Segal machine is defined on the former and the operadic machine is defined on the latter.  We have generalized both machines 
so that they accept special $\sD_G$-$G$-spaces as input.  Moreover, we have compared inputs and shown that both 
special $\sF_G$-$G$-spaces and $\sC_G$-spaces are equivalent to special $\sD_G$-$G$-spaces and therefore to each other.
Further, we have compared outputs.  We have shown that application of the generalized Segal machine to $\sD_G$-$G$-spaces is 
equivalent to application of the original homotopical Segal machine to $\sF_G$-$G$-spaces, and that application of the 
generalized operadic machine to $\sD_G$-$G$-spaces is equivalent to application of the original operadic machine to $\sC_G$-spaces.  

In more detail, \gen-special $\sF$-$G$-spaces, \gen-special $\sD$-$G$-spaces, special  $\sF_G$-$G$-spaces, 
and special $\sD_G$-$G$-spaces are all equivalent by Theorems \ref{Segalin2} and 
\ref{Segalin4}, and the Segal machines on all four equivalent inputs give equivalent output by \autoref{Segalouttoo}.
We may therefore focus on the Segal machine $\bS_G=\bS_G^{\sD_G}$ defined 
on special $\sD_G$-$G$-spaces $Y$.  
Similarly, $\sC_G$-spaces, \gen-special $\sD$-$G$-spaces, and special $\sD_G$-$G$-spaces are equivalent by Propositions \ref{RLC} and \ref{Mayin}, 
and the operadic machine on $\sC_G$-spaces is a special case of the operadic machine on $\sD_G$-$G$-spaces by 
\autoref{Mayouttoo}.  Thus we may again
focus on the operadic machine $\bE_G=\bE^{\sD_G}_G$ defined on special $\sD_G$-$G$-spaces $Y$.  
Thus, fixing an $E_\infty$ operad  $\sC_G$ with associated category of operators $\sD_G$ over $\sF_G$, 
we consider special $\sD_G$-$G$-spaces $Y$.  Our goal is to give a constructive proof of the following comparison theorem.

\begin{thm}\label{WOW}  For special $\sD_G$-$G$-spaces $Y$, there is a natural zigzag of equivalences of orthogonal $G$-spectra between $\bS_GY$ 
and $\bE_GY$.
\end{thm} 

\subsection{The proof of the comparison theorem}\label{SECWOW}

Here we display the zigzag and then fill in the required constructions and proofs in subsequent sections.  We first introduce generic notations that may enhance readability.

\begin{notn}\label{starstar}  We shall use $\star$ as a place holder for representations $V$ and we shall use $\bullet$ as a placeholder for finite $G$-sets $\bn^{\al}$.  For fixed $V$, $\rbullv$ denotes the contravariant functor on $\sF_G$ that sends
the object $\bn^{\al}$ to  $(S^V)^{\bn^{\al}} = \sT_G(\bn^{\al},S^V)$.  Precomposing with $\xi_G$, we view $\rbullv$ as a contravariant functor $\sD_G\rtarr \sT_G$.  We write $\rbull$ for the functor obtained by letting $V$ vary.
\end{notn}

Since wedges taken over $G$-sets ${\bn^{\al}}$ play a significant role in our 
argument, we introduce the following convenient notation.

\begin{notn}\label{nottwo}  For a based space $A$, let ${^n}\!A$ denote the wedge sum of $n$ copies of $A$. Similarly, for a $G$-set $\bn^{\al}$ and a based $G$-space $A$, let $^{\bf n^{\al}}\!\!A$ denote the wedge sum of $n$ copies of $A$ with $G$-acting on $A$, but also interchanging the wedge summands. We write $(j, a)\in \,^{\bf n^{\al}}\!\!A$ for the element $a$ in the $j$th summand.  The $G$-action is given explicitly by  $g\cdot(j, a)=(\al(g)(j), g\cdot a)$.
\end{notn} 

Recall from \autoref{Dmach} that the $V$th $G$-space of $\bS_G Y$ is
\[ (\bS_G Y)(V)= B(\rbullv, \sD_G, Y). \]
We thus adopt the notation
\[ \bS_G Y = B(\rbull, \sD_G, Y).  \]
Let $\sC_V$ be the product operad $\sC_G \times \sK_V$ and let $\sD_{G,V} = \sD_G(\sC_V)$ with associated monad $\bD_{G,V}$ 
on the category of $\PI_G$-$G$-spaces. As in \autoref{DGmachine}, the $V$th $G$-space of $\bE_G Y$ is
\[  (\bE_GY)(V) = B(\SI^V\bL_G , \bD_{G,V}, Y). \]
We thus adopt the notation
\[ \bE_G Y = B(\SI^{\star}\bL_G,\bD_{G,\star}, Y). \]

Note the different uses of the $\star$ notation.  In both machines, it is a placeholder
for representations $V$.  However, in the Segal machine, we are using cartesian powers of $G$-spheres $S^V$
to obtain functors $\rbullv \colon \sF_G^{op}\rtarr \sT_G$, whereas in the operadic machine we are using the suspension 
functor $\SI^V$ associated to $S^V$ together with the Steiner operad $\sK_V$.   While a two-sided bar construction is used
in both machines, the similarity of notation hides how different these bar constructions really are: the use of categories
and contravariant and covariant functors in one is quite different from the use of monads, (right) actions on functors,
and (left) actions on objects in the other.  

The zigzag of stable equivalences relating the  machines has the following shape.  All of the constructions are two sided categorical bar constructions except the operadic machine at the last step, which is a monadic bar construction.

\begin{equation}\label{WOWWOW}
\xymatrix{
\bS_G Y  \ar@{=}[r] & B(\rbull,\sD_G,Y) & \\
& \ar[u]_{\pi} B(\rbull,\sD_{G,\star},Y) \ar[d]^{i_1}& \\
& B(I_+\sma \rbull,\sD_{G,\star},Y)& \\
& \ar[u]_{i_0} B(\rbullo,\sD_{G,\star},Y)& \\
& \ar[u]_{\io} B(\lbullo, \sD_{G,\star},Y) \ar[d]^{\om}& \\
& B(\SI^{\star}\bL_G,\bD_{G,\star},Y) & \bE_G Y. \ar@{=}[l]\\}
\end{equation} 

\begin{rem} The heart of the construction is to use the Steiner operads to define a  certain homotopy $H$ that specifies the required contravariant action of $\sD_{G,\star}$ on $I_+\sma \rbull$; see \autoref{homotopy}.   At $t=1$, this action restricts to $\rbulll = \rbull$.  At $t=0$, it restricts to a new contravariant action of $\sD_{G,\star}$ on $\rbull$, denoted $\rbullo$.   Crucially, this action restricts to an action, denoted $\lbullo$, on $\lbull$.
\end{rem}

We shall shortly construct the intermediate orthogonal $G$-spectra and maps in this zigzag and prove that all of the maps except $\om$ are stable equivalences.  As we now explain, this will imply that $\om$ is also a stable equivalence.      
Recall that the homotopy groups of a pointed $G$-space $X$ are $\pi_q^H(X) = \pi_q(X^H)$ 
and the homotopy groups of an orthogonal $G$-spectrum $T$ are 
\begin{equation}\label{pistar}
\pi_q^H(T) = \colim_V \pi_q^H(\OM^VT(V)),
\end{equation}
where the colimits are formed using the adjoint structure maps of $T$; our $G$-spectra are all connective, so that their negative 
homotopy groups are zero.  A map $T\rtarr T'$ is a stable equivalence if its induced maps of homotopy groups are isomorphisms. That depends only on large $V$.  Thus we may focus on those $V$ that contain $\bR$, so that the group completions of Theorems \ref{bigSegal} and \ref{approx} are available. 

Applying $\OM^V$ to the $V$th spaces implicit in the diagram \autoref{WOWWOW}, we obtain a diagram of $G$-spaces 
under $Y_1$.  By completely different proofs, both maps 
$$Y_1\rtarr \OM^V \bS_GY(V) \ \ \text{and} \ \  Y_1\rtarr \OM^V \bE_G Y(V)$$
are group completions. Therefore, once we prove that the arrows other than $\om$ are stable equivalences, it will follow that $\om$ is also a
stable equivalence.  Indeed, arranging as we may that our outputs are $\OM$-$G$-spectra and using that they are connective, $\om$ is a stable equivalence if and only if the map $\om_0$ it induces on $0$th $G$-spaces is a weak $G$-equivalence. The displayed group completions imply that 
$\om_0$ induces a homology isomorphism on fixed point spaces. Since these spaces are Hopf spaces, hence simple, it follows  that $\om_0$ 
induces an isomorphism on homotopy groups, so that $\om_0$ is a weak $G$-equivalence.

\begin{rem}\label{damn}
In constructing the diagram, we shall encounter an annoying but minor
clash of conventions.  There is a dichotomy in how one chooses to define the faces and degeneracies of the categorical
bar construction.  We made one choice in \autoref{bar}, but to mesh with the monadic bar construction as defined in \cite[Construction 9.6]{MayGeo},
we must now make the other.  Therefore, on $q$-simplices, we agree to replace the previous $d_i$ and $s_i$ by $d_{q-i}$ and $s_{q-i}$, respectively. 
With the new convention, $d_0$ is given by the evaluation map of the left (contravariant) variable in the categorical 
two-sided bar construction, rather than the right variable. 
\end{rem}

\subsection{Construction and analysis of the map $\pi$}\label{constpi}

Turning to the diagram \autoref{WOWWOW}, we first define the top map $\pi$.
We start by defining its source orthogonal $G$-spectrum $B(\rbull,\sD_{G,\star},Y)$. The $V$th space, as the notation indicates, is defined by plugging in $V$ for $\star$;
 it is  the bar construction $B(\rbullv,\sD_{G,V},Y)$, as defined in \autoref{bar}, namely it is the geometric realization of the simplicial space with $q$-simplices given by
the wedge over all sequences $(\bn_q^{\al_q},\dots, \bn_0^{\al_0})$ of the 
$G$-spaces
$$ (S^V)^{\bn_q^{\al_q}} \sma  \sD_{G,V}(\bn_{q-1}^{\al_{q-1}}, \bn_q^{\al_q})\sma \dots \sma \sD_{G,V}(\bn_0^{\al_0},\bn_1^{\al_1}) \sma Y(\bn_0^{\al_0}).$$
We have implicitly 
composed $Y$ with the evident projections $\sD_{G,V}\rtarr \sD_G$  to regard $Y$ as a $G\sU_{\ast}$-functor defined on each $\sD_{G,V}$, and we
have composed the $\rbullv$ with the composite $\sD_{G,V}\rtarr \sD_G \rtarr \sF_G$ to regard the $\rbullv$
as functors defined on $\sD_{G,V}$. Note that  $B(\rbull,\sD_{G,\star},Y)$ is not the restriction of a $\sW_G$-$G$-space, but it is an $\sI_G$-$G$-space. For an isometric isomorphism $V\rtarr V'$ in $\sI_G$, the map 
\[ B(\rbullv,\sD_{G,V},Y) \rtarr B(\rbullpv,\sD_{G,V'},Y)\] 
is the geometric realization of the map induced at each simplicial level by the maps $\sK_V\rtarr \sK_{V'}$
and $S^V\rtarr S^{V'}$. 

Geometric realization commutes with $\sma$, and the structure maps of the orthogonal $G$-spectrum $B(\rbull,\sD_{G,\star},Y)$ are geometric realizations of levelwise simplicial maps given by the maps  $j \colon \sD_{G,V} \rtarr \sD_{G,V\oplus W}$ induced by the inclusions $\sK_V\rtarr \sK_{V\oplus W}$ and the maps 
\begin{equation}\label{vandw}
i \colon (S^V)^{\bn^{\al}} \sma S^W \rtarr (S^{V\oplus W})^{\bn^{\al}}
\end{equation}
defined by
\[ i\big{(}(v_1,\dots,v_n) \sma w\big{)} = (v_1\sma w, \dots, v_n \sma w). \]
An alternative characterization of this construction is to use \autoref{structuremaps} together with \autoref{barcommuteswithsmash} in \autoref{UhOh} to obtain $G$-maps 
$$B(\rbullv, \sD_{G,V\oplus W}, Y) \sma S^W\rtarr B(\rbullvw, \sD_{G,V\oplus W}, Y)$$ 
and to precompose with the $G$-map 
$$ B(\rbullv, \sD_{G,V}, Y)  \rtarr B(\rbullv, \sD_{G,V\oplus W}, Y) $$ 
induced by $j\colon \sD_{G,V} \rtarr \sD_{G, V\oplus W}$.
One can easily check that these maps do indeed give maps of bar constructions that specify the structure maps for an orthogonal $G$-spectrum. 
The projections $\sD_{G,V}\rtarr \sD_G$ induce the top map $\pi$ of orthogonal $G$-spectra in \autoref{WOWWOW}.

Recall that $\colim_V  \sK_V(j)=\sK_U(j)$, so that  $\colim_V(\sC_G\times \sK_V)$ is the product
$\sC_G\times\sK_U$, which is an $E_{\infty}$ $G$-operad since it is the product of two such operads. Therefore the projection $(\sC_G\times\sK_U)(j)\rtarr \sC_G(j)$ is a $\LA$-equivalence for all $\LA\in \bF_j$ and, by \autoref{lemmaspecial}, the map $\sD_G(\sC_G\times\sK_U) \rtarr \sD_G(\sC_G)$ is  a $G$-equivalence of $G$-COs over $\sF_G$. The projection map 
$$\pi\colon B(\rbull,\sD_{G,\star},Y)\rtarr  B(\rbull,\sD_G,Y) $$ 
is not a level $G$-equivalence, but a direct comparison of colimits shows that $\pi$ is a stable equivalence.
In more detail, in computing $\pi$ on homotopy groups, we start from the commutative diagrams 
\[ \xymatrix{
 \Omega^V B(\rbullv, \sD_{G,V}, Y) \ar[r] \ar[d]_{\pi} &  \Omega^W B(\rbullw, \sD_{G,W}, Y) \ar[d]^{\pi} \\
\Omega^V B(\rbullv, \sD_G, Y) \ar[r] & \Omega^W B(\rbullw, \sD_G, Y), }\]
where $V\subset W$.  We then take $H$-fixed points and pass to homotopy groups. 
Since the inclusions $\sD_{G,V}\rtarr \sD_{G,W}$ 
become isomorphisms on homotopy groups in increasing ranges  of dimensions, by inspection of the homotopy types of the $G$-spaces comprising 
the Steiner operads in \cite[\S1.1]{GM3}, we see that $\pi$ is a stable equivalence.

\subsection{The contravariant functors $I_+\sma \rbullv$  on $\sD_{G,V}$}

In the notation $I_+\sma \rbull$ in \autoref{WOWWOW}, $\star$ is again a place holder for $V$, and the notation stands for $G\sU_{\ast}$-functors 
$$I_+\sma \rbullv\colon (\sD_{G,V})^{op}\rtarr \sT_G$$  that are given on objects by sending 
$\bn^{\al}$ to $I_+\sma (S^V)^{\bn^{\al}},$ where $I$ is the unit interval; we have adjoined a disjoint basepoint and taken
the smash product in order to have domains for based homotopies. The crux of our comparison is to specify the functors on $I_+\sma \rbullv$ on morphisms in terms of homotopies that are deduced from the paths that comprise the Steiner operads. 

Recall from \autoref{prodal} that $(S^V)^{\bn^{\al}}$  is just $\Svn$ with the $G$-action $\cdot_{\al}$ specified by  $$g\cdot_\al (x_1, \dots, x_n) = (gx_{\al(g)^{-1}(1)}, \dots, gx_{\al(g)^{-1}(n)}) = \al(g)_{*}(gx_1,\dots, gx_n).$$  Therefore, by \autoref{compGGG} and  \autoref{actact}, it is enough to instead define $G\sU_{\ast}$-functors
$$I_+\sma \rbullv\colon (\sD_{V})^{op}\rtarr \sT_G$$ given on objects by sending $\mathbf{n}$ to $I_+\sma \Svn$ 
and then apply the functor $\bP$ defined in \autoref{DDGSec} to obtain the desired functors on $\sD_{G,V}$. We choose to do this in order to make the definitions a little less cumbersome. 

We construct the required maps on hom objects as composites
\[ \xymatrix@1{ \sD_V(\mathbf m,\mathbf n) \ar[r] & \sD(\sK_V)(\mathbf m,\mathbf n) \ar[r]^-{\tilde H} & \sT_G(I_+\sma \Svn,I_+\sma \Svm).\\} \]
The first map is the evident projection, and we shall use the same letter for maps and their composites with that projection.  To define $\tilde H$, we shall construct a homotopy 
\begin{equation}\label{homotopy}
H\colon  I_+\sma \Svn \sma \sD(\sK_V)(\mathbf m,\mathbf n) \rtarr \Svm  
\end{equation} 
and then set 
\begin{equation}\label{homtoo}
\tilde{H}(f)(t,v) = (t, H(t,v,f)),
\end{equation}
where $t\in I$, $v\in \Svn$, and $f\in \sD(\sK_V)(\mathbf m,\mathbf n)$.  We have written variables in the order appropriate to thinking of the homotopies $H$ as the core of the evaluation maps of the contravariant functor $I_+\sma \rbullv\colon \sD_V \rtarr \sT_G$.   Note that such evaluation maps, after prolongation to $\sD_{G,V}$, give the zeroth face operation $d_0$ in the simplicial $G$-spaces whose realizations give the central bar constructions in \autoref{WOWWOW}.

Writing $H_t$ for $H$ at time $t$, $H_1$ will relate to the evaluation maps of the represented functor 
$\rbullv$ used in the left variable of the Segal machine
and $H_0$ will relate to the maps that define the action of the monad $\bK_V$ on the functor $\SI^V$ 
that is used in the left variable of the operadic machine. 

The following construction is the heart of the matter.  Recall that 
a Steiner path in $V$ is a map $h\colon I\rtarr R_V$ such that $h(1) = \id$,
where $R_V$ is the space of distance reducing embeddings $V\rtarr V$.  The space $\sK_V(s)$
of the Steiner operad is the space of $s$-tuples of Steiner paths $h_r$ such that the
$h_r(0)$ have disjoint images.  We define a homotopy 
\begin{equation*}\label{gamma}  \ga \colon I \times S^V \times \sK_V(s) \rtarr \Svs
\end{equation*}
with coordinates $\ga_r$, $1\leq r\leq s$, by letting 

\begin{equation*}\label{gamma2} 
\ga_r(t, v, \langle h_1,\dots,h_s\rangle) = \left\{ \begin{array}{ll}
w & \mbox{ if $h_r(t)(w) = v$}\\
\ast & \mbox{ if $v\notin \im(h_r(t)).$}
\end{array} \right. 
\end{equation*}

If $t=1$, this is just the diagonal map  $S^V\rtarr \Svs$, which is
relevant to the Segal machine.  If $t=0$, this map lands in the 
$s$-fold wedge $\sSv$ of copies of $S^V$ since the conditions $v\in \im(h_r(0))$ 
as $r$ varies are mutually exclusive; that is, there is at most one $r$ such that 
$v\in \im(h_r(0))$.  This is relevant to the operadic machine since the action map
\[ \tilde{\al}\colon \SI^V K_V X = K_V X\sma S^V \rtarr X\sma S^V= \SI^V X \]
is given by
\begin{equation*}  \tilde{\al}((\langle h_1,\dots,h_s\rangle, x_1,\dots,x_s),v) = 
\left\{ \begin{array}{ll}
(x_r,w_r) & \mbox{ if $h_r(0)(w_r) = v$}\\
\ast & \mbox{ if $v\notin \im(h_r(0))$ for $1\leq r\leq s.$}
\end{array} \right. 
\end{equation*}
Remember that we understand $\SI^V A$ to be $A\sma S^V$ for a based $G$-space $A$, but we write
the $V$ coordinate on the left when looking at the evaluation maps of the functor $I_+\sma \rbullv$.

We now define the homotopy $H$ of \autoref{homotopy}. Recall from \autoref{defGops} that 
$$\sD(\sK_V)(\mathbf m,\mathbf n)= \coprod_{\phi\colon \mathbf m\rtarr \mathbf n} \prod_{1\leq j\leq n} \sK_V(\ph_j),$$ 
where $\ph_j = |\phi^{-1}(j)|$.  Let $f = (\ph;k_1,\dots,k_n)\in \sD(\sK_V)(\mathbf m,\mathbf n)$, where $\ph\in \sF(\mathbf m,\mathbf n)$ 
and $k_j\in \sK_V(\ph_j)$. For $1\leq i\leq m$, define the $i$th coordinate of $H$ as follows.  
If $\ph(i) = j$, $1\leq j\leq n$, and
$i$ is the $r$th element of $\ph^{-1}(j)$ with its natural ordering as a subset of $\mathbf m$, then
\[ H(t, v_1,\dots,v_n, f)_i =\ga_r(t, v_j, k_j), \]
where $\ga_r$ is the $r$th coordinate of 
\[ \ga\colon I\times S^V \times \sK_V(\ph_j) \rtarr (S^V)^{\ph_j}. \]
If $\ph(i) = 0$, then the $i$th coordinate of $H$ is the trivial map.

It requires some combinatorial inspection to check that these maps do indeed specify a $G\sU_{\ast}$-functor 
$I_+\sma \rbullv\colon \sD_V^{op}\rtarr \sT_G$,
but we leave that to the reader. 
Prolonging  these functors using \autoref{actact}, we obtain the $G\sU_{\ast}$-functors 
$$I_+\sma \rbullv \colon \sD_{G,V}^{op}\rtarr \sT_G, \ \ \ \ \bn^{\al}\mapsto I_+\sma(S^V)^{\bn^{\al}},$$ 
needed to define the two-sided bar constructions 
$B(I_+\sma \rbullv ,\sD_{G,V},Y)$. 
Just as in \autoref{constpi}, the assignment $V\mapsto  B(I_+\sma \rbullv ,\sD_{G,V},Y)$ gives an $\sI_G$-$G$-space, and we can construct structure maps that 
make it into an orthogonal $G$-spectrum.   

We denote by $\rbulllv$ the restrictions of the functors $I_+\sma \rbullv$ to $t=1$. The $\rbulllv$ 
are just the functors $\rbullv $
used to define $B(\rbullv,\sD_{G,V},Y)$.  Similarly, we denote by $\rbullov$ the restrictions of the functors $I_+\sma \rbullv $ to $t=0$. 
For any  based $G$-space $A$, let $i_0$ and $i_1$ denote the inclusions of the top and bottom copy of $A$ into the cylinder $I_+\sma A$, where $G$ acts trivially on the interval $I$. Note that $i_0$ and $i_1$ are $G$-homotopy equivalences. 

The functors $\rbullov$ and $ \rbulllv$ from $\sD_{G,V}^{op}$ to $\sT_G$ are restrictions
of $I_+\sma \rbullv$, hence they commute with the face $d_0$, which is given by the evaluation maps of these functors. 
It is clear that the maps $i_0$ and $i_1$ commute with all other faces and degeneracies.  Since they are levelwise $G$-equivalences
of Reedy cofibrant simplicial $G$-spaces, their realizations 
\begin{equation}\label{i0i1} 
B(\rbullov, \sD_{G,V}, Y) \xrightarrow{i_0} B( I_+\sma \rbullv , \sD_{G,V}, Y)  \xleftarrow{i_1} B(  \rbulllv, \sD_{G,V}, Y), 
\end{equation}
are weak $G$-equivalences. Therefore 
the maps $i_0$ and $i_1$ in  \autoref{WOWWOW} are level equivalences of orthogonal $G$-spectra.

\subsection{Construction and analysis of the map $\io$}\label{iota}

To define the map of orthogonal $G$-spectra labeled $\io$ in \autoref{WOWWOW}, we must look more closely at $\rbullov$. 
Note that $H_0$ sends an element indexed on $\phi\colon \mathbf m\rtarr \mathbf n$ to an element of the product over 
$1\leq j\leq n$ of the wedge sums $\phSv$, where $|\phi^{-1}(j)|=\ph_j$. If we restrict the domain of $H_0$ to  
$\nSv \sma \sD(\sK_V)(\mb{m},\mb{n})\subset \Svn\sma \sD(\sK_V)(\mb{m},\mb{n}) $, we land in $\mSv$ since for 
any element $(v_1,\dots,v_n, f)$ in the domain of $H_0$ such that all but one of the factors 
$v_j$ is $\ast$, only those $i$ such that $\ph(i) = j$ can contribute a non-basepoint image. 
Therefore $\tilde{H}_0$ restricts to a $G\sU_{\ast}$-functor   
$\lbullov \colon (\sD_V)^{op}\rtarr \sT_G$ that on objects sends $\mathbf n$ to $\nSv $ and on morphism 
spaces is the adjoint of the restriction of $H_0$ from products to wedges.  On wedges, we have the composite
$$
\nSv \sma \sD_V(\mb{m},\mb{n}) \rtarr  \   \nSv \sma \sD(\sK_V)(\mb{m},\mb{n}) \rtarr \   \mSv
$$
of projection and the map obtained by unravelling the definition of $H$.   Upon applying prolongation $\bP$, we obtain a $G\sU_{\ast}$-functor 
$$
\lbullov \colon (\sD_{G,V})^{op}\rtarr \sT_G.
$$
It is defined  on objects by sending $\bn^{\al}$ to $^{\bn^{\al}}\! (S^V)$, and it is a subfunctor of the functor
$\rbullov \colon (\sD_{G,V})^{op}\rtarr \sT_G $.  The map $i \colon \Svn \sma S^W \rtarr (S^{V\oplus W})^n$ of \autoref{vandw} restricts to the canonical identification of
 $\nSv \sma S^W$ with $\nSvw$, and this works just as well when the twisted action of $\al$ is taken into account. Just as in \autoref{constpi}, by \autoref{barcommuteswithsmash} in \autoref{UhOh} these maps give rise to the structure maps of the $G$-spectrum $B(\lbullo,\sD_{G,\star},Y)$. The inclusions of wedges into products give the  inclusions of bar 
constructions that together specify the map of $G$-spectra labeled $\io$ in \autoref{WOWWOW}.

It is worth pausing to say what is going on philosophically before showing that $\io$ is a stable equivalence of orthogonal $G$-spectra.  The contravariant functor 
$\rbullv $ from $\sD_{G,V}$ to  $\sT_G$ is purely categorical since it factors through $\sF_G$ and applies just as well to give a  functor $A^{\bullet}$ for any $A$.
The action of $\sF_G$ on $\rbullv $ does {\em not} restrict to an action on the system of subspaces $\lbullv$.  Use of the Steiner operad in effect gives a new and more geometric functor  $\rbullov$.  It is again defined on products, but it depends on the geometry encoded in the Steiner operads and it {\em does} restrict  to a functor defined on $\lbullov$.  That is, we have commutative diagrams
 \[ \xymatrix{
 ^{\bn^{\be}}\! (S^V) \sma \sD_{G,V}(\bm^{\al},\bn^{\be}) \ar[d]_{\io\sma \id} \ar[r] & ^{\bm^{\al}}\! (S^V) \ar[d]^{\io}\\
 (S^V)^{\bn^{\be}} \sma  \sD_{G,V}(\bm^{\al},\bn^{\be}) \ar[r] & (S^V)^{\bm^{\al}} \\} \]
 where the horizontal arrows are evaluation maps of the functors  $\lbullov$, and $\rbullov$, respectively, and the vertical arrows are given by inclusions of wedges in products.
 The diagram displays the essential part of the map $d_0$ in the simplicial bar constructions that $\io$ compares; $\io$ 
 is the identity on all factors in $\sD_{G,V}$ or $Y$ of the $G$-spaces of $q$-simplices in these bar constructions.  From here, it is routine to check that these maps of bar constructions specify a map of orthogonal $G$-spectra.
 
We claim that $\io$ is a stable equivalence. Note that both the source and target of $\io$ are orthogonal $G$-spectra which at level $V$ are geometric realizations of simplicial $G$-spaces. Before passage to geometric realization, we have functors $\DE^{op}\times \sI_G\rtarr \sT_G$. It is not hard to see that the simplicial structure commutes with the spectrum structure maps, so that we can view these as simplicial orthogonal $G$-spectra, 
as discussed in \autoref{specpre}. Therefore, we can view the source and target of $\io$ as geometric realizations of simplicial orthogonal $G$-spectra and we can view $\io$ as the geometric realization of a map $\io_*$ of simplicial orthogonal $G$-spectra.   By the following two lemmas, \autoref{reedyforspectra} applies to prove the claim.

\begin{lem}  The map $\io_q$ of $q$-simplices is a stable equivalence of orthogonal $G$-spectra for each $q$. 
\end{lem}
\begin{proof} This means that $\io_q$ induces isomorphisms
\[\colim_V \pi_{*}^H\big(\OM^V B_q(\lbullov,\sD_{G,V},Y)\big)  \rtarr \colim_V \pi_{*}^H(\OM^V B_q( \rbullov,\sD_{G,V},Y)).\]  
To prove this, recall that it is standard that finite wedges are finite products in the stable category \cite[Proposition III.3.11]{Adams};
the same proof works equivariantly. The maps $j\colon \sD_{G,V} \rtarr \sD_{G,V\oplus W}$ of Steiner operads also induce isomorphisms
 on homotopy groups in increasing dimensions. The result follows by inspection of colimits. 
 \end{proof}
 
 \begin{lem}  The  source and target simplicial orthogonal $G$-spectra of $\io_*$ are Reedy $h$-cofibrant. 
\end{lem}
\begin{proof} Since for each $V$, the simplicial $G$-spaces are bar constructions,
they  are  Reedy cofibrant, so it remains to show that the homotopy extensions can be made compatibly with the orthogonal $G$-spectrum structure. To see this, note that the latching maps are constructed from the $G$-cofibration $\ast \rtarr \sC_V(1)=\sC_G(1)\times \sK_V(1)$ given by the inclusion of the identity element in $\sC_V(1)$. That map is a $G$-cofibration since $\ast \rtarr \sC_G(1)$ is a $G$-cofibration by \autoref{opass} and  $\{\id\} \hookrightarrow \sK_V(1)$ is the inclusion of a $G$-deformation retract. The explicit retraction from $\sK_V(1)$ onto $\{\id\}$ is easily seen to be compatible with the $\sI_G$-$G$-space structure and with the inclusions $j\colon \sK_V(1) \rtarr \sK_{V\oplus W}(1)$, so it is compatible with the structure maps.  This in turn implies the compatibility of the retracts for the latching maps of the bar constructions as $V$ varies.
\end{proof}
 
\subsection{Construction of the map $\om$} 

To construct the map $\om$ in \autoref{WOWWOW} and thus to complete the proof of \autoref{WOW}, we must define maps
$$B(\lbullov, \sD_{G,V}, Y)\rtarr B(\SI^V \bL_G,\bD_{G,V},Y).$$
Remember that $\bL_G Y = Y_1$.
Both source and target are realizations of simplicial (based) $G$-spaces, and we define $\om$ as the realization of a
map of simplicial $G$-spaces.   On the spaces of $0$-simplices we define 
$$\om_0\colon B_0(\lbullov, \sD_{G,V}, Y)= {\bigvee _{\bn^{\al}}}\,  ^{\bn^{\al}}\!\!S^V\sma Y(\bn^{\al}) \ \rtarr\  \SI^V Y_1= B_0(\SI^V \bL_G,\bD_{G,V},Y)$$  
to be given on each wedge summand by the composite
\[ \xymatrix@1{^{\bn^{\al}}\!\!S^V\sma Y(\bn^{\al}) \ar[r]^-{\id\sma \de} & ^{\bn^{\al}}\!\!S^V\sma Y_1^{\bn^{\al}} 
\ar[r]^-{\nu} & S^V \sma Y_1  \ar[r]^-{\ta} & Y_1\sma S^V=\SI^V Y_1.\\} \]
\noindent 
Here, for based spaces $A$ and $B$, define $\nu\colon {^n}\!A\sma B^n  \rtarr A\sma B$ by 
\[  \nu((i,a),(b_1,\dots,b_n)) = (a ,b_i)  \]
where $(i,a)$ denotes the element $a\in A$ of the $i$th wedge summand of ${^n}\!A$ and $b_j\in B$, $1\leq j\leq n$.  We check explicitly that $\nu$ is $G$-equivariant:
$$g\cdot((i, a), (b_1, \dots, b_n))=\big((\al(g)(i), g\cdot a) , (g\cdot b_{\al(g)^{-1}(1)},\dots ,g\cdot b_{\al(g)^{-1}(n)})\big)\mapsto (g\cdot a, g\cdot b_i)$$ since the $\al(g)(i)$ position is $\al(g)^{-1}\al(g)(i)=i$.  The map $\ta\colon A\sma B\rtarr B\sma A$ is the usual twist.  Then all of the maps in the definition of 
$\om_0$ are equivariant.

\begin{notn}
For $q>0$, we may write the space of 
$q$-simplices of $B(\lbullov, \sD_{G,V}, Y)$ as the wedge over pairs $(\bm^{\al},\bn^{\be})$ of
the spaces 
\[ ^{\bn^{\be}}\!S^V \sma \sD_{G,V}(\bm^{\al},\bn^{\be}) \sma Z(\bm^{\al}), \]
where $Z(\bm^{\al})$ denotes the wedge over sequences $(\bm_0^{\al_0},\dots, \bm_{q-2}^{\al_{q-2}}))$ of the spaces 
\[  \sD_{G,V}(\bm_{q-2}^{\al_{q-2}},\bm^{\al})\sma \cdots  \sma \sD_{G,V}(\bm_0^{\al_0},\bm_1^{\al_1})\sma  Y(\bm_0^{\al_0}). 
\]
\end{notn}

Recall the definition of $\bD_{G,V} Y$ from \autoref{DGY}. The $G$-space $(\bD_{G,V} Y)(\bn^{\beta})$ is a quotient of the wedge over all $\bm^{\al}$ of the $G$-spaces
$$\sD_{G,V}(\bm^{\al},\bn^{\be})\sma Y(\bm^{\al}).$$
Therefore the space $\SI^V\bL_G \bD_{G,V}^q Y$ 
of $q$-simplices of $B(\SI^V\bL_G,\bD_{G,V},Y)$ is a quotient of the wedge over all $\bm^{\al}$ of the spaces
\[  \sD_{G,V}(\bm^{\al}, \mathbf{1}) \sma Z(\bm^{\al})\sma S^V. \]
Define $\om_q$ by passage to wedges over $\bm^{\al}$ from the composites 
\[ \xymatrix{ 
^{\bn^{\be}}\!S^V \sma \sD_{G,V}(\bm^{\al},\bn^{\be}) \sma Z(\bm^{\al})
\ar[d]^{\id\sma\de\sma\id}\\
^{\bn^{\be}}\!S^V \sma \sD_{G,V}(\bm^{\al}, {\bf 1})^{\bn^{\be}} \sma Z(\bm^{\al}) \ar[d]^{\nu\sma \id}\\
S^V \sma \sD_{G,V}(\bm^{\al}, \mathbf{1}) \sma Z(\bm^{\al})  \ar[d]^{\ta}\\
 \sD_{G,V}(\bm^{\al}, \mathbf{1}) \sma Z(\bm^{\al}) \sma S^V.}
\]
Since we know that these maps are $G$-equivariant, to check commutative diagrams we may drop the $\al$'s and $\be$'s
 from the notation and only consider the underlying nonequivariant spaces. On underlying spaces, the composite above is 
\[ \xymatrix{ 
{^n}\!S^V \sma \sD_V(\mathbf m,\mathbf n) \sma Z_m 
\ar[d]^{\id\sma \de\sma\id}\\
{^n}\!S^V \sma \sD_V(\mathbf m,\mathbf 1)^n \sma Z_m \ar[d]^{\nu\sma \id}\\
S^V \sma \sD_V(\mathbf{m},\mathbf{1})\sma Z_m \ar[d]^{\ta}\\
\sD_V(\mathbf m,\mathbf 1) \sma Z_m\sma S^V.\\}
\]
We must show that these maps $\om_q$ specify a map of simplicial spaces.  Again recall \autoref{damn}. 
In both the categorical and monadic bar constructions, 
the face maps $d_i$ for $i>1$ are induced by composition 
in $\sD_V$ and the action of $\sD_V$ on $Y$.  Commutation of $\om$ with these face
maps is evident.  Similarly, commutation with the degeneracy maps $s_i$ for $i>0$ is evident.
We must show commutation with $s_0$, $d_0$, and $d_1$.  An essential point is that the components of the
Segal maps are in $\PI_G$, and we are taking the categorical tensor product over $\PI_G$ in the target.    
First consider $s_0$ on zero simplices.  For $(i, v)$ in $^n\!S^V$, that is, $v$ in the $i$th summand, 
and $y\in Y_n$,
\begin{eqnarray*}
\om s_0((i,v),y) & = & \om((i,v),\id_n,y) \\
               & = & (\de_i,y)\sma v\\
               & = & (\id_1,\de_i(y))\sma v \\
               & = & s_0 \om((i,v),y).\\
\end{eqnarray*}
Here $\id_n \in \sD_V(\mathbf n,\mathbf n)$, the third equation uses $\de_i=\id_1\com \de_i$ 
and the equivalence relation defining $\bD_V Y$, and the last equation uses that $\om((i,v),y) = \de_i(y)\sma v$.
The commutation of $\om$ and $s_0$ on $q$-simplices for $q>0$ is similar.   The following diagrams prove the
commutation of $\om$ with $d_0$ and $d_1$ on $1$-simplices, and the argument for $q$-simplices for $q>1$ is 
similar.   Again, we only write the maps of underlying spaces, dropping from the 
notation the indices that indicate $G$-actions since the $G$-action is not relevant to checking that the diagrams commute.

\[ \xymatrix{ 
{^n}\!S^V \sma \sD_V(\mathbf m,\mathbf n)\sma Y_m \ar[d]_{\mathrm{\id} \sma \delta \sma \mathrm{id}}  
\ar[rrr]^-{H_0 \sma \mathrm{id}} & & & {^m}\!S^V  \sma Y_m \ar[d]^{\mathrm{id}\sma \de} \\
{^n}\!S^V \sma \sD_V(\mathbf m,\mathbf 1)^n \sma Y_m  \ar[d]_{\nu\sma \mathrm{id}} 
&&& {^m}\!S^V \sma Y_1^m \ar[d]^{\nu}\\
S^V\sma \sD_V(\mathbf m,\mathbf 1) \sma Y_m  \ar[rrruu]_-{H_0\sma \mathrm{id}}  \ar[d]
& &  &  S^V\sma Y_1  \ar[d] \\
\Sigma^V \bL\bD_V Y \ar[r]_-\cong  & \Sigma^V \bC_V \bL Y \ar[rr]_-{\hat{\alpha}} && \Sigma^V Y_1\\
}\]

\[ \xymatrix{ 
{^n}\!S^V \sma \sD_V(\mathbf m,\mathbf n)\sma Y_m \ar[d]_{\mathrm{\id} \sma \delta \sma \mathrm{id}}  
\ar[rrr]^-{\id \sma \tha} & & & {^n}\!S^V \sma Y_n \ar[d]^{\mathrm{id}\sma \de} \\
{^n}\!S^V \sma \sD_V(\mathbf m,\mathbf 1)^n \sma Y_m  \ar[d]_{\nu\sma \mathrm{id}} 
&&& {^n}\!S^V \sma Y_1^n \ar[d]^{\nu}\\
S^V\sma \sD_V(\mathbf m,\mathbf 1) \sma Y_m  \ar[rrr]^-{\id\sma \tha}  \ar[d]
& & &  S^V\sma Y_1  \ar[d] \\
\Sigma^V \bL\bD_V Y \ar[rrr]_-{\SI^V \bL\tha}  & && \Sigma^V Y_1\\
}\]

Here $\tha$ denotes the action of $\bD_V$ on $Y$, which is given by the adjoints of the $G$-maps $Y\colon \sD_V(\mathbf m,\mathbf n)\rtarr \sT_G(Y_m,Y_n)$. 
Both top pieces of the diagrams commute by formal inspection, the trapezoid in the first diagram commutes by the definitions of $H_0$ and $\hat{\al}$, as recalled
in the previous section, and the lower rectangle in the second diagram commutes by definition.  It is not hard to check that the maps $\om$ are maps of $\sI_G$-$G$-spaces 
and that they are compatible with the structure maps, so that they give a map of orthogonal $G$-spectra.

\section{Proofs of technical results about  the operadic machine}\label{SegOpPf}

We prove  Theorems  \ref{lemmaspecial} and  \ref{homotopMT} in this section.  

\subsection{Levelwise $G$-equivalences of categories of operators over $\sF_G$}\label{levelwise}
With the notations and hypotheses of \autoref{lemmaspecial}, we must prove that the map
 \[\nu_G\colon \sD_G(\bm^{\al}, \bn^{\be}) \rtarr \sD'_G(\bm^{\al}, \bn^{\be})\]
 is a weak $G$-equivalence for all $\bm^{\al}$ and $\bn^{\be}$. Recall that $\sD_G(\bm^{\al}, \bn^{\be})$ is 
 just $\sD(\mathbf m,\mathbf n)$ with the $G$-action given by $ g\cdot f= \be(g) \circ (gf) \circ \al(g^{-1})$. 
 
 Let $H$ be a subgroup of $G$. We claim that there is a homeomorphism
 \[\big[\sD_G(\bm^{\al},\bn^{\be})\big]^H \cong \coprod_{\phi} \prod _i \sC(d_i)^{\LA _i},\]
where $\phi$ runs over the $H$-equivariant maps $\bm^{\al} \rtarr \bn^{\be}$, $i$ runs over the $H$-orbits of $\bn^{\be}$, and $d_i$ and $\LA_i$ depend only on $\phi$ and $H$, with $\LA_i$ in $\bF_{d_i}$. This homeomorphism is moreover compatible with the map $\nu_G$. Thus  our hypothesis gives that the map
\[(\nu_G)^H\colon \big[\sD_G(\bm^{\al}, \bn^{\be})\big]^H \rtarr \big[\sD'_G(\bm^{\al}, \bn^{\be})\big]^H\]
 is a weak equivalence for all subgroups $H$, as wanted.
 
 To prove the claim, recall that 
 \[\sD(\mb{m},\mb{n})=\coprod_{\phi\colon \mathbf{m}\rtarr \mathbf{n}} \prod_{1\leq j\leq n} \sC(\phi_j),\]
 where $\phi_j=|\phi^{-1}(j)|$.
 Using \autoref{defGops} and in particular \autoref{obsGops}, the new action on $\sD(\mb{m},\mb{n})$ is given by
 {\small{
\[
 g\cdot(\phi; x_1, \dots, x_n) = (\be(g)\phi\al(g^{-1}); gx_{\be(g^{-1})(1)}\si_{\be(g^{-1})(1)}(g^{-1}), 
 \dots, gx_{\be(g^{-1})(n)}\si_{\be(g^{-1})(n)}(g^{-1})).
\]}}

\noindent
Here $\si_j(g^{-1})\in \Sigma_{(\phi\com\al(g^{-1}))_j}= \Sigma_{\phi_j}$ is that permutation of 
$(\phi\com\al(g^{-1}))_j$ letters which converts 
the natural ordering of $(\phi\com\al(g^{-1}))^{-1}(j)$ as a subset of $\{1,\dots,m\}$ to
its ordering obtained by regarding it as $\coprod_{\ph(i)=j}\al(g)(i)$, so ordered that $\al(g)(i)$ precedes $\al(g)(i')$ if $i< i'$. 

Note that the component corresponding to $\phi$ is nonempty in the $H$-fixed points if and only if $\phi$ is $H$-equivariant. 
In what follows we fix such a $\phi$. Careful analysis of the definition of $\si_j(h)$ shows that for $j\in\{1,\dots,n\}$, and $h,k\in H$, we have
\begin{equation}\label{sigma}
\si_j(hk)=\si_j(h)\si_{\be(h^{-1})(j)}(k).
\end{equation}

The $H$-action shuffles the indices within each $H$-orbit of $\bn^{\be|_H}$, so it is enough to consider each $H$-orbit separately. We can assume then that the $H$-action on $\bn^{\be|_H}$ is transitive. The rest of the proof is analogous to the proof of \autoref{lem:lafixedpoints}. 

Since $\phi$ is $H$-equivariant and the action is transitive, all the sets $\phi^{-1}(j)$ have the same cardinality, say $d$. Let $K$ be the stabilizer of $1\in \mb{n}$ under the action of $H$. By \autoref{sigma}, $\si_1$ restricted to $K$ is ahomomorphism, and thus 
\[\LA=\{(k,\si_1(k)\mid k\in K\} \subseteq G \times \SI_d\]
is a subgroup that belongs to $\bF_d$. To complete the proof of the claim, we note that the projection to the first coordinate induces a homeomorphism
\[\big(\sC(d) \times \cdots \times \sC(d)\big) ^H \rtarr \sC(d)^{\LA}.\]
One can easily check that if $(x_1,\dots,x_n) \in \sC(d)^n$ is an $H$-fixed point, then $x_1$ is a $\LA$-fixed point, since for all $k\in K$ we have that
\[(k,\si_1(k))\cdot x_1 = kx_1\si_1(k^{-1}) = kx_{\be(k^{-1})(1)}\si_{\be(k^{-1})(1)}(k^{-1})=x_1,\]
the last equality being true by the assumption that $(x_1,\dots, x_n)$ is fixed by $H$.
To construct an inverse, for every $j$ choose $h_j\in H$ such that $\be(h_j)(1)=j$. Note that choosing these amounts to choosing a system of coset representatives for $H/K$. Consider the map $\sC(d) \rtarr \sC(d)^n$ that sends $x$ to the $n$-tuple with $j$th coordinate
\[x_j=h_jx\si_1(h_j^{-1}).\]
We leave it to the reader to check that this map restricts to the fixed points and is inverse to the projection.

\subsection{The structure of $\bD X$}\label{BDXStruc}

In the rest of this section,   let $\sC_G$ be a $\SI$-free $G$-operad and let $\bD$ be the monad on
$\PI$-$G$-spaces associated to the category of operators $\sD=\sD(\sC_G)$.  Part (i) of  \autoref{homotopMT} asserts that $\bD$ preserves \gen-equivalences,
and its proof is the hardest equivariant work we face.  It involves a detailed combinatorical analysis of the structure of $\bD X$ for a $\PI$-$G$-space $X$. This entails 
combinatorial analysis of $\PI$ and $\sF$ that will also be relevant to the technical proofs 
for the Segal machine in \autoref{SecSegDet}. 

Fix $n$.   Then the definition of  $(\bD X)_n$ given in \autoref{DX} implies that it is the quotient
\[\bigg( \bigvee_{q} \sD(\mb{q}, \mb{n}) \sma X_q\bigg)/(\sim) \]
where $\sim$ is the equivalence relation specified by 
$$(\ps^*d; x)\sim (d; \ps_{*}x)$$ 
for $d \in \sD(\bf q,\bf n)$, $\ps\in \PI(\bf p, \bf q)$ and $x\in X_p$.  By \autoref{OK}, we can replace wedges and smash products by disjoint unions and products, that is, $(\bD X)_n$ is the quotient 
\[\bigg( \coprod_{q} \sD(\mb{q}, \mb{n}) \times X_q\bigg)/(\sim). \]

Recall that the morphism space $\sD(\bf q,\bf n)$ is given by the disjoint union of components indexed on all $\ph\colon \bf{q}\rtarr \bf{n}$ in $\sF$ 
$$\sD(\mb{q},\mb{n}) = \coprod_{\ph\in \sF( \mb{q}, \mb{n})} \prod_{1\leq j\leq n}   \sC_G( \ph_j),$$ where $\ph_j= |\ph^{-1}(j)| $.
The basepoint is the component indexed on $\ph = 0_{q,n}$.  We write a non-basepoint morphism as $(\phi;c)$, where 
$c = (c_1,\dots, c_n)$ with $c_j \in \sC_G(\ph_j)$. For  a morphism $\ps\colon \bf  p\rtarr \bf q$ in $\PI$,  write 
\[\ps^*\colon \prod_{1\leq j \leq n} \sC_G(\ph_j) \rtarr \prod_{1\leq j \leq n} \sC_G((\ph\circ\ps)_j)\] 
for the map $\sD(\bf q,\bf n)\rtarr \sD(\bf p,\bf n)$ induced by $\ps$ from the component of $\ph$  to the component of $\ph\com \ps$, and write  
$\ps_{*}\colon X_p \rtarr X_q$ for the induced morphism
giving the covariant functoriality of $X$.     Then the equivalence relation $\sim$ takes the form
\begin{equation}\label{sim}
(\ph\com \ps; \ps^*c; x)\sim (\ph; c; \ps_{*}x)
\end{equation}
for $(\ph;c)\in \sD(\bf q,\bf n)$ and $x\in X_p$. We shall use the identifications induced by $\sim$ to cut down on the number of components 
that need be considered.  To this end, we first describe the structure of $\PI$ and $\sF$, partially following \cite[\S5]{MT}.  

\begin{defn}\label{PIF}
Recall that $\PI$ is the subcategory of $\sF$ with the same objects
and those maps $\ph\colon \mathbf{m} \rtarr \mathbf{n}$ such that $\ph_j \leq 1$ for $1\leq j\leq n$.  
A map $\pi\in \PI$ is a \emph{projection} if $\pi_j= 1$ for $1\leq j\leq n$.  A map
$\io\in \PI$ is an \emph{injection} if $\io^{-1}(0) = \{0\}$. The permutations are the maps in $\PI$ that 
are both injections and projections. A projection or injection is \emph{proper} if it is not a permutation.
Recall that $\inj$ is the subcategory of injections in $\PI$.  
\end{defn}

\begin{defn}\label{effective}
A map $\ph\in \sF$ is \emph{ordered} (or more accurately monotonic) if $i<j$ implies $\ph(i) \leq \ph(j)$; 
note that this does not restrict the ordering of those $i$ such that $\ph(i) = k$ for some fixed $k$.
A map $\epz\in \sF$ is {\em effective} if $\epz^{-1}(0) = \{0\}$, and an effective map $\epz$ 
is {\em essential} if it is surjective, that is, if $\epz_j \geq 1$ for $1\leq j\leq n$.
\end{defn}

Observe that every morphism of $\PI$ is a composite of proper projections, proper injections, and 
permutations, and that $\sF$ is generated under wedge sum and composition by $\PI$ and the single
product morphism $\varphi_2\colon \mathbf{2}\rtarr \mathbf{1}$.

\begin{lem}\label{factor} A map $\ph\colon \mathbf{m}\rtarr \mathbf{n}$ in $\sF$ factors as the composite 
$\io\com \epz\com \pi$ of a projection $\pi$, an essential 
map $\epz$ and an injection $\io$, uniquely up to permutation.  That is, given two such decompositions
of $\ph$, there are permutations $\si$ and $\ta$ making the following diagram commute.
\[ \xymatrix{
& \mathbf{q} \ar[rr]^-{\epz} \ar[dd]^{\si} & & \mathbf{r}  \ar[dr]^{\io} \ar[dd]^{\ta} &\\
\mathbf{m} \ar[ur]^-{\pi} \ar[dr]_{\pi'} & & & & \mathbf{n} \\
& \mathbf{q} \ar[rr]_-{\epz'} & & \mathbf{r} \ar[ur]_-{\io'} & \\} \]
\end{lem}
\begin{proof} The projection $\pi$ is determined up to order by which $i\geq 1$ in $\mathbf{m}$ are mapped to 
$0$ in $\mathbf{n}$. The injection $\io$ is determined up to order by which $j\geq 1$ in $\mathbf{n}$ are not in 
the image of $\ph$. Up to order, $\epz$ is the wedge sum in $\sF$ of the product maps 
$\varphi_{\ph_j}\colon \mathbf{\boldsymbol\ph_j}\rtarr \mathbf{1}$, where $\ph_j = |\ph^{-1}(j)|$ for those $j$ such that $1\leq j\leq n$ and
$\ph^{-1}(j)$ is nonempty. Up to permutation, these $\ph_j$ run through the numbers $\epz_j$, $1\leq j\leq r$. 
\end{proof}

\begin{rem}\label{matchup}  We remark that $\PI$ and $\sF$ are 
dualizable Reedy categories,  as 
defined by Berger and Moerdijk \cite[Definition 1.1 and Example 1.9(b)]{BergMoer}, 
although we have chosen not to use that language. They write 
$\sF^+$ for the monomorphisms and $\sF_{-}$ for the epimorphisms in $\sF$.  We have factored epimorphisms into 
composites of projections and essential maps to make the structure clearer.
\end{rem}

A map $\ph\colon \bf  q\rtarr \bf n$ in $\sF$ is ineffective if and only if it factors as a composite 
$$\xymatrix@1{\bf q  \ar[r]^-{\pi} & \bf p \ar[r]^-{\ze} & \bf n,\\} $$
where $p=q- \ph_0$, $\pi$ is the proper ordered projection such that $\pi(k) = 0$ if and only if $\ph(k) = 0$, and $\ze$ is an effective morphism.
Then $\ze_j = \ph_j$ for $j\geq 1$ and, as in \autoref{obsGops},  $\pi^*(c)=c$ for any $c\in \prod_j \sC_G(\ze_j)$. Therefore 
$$(\ph;c;x)=(\ze\circ\pi;\pi^*(c);x)\sim (\ze;c;\pi_{*}(x)).$$ 
This says both that we may restrict to those wedge summands that are indexed on the effective morphisms of $\sF$, ignoring the ineffective
ones, and that we can ignore the proper projections in $\PI$, restricting further analysis to $\sim$ applied to morphisms of $\inj\subset \PI$. 
Here we must start paying attention to permutations.

\begin{lem}\label{order}
If $\epz\colon \mathbf{p}\rtarr \mathbf{n}$ is an effective morphism in $\sF$, there is a permutation $\nu\in \SI_p$ such that 
$\epz\com\nu$ is ordered; $\nu$ is not unique, but the ordered morphism $\epz \circ \nu$ is independent of the choice. 
\end{lem}

Applying $\sim$ to the permutations $\nu$, we can further restrict to components indexed on ordered effective morphisms.  
We abbreviate notation.

\begin{notn} We say that an ordered effective morphism in $\sF$ is an $OE$-function.  If $\epz$ is effective and $\epz\com \nu$ is ordered, we call it the $OE$-function associated to $\epz$.  We let $\sE(\bf{p},\bf{n})$ denote the set of 
all $OE$-functions $\bf{p}\rtarr \bf{n}$.
\end{notn} 

\begin{defn} Let $\epz \colon \mb{p} \to \mb{n}$ be an $OE$-function, with  $\epz_j = |\epz^{-1}(j)|$.  The sum of the $\epz_j$ is $p$ and we define 
$$\SI(\epz) =\SI_{\epz_1}\times \cdots \times \SI_{\epz_n} \subset \SI_p,$$
where the inclusion is determined by identifying  $\bf p$ with $\bf{\boldsymbol\epz_1} \wed \cdots \wed \bf{\boldsymbol\epz_n}$. 
In other words, we partition $\{1, \dots, p\}$ into $n$ blocks of letters, as dictated by $\epz$. 
\end{defn}

\begin{lem}\label{order2} 
If $\epz\colon \mb{p} \to \mb{n}$ is an $OE$-function and $\nu\in \SI_p$, then $\epz \com \nu$ is ordered (and hence equal to $\epz$) if and only 
$\nu$ is in the subgroup $\SI(\epz)$.
\end{lem}

The equivalence relation $\sim$ is defined in terms of precomposition of morphisms of $\sF$ with morphisms of $\PI$,
while the action of $\PI$ on $\bD X$ is defined in terms of postcomposition of morphisms of $\sF$ with morphisms of $\PI$.
Especially for permutations, these are related.  We discuss composition with permutations on both sides in the following three remarks.

\begin{rem}\label{tau}
Let $\epz\colon \bf p\rtarr \bf n$ be an $OE$-function and let $\si\in \SI_n$. Define $\ta(\si)\in \SI_p$  to be the permutation that 
permutes the $n$ blocks of size $\{\epz_1, \dots, \epz_n\}$  as  $\si$ permutes $n$ letters. 
Then $\si \com \epz \com  \ta(\si)^{-1}$ is again ordered, and it is the $OE$-function associated to $\si\com \epz$.  Moreover, the 
function $\ta = \ta_{\epz}\colon \SI_n\rtarr \SI_p$ is a homomorphism. Inspecting our identifications and using \autoref{obsGops}, 
we see that  postcomposition with $\si$ sends a point with representative $(\epz;c_1,\dots,c_n;x)$ to the point with representative
\begin{equation}\label{sigmaact}
 \big(\si\epz\tau(\si)^{-1}; c_{\si^{-1}(1)} , \dots , c_{\si^{-1}(n)}; \tau(\si)_\ast x\big).
 \end{equation}
\end{rem}

\begin{rem}\label{partition}  One can think of an $OE$-function $\epz\colon \bf{p}\rtarr \bf{n}$ as an ordered partition of the set $\{1,\dots,p\}$ into
$n$ ordered subsets.  Observe that $\epz$ is essential if and only if $\epz_j > 0$ for all  $1\leq j\leq n$, so that our $n$ subsets are all nonempty.  
Define the signature of $\epz$ to be the unordered set of numbers $\{\epz_1,\dots,\epz_n\}$.   The action of $\SI_n$ permutes partitions with the same signature (and thus the same $p$).  The $OE$-functions $\epz$ and $\si \com \epz \com \tau(\si)^{-1}$ have the same signature, and any two ordered partitions with the same signature are connected this way.
\end{rem}

\begin{rem} If $\epz\colon \bf p\rtarr \bf n$ is an $OE$-function and $\rh\in \SI_p$, then 
$$(\epz\com \rh; c; x)  \sim (\epz; (\rh^{-1})^*c;\rh^{-1} _{*} x).$$
\end{rem} 

We have now accounted for $\sim$ applied to all proper projections and to all permutations.  It remains to consider proper
injections $\io$. For any such $\io\colon \bf p\rtarr \bf q$, there is a permutation $\nu\in \SI_p$ such that $\io \com \nu$ is ordered.
Recall that we have the ordered injections $\si_i\colon  \bf{p-1} \rtarr \bf p$ that skip $i$, $1\leq i\leq p$.  Every proper ordered
injection is a composite of such $\si_i$, so it remains to account for $\sim$ applied to the $\si_i$. 
These give an equivalence relation on

$$\coprod_q \coprod_{\epz\in \sE(\bf{q},\bf{n})} \bigg( \prod_{1\leq j\leq n} \sC_G(\epz_j)\bigg) \times_{\SI(\epz)} X_{q}$$
whose quotient is $(\bD X)_n$.   The component with $q=0$ gives the basepoint.

The monad $\bD$, like the monad $\bC_G$, is filtered.   Its $p$th filtration at level $n$, denoted $F_p(\bD X)_n$, is the image of the components indexed 
on $q\leq p$. We can think of the quotient as given by filtration-lowering  ``basepoint identifications'', namely 
\begin{equation}
(\epz; c_1,\dots, c_n; (\si_i)_{*} x) \sim (\epz\com \si_i; c_1, \dots, c_{i-1}, \si_r c_i, c_{i+1}, \dots,c_n; x),
\end{equation} 
for some $i=1,\dots,q$. Here $r$ is the position of $i$ within its block of $\epz_j = |\epz^{-1}(j)|$ letters, where $j= \epz(i)$, and $\si_r\colon \sC_G(\epz_j) \rtarr \sC_G(\epz_j -1)$ 
is the map from 
\autoref{degensig}.\footnote{For consistency of notation here, we might have written $\si_r^*$ instead of $\si_r$, as appears in \autoref{degensig}.} In equivalent abbreviated notation, we write this as
\begin{equation}\label{Yes2}
(\epz; c ; (\si_i)_{*} x) \sim (\epz\com \si_i; \si_i^*c ; x)
\end{equation} 

We give a definition of \emph{latching  spaces} for functors $ \PI \rtarr \UG$, in the spirit of the usual definition for simplicial spaces. This definition coincides with the one in \cite{BergMoer} when considering $\PI$ as a dualizable Reedy category (see \autoref{matchup}).
Via \autoref{PIReedytrue} below, this concept will allow us to make inductive arguments using the pushout diagram \autoref{pushout}.  This is similar to classical arguments involving simplicial spaces.

\begin{defn}\label{PIReedy}  For a functor $X\colon \PI \rtarr \UG$, the ordered injections $\si_i\colon \mathbf{n-1}\rtarr \mathbf{n}$ 
induce maps $(\si_i)_*\colon X_{n-1} \rtarr X_n$.  A point of $X_n$ in the image of some $(\si_i)_\ast$ is said to be degenerate.   We define the $n$th 
latching space of $X$
to be the set of degenerate points in $X_n$:
\[L_nX = \bigcup _{i=1}^n (\si_i)_*(X_{n-1}).\]
It is a $(G\times \SI_n$)-space, and the  inclusion $L_nX\rtarr X_n$ 
is a $(G\times \SI_n)$-map. For a $G$-space $X$, $L_n\bR X$ is the subspace 
of $X^n$ consisting of those points at least one coordinate of which is the basepoint.
\end{defn} 

\begin{lem}\label{PIReedytrue}  If $X$ is a $\PI$-$G$-space, then the inclusion  $L_n X\rtarr X_n$ is a $(G\times\SI_n)$-cofibration.
\end{lem}
\begin{proof}
This is a direct application of  \cite[Theorem A.2.7]{BVbook}.
\end{proof}

Following up \autoref{PIReedy}, it is helpful to be more explicit about degenerate elements.

\begin{defn} Fixing an $\epz\in \sE(\bf p, \bf n)$, write $[c;x]$ for an element of
$$ \prod_{1\leq j\leq n} \sC_G(\epz_j )\times_{\SI(\epz)} X_{p},$$ 
meaning that $(c;x)$ is a representative element for an orbit $[c;x]$  
under the action of $\SI(\epz)$. The element  $x\in X_p$ is degenerate  if $x\in L_pX$, that is, if $x =(\si_i)_{*} y$ for some $y\in X_{p-1}$ and some $i$, and we say that $(\epz;[c;x])$  is degenerate if $x$ is degenerate. Since $\tau\com \si_i$ is a proper injection for any $\ta\in \SI(\epz)$ and any $i$, the 
condition of being degenerate is independent of the orbit representative $(c;x)$.
\end{defn}

By use of  \autoref{Yes2}, we reach the following description of the elements of $(\bD X)_n$.    

\begin{lem}  A point of $(\bD X)_n$ has a unique nondegenerate representative $(\epz; [c;x])$.
\end{lem}

Just as nonequivariantly  (\cite[p. 218]{MT}), we have pushouts of $(G\times \SI_n)$-spaces
\begin{equation}\label{pushout} \xymatrix{  \coprod_{\epz\in \sE(\bf{p},\bf{n})} \big{(} \prod_{1\leq j\leq n} \sC_G(\epz_j)\big{)} \times_{\SI(\epz)}  L_{p}X \ar[r]^-{\nu} \ar[d]
& F_{p-1} (\bD X)_n \ar[d] \\
\coprod_{\epz\in \sE(\bf{p},\bf{n})}  \big{(}\prod_{1\leq j\leq n} \sC_G(\epz_j)\big{)} \times_{\SI(\epz)} X_{p} \ar[r] 
& F_{p} (\bD X)_n \\} \end{equation}
With notation as in \autoref{Yes2}, the map $\nu$ sends a point with orbit representative $(\epz; c; (\si_i)_{*} x)$ to the point with orbit 
representative $(\epz\com \si_i;  \si_i^*c; x)$.

By \autoref{PIReedytrue}, when $X$ is a $\PI$-$G$-space  the inclusion $L_pX \rtarr  X_p$ is a 
$(G\times \SI_p)$-cofibration for each $p$.   Then each component of the left vertical map is a $(G\times \SI_p)$-cofibration before passing to  orbits under 
 $\SI(\epz)$, so the map on orbits is a $G$-cofibration by \cite[Lemma A.2.3]{BVbook}. Since $\SI_n$ acts by permuting components and the factors of the displayed products, it follows that the left vertical map is  a 
$(G\times\SI_n)$-cofibration, hence so is the right vertical arrow.  

\subsection{The proof that $\bD$   preserves $\PI$-$G$-spaces}\label{DXReedy}We here complete the proof of \autoref{Dgood} by proving that $\bD$ preserves the cofibration condition in \autoref{Fspace}. By \autoref{onlyordered}, we must prove that for every ordered injection $\phi \colon \bm \rtarr \bn$ in $\PI$, the induced map $\phi_* \colon (\bD X)_m \rtarr (\bD X)_n$ is a $(G\times \SI_\ph)$-cofibration. To do so, we adopt the following notation.

Let $T\subset \mb{n}\setminus \{0\}$. We write $ \SI_T$ for the subgroup of $\SI_n$ of those permutations $\tau$ such that $\tau(T)=T$. Note that $\SI_T$ consists of those permutations that act separately within $T$ and its complement. As a group, $\SI_T$ it is isomorphic to $\SI_{|T|} \times \SI_{n-|T|}$. 

\begin{notn}\label{sigmaT}
We denote by $\si_T \colon \mb{n-|T|} \rtarr \mb{n}$ the ordered injection that misses the elements of $T$. It can be written as $\si_T=\si_{i_k}\cdots \si_{i_1}$ where $i_1<\cdots < i_k$ are the elements of $T$. If $T$ is empty we use the convention that $\si_T=\id$ (which makes sense as the empty composition). For a $\PI$-$G$-space $X$, there is an induced map $(\si_T)_\ast\colon X_{n-|T|} \rtarr X_n$, which for the sake of readability we denote by $\si_T$. Then, for a $\PI$-$G$-space $X$, 
\[\si_T X_{n-|T|} = \bigcap _{i\in T} \si_i X_{n-1}.\] 
For the case $T=\emptyset$, we interpret this as $\si_T X_n = X_n$.
\end{notn}

Note that any ordered injection $\phi$ is of the form $\si_T$ where $T=\bn \setminus \im(\phi)$, and in that case, $\SI_\ph=\SI_T$. Thus, it suffices to show that for all subsets $T$, the maps 
\begin{equation}\label{indT}
\si_T\colon (\bD X)_{n-|T|} \rtarr (\bD X)_n
\end{equation} 
are  $(G\times \SI_T)$-cofibrations. Note that the action of $\SI_T$ on $(\bD X)_{n-|T|}$ is given by restricting to the action on the block ${\bf n}\setminus T$ and identifying it with the 
set ${\bf n -|T|}$.  

Consider the cube obtained by mapping the pushout square of \autoref{pushout} for $(\bD X)_{n-|T|}$ to the pushout square \autoref{pushout} for $ (\bD X)_n$. 
Write $\si_T$ for the maps from the four corners of the first square to the four corners of the second square.
We will prove by induction on $p$ that the map 
\begin{equation}\label{indp}
\si_T \colon F_p(\bD X)_{n-|T|} \rtarr F_p(\bD X)_n
\end{equation}
is a $(G\times \SI_T)$-cofibration, and we assume this for the map with $p$ replaced by $p-1$.  The vertical maps in the diagram \autoref{pushout} for $\bD X_{n-|T|}$ are 
$(G\times \SI_{n-|T|})$-cofibrations, so they are $(G\times \SI_T)$-cofibrations via the action described above. The vertical maps in the diagram \autoref{pushout}  for $\bD X_n$ are 
$(G\times \SI_n)$-cofibrations, so in particular they are also $(G \times \SI_T)$-cofibrations.

The map $\si_T$ on the left corners of the diagram is given by
\[(\epz, [(c_1, \dots, c_{n-|T|}), x)] \mapsto (\si_T\circ \epz, [(d, x)]),\]
where $d$ is the $n$-tuple given by
\[d_j=\begin{cases} 0\in \sC_G(0) &\text{if } j \in T\\
 c_{\si_T^{-1}(j)} & \text{if } j \not\in T
 \end{cases}
 \]
 It is not hard to see that for the left entries of the pushout diagram, the map $\si_T$ is the inclusion of those components labeled by maps $\epz \colon \mb{p} \rtarr \mb{n}$ that miss the elements of $T$. The groups $\SI_{n-|T|}$ and $\SI_n$ act on  the source and target, respectively, as stated in \autoref{tau}. In particular, both actions shuffle components, hence the inclusion of components is a $(G\times \SI_T)$-cofibration.  By the induction hypothesis,  the map connecting the top right corners of the cube is also a $(G\times \SI_T)$-cofibration. 

\autoref{glueingcof} below, which concerns the preservation of cofibrations by pushouts, implies that the map \autoref{indp} connecting the bottom right corners of the cube is a $(G\times \SI_T)$-cofibration,  as claimed, noting that
\[\si_T \Bigg(\coprod_{\epz\in \sE(\bf{p},\bf{n-|T|})}  \big{(}\prod_{j=1}^{n-|T|} \sC_G(\epz_j)\big{)} \times_{\SI(\epz)} X_{p}\Bigg) \cap \Bigg( \coprod_{\epz\in \sE(\bf{p},\bf{n})}  \big{(}\prod_{j=1}^n \sC_G(\epz_j)\big{)} \times_{\SI(\epz)} L_pX \Bigg)\]
is equal to
\[\si_T \Bigg( \coprod_{\epz\in \sE(\bf{p},\bf{n-|T|})}  \big{(}\prod_{j=1}^{n-|T|} \sC_G(\epz_j)\big{)} \times_{\SI(\epz)} L_pX\Bigg).\] 

\autoref{sequentialcolim} below, which concerns the preservation of cofibrations by sequential colimits implies that the map of colimits
\[\si_T\colon (\bD X)_{n-|T|}= \colim_p F_p (\bD X)_{n-|T|} \rtarr \colim_p F_p (\bD X)_n= (\bD X)_n\]
is  a $(G\times \SI_T)$-cofibration. To see this, we must check that the intersection condition  $$\si_T(F_p(\bD X)_{n-|T|}) \cap F_{p-1}(\bD X)_n = \si_T(F_{p-1}(\bD X)_{n-|T|})$$ 
of  \autoref{sequentialcolim} is satisfied. One inclusion is obvious. For the other, take an element $(\epz,[c,x])\in F_p(\bD X)_{n-|T|} \setminus F_{p-1}(\bD X)_{n-|T|}$; in particular, $x\in X_p\setminus L_pX$. Then $\si_T(\epz,[c,x])=(\si_T\epz,[\si_T^*c,x])\in F_p (\bD X)_n \setminus F_{p-1} (\bD X)_n$, as required.  

This completes the proof that the map \autoref{indT} is a $(G\times \SI_T)$-cofibration.

\subsection{The proof that $\bD$ preserves \gen-equivalences}\label{HARD}

We assume given an \gen-equivalence $f\colon X\rtarr Y$ between $\PI$-$G$-spaces.
\autoref{homotopMT}(i) asserts that the map $\bD f\colon \bD X\rtarr \bD Y$ is an \gen-equivalence.  We shall prove it by showing by induction on $p$
that  $f$ induces an $\bF_n$-equivalence  $F_{p} (\bD X)_n\rtarr F_{p} (\bD  Y)_n$ for each $n$ and each $p\geq 0$, 
there being nothing to prove when $p=0$.  By the usual gluing lemma on pushouts, proven equivariantly in
\cite[Theorem A.4.4]{BV} (but also a model theoretic formality), it suffices to prove that the maps induced by $f$ on the source 
and target of the left vertical arrow in \autoref{pushout}  induce equivalences on $\LA$-fixed point spaces, where $\LA\subset G\times \SI_n$ 
and $\LA\cap \SI_n= \{e\}$.   We have $\LA =\{ (h,\al(h)) \mid h\in H\}$ for some subgroup  $H$ of $G$ and homomorphism 
$\al\colon H\rtarr \SI_n$, and we regard $\bf n$ as a based $H$-set via  $\al$.  Fixing $\LA$ for the rest of the section, 
we shall prove \autoref{homotopMT}(i) by analyzing $\LA$-fixed points.  

We first consider the target,  that is the lower left corner of the diagram.
To clarify the argument, we separate out some of its combinatorics before proceeding.

\begin{defn} Let $\epz\colon \mathbf{p}\rtarr \mathbf{n}$ be an $OE$-function, let $\ta\colon \SI_n\rtarr \SI_p$ be the homomorphism
determined by $\epz$ as defined in \autoref{tau}, and define $\be\colon H\rtarr \SI_p$ to be the composite homomorphism $\ta\al$.  Say that 
$\epz$ is $\LA$-fixed if $\al(h)\epz = \epz \be(h)$ for all $h\in H$. Note that this implies that $\epz_j=\epz_k$ if $j$ and $k$ are in the same 
$H$-orbit.

Define $\sE(\mathbf p, \mathbf n)^{\LA}$ to be the set of all $\LA$-fixed $OE$-functions $\mathbf{p}\rtarr \mathbf{n}$. Fix 
$\epz \in \sE(\mathbf p, \mathbf n)^{\LA}$. Say that a function 
$$\ga= (\ga_1,\dots,\ga_n)\colon H \rtarr \SI(\epz)$$ 
is {\em{admissible}}, or admissible with respect to $\al$, if
 \begin{equation}\label{sigma2}
 \ga_j(hk)=\ga_j(h)\ga_{\al(h)^{-1}(j)}(k)
 \end{equation}
 for $h,k\in H$ and $1\leq j\leq n$.   For any function $\ga \colon H \rtarr \SI(\epz)$, 
 define a function\footnote{We use the notation $\cdot$ since we often use juxtaposition to mean composition in this section.} 
 $\ga\cdot \be\colon H\rtarr \SI_p$ by
 $ (\ga\cdot \be)(h) = \ga(h)\be(h).$
 \end{defn}
 
We leave the combinatorial proof of the following lemma to the reader. When the action of $H$ on ${\bf n}\setminus \{0\}$ 
has a single orbit, there is a conceptual rather than combinatorial proof using wreath products.
 
  \begin{lem}  Fix $\epz\in \sE(\mathbf{p},\mathbf{n})^{\LA}$.  A function $\ga\colon H\rtarr \SI(\epz)$ is admissible with respect to $\al$ if and only 
  if $\ga\cdot \be$  is a homomorphism $H\rtarr \SI_p$. 
 \end{lem}
 
The following result is the central step of the proof of \autoref{homotopMT}(i). It identifies the $\LA$-fixed points of the bottom left corner of the pushout diagram \autoref{pushout}.

\begin{prop}\label{heart} Let $\LA=\{(h, \al(h))\}$ and assume that the action of $H$ on ${\bf n}\setminus \{0\}$ defined by $\al$ is transitive. Then there is a 
natural homeomorphism 
\[
 \xymatrix{ 
  \Big(\coprod_{\epz\in\sE(\bf{p},\bf{n})} \big( \big{(}\prod_{1\leq j\leq n} \sC_G(\epz_j)\big{)} 
\times_{\SI(\epz)} X_{p}  \big) \Big)  ^\LA 
\ar[d]^-{\om} \\ 
\coprod_{\epz \in \sE(\mb{p},\mb{n})^\LA} \Big(\coprod_{\ga\colon H\rtarr \SI(\epz)} \sC_G(\epz_1)^{\LA_{\ga}^1}\times X_p^{\LA_{\ga}} \Big) /\SI(\epz).}
\]
In the target, the second wedge runs over all admissible functions 
$$\ga=(\ga_1,\dots,\ga_n)\colon H\rtarr \SI(\epz);$$
the groups $\LA_{\ga}$ and $\LA_{\ga}^1$ are specified by
 $$\LA_{\ga}=\{(h, (\ga\cdot \be)(h)) | h\in H\} \subset  G\times \SI_p $$
 and 
$$\LA_{\ga}^1 =\{(k, \ga_1(k)) | k\in K\}\subset G\times \SI_{\epz_1},$$
 where $K\subset H$ is the isotropy group of $1$ under the action of $H$ on $\bf n$ given by $\al$. 
 \end{prop} 
 \begin{proof} 
 Since $\al(k)(1) = 1$ for $k\in K$, $\ga_1$ is a homomorphism $K\rtarr \SI_{\epz_1}$ by specialization of \autoref{sigma2}.  In the target, we pass to orbits 
 from the $\SI(\epz)$-action defined on the term in parentheses by
   \[ \rh(\ga;c;x) = (\rh\ast \ga;c\rh_1^{-1}; \rh_{*}x), \] 
 Here $\rh = (\rh_1,\dots, \rh_n)$ is in $\SI(\epz)$,  $\ga$ is admissible, 
 $c\in \sC_G(\epz_1)^{\LA_\ga^1}$, and $x\in X_p^{\LA_{\ga}}$.
The $j$th coordinate of $\rh\ast \ga$ is defined by 
 \begin{equation}\label{star}
  (\rh \ast \ga)_j(h) = \rh_j\ga_j(h)\rh^{-1}_{\al(h)^{-1}(j)}.
  \end{equation}
 A quick check of definitions shows that
\[ (\rho \ast \ga) \cdot \be = \rho (\ga \cdot \be) \rho ^{-1},\]
which also implies that $\rh\ast\ga$ is admissible since $\ga$ is admissible. 
 Similarly, $c\rh^{-1}_1$ is fixed by $\LA_{\rh\ast \ga}^1$
since $c$ is fixed by $\LA_{\ga}^1$ and $\rh_{*}(x)$ is fixed by $\LA_{\rh\ast\ga}$ since $x$ is fixed by $\LA_{\ga}$.   Thus the action makes sense.
Moreover, as we shall need later, this action is free.  If  $\rh(\ga;c;x) =  (\ga;c;x)$, then $c\rh_1^{-1} = c$ and thus $\rh_1 = 1$ since $\SI_{\epz_1}$ 
acts freely on $\sC_G(\epz_1)$.   Also, $\rh\ast \ga =\ga$ and thus $\rh_j^{-1}\ga_j(h)\rh_{\al(h)^{-1}(j)} = \ga_j(h)$ for all $h$. Taking $j=1$, this 
implies that $\rh_{\al(h)^{-1}(1)} = 1$ for all $h$. Since we are assuming the action of $H$ induced by $\al$ on $\mb{n}\setminus \{0\}$ is transitive, this implies that $\rh = 1\in \SI(\epz)$. 

We turn to the promised homeomorphism. By \autoref{sigmaact}, for a point $z$ represented by 
$(\epz;c_1,\dots,c_n;x)$, $c_j\in \sC_G(\epz_j)$ and $x\in X_p$,
$(h,\al(h))z$ is represented by
$$ \big(\al(h)\epz \be(h)^{-1}; hc_{\al(h)^{-1}(1)}, \dots, hc_{\al(h)^{-1}(n)};\be(h)_\ast(hx)\big)$$
where, as before, $\be=\ta\al$.  Assume that $z$ is fixed by $\LA$. Then we must have $\al(h)\epz \be(h)^{-1}  = \epz$, 
so that $\al(h)\epz = \epz \be(h)$ and thus $\epz\in \sE(\mathbf{p},\mathbf{n})^{\LA}$.  We must also have 
 $$(c_1, \dots, c_n; x)\sim \big(hc_{\al(h)^{-1}(1)}, \dots hc_{\al(h)^{-1}(n)}; \be(h)_{*}(hx)\big ),$$ 
so that for each $h\in H$ there exists $\ga(h)=\ga_1(h)\times \cdots \times \ga_n(h)\in \SI(\epz)$ such that
\begin{equation}\label{fixed} c_j\ga_j (h)= hc_{\al(h)^{-1}(j)} \ \text{   and   } \ x=\ga(h)_\ast \be(h)_\ast (hx) .
\end{equation}
Note that for any given $n$-tuple $(c_1, \dots, c_n)$, $\ga(h)$ is unique since the action of $\SI(\epz)$ 
on $\prod_{1\leq j\leq n} \sC_G(\epz_j)$ is free.  For $h,k\in H$ and $1\leq j\leq n$,
 \begin{align*} c_j\ga_j(hk)&=(hk)c_{\al(hk)^{-1}(j)}\\
 &=h(kc_{\al(k)^{-1}(\al(h)^{-1}(j))})\\
 &=hc_{\al(h)^{-1}(j)}\ga_{\al(h)^{-1}(j)}(k)\\
 &=c_j\ga_j(h)\ga_{\al(h)^{-1}(j)}(k).
 \end{align*}
 Since the action of $\SI_{\epz_j}$ on $\sC_G(\epz_j)$ is free, this implies that \autoref{sigma2} holds, so that $\ga$ is admissible.  Note that we have not yet used that the action given by $\al$ is transitive.
   
 Now the map $\om$  is defined by
 $$\om(\epz;c_1,\dots, c_n;x) = (\epz;\ga; c_1; x).$$ 
We see from \autoref{fixed} that $x$ is in $X_p^{\LA_{\ga}}$ and that $c_1$ is in $\sC_G(\epz_1)^{\LA_{\ga}^1}$, the latter using the fact that $K$ is the isotropy group of $1\in \mb{n}\setminus \{0\}$.
 We must check that $\om$  is well-defined.  Thus suppose that
 $$(\epz; c_1,\dots, c_n; x)\sim (\epz; d_1,\dots, d_n; y).$$ 
 
  Then there exists $\rh\in \SI(\epz)$ such that $c_j=d_j\rh_j$ and $y=\rh_\ast x$. Using \autoref{star} and  \autoref{fixed},     we see that 
  \begin{eqnarray*}
 hd_{\al(h)^{-1}(j)} &=& h c_{\al(h)^{-1}(j)}\rh^{-1}_{\al(h)^{-1}(j)}\\
 &=&c_j\ga_j(h)\rh^{-1}_{\al(h)^{-1}(j)}\\
 &=& d_j\rh_j\ga_j(h)\rh^{-1}_{\al(h)^{-1}(j)}\\
 &=& d_j(\rh\ast\ga(h))_j.
 \end{eqnarray*}
 Thus, comparing with \autoref{fixed} for $(d_1,\dots,d_n; y)$, and using the freeness of the action, we see that 
 $$ \om (\epz;d_1,\dots, d_n; y) = (\epz;\rh\ast \ga, d_1, y) = (\epz;\rh\ast \ga, c_1 \rh_1^{-1},\rh_{*} x) = \rh(\epz;\ga;c_1;x),$$ 
 so that the targets of our equivalent elements are equivalent. 
  
 Clearly $\om$ is continuous since it is obtained by passage to orbits from a (disconnected) cover by restriction to subspaces 
 of the projection that forgets the  coordinates $(c_2,\dots,c_n)$.
 
 To define $\om^{-1}$, first choose coset representatives for $H/K$ where $K$ is the isotropy group of $1$, that is, choose $h_j\in H$ such that $\al(h_j)(1)=j$ for $1\leq j\leq n$,
 taking $h_1 = e$.  Then define $\om^{-1}$ by
 $$\om^{-1} (\epz;\ga; c; x) = (\epz; c_1,\dots, c_n; x)$$ 
 where  $c_j=h_jc\ga_j(h_j)^{-1}$.  Note that $c_1 = c$ and that $\om^{-1}$ does not depend on the choice of coset representatives.  Here $\epz$ is $\LA$-fixed,  $\ga\colon H \rtarr \SI(\epz)$ is admissible, $c\in \sC_G(\epz_1)^{\LA_{\ga}^1}$ and $x\in X_p^{\LA_{\ga}}$. 
 We must show that $\om^{-1} (\epz;\ga; c; x)$ is fixed by $\LA$. First, note that $(h,\al(h))$ sends $\epz$ to $\al(h)\epz \be(h)^{-1}=\epz$, since $\epz \in \sE(\mb{p},\mb{n})^{\LA}$. Omitting $\epz$ from the notation for readability, 
  \begin{eqnarray*}
  (h,\al(h))(c_1,\dots, c_n,x)&=& (hc_{\al(h)^{-1}(1)},\ \dots,\  hc_{\al(h)^{-1}(n)}, \be(h)_\ast(hx)) \\
 &=& ( hc_{\al(h)^{-1}(1)},\  \dots,\  hc_{\al(h)^{-1}(n)}, \ga(h)^{-1}_\ast(x))\\
 &\sim & ( hc_{\al(h)^{-1}(1)}\ga_1(h)^{-1},\ \dots,\ hc_{\al(h)^{-1}(n)}\ga_n(h)^{-1},x).
 \end{eqnarray*}
 We claim that $hc_{\al(h)^{-1}(j)}\ga_j(h)^{-1}=c_j.$ The definition of $c_{\al(h)^{-1}(j)}$ gives us the identification
 \begin{eqnarray*}
 hc_{\al(h)^{-1}(j)}\ga_j(h)^{-1} &=& h\big( h_{\al(h)^{-1}(j)}c\ga_{\al(h)^{-1}(j)}(h_{\al(h)^{-1}(j)})^{-1}\big)\ga_j(h)^{-1}.
 \end{eqnarray*}
 Now note that $$\al(hh_{\al(h)^{-1}(j)})(1)=\al(h)(\al(h)^{-1}(j))=j,$$ 
 hence $hh_{\al(h)^{-1}(j)}$ is in the coset represented by $h_j$. Thus there exists a $k\in K$ such that $hh_{\al(h)^{-1}(j)}=h_jk$. Since $\ga$ satisfies equation \autoref{sigma2}, 
 \begin{eqnarray*}
 \ga_j(h)\ga_{\al(h)^{-1}(j)}(h_{\al(h)^{-1}(j)}) &=& \ga_j(hh_{\al(h)^{-1}}(j))\\
 &=&\ga_j(h_jk)\\
 &=&\ga_j(h_j)\ga_{\al(h_j)^{-1}(j)}(k)\\
 &=& \ga_j(h_j)\ga_1(k).
 \end{eqnarray*}
Thus $$hc_{\al(h^{-1})(j)}\ga_j(h)^{-1}=h_jkc\ga_1(k)^{-1}\ga_j(h_j)^{-1}=h_jc\ga_j(h_j)^{-1}=c_j$$ as claimed. Therefore the map $\om^{-1}$ really does land in the $\LA$-fixed points.

To show that $\om^{-1}$ is well-defined, note that if $(\epz;\ga;c;x)\sim (\epz,\rh\ast \ga;c\rh_1^{-1},\rh_{*}x)$ for some $\rh \in \SI(\epz)$, then  $\om^{-1}$ sends the latter to $(\epz;d_1,\dots,d_n;\rh_{*}x)$, where
\begin{align*}
d_j=&h_jc\rh_1^{-1} (\rh_j\ga_j(h_j)\rh_{\al(h_j)^{-1}(j)}^{-1})^{-1}\\
=&h_jc\rh_1^{-1} (\rh_j\ga_j(h_j)\rh_1^{-1})^{-1}\\
=& h_j c \rh_1^{-1}\rh_1\ga_j(h_j)^{-1}\rh_j^{-1}\\
=& c_j\rh_j^{-1}.
\end{align*}
Thus 
$$ (d_1,\dots,d_n,\rh_\ast x) = \rh \cdot(c_1,\dots, c_n, x),$$ 
so that $\om^{-1}$ is well-defined. This map is clearly continuous.

It is easy to see that the map forward and the map backward composed in either order are the identity, hence we have the claimed homeomorphism.
\end{proof}

The restriction to transitive action by $\al$ in the previous result serves only to simplify the combinatorics.  The following remark indicates the changes that are needed to deal with the general case.

\begin{rem} When the action of $H$ on ${\bf n}\setminus \{0\}$ is not transitive, we argue analogously to \autoref{lem:lafixedpoints} and \autoref{lemmaspecial}
to obtain an analogous homeomorphism.  We break the $H$-set ${\bf n}\setminus \{0\}$ given by $\al$ into a disjoint union of orbits $H/K_a$  of size
$n_a = |H/K_a|$,  where $\sum_a n_a= n$ and $K_a$ is the isotropy group of the initial element, denoted $1_a$,  in its orbit in  ${\bf n}\setminus \{0\}$.  That 
breaks $\bf n$ into the wedge of subsets ${\bf  n}_a$ and breaks $\SI(\epz)$ into a product of subgroups 
$\SI(\epz(a)) = \prod_{j\in H/K_a} \SI(\epz_j)$.   
Paying attention to the ordering, the product of the 
$\sC_G(\epz_j)$ in the source of $\om$ breaks into the product over $a$ of those $\sC_G(\epz_j)$ such that $j$ is in 
the $a$th orbit of $\mathbf{n}\setminus \{0\}$. To generalize the target 
of $\om$ accordingly, define subgroups 
$$\LA_{\ga}^{1_a} =\{(k, \ga_{1_a}(k)) | k\in K_a\}\subset G\times \SI_{\epz_{1_a}}$$ 
and replace $\sC_G(\epz_1)^{\LA_{\ga}^{1}}$ by the product over $a$ of the $\sC_G(\epz_{1_a})^{\LA_{\ga}^{1_a}}$.  With these changes of source and target and just a bit of extra bookkeeping, it is straightforward to state and prove the general analogue of \autoref{heart}. 
\end{rem}

In the single orbit case, observe that if $f\colon X_p\rtarr Y_p$ is a $\LA_{\ga}$-equivalence, it induces an equivalence on the target of
$\om$ before passage to quotients under the action of $\SI(\epz)$.  Since the $\SI(\epz)$ action is free, the equivalence passes to the quotients. Generalizing to the multi-orbit case, this concludes our proof that we have a $\LA$-equivalence in the lower left corner of the pushout diagram \autoref{pushout}. 

We next consider the upper left corner of \autoref{pushout}.  Precisely the same argument as that  just given, but with $X_p$ replaced by $L_pX$, identifies the 
$\LA$-fixed subspace of the upper left corner in terms of appropriate fixed point subspaces of  $L_pX$.  Therefore the same 
argument as that just given shows that the following result implies that $f$ induces an equivalence on the upper left corner, as needed to
complete the proof of \autoref{homotopMT}(i).

\begin{prop}\label{genlatching} If $f \colon X \rtarr Y$ is an \gen-equivalence of $\PI$-$G$-spaces, 
then \linebreak
$f\colon L_nX \rtarr  L_nY$ is an $\bF_n$-equivalence for each $n$.
\end{prop} 

To prove this, we need more combinatorics to identify $(L_nX)^{\LA}$, where $\LA \in \bF_n$.   We again take
$\LA = \{(h,\al(h)) \mid   \al\colon H\rtarr \SI_n \}\subset G\times \SI_n$ and again view $\mb{n}\setminus \{0\}$ as an $H$-set via $\al$.  For a 
subset $U$ of $\mb{n}\setminus \{0\}$ with $u$ elements, recall from \autoref{sigmaT} that $\si_U\colon \mb{n-u} \rtarr \mb{n}$ 
denotes the ordered injection that skips the elements in $U$.  It 
is a composite of degeneracies $\si_U=\si_{i_k} \cdots \si_{i_1}$ where $i_1 < \cdots < i_k$ are the elements of $U$.  Given $1\leq i\leq n$, let $\pi_i \colon \mb{n} \rtarr \mb{n-1}$ be the ordered projection that sends $i$ to $0$. Similarly, define $\pi_U \colon \mb{n} \rtarr \mb{n-u}$ to be the ordered projection that sends the elements of $U$ to $0$; explicitly, $\pi_U=\pi_{i_1} \cdots \pi_{i_k}$.  Note that $\pi_U \si_U$ is the identity, and 
\[\si_U\pi_U(i)=\begin{cases} i & \text{if }i\not\in U\\
0 & \text{if } i\in U.
\end{cases}\]

\begin{rem} The maps $\si_i$ correspond to the degeneracies in $\DE^{op}$ via the inclusion $F\colon \DE^{op} \rtarr \sF$, except there is a shift since we are indexing on the non-zero elements of $\mb{n}$. The maps $\pi_i$ are mostly invisible to $\DE$. The collection of maps $\{\si,\pi\}$ satisfies the following subset of the simplicial relations, as can be easily checked.
\begin{align*}
 \pi_i \pi_j & = \pi_{j-1}\pi_i  \quad\text{if } i<j\\
 \si_j\si_i&=\si_i\si_{j-1}  \quad\text{if } i< j\\
 \pi_i\si_j&=\si_{j-1}\pi_i \quad\text{if } i< j\\
 \pi_i \si_i &= \id.
\end{align*}
\end{rem}

Now assume that $U\subset \mb{n}\setminus \{0\}$ is a $H$-subset of $\mb{n}\setminus\{0\}$ and note that its complement is also
a $H$-subset of $\mb{n}\setminus \{0\}$. For $h\in H$, define
$$\al_U(h) = \pi_U\al(h) \si_U\colon  \mb{n-u}\rtarr  \mb{n-u}.$$
This is essentially the restriction of $\al(h)$ to $\mb{n}\setminus U$, but using the ordered inclusion $\si_U$ to identify 
that set with $\mb{n - u}$.  Note that $\al_U(e) = \id$ and that
\begin{equation}\label{alphaj}
\al(h)\si_U= \si_U\pi_U\al(h) \si_U = \si_U\al_U(h)
\end{equation}
since $\si_U\pi_U$ is the identity on $\mb{n}\setminus U$ and $0$ on $U$ and since $\al(h)\si_U$ is $0$ on $U$ and takes
$\mb{n}\setminus U$ to itself.  This implies that $\al_U$ is a homomorphism $H \rtarr \SI_{n-u}$ since
\[\al_U(h)\al_U(k)=\pi_U\al(h)\si_U\pi_U\al(k)\si_U=\pi_U\al(h)\al(k)\si_U=\pi_U\al(hk)\si_U=\al_U(hk),\]
where the second equality uses \autoref{alphaj} with $h$ replaced by $k$.  Thus we can define
$$\LA_U=\{(h,\al_U(h)) \mid h \in H\}.$$ 
We have the following identification of $(L_nX)^{\LA}$.  Henceforward
we abbreviate notation for the action of $\sF$ on $X$, writing $\si_U$ for ${\si_U}_{*}$ and so forth.

\begin{lem}\label{lambdalatching} 
\[ (L_nX)^\LA = \bigcup \si_U\bigg((X_{n-u})^{\LA_U}\bigg)\]
where the union runs over the $H$-orbits $U\subset \mb{n}\setminus \{0\}$. 
\end{lem}

\begin{proof}  For $U \subset \mb{n}\setminus \{0\} $ and $z\in (X_{n-u})^{\LA_U}$,  $\si_Uz$ is a $\LA$-fixed point since
\[(h,\al(h))\cdot (\si_Uz) = \al(h)(h\si_Uz)=\al(h)\si_U(hz)=\si_U\al_U(h)(hz)=\si_Uz,\]
for $h\in H$; the next to last equality holds by \autoref{alphaj} and the last holds since $z$ is a $\LA_U$-fixed point. 
This gives one inclusion. 

For the other inclusion, we first note that the action of $\SI_n$ on $L_nX$ can be expressed as follows.   
Let $\rh\in \SI_n$ and  $x\in L_nX$, so that $x=\si_iy$ for some $i$, $1\leq i\leq n$, and  some $y\in X_{n-1}$.  
Then
\[\rh x = \rh \si_i y = \si_{\rh(i)} \tilde{\rh} y\]
where $\tilde{\rh}=\pi_{\rh(i)}\rh \si_i$ is a permutation in $\SI_{n-1}$, as is easily checked. Now suppose that 
$x$ is a $\LA$-fixed point and let $U$ be the $H$-orbit of $i$ in $\mb{n}\setminus \{0\}$. Then $x\in \si_j(X_{n-1})$ for all $j\in U$
since
\[x=(h,\al(h))\cdot x = \al(h)(hx)=\al(h)(h\si_iy)=\al(h)\si_i(hy)=\si_{\al(h)(i)}\widetilde{\al(h)}(hy)\]
for $h\in H$. It follows that $x = \si_Uz$, where $z=\pi_U x\in X_{n-u}$.  We claim that $z$ is a $\LA_U$-fixed
point.  Indeed 
\[\si_Uz=x=\al(h)(hx)=\al(h)\si_U(hz)=\si_U\al_U(h)(hz),\]
and the claim follows since $\si_U$ is injective.  This gives the other inclusion.
\end{proof}

Next we consider the intersection of the subspaces corresponding to two such subsets $U$.
 
\begin{lem}\label{intersection}
Let $U$ and $V$ be disjoint $H$-subsets of the action of $H$ on $\mb{n}\setminus \{0\}$ given by $\al$. Let $U$ have $u$ elements and
$V$ have $v$ elements.  Then
\begin{align*}
\si_U \bigg((X_{n-u})^{\LA_U}\bigg)  \cap \si_V \bigg((X_{n-v})^{\LA_V}\bigg) 
& =\si_{U \cup V} \bigg((X_{n-u-v})^{\LA_{U\cup V}}\bigg) \\
& =\si_U\si_{\tilde{V}} \bigg( (X_{n-u-v})^{(\LA_U)_{\widetilde{V}}}\bigg),
\end{align*}
where $\widetilde{V}$, also with $v$ elements, is the subset of $\mb{n}\setminus (U\cup \{0\})$ that $\si_U$ maps 
onto the subset $V$ of $\mb{n}\setminus\{0\}$.
\end{lem}
\begin{proof}
The notation obscures the fact that $\tilde{V}$ depends on $U$, but we always use it directly following the $U$ to which it pertains. 
Note that the last equality follows from the facts that 
$\si_{U\cup V}=\si_U\si_{\tilde{V}}$ and $\LA_{U\cup V}=(\LA_U)_{\tilde{V}}$.
Suppose that $x$ is in the intersection. Then $x=\si_iy_i$ for all $i\in U$ and also for all $i\in V$, hence 
 $x=\si_{U\cup V} z$, where $z=\pi_{U\cup V} x \in X_{n-|U\cup V|}$.  By the same argument as in the
 proof of the previous lemma, $z$ is a $\LA_{U\cup V}$-fixed point.

For the opposite inclusion, let $x=\si_{U\cup V}z$, where $z\in (X_{n-|U\cup V|})^{\LA_{U\cup V}}$.  Then $x = \si_U y_U$
where $y_U=\si_{\tilde{V}}z$ is a $\LA_U$-fixed point by the same argument as in the previous proof.  Indeed, its first part
works for all $H$-subsets, not just orbits, to give that $x$ is a $\LA$-fixed point, and then its second part gives that
$y_U$ is a $\LA_U$-fixed point.  The symmetric argument gives that $x = \si_Vy_V$ where $y_V$ is a $\LA_V$-fixed point.
\end{proof}

Finally, we use these lemmas to prove \autoref{genlatching}.
 
\begin{proof}[Proof of \autoref{genlatching}] 
We first observe that an argument similar to the proof of \autoref{lambdalatching} shows that for all orbits (in fact, all $H$-subsets) $U$ of $\mb{n}$,
\[\si_U\big((X_{n-u})^{\LA_U}\big)=\big(\si_U(X_{n-u})\big)^\LA=\si_U(X_{n-u})\cap X_n^\LA.\]
Since $\si_U$ is a closed inclusion, this shows that $\si_U\big((X_{n-u})^{\LA_U}\big)$ is closed in $X_n^\LA$. 
Let $U_1, \dots, U_m$ be the orbits of the $H$-set $\mb{n}\setminus \{0\}$, with corresponding cardinalities $u_1,\dots,u_m$. For $1\leq k\leq m$, we have
\begin{eqnarray*}
 \si_{U_k}(X_{n-u_k}^{\LA_{U_k}}) \cap \bigg( \bigcup_{i=1}^{k-1} \si_{U_i}(X_{n-u_i}^{\LA_{U_i}}) \bigg)
 &=&  \bigcup_{i=1}^{k-1}  \si_{U_k}(X_{n-u_k}^{\LA_{U_k}}) \cap \si_{U_i}(X_{n-u_i}^{\LA_{U_i}})\\
&=& \bigcup_{i=1}^{k-1}  \si_{U_k}\si_{\tilde{U_i}} (X_{n-u_k-u_i}^{(\LA_{U_k})_{\tilde{U_i}}})\\
&=&  \si_{U_k} \big( \bigcup_{i=1}^{k-1} \si_{\tilde{U_i}} (X_{n-u_k-u_i}^{(\LA_{U_k})_{\tilde{U_i}}}) \big)
\end{eqnarray*}
where the next to last equality holds by \autoref{intersection} and the others are standard set manipulations.  We therefore have inclusions
which give the following pushout diagram.
\begin{equation}\label{pushouttoo}
\xymatrix{
\si_{U_k}\big( \bigcup_{i=1}^{k-1} \si_{\tilde{U_i}}(X_{n-u_k-u_i}^{(\LA_{U_k})_{\tilde{U_i}}})\big) \ar[d] \ar[r] & \bigcup_{i=1}^{k-1} \si_{U_i}(X_{n-u_i}^{\LA_{U_i}}) \ar[d]\\
 \si_{U_k}(X_{n-u_k}^{\LA_{U_k}}) \ar[r]& \bigcup_{i=1}^{k} \si_{U_i}(X_{n-u_i}^{\LA_{U_i}}) }
 \end{equation}
This diagram is clearly a pushout of sets. It is a pushout of spaces since the lower horizontal and the right vertical
arrows are closed inclusions by our first observation, so that their target has the topology of the union.  By \autoref{lambdalatching}, 
the lower right corner is $(L_nX)^{\LA}$ when $k=m$. 

We claim that the left vertical arrow and therefore the right vertical arrow is a cofibration for each $k\leq m$,
the assertion being vacuous if $k=1$.  We prove this by induction on $n$.  Thus suppose it holds for all values 
less than $n$. In particular, assume that it  holds for each $n-u_k$. Note that the orbits of 
$\mb{n}\setminus (U_k\cup \{0\})$ are $\tilde{U}_1,\dots ,\tilde{U}_{k-1}, \tilde{U}_{k+1}, \dots, \tilde{U}_m$. 
Since $\si_{U_k}$  (or any restriction of it to a subspace) is a homeomorphism onto its image, it suffices to prove that
the left vertical map is a cofibration before application of $\si_{U_k}$.  
With $n$ replaced by $n-u_k$ and with each 
$\tilde{U_i}$ referring to $U_k$, the induction hypothesis applied to right vertical arrows gives that 
$$\bigcup_{i=1}^{k-1}  \si_{\tilde{U_i}}(X_{n-u_k-u_i}^{(\LA_{U_k})_{\tilde{U_i}}}) 
\rtarr \bigcup_{i=1,\dots,k-1,k+1}  \si_{\tilde{U_i}}(X_{n-u_k-u_i}^{(\LA_{U_k})_{\tilde{U_i}}}) $$
and each map
$$\bigcup_{i=1,\dots,k-1,k+1,\dots,k+j-1}  \si_{\tilde{U_i}}(X_{n-u_k-u_i}^{(\LA_{U_k})_{\tilde{U_i}}})  
\rtarr \bigcup_{i=1,\dots,k-1,k+1,\dots,k+j}  \si_{\tilde{U_i}}(X_{n-u_k-u_i}^{(\LA_{U_k})_{\tilde{U_i}}}), $$
$2\leq j\leq m$, is a cofibration.  When $j=m$, the target of the last map is $(L_{n-u_k}X)^{\LA_{U_k}}$, and
$$(L_{n-u_k}X)^{\LA_{U_k}}\rtarr (X_{n-u_k})^{\LA_{U_k}}$$ 
is a cofibration by \autoref{PIReedytrue}.   Application of $\si_{U_k}$ to the composite of these cofibrations 
gives the left vertical arrow, completing the proof of our claim.

This allows us to prove by induction on $k$ that we have a weak equivalence
$$ \bigcup_{i=1}^{k} \si_{U_i}(X_{n-u_i}^{\LA_{U_i}}) \rtarr  \bigcup_{i=1}^{k} \si_{U_i}(Y_{n-u_i} ^{\LA_{U_i}}) $$ 
for any $k\leq m$.  Since the $\si_{U_i}$ are homeomorphisms onto their images, the base case $k=1$ holds by our 
assumption on $f$. The inductive step is an application of the gluing lemma to the pushout diagram \autoref{pushouttoo}. 
For $k=m$, this gives that 
$$(L_nX)^\LA \rtarr (L_nY)^\LA$$ is a weak equivalence, as required. 
\end{proof}

\section{Proofs of technical results about the Segal machine}\label{SEGALPf}

We prove Propositions \ref{consist} and \ref{cute} in \autoref{SecSegDet},
focusing on the combinatorial analysis of the simplicial, conceptual, and homotopical versions of the Segal machine.
We prove the group completion property, \autoref{Ggpcomp}, in \autoref{groupcomp}.  We prove the positive linearity property, 
\autoref{Omnibus}, in \autoref{cof}.  The main step is to verify that the relevant $\sW_G$-$G$-spaces satisfy the wedge axiom
formulated in \autoref{wedax1}, and we prove that in \autoref{WEDGE}.  We  prove \autoref{keystone} in \autoref{heredity}.

\subsection{Combinatorial analysis of $A^{\bullet}\otimes_{\sF} X$}\label{SecSegDet}

Let $X$ be an $\sF$-$G$-space. We first compare the functor $A^{\bullet}\otimes_{\sF} X$ with geometric realization.  
Recall that the objects of $\DE$ are the ordered finite sets $[n] =\{0,1,\dots,n\}$ and its 
morphisms are the nondecreasing functions.  As in \autoref{SEGALsimp}, let $F$ denote 
the simplicial circle $S^1_s = \Delta[1]/\partial \Delta[1]$ viewed as a functor $\DE^{op} \rtarr \sF$.
Take the topological circle to be $S^1 = I/\pa I$.

\begin{rem}\label{delf}  The functor $F$ sends the ordered set $[n]$ to the based set $\mathbf{n}$. For a map $\ph\colon [n]\rtarr [m]$ 
in $\DE$,  $F\ph\colon \mathbf{m}\rtarr \mathbf{n}$ sends $i$ to $j$ if $\ph(j-1) < i \leq \ph(j)$
and sends $i$ to $0$ if there is no such $j$.  Thus
\[ (F\ph)^{-1}(j) = \{ i \mid \ph(j-1) < i \leq \ph(j)\} \ \ \text{for} \ \ 1\leq j\leq n. \]
If $\de_i\colon [n-1] \rtarr [n]$ and $\si_i\colon [n+1] \rtarr [n]$,
$0\leq i\leq n$, are the standard face and degeneracy maps that skip or repeat $i$ in the target,
then $F\de_i = d_i\colon \mathbf{n}\rtarr \mathbf{n-1}$ is the ordered surjection that repeats $i$ for $i < n$ but
sends $n$ to $0$ if $i = n$,  and 
$F\si_i = s_i\colon \mathbf{n}\rtarr \mathbf{n+1}$ is the ordered injection\footnote{In \autoref{operadicmachine} and \autoref{SegOpPf}  it was 
denoted $\si_{i+1}$ as a map in the category $\inj$ of finite sets and injections.}  that skips $i+1$.  Note in 
particular that $F\de_1 = \varphi_2\colon \mathbf{2} \rtarr \mathbf{1}$, which sends $1$ and $2$ to $1$. 
In $\sF$, we also have ordered projections $\pi_i\colon \mathbf{n}\rtarr \mathbf{n-1}$, used in \autoref{HARD}, that are mostly invisible to $\DE$. 
The map $\pi_i$ sends $i$ to $0$ and it sends $j$ to $j$ if $j<i$ and to $j-1$ if $j>i$. 
\end{rem}

To prove \autoref{cute}, we must compare 
$$ |X| = X \otimes_{\DE} \DE \ \ \text{with} \ \  X(A) : = \bP(X)(A) = A^{\bullet} \otimes_{\sF} X$$
when $A=S^1$. To aid in the comparison, we rewrite $|X|$ as $\DE \otimes_{\DE^{op}} X$. Here $\DE$ on 
the left is the covariant functor $\DE\rtarr \sU$  that sends $[n]$ to the topological simplex 
$$\DE_n = \{(t_1,\dots, t_n) \mid 0 \leq t_1 \leq \cdots \leq t_n \leq 1\}.$$
Nowadays it is more usual to use tuples $(s_0,s_1, \dots, s_n)$ such that $0\leq s_i\leq 1$
and $\sum_i s_i = 1$, but the formulae $s_i = t_{i+1}-t_i$ and $t_i = s_0 + \cdots + s_{i-1}$ 
translate between the two descriptions. For $0\leq i\leq n$, the face map 
$\de_i\colon \DE_{n-1}\rtarr \DE_{n}$ and the degeneracy map $\si_i\colon \DE_{n+1}\rtarr \DE_n$ 
are given by
\[\de_i(t_1,\dots, t_{n-1}) = \left\{ \begin{array}{lll}
(0,t_1,\dots, t_{n-1}) & \mbox{if $i=0$} \\
(t_1, \dots t_{i-1}, t_i, t_i, t_{i+1},\dots t_{n-1}) & \mbox{if $0 < i < n$} \\
(t_1,\dots, t_{n-1},1) & \mbox{if $i = n$}
\end{array} \right. \]
\[ \si_i (t_1,\dots, t_{n+1}) = (t_1, \dots, t_{i}, t_{i+2}, \dots, t_{n+1}). \]
Map  $\DE_n$ to $(S^1)^n$  by sending  $(t_1,\dots,t_n)\in \DE_n$ 
to $(t_1,\dots,t_n)\in (S^1)^n$. Looking at the definition of the functor $F$, we see that this defines
a map $\xi\colon \DE\rtarr (S^1)^{\bullet}$ of cosimplicial spaces,\footnote{Warning: we are thinking of 
both source and target as cosimplicial {\em unbased} spaces.}  where $(S^1)^{\bullet}$ is a cosimplicial 
space by pullback along $F$. Therefore $\xi$ induces a natural map 
\[ \xi_\ast \colon |X| = \DE \otimes_{\DE^{op}} X\rtarr (S^1)^{\bullet}\otimes_{\sF} X = X(S^1). \]

Recall that every point of $|X|$ is represented by a unique point $(u,x)$ such that $u\in \DE_p$ 
is an interior point and $x\in X_p$ is a nondegenerate point \cite[Lemma 14.2]{Mayss}. Said another way,
$|X|$ is filtered with strata 
\[ F_p|X| \setminus F_{p-1}|X| = (\DE_p \setminus \pa \DE_p) \times (X_p \setminus L_pX), \]
where $L_pX$, the $p$th latching space,  is the union of the subspaces $s_i(X_{p-1})$ (see \autoref{latch}).  
The construction of
$F_p|X|$ from $F_{p-1}|X|$ is summarized by the concatenated pushout diagrams 
\begin{equation}\label{grpushout}  \xymatrix{
\pa \DE_p \times L_pX \ar[r] \ar[d] &  \DE_p \times L_p X \ar[d] &\\
\pa \DE_p \times X_p\ar[r] & \DE_p \times L_pX \cup \pa \DE_p \times X_p \ar[d] 
\ar[r] & F_{p-1} |X| \ar[d] \\
& \DE_p\times X_p \ar[r] & F_p|X|\\} \end{equation}

We shall describe $X(A)$ similarly for all $A\in G\sW$, and we shall specialize to $A=S^1$ to 
see that $\xi_{*}$ is a natural homeomorphism, using results about the structure of $\sF$ 
recorded in \autoref{BDXStruc}. 

\begin{rem}\label{matchup2}  The functor $F$ is a map of generalized Reedy categories in the sense of \cite{BergMoer}.
Recall that the latching $G$-space $L_pX\subset X_p$ of an $\sF$-$G$-space $X$ is defined to be the
latching space of its underlying $\PI$-$G$-space, as defined in \autoref{PIReedy}.  The $\sF$-$G$-space $X$ also has a latching space 
when regarded as a simplicial $G$-space via $F$.  Direct comparison of definitions shows that these two
latching spaces are the same.
\end{rem}

By \autoref{OK}, the $G$-space $X(A)$ is the quotient of $\coprod_{n\geq 0} A^n\times X_n$ obtained by identifying
$(\ph^*(a),x)$ with $(a,\ph_{*}(x))$ for all $\ph\colon \mathbf{m}\rtarr \mathbf{n}$ in $\sF$, $a\in A^n$, and 
$x\in X_m$.  Here $\ph^*(a_1,\dots,a_n) = (b_1,\dots,b_m)$ where $b_i = a_{\ph(i)}$, with 
$b_i = \ast$ if $\ph(i) = 0$, and $\ph_{*}(x)$ is given by the covariant functoriality of $X$. 
The image of $\coprod_{n\leq p}A^n\times X_n$ is topologized as a quotient and denoted $F_pX(A)$, 
and $X(A)$ is given the topology of the union of the $F_pX(A)$. 

\begin{notn} For an unbased $G$-space $U$, the {configuration space} $\mb{Conf}(U,p)$ is the $G$-subspace 
of $X^p$ of points $(u_1,\dots,u_p)$ such that $u_i\neq u_j$ for $i\neq j$.  For a based $G$-space
$A$, the {\em based fat diagonal}\, $\de A^p\subset A^p$ is the $G$-subspace of points 
$(a_1,\dots,a_p)$ such that either some $a_i$ is the basepoint or $a_i=a_j$ for some $i\neq j$.  
Observe that
$$ A^p \setminus \de A^p = \mb{Conf}(A\setminus \{\ast\},p).$$
\end{notn}

\begin{lem}\label{strata} $F_0X(A) = \ast\times X_0$.  For $p\geq 1$, the stratum $F_pX(A)\setminus F_{p-1}X(A)$ is 
$$\mb{Conf}(A\setminus \{\ast\},p) \times_{\SI_p} (X_p\setminus L_pX).$$ 
The construction of $F_pX(A)$ from $F_{p-1} X(A)$ is summarized by the concatenated 
pushout diagrams
\[  \xymatrix{
\de A^p \times_{\SI_p} L_pX \ar[r] \ar[d] &  A^p \times_{\SI_p} L_pX \ar[d]& \\
 \de A^p \times_{\SI_p} X_p \ar[r] & A^p \times_{\SI_p} L_pX \cup \de A^p \times_{\SI_p} X_p \ar[d] 
\ar[r] &  F_{p-1} X(A) \ar[d] \\
& A^p\times_{\SI_p} X_p \ar[r] & F_pX(A)\\} \]
\end{lem}
\begin{proof} Using projections in $\sF$, every point of 
$\coprod_{n\geq 1}A^n\times X_n$ is equivalent to a point $(a,x)$ such that either $n=0$ or no 
coordinate of $a$ is the basepoint of $A$. Using permutations and multiplication maps 
$\varphi_i\colon \mathbf{i} \rtarr \mathbf{1} $ when $i$ coordinates of $a$ are equal, every point is 
equivalent to a point $(a,x)$ such that $a$ has no 
repeated coordinates.  We must take orbits under the action of $\SI_p$ as stated to avoid 
double counting of elements. Using injections, every point is equivalent to a point 
$(a,x)$ such that $x$ is nondegenerate.   Taking care of the order in which the cited 
operations are taken, using \autoref{factor}, the conclusion follows.
\end{proof}

It is now easy to see that $\xi\colon |X|\rtarr X(S^1)$ is a homeomorphism.

\begin{proof}[Proof of \autoref{cute}] As noted in \autoref{matchup2}, the latching subspaces $L_pX$ for $X$ as a $\DE^{op}$-$G$-space 
and as an $\sF$-$G$-space agree. We consider the strata on the filtrations for $|X|$ and $X(S^1)$ and find that $\xi$ defines homeomorphisms
\[(\DE_p\setminus \partial \DE_p) \times (X_p \setminus L_p X) \rtarr \mb{Conf}(S^1\setminus \{*\},p)\times _{\SI_p} (X_p \setminus L_pX). \]
To see this, identify $\mb{Conf}(S^1\setminus \{\ast\},p)$ with $\mb{Conf}(I\setminus \{0,1\},p)$.  Then  $\xi$ sends  a point $((t_1,\dots,t_p),x)$  
in the domain with $0<t_1<\cdots<t_p<1$  to the equivalence class of $((t_1,\dots,t_p),x)$ in the target. Note that  $((t_1,\dots,t_p),x)$  is the unique representative of its class such that the coordinates $t_i$ are in increasing order.\end{proof}

\begin{proof}[Proof of \autoref{consist}] Recall from \autoref{class} that we have the classifying $\sF$-$G$-space $\bB X$ 
with $p$th space $|X[p]|$ and that $(\bS_G^CX)_n = (\bB^nX)_1$.  We must prove that $(\bB^nX)_1$ is 
homeomorphic to $X(S^n)$, and  \autoref{cute} shows that $|X[p]|\iso X[p](S^1)$.  For any $A\in G\sW$,
we have an $\sF$-$G$-space $A\otimes X$ with $p$-th space $(X[p])(A)$. Thus $\bB X \iso S^1\otimes X$.
We claim that $S^n\otimes X$ is isomorphic to $\bB^n X$; evaluating at $p=1$, this gives the 
desired homeomorphism.  Since $S^n = S^1\sma S^{n-1}$, the claim is an immediate induction from
the following result, which is essentially Segal's \cite[Lemma 3.7]{Seg}.
\end{proof}

\begin{prop}\label{ravel} For $A,B\in G\sW$, $A\otimes (B\otimes X)$ and $(A\sma B)\otimes X$
are naturally isomorphic.
\end{prop}
\begin{proof} Recall that $X[p]$ has $j$th space $X_{pj}$, that is $X(\bp\sma\bj)$.
Thus $X[p](B)$ is a quotient of $\coprod_{j} B^j \times X(\bp\sma \bj)$. 
We can write it schematically as $B^{\bullet}\otimes _{\sF} X(\bp \sma \bullet)$.
Writing $\star$ for another schematic variable, we can write the $p$-th space of 
$A\otimes (B \otimes X)$ as
\[ A^\star \otimes_\sF(B^\bullet \otimes_\sF X(\bp \sma \star \sma \bullet)).\]
It is a quotient of $\coprod_{i,j} A^i \times B^j \times X_{pji}$. We define a map
\[A\otimes (B\otimes X) \rtarr (A\sma B)\otimes X\] 
by passage to coequalizers from the maps that send 
\[((a_1,\dots, a_i),(b_1,\dots,b_j),x) \ \ \ \text{to} \ \ \ ((a_r\sma b_s),x) \]
where the $a_r$'s are in $A$, the $b_s$'s are in $B$, and $x\in X_{pij}$. Here $(a_r\sma b_s)$
means the set of $a_r\sma b_s$ in lexicographic order. Indeed, since $\mb{i }\sma \mb{j}$ is ordered lexicographically and $X_{pij}$ is defined in terms of $\bi\sma \bj$, we must order the set $\{a_r\sma b_s\}$ to match. In the other direction, given

$a_t\sma b_t$ for $1\leq t\leq k$ and $x\in X_{pk}$ we map 
$$((a_1\sma b_1,\dots, a_k\sma b_k),x)\ \ \ \text{to}\ \ \ ((a_1,\dots, a_k),(b_1,\dots,b_k),\DE_{*}(x)),$$
where $\DE\colon \bp \sma \bk \rtarr \bp \sma \bk \sma \bk$ is $\id \sma \DE$. Following
Segal \cite[p. 304]{Seg}, ``we shall omit the verification that the two maps are well-defined and inverse to each other".
It can be seen in terms of the explicit description of the filtration strata in \autoref{strata}.  
\end{proof}

\subsection{The proof of the group completion property}\label{groupcomp}

We sketch the proof of \autoref{Ggpcomp} by reducing it to details in \cite[Section 15]{MayClass}.  We also give some relevant background and references.  We first reduce the proof to that of a nonequivariant generalization, \autoref{Classthm},  of the result for classifying spaces of monoids that is proven in \cite{MayClass}.  We then exlain how the details in \cite{MayClass} actually apply to prove \autoref{Classthm}, even though its statement does not seem to appear in the literature.
 
Let $X$ be a special $\sF$-$G$-space.  Then the Segal maps 
$$\de^H\colon X^H_n\rtarr (X_1^n)^H = (X_1^H)^n$$ 
are weak equivalences and $X^H$ is a nonequivariant special $\sF$-space. We emphasize that we
only need this classical condition: we do not require $X$ to be \gen-special.

It is convenient but not essential to modify the definition of a special $\sF$-$G$-space by
requiring the Segal maps $\de$ to be $G$-homotopy equivalences rather than just weak
$G$-equivalences, and then their fixed point maps $\de^H$ are homotopy equivalences. 
We can make this assumption without loss of generality since we can replace given special $\sF$-$G$-spaces 
by cofibrant approximations, for which the assumption holds.

We give $X_1$ a Hopf $G$-space structure by choosing a $G$-homotopy inverse to $\de$ when $n=2$
and using $\varphi_2$. Then $X_1$ and each $X_1^H$ are homotopy associative and commutative, as in
our standing conventions about Hopf $G$-spaces in \autoref{prelimspace}.  We could instead work with weak 
Hopf $G$-spaces, but doing so explicitly only obscures the exposition.

We must prove that the canonical map $\et\colon X_1\rtarr \OM |X|$ is a group completion in the sense 
of \autoref{Defngpcomp}. Passage to $H$-fixed point spaces commutes with realization, as it is a finite limit. It also commutes with taking loops since $G$ acts trivially on $S^1$.  Thus the equivariant case of \autoref{Ggpcomp} 
follows directly from the nonequivariant case.  We therefore take $G=e$ and ignore equivariance in 
the rest of this section.

If $M$ is a topological monoid, we use its product to define a simplicial space $B_{*} M$ with 
$B_nM = M^n$. Then $|B_{*}M|$ is just the classical classifying space $BM$. When $M$ is commutative, 
$B_{*} M$ is the simplicial space obtained by pullback of the evident special $\sF$-space with 
$n$th space $M^n$. When $X$ is a special $\sF$-space its first space 
$X_1$ plays the role of $M$. Since $X_0=*$, $X_1$ has a specified unit
element $e$. Spaces of the form $X_1$ for a special $\sF$-space $X$ give the Segal version of 
an $E_{\infty}$-space.

It makes sense to ask that a simplicial space $X$ be reduced and special, since we can
use iterated face maps to define Segal maps $X_n\rtarr X_1^n$. The Segal maps of $\sF$-spaces
pull back along the map $F\colon \DE^{op}\rtarr \sF$ of \autoref{SEGALsimp} to examples of these more general Segal maps. 
Then $X_1$ is a Hopf space  with product induced by a homotopy inverse to the second Segal map and the map $d_1\colon X_2\rtarr X_1$.  Since $\varphi_2 = Fd_1$, this product on reduced special simplicial spaces generalizes the product on the space $X_1$ of a special $\sF$-space $X$.  Spaces of the form $X_1$ for a reduced special  simplicial space $X$ give the Segal version of an $A_{\infty}$-space.  The group completion theorem for (reduced) special $\sF$-spaces is a special case of the following more general result.

\begin{thm}\label{Classthm} If $X$ is a reduced special simplicial space such that 
$X_1$ and $\OM |X|$ are homotopy commutative,  then  
$\et\colon X_1\rtarr \OM |X|$ is a group completion.
\end{thm}

Rather than give full details, we describe the conceptual framework for this nonequivariant result.  Just as there is a uniqueness result for infinite loop space machines, there is a uniqueness result for $1$-fold delooping machines.  It was first developed by Thomason \cite{Thom} and Fiedorowicz \cite{Fied}.  A later paper by Dunn \cite{Dunn} combined features of \cite{Fied, Thom} to give a treatment closely parallel to that in \cite{MT}, and it  gives a similar uniqueness theorem for $n$-fold delooping machines.   

\autoref{Classthm} is a result about Segal's $1$-fold delooping machine.   The papers just cited make clear that $X_1$ is naturally equivalent to a topological monoid $M$ and that $\et$ is equivalent to the natural inclusion $\ze\colon M\rtarr \OM BM$.  The proof makes use of the Moore loop space functor.  See \cite[Theorem 2.5 and Lemma 2.8]{Thom} or \cite[Example 2.8]{Dunn}.  These sources also show that $\et$ and $\ze$ induce group completions on $\pi_0$; an easy direct argument proves that when $\pi_0(M)$ is abelian \cite[Lemma 15.2]{MayClass}.  

The sources \cite{Thom, Fied, Dunn} do not consider the homological group completion property.  That would not hold in the generality in which they work since it requires some homotopy commutativity.  It was proven for $\ze\colon M\rtarr \OM BM$ 
with homotopy commutativity weakened to the assumption that left and right multiplication by any element are homotopic in \cite[\S15]{MayClass}.  However, the proof is actually given in the context of the Segal machine,  and it  can easily be adapted to prove \autoref{Classthm} directly for $\et$.  

Indeed, its starting point is Segal's based variant of the standard adjunction between topological spaces and simplicial sets.  That gives an adjunction $(T,S)$ relating the category of reduced special  simplicial spaces, denoted $\sS^+\sT$ in \cite[Section 15]{MayClass}, to the category $\sT$.   Here the functor $T = |-|$ is geometric realization. The functor $S$ is a reduced version of the total singular complex. For a  based space $K$, $S_pK$ is the set of $p$-simplices $\DE_p\rtarr K$ that map all vertices to the basepoint.  In particular, $S_1K = \OM K$.  Let $\ph\colon TS K \rtarr K$ and $\ps\colon X\rtarr ST X$ be the counit and unit of the adjunction.  Then \cite[Proposition 15.5]{MayClass} gives the following result.

\begin{prop}  If $K\in \sT$ is path connected, then $\ph\colon TS K\rtarr K$ is a weak equivalence.
For any $X\in \sS^+\sT$,  $T\ps\colon TX\rtarr TSTX$ is a weak equivalence.
\end{prop}

From here, the main tool is a comparison of the standard homology spectral sequence of the filtered space $TX= |X|$
to  the analogous spectral sequence for $TST X$, where $X$ is as in \autoref{Classthm}.  The map $\ps$ induces a map of spectral sequences for homology with coefficients in a field $k$ that converges to the homology isomorphism induced by the weak equivalence  $TX\rtarr TST X$.   It is clearly an isomorphism on $E^{\infty}$, and its map of $E_2$-terms is
\[  Tor^{\et_*}(\id,\id) \colon Tor^{H_*(X_1)}(k,k) \rtarr Tor^{H_*(\OM |X|)}(k, k). \]
A version of the comparison theorem for spectral sequences gives sufficient information about $E^2$ to deduce \autoref{Classthm} from a classical result of Cartan and Eilenberg  \cite[Proposition VI.4.1.1]{CE}  to the effect that $Tor$ commutes with localization. We refer the reader to \cite[Section 15]{MayClass} for the details in the case of $\ze\colon M\rtarr \OM BM$, but the details are exactly the same for $X$ as in \autoref{Classthm}.

\subsection{The positive linearity theorem}\label{cof}

We prove \autoref{Omnibus} in this section and the next. In both, let $G$ be a group and let $X$ be an $\sF$-$G$-space.
For a based $G$-CW complex $A$, write $Y(A)$ ambiguously for either $B(A^{\bullet},\sF,X)$, where $X$ is special, or 
$B(A^{\bullet},\sF_G,\bP X)$, where $X$ is \gen-special; here $\bP X$ is a special $\sF_G$-$G$-space, and up to ismorphism all special 
$\sF_G$-$G$-spaces are of this form.  Under either assumption, we write $Y_*(A)$ for the simplicial bar construction
whose realization is $Y(A)$.   These notations remain fixed throughout these two subsections.  Recall \autoref{simpre}.  As asserted in
\autoref{blanket} and is easily checked using that the degeneracy maps are given by identity maps of $\sF$ and $\sF_G$, $Y_{*}(A)$ is Reedy cofibrant in the simplicial sense.  Remember that we are assuming that given based $G$-spaces are nondegenerately based.

We start with the proof of the easier part of \autoref{Omnibus}.

\begin{lem}\label{conn2} The $\sW_G$-$G$-space $Y$ preserves connectivity.
\end{lem}
\begin{proof}
Let $A\in G\sW$ be $G$-connected, so that $A^H$ is connected for all $H$. We must show that $Y(A)$ is $G$-connected. 
Using that passage to fixed points commutes with pushouts one leg of which is a closed inclusion, it is 
easily checked that geometric realization and the bar construction commute with passage to $H$-fixed points. 
Using this, we see 
that $Y(A)^H$ is the geometric realization of a simplicial space with $0$-simplices given by 
\begin{equation}\label{1}\bigvee_{n} (A^H)^{n} \sma X_n^H
\end{equation} 
or
\begin{equation}\label{2}\bigvee_{\bn^{\al}} (A^{\bn^{\al}})^H \sma  (( \bP X)(\bn^{\al}))^H,\end{equation} 
depending on whether $Y(A)$ is $B(A^{\bullet},\sF,X)$ or 
$B(A^{\bullet},\sF_G, \bP X)$ for an  $\sF$-$G$-space $X$.
We claim that the space of $0$-simplices is connected in either case, so that the geometric realization $Y(A)^H$ is also connected. In the first case, it is clear that the space \autoref{1} is connected since we assume that $A^H$ is connected, and in the second case, the space \autoref{2} is connected because $(A^{\bn^{\al}})^H$ is connected. To see that, note that $(A^{\bn^{\al}})^H\cong (A^n)^\LA$ where $\LA = \{(h,\al(h)) \mid h \in H\}\subset G\times \SI_n$.  Thus by \autoref{lem:lafixedpoints},
$$(A^{\bn^{\al}})^H \cong \prod A^{K_i},$$
where the product is taken over the orbits of the $H$-set $\bn^{\al|_H}$ and the $K_i\subset H$ are the stabilizers of elements in the corresponding orbit. Again, by our assumption that the $A^{K_i}$ are connected  it follows that $(A^{\bn^{\al}})^H$ is connected.
\end{proof}

To prove that the $\sW_G$-$G$-space $Y$ is positive linear, as defined 
in \autoref{lineardefn},  we first isolate properties that, together with preservation of connectivity, imply positive linearity.

\begin{defn}\label{wedax1}  A $\sW_G$-$G$-space $Z$ satisfies the
wedge axiom if for all $A$ and $B$ in $G\sW$ the natural map
\[   \pi\colon Z(A\wed B) \rtarr Z(A)\times Z(B) \]
induced by the canonical $G$-maps $A\wed B\rtarr A$ and $A\wed B\rtarr B$ is a weak $G$-equivalence.
\end{defn}

\begin{defn}\label{realax1} Let $Z$ be a $\sW_G$-$G$-space and consider simplicial based $G$-CW complexes $A_{*}$.
\begin{enumerate}[(i)]
\item  $Z$ commutes with geometric realization if the natural $G$-map 
$$|Z(A_{*})|\rtarr Z(|A_{*}|)$$ 
is a homeomorphism.  
\item $Z$ preserves Reedy cofibrancy if the simplicial $G$-space $Z(A_{*})$ is Reedy cofibrant when $A_{*}$ 
is Reedy cofibrant. 
\end{enumerate}
\end{defn} 

Our $\sW_G$-$G$-space $Y$ satisfies all of these properties. We record the results here, with the proofs delayed to later.

\begin{prop}\label{wedax2} The $\sW_G$-$G$-space $Y$ satisfies the wedge axiom.
\end{prop} 

\begin{prop}\label{realax2} The $\sW_G$-$G$-space $Y$ commutes with realization and preserves Reedy cofibrancy.
\end{prop} 

Granting these results, \autoref{Omnibus} is an application of the following result.

\begin{thm}\label{Omnibus2} Let $Z$ be a $\sW_G$-$G$-space that satisfies the wedge axiom, commutes with geometric realization,
preserves Reedy cofibrancy, and preserves connectivity.  Then $Z$ is positive linear.
\end{thm}

The proof centers around the following construction of cofiber sequences in terms of wedges and
geometric realizations; it is a corrected version and equivariant generalization of a construction due to Woolfson \cite{Woolf}.

\begin{con} Let $f\colon A\rtarr B$ be a map in $G\sW$ with cofiber $i\colon B\rtarr Cf$, where 
$A$ is $G$-connected. We give an elementary simplicial description of $Cf$ in terms of wedges.
 Let $^q\! A$ denote the wedge of $q$ copies of $A$, labelling the $i$th wedge
summand as $A_i$ and setting $^0\!A = \ast$. We define a simplicial $G$-space $W_{*}(B,A)$  whose 
space of $q$-simplices 
$W_q(B,A)$ is $B\wed\, ^{q}\!A$.  Define face and degeneracy operators $d_i$ and $s_i$ with domain $W_q(B,A)$ 
for $0\leq i\leq q$ as follows.
\begin{enumerate}[$\bullet$]
\item All $d_i$ and $s_i$ map $B$ onto $B$ by the identity map.
\item $d_0$ maps $A_1$ to $B$ via $f$ and maps $A_j$ to $A_{j-1}$ by the identity map if $j>1$.
\item $d_i$, $0 < i < q$, maps $A_j$ to $A_j$ by the identity map if $j<i$ and maps $A_j$ to $A_{j-1}$ 
by the identity map if $j>i$. 
\item $d_q$ maps $A_j$ to $A_j$ by the identity map if $j<q$ and maps $A_q$ to $\ast$.
\item $s_i$ maps $A_j$ to $A_j$ by the identity map if $j\leq i$ and 
maps $A_j$ to $A_{j+1}$ by the identity map if $j > i$.
\end{enumerate}

The simplicial identities are easily checked.\footnote{Our specification of faces is a correction of the specification in \cite[Lemma 1.10]{Woolf}.} Note that there is not much choice: the $s_i$ must be 
inclusions with $s_is_j = s_{j+1}s_i$ for $i\leq j$ and they must satisfy $d_is_i = \id = d_{i+1}s_i$.
If we specify the $s_i$ as stated, then the $d_i$ must be as stated except in the exceptional cases
noted for $d_0$ and $d_q$.   
\end{con}

\begin{lem} The realization $|W_{*}(B,A)|$ is homeomorphic to $Cf$.
In particular, the realization $|W_{*}(\ast,A)|$ is homeomorphic to $\SI A$.
\end{lem}
\begin{proof}  Clearly every point of $W_q(B,A)$ for $q\geq 2$ is degenerate. Therefore, identifying $\DE_1$ with
$I$, the realization is the quotient of $B\times\{*\}\amalg (B\wed A)\times I$ 
obtained by the identifications
\[  (b,t) \sim (b,*) \ \ \text{for} \ \  b\in B \ \ \text{and} \ \ t\in I \ \ \text{since} \ \ s_0 b = b \]
\[  (a,0) \sim (f(a),*) \ \  \text{for}\ \ a\in A \ \ \text{since} \ \ d_0(a) = f(a) \]
\[  (a,1) \sim (\ast,*) \ \  \text{for}\ \ a\in A \ \ \text{since} \ \  d_1(a) = \ast \]
It is simple to verify that the result is homeomorphic to $Cf = B\cup_f (A\sma I)$.
\end{proof}
 
We record the following result, which is the equivariant generalization of  \cite[Theorem 12.7]{MayGeo}. It is proven
by applying that result to $H$-fixed point simplicial spaces for all subgroups $H$ of $G$.  A $G$-map $f$ is a 
$G$-quasifibration if each $f^H$ is a quasifibration. Similarly, a map $p$ of simplicial based $G$-spaces is a simplicial 
based Hurewicz $G$-fibration if each $p^H$ is a simplicial based Hurewicz fibration in the sense of \cite[Definition 12.5]{MayGeo}).

\begin{thm}\label{GeoTheo} Let $E_{*}$ and $B_{*}$ be simplicial based $G$-spaces and let $p_{*}\colon E_{*}\rtarr B_{*}$ be a simplicial based Hurewicz $G$-fibration with fiber $F_{*} = p_{*}^{-1}(\ast)$. If each $B_q$ is $G$-connected and $B_{*}$ is Reedy cofibrant, then the realization $|p_{*}|\colon |E_{*}|\rtarr |B_{*}|$ is a $G$-quasifibration 
with fiber $|F_{*}|$.
\end{thm}

\begin{proof}[Proof of \autoref{Omnibus2}] Let $f\colon A\rtarr B$ be a map in $G\sW$, where $A$ is $G$-connected. We must show that application of 
$Z$ to the cofiber sequence 
$$ \xymatrix@1{ A \ar[r]^-{f} & B \ar[r]^-{i} & Cf \\} $$
gives a fiber sequence.

Let $p_{*}\colon W_{*}(B,A)\rtarr W_{*}(\ast,A)$ be the map of simplicial $G$-spaces given by
sending $B$ to the basepoint and let $i_{*}\colon B_{*} \rtarr W_{*}(B,A)$ be the inclusion, where $B_{*}$ denotes
the constant simplicial space at $B$.  On passage to realization, these give the 
canonical maps
\[ \xymatrix@1{  B \ar[r]^-{i}  & Cf \ar[r]^-{p} & \SI A\\}. \]
We apply our given $\sW_G$-$G$-space $Z$ to these maps to obtain the sequence
\[ \xymatrix@1{  Z(B) \ar[r]^-{Z(i)}  & Z(Cf) \ar[r]^-{Z(p)} & Z(\SI A)\\}. \]
We claim that it is a fiber sequence.  Since $Z$ commutes with realization and since $Z(B) \iso |Z(B_{*})|$,
this sequence is $G$-homeomorphic to the sequence 
\[ \xymatrix@1{    |Z(B_{*})| \ar[rr]^-{|Z(i_{*})|} & & |Z(W_{*}(B,A))| \ar[rr]^-{|Z(p_{*})|} 
& & |Z(W_{*}(\ast,A))|.  \\}  \]
On $q$-simplices, before realization, we have the commutative diagram
 \[ \xymatrix{    
Z(B) \ar[rr]  \ar@{=}[d]& & Z(B\wed\, ^q\!A ) \ar[rr] \ar[d] & & Z(^q\!A) \ar[d] \\
Z(B) \ar[rr] && Z(B)\times Z(A)^q \ar[rr] && Z(A)^q \\} \]
where the horizontal arrows are given by the evident inclusions and projections.  The vertical arrows
are the canonical maps and are weak $G$-equivalences by the wedge axiom.   The simplicial $G$-spaces  $W_{*}(B,A)$ and $W_{*}(\ast,A)$ are
trivially Reedy cofibrant, hence so are $Z(W_{*}(B,A))$ and $Z(W_{*}(\ast,A))$. The nondegeneracy of 
basepoints implies that  $Z(B)\times Z(A)^{\bullet}$ and $Z(A)^{\bullet}$ are also Reedy cofibrant.  Therefore
the vertical arrows become weak $G$-equivalences after passage to realization.   Since $A$ and therefore $Z(A)$ is
$G$-connected, \autoref{GeoTheo} applies to 
show that the realization of the bottom row is a fibration sequence up to homotopy.
Thus we have the fiber sequence 
\[ Z(B)\rtarr Z(Cf) \rtarr Z(\SI A) \]
and therefore also the fiber sequence
\[ \OM Z(\SI A) \rtarr Z(B)\rtarr Z(Cf). \]
Specializing to the map $\id\colon A \rtarr A$, we also 
have a fiber sequence
\[ Z(A) \rtarr Z(CA) \rtarr Z(\SI A). \]
Since $Z(CA)$ is $G$-contractible, this implies that we 
have a weak $G$-equivalence $Z(A) \rtarr \OM Z(\SI A)$ and therefore the desired fiber sequence
\[ Z(A) \rtarr Z(B)\rtarr Z(Cf). \qedhere\]
\end{proof}

The following result will apply to show that $Y$ commutes with realization.

\begin{lem}\label{realize2}
Let $Z$ be a $\sW_G$-$G$-space such that $Z(A)$ is naturally isomorphic to $|Z_{*}(A)|$ for some functor 
$Z_{*}$ from $\sW_G$ to simplicial based $G$-spaces.  If the natural $G$-map
$$|Z_q(A_{*})|\rtarr Z_q(|A_{*}|)$$ 
is a $G$-homeomorphism for all simplicial based $G$-spaces $A_{*}$ and all $q\geq 0$, then the natural map
$$|Z(A_{*})|\rtarr Z(|A_{*}|)$$
is a $G$-homeomorphism.
\end{lem}
\begin{proof} For a bisimplicial $G$-space, realizing first in one direction and then the other gives a space 
that is $G$-homeomorphic to the one obtained by realizing in the opposite order. Let $A_{*}$ be a simplicial $G\sW$-space.
Then $Z(A_p) = |q\mapsto Z_q(A_p)|$, so 
\[|Z(A_{*})|=|p\mapsto |q\mapsto Z_q(A_p)||.\]
By the assumption on $Z_q$,
 \[Z(|A_{*}|)=|q\mapsto Z_q(|p\mapsto A_p|)|\cong |q\mapsto |p\mapsto Z_q(A_p)||.\]
The result follows.
\end{proof}

\begin{proof}[Proof of \autoref{realax2}]  We first prove that $Y$ commutes with realization by checking that 
the natural map 
$$ |Y_q(A_{\bullet})| \rtarr Y_q(|A_{\bullet}|)$$
is a $G$-homeomorphism for all $q\geq 0$, so that \autoref{realize2} applies.  Writing out the simplicial bar construction, we see that $Y_q(A_{\ast})$ 
is a wedge of simplicial $G$-spaces of the form  $A_{\ast}^n\sma K$ for some $n\geq 0$ and some $G$-space $K$. 
Since realization commutes with wedges and smash products with $G$-spaces, this gives the proof.

Now assume that $A_{*}$ is  Reedy cofibrant. We must prove that $Y(A_{*})$ is Reedy cofibrant. 
Let $Y= B((-)^\bullet, \sF, X)$; the proof in the case $Y=B((-)^\bullet, \sF_G, \bP X)$ is the same.   By \autoref{PlenzReedy}, it suffices to show 
that all of the degeneracy maps
$$ S_i\colon B((A_{n-1})^\bullet,\sF,X) \xrightarrow{B((s_i)^\bullet,\id,\id)} B((A_n)^\bullet,\sF,X),$$ 
are $G$-cofibrations. Using a shorthand notation, the maps $S_i$ are geometric realizations of maps of Reedy cofibrant simplicial $G$-spaces 
that are given on $q$-simplices by the $G$-cofibrations
$$(s_i)^{n_q}\sma (\id)  \colon \bigvee_{n_0,\dots, n_q} (A_{n-1})^{n_q} \sma (-) \rtarr
\bigvee_{n_0,\dots, n_q} (A_{n})^{n_q} \sma (-).$$
By \autoref{PlenzCof}, the $S_i$  are therefore $G$-cofibrations.
\end{proof} 

\subsection{The proof that the wedge axiom holds}\label{WEDGE}
In this section we prove \autoref{wedax2}, which says that the $\sW_G$-$G$-space $Y$ satisfies the wedge axiom. 
 \begin{proof}[Proof of \autoref{wedax2}]  
 We write the proof for $Y(A) = B(A^{\bullet},\sF_G,\bP X)$, where $X$ is \gen-special.
The proof for $Y(A) = B(A^{\bullet},\sF,X)$, where $X$ is (classically) special, is essentially the same, but  simpler since it is simpler to keep track of equivariance in that case.  For convenience of notation, we write $a,b,c,d,e,f$ for based finite $G$-sets, that is objects of $\sF_G$.  We write $X(a)$ for the value of $\bP X$ on $a$ and $A^a$ for the based $G$-space $\sT_G(a,A)$, with $G$ acting by conjugation. 

Recall from \autoref{Bsmash} that $B(A^\bullet,\sF_G,\bP X)=B^\sma(A^\bullet,\sF_G,\bP X)$. As in \cite[Theorems 2.5 and 3.15]{GMMO}, the quotient map 
\[B^\times(A^\bullet,\sF_G, \bP X) \rtarr B(A^\bullet,\sF_G, \bP X)\]
is a $G$-equivalence.  Thus it suffices to prove the result for $Y(A)=B^\times(A^\bullet,\sF_G, \bP X)$.\footnote{The defects of $B^\times$ pointed out in
\autoref{UhOh} are irrelevant here.}

To do so, we shall construct a $G$-homotopy commutative diagram of $G$-spaces
\begin{equation}\label{wedgy}
\xymatrix{
& Z(A,B) \ar[dl]_F \ar[rd]^Q &\\
Y(A\wed B) \ar[rr]_-{P}  && Y(A)\times Y(B)\\}
\end{equation}
in which $F$ and $Q$ are weak $G$-equivalences and $P$ is the canonical map induced by the projections 
$\pi_A \colon A\wed B\rtarr A$ and $\pi_B\colon A\wed B\rtarr B$.  
This will prove that $P$ is a weak $G$-equivalence.  

In this subsection, we abbreviate notation by writing $\sC(A;X)$ for the category internal to $G\sU$ whose nerve is $B^\times_{*}(A,\sF_G,\bP X)$ and whose 
classifying $G$-space is therefore $Y(A)$.  Recall from \autoref{Groth} that the object and morphism $G$-spaces of $\sC(A; X)$ are
\[\coprod_{a} A^a \times X(a)\] 
and 
\[\coprod_{a,c} A^c \times \sF_G(a,c) \times X(a).\]
Its source and target $G$-maps $S$ and $T$ are induced from the evaluation maps of the contravariant $G\sU_{\ast}$-functor $A^{\bullet}$ and the covariant
$G\sU_{\ast}$-functor $X$ from $\sF_G$ to $\sT_G$.  The identity and composition $G$-maps $I$ and $C$ are induced from identity morphisms and composition in $\sF_G$.

Analogously, we define $Z(A,B)$ to be the classifying $G$-space of the category $\sC(A,B;X)$  internal to $G\sU$ 
whose nerve is the simplicial bar construction
$$B^\times_{*}(A^\bullet \times B^\star;\sF_G \times \sF_G; (\bP X)(\bullet \wed \star)).$$  
Its object and morphism $G$-spaces are 
\[\coprod_{a,b} (A^a\times B^b) \times X(a\wed b)\]
and 
\[\coprod_{a,b,c,d } (A^c \times B^d) \times \sF_G(a,c)\times \sF_G(b,d) \times X(a\wed b).\]

Here $S$ and $T$ are induced from the evaluation maps of $A^\bullet \times B^\star$ and the composite of 
$\wed \colon \sF_G \times \sF_G \rtarr \sF_G$ with the evaluation maps of $\bP X$.  Again, identity and composition maps are induced
from identity morphisms and composition in $\sF_G$. 

Using categories, functors, and natural transformations to mean these notions internal to $G\sU$ in what follows,
we shall define functors giving the following diagram of categories and shall prove that it is commutative up to 
natural transformation.
\begin{equation}\label{wedgytoo}
\xymatrix{
& \sC(A,B;X) \ar[dl]_{F} \ar[rd]^{Q}  &\\
\sC(A\wed B;X) \ar[rr]_-{P}  && \sC(A;X)\times \sC(B;X)\\} 
\end{equation}
Passing to classifying $G$-spaces, this will give the diagram \autoref{wedgy}.   

To define $F$ and $Q$, it is convenient to write elements of
$A^a$ as based maps $\mu\colon a\rtarr A$, with $G$ acting by conjugation on maps.    The functor $F$ sends an object  $(\mu,\nu, x)$ to $(\mu \wed \nu, x)$
and sends a morphism  $(\mu ,\nu, \ph,\ps,x)$ to $(\mu \wed \nu, \ph \wed \ps,x)$.  The functor $Q$ sends an object $(\mu,\nu, x)$ to $(\mu,x_a)\times (\nu,x_b)$ 
where $x_a$ and $x_b$ are obtained from $x\in X(a\wed b)$ by using the $G$-maps induced
by the projections $\pi_a\colon a\wed b\rtarr a$ and $\pi_b\colon a\wed b\rtarr b$.  It sends a morphism $(\mu,\nu,\ph,\ps,x)$ to the morphism
$(\mu,\ph,x_a)\times (\nu,\ps,x_b)$.  As in \autoref{wedgy}, $P$ is induced by the projections 
$\pi_A$ and $\pi_B$.  Noting that $\pi_a$ and $\pi_b$ are $G$-fixed morphisms of $\sF_G$ and that $\pi_A\com (\mu\wed \nu) = \mu\com \pi_a$ and  
$\pi_B\com (\mu\wed \nu) = \nu\com \pi_b$, we see that the morphisms 
$$(\mu,\pi_a,x)\times (\nu,\pi_b, x)\colon (\pi_A\com (\mu\wed \nu),x)\times (\pi_B\com (\mu\wed \nu),x) \rtarr (\mu,x_a)\times (\nu,x_b)$$
give a natural transformation $P\com F\rtarr Q$ in diagram \autoref{wedgytoo} that induces a $G$-homotopy $P\com F \rtarr Q$ in diagram \autoref{wedgy}.

While $Q$ need not be an equivalence of categories of any sort, we see from the assumption that $\bP X$ is special and our use of the projections $\pi_a$ and $\pi_b$ that $Q$  gives a level weak equivalence of simplicial $G$-spaces on passage to nerves. Reedy cofibrancy of the bar constructions then implies that the induced map $Q$ in \autoref{wedgy} 
is a weak equivalence of classifying $G$-spaces.  To complete the proof, we shall construct  a functor $F^{-1}\colon \sC(A\wed B;X)\rtarr \sC(A,B; X)$  and natural transformations 
$ \Id\rtarr F^{-1}\com F$ and $\Id \rtarr F\com F^{-1}$.  Passing to classifying $G$-spaces, this will imply that $F$ in \autoref{wedgy} is a $G$-homotopy equivalence.

For $\om\colon f\rtarr A\wed B$, define $\om_A = \pi_A\com \om$ and $\om_B = \pi_B\com\om$.  Define 
$\si_{\om}\colon f\rtarr f\wed f$, called the splitting of $\om$, by 
 \[
 \si_{\om} (j) = \begin{cases}
 j & \text{in the first copy of $f$ if     }  \om(j)\in A\setminus \ast\\
 j & \text{in the second copy of $f$ if  }  \om(j)\in B\setminus \ast\\
 \ast & \text{if }  \om(j)= \ast
 \end{cases}
\]
Observe that $\om$ factors as the composite 
$$ \xymatrix@1{ f \ar[r]^-{\si_{\om}} & f\wed f \ar[r]^{\om_A \wed \om_B} & A\wed B.\\}$$
Define $F^{-1}$ on objects by 
$$ F^{-1}(\om,x) = (\om_A,\om_B, (\si_{\om})_{*}(x)). $$
For a morphism $(\om, \ph, x)$, $\om\colon f\rtarr A\wed B$, $\ph\colon  e\rtarr f$, and $x\in X(e)$, observe that
$$ (\om \com \ph)_A  = \om_A \com \ph   \ \ \text{and} \ \ (\om \com \ph)_B = \om_B \com \ph.$$
Define $F^{-1}$ on morphisms by
$$F^{-1}(\om,\ph,x) = (\om_A,\om_B, \ph,\ph, (\si_{\om\com \ph})_{*}(x)).$$
A check of definitions using the commutative diagram
 \[ \xymatrix{
e \ar[r]^{\phi} \ar[d]_{\si_{\om\com \ph}}& f \ar[d]_-{\si_{\om}} \ar[r]^-{\om} & A\wed B \\
e\wed e \ar[r]_-{\ph \wed \ph} & f\wed f \ar[ru]_-{\om_A\wed \om_B} \\ } \]
 shows that $S\com F^{-1} = F^{-1} \com S$ and $T\com  F^{-1} = F^{-1} \com T$, and $F^{-1}$ is 
 clearly compatible with composition and identities. It is easily checked that $F^{-1}$ is continuous on object and 
 morphism $G$-spaces, but equivariance is a little more subtle. We first claim that $\si_{g\cdot \om} = g\cdot \si_{\om}$. 
 Since $(g\cdot \om)(j) = g\om(g^{-1}j)$, 
  \[
 \si_{g\cdot \om} (j) = \begin{cases}
 j & \text{in the first copy of $f$ if     }  g\om(g^{-1}j)\in A\setminus \ast\\
 j & \text{in the second copy of $f$ if  }  g\om(g^{-1}j)\in B\setminus \ast\\
 \ast & \text{if }  g\om(g^{-1}j)= \ast.
 \end{cases}
\]
On the other hand, using the definition of $\si_{\om}$ and the fact that $gg^{-1} = 1$, 
\[ g\cdot \si_{\om}(j) = g\si_{\om}(g^{-1}j) = 
\begin{cases}
 j & \text{in the first copy of $f$ if     }  \om(g^{-1}j)\in A\setminus \ast\\
 j & \text{in the second copy of $f$ if  }  \om (g^{-1}j)\in B\setminus \ast\\
 g \ast = \ast & \text{if }  \om(g^{-1}j)= \ast.
 \end{cases}
 \]
 Observing that $gz \in A \setminus \ast$ if and only if $z\in A \setminus \ast$ and similarly for $B$, we see that 
 these agree, proving the claim.  Since $\pi_A$ is a $G$-map, we also have 
 \[ g\cdot \om_A = g\cdot (\pi_A\com \om) = g\pi_A\om g^{-1} =  \pi_A g\om g^{-1} = \pi_A\com g\cdot \om = (g\cdot \om)_A, \]
 and similarly for $B$. Putting these together gives 
 \begin{eqnarray*}
 F^{-1}(g\cdot(\om,x)) = F^{-1}(g\cdot \om, gx) & = &\big((g\cdot\om)_A,(g\cdot \om)_B,(\si_{g\cdot \om})_{*}(gx)\big) \\
 & = & \big (g\cdot \om_A, g\cdot \om_B, (g\cdot \si_{\om})_{*}(gx)\big).
 \end{eqnarray*}
 Since the evaluation map $\sF_G(a,b) \sma X(a)\rtarr X(b)$ is a $G$-map, this is equal to 
 \[ \big(g\cdot \om_A, g\cdot \om_B, g\cdot ((\si_{\om})_{*}x)\big) = g\cdot F^{-1}(\om,x). \]
 This shows that $F^{-1}$ is a $G$-map on objects. Following the same argument and using that 
 $g\cdot (\om\com \ph) = (g\cdot \om)\com (g\cdot \ph)$, we see that $F^{-1}$ is a $G$-map on morphisms.

Now consider the composite $F\com F^{-1}$. It sends the object $(\om,x)$ in $\sC(A\wed B;X)$ to the object $(\om_A\wed \om_B, (\si_{\om})_{*}x)$. Here
$(\om_A\wed \om_B, \si_{\om}, x)$ is a morphism from $(\om,x)$ to $(\om_A\wed \om_B,(\si_{\om})_{*}x)$.  We claim that this morphism
is the component at $(\om,x)$ of a natural transformation $\Id\rtarr F\com F^{-1}$. These morphisms clearly give a continuous map from the object $G$-space of $\sC(A\wed B;X)$ to its morphism $G$-space, and it is not hard to check naturality 
using the diagram just above.  To show equivariance, if $g\in G$ then
\[g\cdot(\om_A\wed \om_B, \si_{\om}, x) = (g\cdot(\om_A \wed \om_B), g\cdot \si_{\om}, gx) = ((g\cdot\om)_{A}\wed (g\cdot\om)_{B}, \si_{g\cdot \om}, gx),\]
which is precisely the component at $g\cdot(\om,x) = (g\om, gx)$.
 
The composite $F^{-1}\com F$ sends an object $(\mu,\nu,x)$ in $\sC(A,B;X)$ to the object
$$(\pi_A \circ (\mu\wed \nu),\pi_B \circ (\mu\wed \nu), (\si_{\mu \wed \nu})_{*}x)=(\mu\circ \pi_a, \nu\circ \pi_b, (\si_{\mu\wed \nu})_{*}x).$$
Note that $\si_{\mu\wed \nu}$ is given by $\tilde \si_{\mu} \wed \tilde \si_{\nu}$, 
where $\tilde\si{_\mu} \colon a \rtarr a \wed b$ is the inclusion, 
except that if $\mu(j)=\ast $, then $\tilde\si_{\mu} (j)=*$, and similarly for $\tilde\si_{\nu}\colon b\rtarr a\wed b$. 
Here  
$$ \big( \pi_A \circ (\mu\wed \nu), \pi_B \circ (\mu \wed \nu), \tilde\si_{\mu}, \tilde{\si}_{\nu}, x \big)$$
is a morphism in $\sC(A,B;X)$ from $(\mu,\nu,x)$ to $\big(\pi_A \circ (\mu\wed \nu), \pi_B \circ (\mu \wed \nu), (\si_{\mu\wed \nu})_{*}x\big)$. 
To see that the source of this morphism is as claimed, observe that
$\mu \circ \pi_a \circ \tilde\si_{\mu} = \mu$
since $\pi_a \circ \tilde{\si}_{\mu}=\id$ except on those $j$ such that $\mu(j)=\ast$, and similarly for $\nu$. 
The naturality follows from the equations 
\[\tilde\si_{\mu} \circ \phi = (\phi \wed \psi) \circ \tilde\si_{\mu\com\phi} \ \ \text{and} \ \  
\tilde\si_{\nu} \circ \psi = (\phi \wed \psi) \circ \tilde\si_{\nu\com\psi},\]
which are easily checked. The continuity of the assignment is also easily verified.  For $g\in G$,  a verification 
similar to that for $F\com F^{-1}$ shows that $g$ acting on the component of our natural transformation at 
$(\mu,\nu,x)$ is the component of the transformation at  $g\cdot (\mu,\nu,x) = (g\mu,g\nu,gx)$.  
\end{proof}

\subsection{The heart of the Segal machine}\label{heredity}  We prove \autoref{part2} and \autoref{keystone} in this subsection.  Recall from \autoref{specialWG} that $Z$ is a special $\sW_G$-$G$-space if it preserves connectivity and its restriction to $\sF_G$ is a special $\sF_G$-$G$-space.   Recall also that  $Z[A](B) = Z(A\sma B)$  for  $A, B\in \sW_G$. Since $A\sma B$ is $G$-connected if either $A$ or $B$ is $G$-connected and since smashing preserves cofiber sequences, the proof of the following lemma is immediate from the definitions. 

\begin{lem}\label{lindetail}  Suppose that  $Z$ is positive linear and preserves connectivity.  Then $Z[A]$ is positive linear and preserves connectivity for any $A$.   If $A$ is $G$-connected, then $Z[A]$ is linear rather than just positive linear.
\end{lem}

Given \autoref{lindetail}, to prove \autoref{part2} it remains to prove that the restriction of a $\sW_G$-$G$-space $Z[A]$  to a $\sF_G$-$G$-space is special when $A$ is $G$-connected.  After the preliminary \autoref{hearty}, we prove this as \autoref{specious}. 

Until otherwise specified, we assume that $Z$ is a positively linear and special $\sW_G$-$G$-space.

\begin{lem}\label{hearty} Let $\al$ be a finite $G$-set and let $Y[\al]$ denote the restriction of $Z[\al]$ to $\sF_G$.  Then $Y[\al]$ is a special 
$\sF_G$-$G$-space.
\end{lem}
\begin{proof}  We are given that $Y[\mathbf{1}]$ is a special $\sF_G$-$G$-space.   For a finite $G$-set $\be$,  we have the analogue for $\sF_G$ of the diagram \autoref{Bspecial}:
\[ \xymatrix{
Y[\al](\be) \ar[r]^-{\de}  \ar@{=}[d]& Y[\al](\mathbf{1})^{\be} \ar@{=}[r] & Y(\al)^{\be} \ar[d]^{\de^{\be}}\\
Y(\al\sma\be) \ar[r]_-{\de} & Y(\mathbf{1})^{\al\sma\be} \ar@{=}[r] & (Y(\mathbf{1})^{\al})^{\be}. \\} \]
The right and bottom arrows $\de$ are weak $G$-equivalences, hence so is the top arrow $\de$.  Thus $Y[\al]$ is special. 
\end{proof}

\begin{lem}\label{specious} Let $A \in \sW_G$ be $G$-connected and let $Y[A]$ denote the restriction of $Z[A]$ to $\sF_G$.  Then $Y[A]$ is a special $\sF_G$-$G$-space.
\end{lem}
 \begin{proof}  Let $S^1_s$ again be the simplicial circle.  Its $q$-simplices ${(S^1_s)}_q$ are finite based sets, hence are finite based $G$-sets with trivial $G$-action.  By \autoref{hearty}, for each finite $G$-set $\al$ we have the simplicial special $\sF_G$-space  with $q$-simplices 
$Y[({S^1_s)}_q\sma \al]$.  Its geometric realization gives a special $\sF_G$-$G$-space isomorphic to $Y[S^1\sma \al]$.  In particular, we can take $\al$ to be a finite wedge of based $G$-sets $(G/H)_+$.  

By positive linearity, for a cofiber sequence
$$ \xymatrix@1{ A\ar[r] & B \ar[r] & C} $$
in $\sW_G$ where $A$ is $G$-connected and a finite based $G$-set $\al$ we have  an induced map of fiber sequences 
\[ \xymatrix{ Z(A\sma \al) \ar[r] \ar[d]_{\de}& Z(B\sma\al) \ar[r] \ar[d]^{\de} & Z(C\sma\al) \ar[d]^{\de}\\
Z(A)^{\al} \ar[r] & Z(B)^{\al}  \ar[r] & Z(C)^{\al}. \\} \]
Therefore, if $Y[A]$ and $Y[B]$ are special, then so is $Y[C]$.
By definition, any $A\in \sW_{G}$ is a finite based $G$-CW complex.   When $A$  is $G$-connected, we can replace $A$ by a $G$-homotopy equivalent based $G$-CW complex whose $0$-skeleton is a point and whose attaching maps are based $G$-maps defined on $G$-spheres $G/H_+\sma S^n$ for $n\geq 1$.
By induction on $n$, using the cofiber sequences $S^{n-1} \rtarr D^n \rtarr S^n$, each $Y[G/H_+\sma S^n]$  is special.  From here, we see that $Y[A]$ is special by induction on the number of cells in $A$.
\end{proof}

The rest of this section is devoted to the proof of \autoref{keystone}. From now on, we assume that $Z$ is special and linear (not just positively linear).  We have in mind the $Z[A]$ for a $G$-connected $A \in \sW_G$.  The proof of \autoref{specious} simplifies to give the following analog.

\begin{lem}\label{specious2}   When $Z$ is special and linear, the restriction $Y[A]$ of $Z[A]$ to $\sF_G$ is a special $\sF_G$-$G$-space
for any  $A \in \sW_G$.
\end{lem}

We restate \autoref{keystone}.  The rest of this subsection is devoted to its proof, following \cite{Seg2, Shim}.

\begin{thm}\label{part3}  If $Z$ is a special and linear $\sW_G$-$G$-space, then the structure map 
$$\tilde{\si}\colon Z(S^0) \rtarr \OM^VZ(S^V)$$ 
is a weak $G$-equivalence.
\end{thm}

We first recall from \autoref{structuremaps} the general definition of such structure maps.
For based $G$-spaces $A, B, C$ in $\sW_G$ and a based $G$-map 
$$\mu\colon B\rtarr \sW_G(A,C),$$
composition with an application of $Z$ gives a based $G$-map 
$$ B\rtarr \UG(Z(A),Z(C)).$$
In turn, by adjunction, this gives a based $G$-map
$$ \hat{\mu}\colon Z(A)\rtarr \UG(B,Z(C)).$$
Explicitly, for $z\in Z(A)$ and $b\in B$,
$$\hat{\mu}(z)(b) = Z(\mu(b))(z).$$
When  $\mu\colon B \rtarr \sW_G(A,A\sma B)$ is the adjoint of the identity map of $A\sma B$, this is a natural map 
$$\tilde{\si} \colon Z(A) \rtarr \UG(B, Z(A\sma B)).$$
Taking  $A=S^0$ and $B = S^V$, this is the adjoint structure map
$$ \tilde{\si}\colon Z(S^0)\rtarr  \OM^V Z(S^V).$$ 

For a subspace $M$ of $V$, let $M_{\epz}$ be the open $\epz$-neighborhood of $M$ in $V$.  Explicitly, 
$$  M_{\epz} =\{v\in V \, |  \,  \text{there exists\ } m\in M \text{\ such that\ } |v-m| < \epz \} \subset  V.$$
Write $(-)^c$ for one-point compactification, which we recall is functorial on proper  maps between locally compact spaces,  and  denote points at infinity by $\infty$.  Recall that $(-)^c$ converts cartesian products to smash products and that  $A^c = A_+$ if $A$ is compact.  For a  small $\epz$, let  $0_{\epz}$ be the $\epz$-neighborhood of the origin in $V$.  Then $0_{\epz}^c$ is isomorphic to the one-point compactification $S^V$ of $V$ and we treat the evident isomorphism as an identification from now on. 

For any compact (unbased) $G$-subspace $M$ of $V$ we have a $G$-map 

\begin{equation}\label{epzmap} M \rtarr \ul{G\sU}(0_\epz, M_\epz),\ \ \ \ 
x \longmapsto (v \mapsto x+v).\end{equation}  Applying a  Pontryagin-Thom construction gives a based $G$-map 
$$\mu\colon M_{+}\rtarr \sW_G(M_{\epz}^c, 0_{\epz}^c).$$
Explicitly, for  $x\in M$ and  $y\in M_{\epz}$, $\mu(x)(y) = y-x$ if $|y-x| < \epz$ and $\mu(x)(y) = \infty$ otherwise.  
As described above, this gives rise to a map
\begin{equation}\label{hatmu}\hat{\mu}_M \colon Z(M_{\epz}^c) \rtarr \UG(M_+, Z(S^V)).\end{equation}

 \begin{rem} We may take $M = \{0\}$.  Then $\mu\colon S^0 = M_+ \rtarr \sW_G(0_{\epz}^c,0_{\epz}^c)$
 sends $0$ to the identity map and thus $\hat{\mu}_{\{0\}}$  can be identified with  the identity map of $Z(S^V)$.
 \end{rem}

We will prove that when $M=S$ is the unit sphere in $V$, the map $\hat{\mu_S}$ is an equivalence. Before we prove this, we set up a little bit more notation.
Consider the radial projection $$\pi\colon S_\epz= \{ v \in V \mid  1-\epz<|v|<1+\epz\} \rtarr S$$ given by $\pi(v)=v/|v|$. For a subset $X$ of $S$, let
\[\widehat{X} = \pi^{-1} (X) = \{ v \in V \mid  1-\epz<|v|<1+\epz \text{ and } \pi(v)\in X\}.\]
Thus $\widehat{X}$ is the subset $X$ on the sphere, thickened by $\epz$ in the radial direction.  Note that $\widehat{X} \subseteq X_\epz$, but these are in general not equal. They are equal in the case of $X=S$.
Finally, again for $X$ a subset of $S$, let 
\[\widetilde{X}=X-(S-X)_{2\epz}=\{x\in X \mid |x-v|\geq 2\epz \text{ for all } v\in S-X\}\footnote{The original sources \cite{Seg, Shim} use $\epz$ in this definition, but it seems to us that we need $2\epz$ in order to obtain the restricted map \autoref{epzmaprestricted} for subspaces $X$ of $S$}.\]  Thus $\widetilde{X}$ is the subset $X$ on the sphere, with a $2\epz$-neighborhood cropped out in the tangent direction. Note that $\widetilde{S}=S$.  With this definition,  for any subset $X$ of $S$ the $G$-map of \autoref{epzmap} with  $M=S$ restricts to a $G$-map 
\begin{equation}\label{epzmaprestricted}
\al_X\colon \widetilde{X} \rtarr \ul{G\sU}(0_\epz,\widehat{X}).
\end{equation}
Again using the Pontryagin-Thom construction  we obtain a map
\[\mu_X \colon \widetilde{X}_+ \rtarr \sW_G(\widehat{X}^c,0_\epz^c)=\sW_G(\widehat{X}^c,S^V),\]
which again gives rise to a map
\[\hat{\mu}_X \colon Z(\widehat{X}^c) \rtarr \UG(\widetilde{X}_+,Z(S^V))\] for any subset $X$ of $S$.  When $X=S$, this is the map $\hat{\mu}_S$ that we constructed in \autoref{hatmu} above (with $M=S$).

\begin{thm}\label{Segalpf}   Let $Z$ be a special  $\sW_G$-$G$-space. Then the map \[\hat{\mu}_S \colon Z(S_\epz^c) \rtarr \UG(S_+,Z(S^V))\]
is a weak $G$-equivalence. 
\end{thm}

\begin{proof}
By Illman \cite{Illman2}, there exists an equivariant triangulation of $S$.  Let $\{C_\la\}_{\la \in \LA}$ be the collection of open stars of open simplexes.  This is an open cover of $S$. We index the stars by the barycenter of the simplex, that is, $\la$ is the barycenter of the simplex for which $C_\la$ is the star. Taking a refinement if necessary, we can assume that for all $g\in G$, either $g\cdot C_\la=C_\la$, or $g\cdot C_\la$ and $C_\la$ are disjoint. From now on, we let $2\epz$ be small enough compared to all the radii of all the simplexes in the triangulation, so that  $X$  deformation retracts onto $\widetilde{X}$ for any $X$ which is  a $G$-stable union of $C_\la$'s.

Let $X=\bigcup_{\la \in T} C_\la$ be a $G$-stable union of $C_\la$'s. We claim that the map $\hat{\mu}_X$, and in particular $\hat{\mu}_S$, is a $G$-equivalence.  We proceed by induction on the number of $G$-orbits contained in the indexing set $T$ for the open stars in $X$, which we recall is a set of barycenters. When $T$ is a single orbit, the inclusion $T \rtarr X$ induces $G$-equivalences $T_+ \rtarr \widetilde{X}_+$ and $T_+ \sma S^V \rtarr \widehat{X}^c$. Via these equivalences, $\hat{\mu}_X$ can be identified with
 \[Z(T_+\sma S^V) \rtarr \UG(T_+,Z(S^V)),\]
 which is a $G$-equivalence since $Z[S^V]$ is special by \autoref{part2}.  
 For the inductive step, we take $G$-stable unions of stars $X$ and $Y$, assume the conclusion for $X$, $Y$ and $X\cap Y$, and prove the conclusion for $X\cup Y$.  The open stars of a triangulation form a covering in the sense that the intersection of any two stars is a star, so that $X\cap Y$ is indeed also a $G$-stable union of stars, possibly empty. The conclusion that $\hat{\mu}_{X\cup Y}$ is an equivalence follows by induction from the following diagram, using the claim that its front and back squares are homotopy cartesian since we are assuming  that $\hat{\mu}_{X}$,  $\hat{\mu}_{Y}$ and  $\hat{\mu}_{X\cap Y}$ are equivalences.
 \[
\xymatrix{
Z(\widehat{(X\cup Y)}^c)  \ar[rr]^-{}  \ar[dd]_{} \ar[dr]^{\hat{\mu}_{X\cup Y}}& & Z(\widehat{X}^c) \ar[dd]^(.3){}|\hole \ar[dr]^{\hat{\mu}_X}& \\
&  \UG( \widetilde{X\cup Y}_+,Z(S^V)) \ar[rr]^(.3){}  \ar[dd]_(.3){} & & \UG(\widetilde{X}_+,Z(S^V))\ar[dd]^{}  \\
Z(\widehat{Y}^c)  \ar[rr]^(.3){}|(.48)\hole  \ar[dr]_{\hat{\mu}_Y}& & Z(\widehat{(X\cap Y)})^c) \ar[dr]^{\hat{\mu}_{X\cap Y}}    & \\
&  \UG(\widetilde{Y}_+,Z(S^V)) \ar[rr]^(.3){}_(.3){}  & &\UG(\widetilde{X \cap Y}_+,Z(S^V))
   \\
}
\]
It thus remains to show that the front and back squares are homotopy cartesian. The front square is homotopy cartesian because the square 
 \[\xymatrix{
 \widetilde{X\cap Y} \ar[r] \ar[d] & \widetilde{X}\ar[d]\\
 \widetilde{Y} \ar[r] & \widetilde{X \cup Y}
 }
 \]
 is homotopy cocartesian, since $X$, $Y$ are open in $X\cup Y$, and $\epz$ was chosen small enough so that there is a $G$-deformation retraction $\widetilde{W}\simeq W$ for all $G$-stable unions of open stars $W$ in our triangulation. For the back square, note that 
 \[(\widehat{X\cup Y}- \widehat{X})^c\rtarr \widehat{X\cup Y}^c \text{   and   } (\widehat{Y}- \widehat{X\cap Y})^c\rtarr \widehat{Y}^c\]
  are $G$-cofibrations. Moreover, we have that 
 \[ \widehat{X\cup Y}- \widehat{X} = \widehat{Y}- \widehat{X\cap Y}\]
   since both sets consist of  those $v\in S_\epz$ such that $\pi(v) $ is in $Y$ but not in $X$.\footnote{Note that this equality would not be true if instead of taking the inverse images under the radial projections $S_\epz\to S$ we took actual $\epz$ neighborhoods $X_\epz$ of subsets of the sphere.}  Identifying the cofibers, we then get a diagram of homotopy cofiber sequences: 
  \[\xymatrix{
(\widehat{X\cup Y} - \widehat{X})^c \ar[r] \ar@{=}[d] & (\widehat{X\cup Y})^c \ar[d] \ar[r] & \widehat{X}^c \ar[d]\\
(\widehat{Y} - \widehat{X\cap Y})^c \ar[r] & \widehat{Y}^c \ar[r] & \widehat{X\cap Y}^c.
 }
 \]
 Since $Z$ is linear, after applying $Z$, we get a diagram of fiber sequences
 \[\xymatrix{
Z((\widehat{X\cup Y} - \widehat{X})^c) \ar[r] \ar@{=}[d] & Z((\widehat{X\cup Y})^c) \ar[d] \ar[r] & Z(\widehat{X}^c) \ar[d]\\
Z((\widehat{Y} - \widehat{X\cap Y})^c) \ar[r] & Z(\widehat{Y}^c) \ar[r] & Z(\widehat{X\cap Y}^c).
 }
 \]
 In particular, the right square in the diagram is homotopy cartesian.
 \end{proof}
 
 We now give the proof of \autoref{part3}.
 
 \begin{proof}[Proof of \autoref{part3}]
  Let $S$ and $D$ be the unit sphere and disk in $V$, and let $i\colon S\to D$ be the inclusion.
 
 Define 
$p\colon D_{\epz}^c \rtarr S_{\epz}^c$  by $p(x) = x$ if $x \in S_{\epz}$ and  $p(x)=\infty$ otherwise.  Then the following diagram commutes.
\[ \xymatrix{
S_{+} \ar[d]_{i} \ar[r]^-{\mu}  & \sW_G(S_{\epz}^c,0_{\epz}^c) \ar[d]^{p^*} \\
D_{+} \ar[r]_-{\mu}  & \sW_G(D_{\epz}^c,0_{\epz}^c) \\}
\]  
Note that $D_\epz\backslash S_\epz \to D_\epz$ is proper, let $j \colon (D_{\epz} \setminus S_{\epz})^c \rtarr D_{\epz}^c$ be  the inclusion,
and observe that the quotient of $j$ is the $G$-space $S_{\epz}^c$ with quotient map $p$.  

Define  
\begin{equation}\label{subtle}
\nu \colon S^V\cong D/S \rtarr \sW_G((D_{\epz} \setminus S_{\epz})^c, 0_{\epz}^c)
\end{equation}
by applying a relative Pontryagin-Thom construction.  Explicitly, $\nu$ is the based $G$-map such that if $x\in D\setminus S$ and 
$y\in D_{\epz}\setminus S_{\epz}$, then $\nu(x)(y) = y-x$
if $|y-x|< \epz$ and $\nu(x)(y) = \infty$ otherwise.   Then the following diagram commutes, where $q$ is the quotient map.   
\[ \xymatrix{
D_{+} \ar[d]_{q} \ar[r]^-{\mu}  & \sW_G(D_{\epz}^c, 0_{\epz}^c) \ar[d]^{j^*} \\
S^V\cong D/S \ar[r]_-{\nu}  & \sW_G((D_{\epz} \setminus S_{\epz})^c ,0_{\epz}^c)\\}
\]

Just as $\mu$ gives rise to $\hat{\mu}$, $\nu$ gives rise to the map $\hat{\nu}$ in the following diagram.   By easy diagram chases, writing $S^V$ for $0_{\epz}^c$, the squares above imply that the squares in the diagram commute.
\begin{equation}\label{keydia}
\xymatrix{
Z((D_{\epz}\setminus S_{\epz})^c) \ar[r]^-{\hat{\nu}}  \ar[d]_{Z(i)} & \UG(D/S, Z(S^V))\ar[d]^{q^*}\\
Z(D_{\epz}^c) \ar[d]_{Z(p)}  \ar[r]^-{\hat{\mu}} & \UG(D_+, Z(S^V)) \ar[d]^{j^*} \\
Z(S_{\epz}^c) \ar[r]_-{\hat{\mu}} &  \UG(S_+, Z(S^V))  \\}
\end{equation} 
The right column is clearly a fiber sequence.  We claim that 
 \begin{equation}\label{yeahy}
\xymatrix@1{(D_{\epz} \setminus S_{\epz})^c  \ar[r]^-{j} &  D_{\epz}^c \ar[r]^-{p}  & S_{\epz}^c.\\}
\end{equation}
is a cofiber sequence so that, since $Z$ is linear, the left hand column  is also a fiber sequence.

 Observe that $D_{\epz} \setminus S_{\epz}$ is the closed disk $D(1-{\epz})$ of radius $1-\epz$ in $V$.   Applying $(-)^c$ adds a disjoint basepoint. Therefore the inclusion 
$$ S^0 =\{0\}\amalg \{\infty\} \rtarr (D_{\epz} \setminus S_{\epz})^c \iso D(1-{\epz})_+$$ 
is a $G$-homotopy equivalence. The inclusion 
$$ j\colon  D(1-{\epz})_+ \rtarr D_{\epz}^c \iso S^V,$$
which sends the disjoint basepoint to the north pole (point at $\infty$) and sends the disk of radius $1-\epz$ to its image around the south pole (the point $0\in V$), is a $G$-cofibration. We conclude that \autoref{yeahy} is $G$-equivalent to the cofiber sequence
$$ S^0 \rtarr S^V \rtarr S^V/S^0.$$

In the diagram \autoref{keydia} the bottom horizontal arrow is a weak $G$-equivalence by \autoref{Segalpf}, and the middle horizontal arrow is easily seen to be a weak equivalence since $D_\epz^c\cong S^V$ and $D_+\simeq S^0$. Hence the top horizontal arrow $\hat{\nu}$ is also a weak equivalence.  Its domain is weak $G$-equivalent to  $Z(S^0)$ and its target is weak $G$-equivalent to $\OM^VZ(S^V)$.  Under these equivalences, $\hat{\nu}$ agrees with $\hat{\si}$, at least up to sign.  On the one hand, for $z\in Z(S^0)$ and $y\in S^V$, 
$$\tilde{\si}(z)(y) = Z(\id(y))(z),$$
where $\id(y)\colon S^V\rtarr \sW_G(S^0,S^V)$ is given by $\id(y)(0)=y$.  On the other hand, for $x\in D/S\iso D(1-\epz)/S(1-\epz)$, we have 
$$\hat{\nu}(z)(x) = Z(\nu(x))(z),$$ 
where $\nu\colon D(1-\epz)_+ \rtarr \sW_G(S^0,S^V)$ satisfies  $\nu(x)(0) = x \in U_{\epz}^c$. 
Therefore $\tilde{\si}$ and $\hat{\nu}$ agree modulo identifications of spheres.  The standard choice of homeomorphism $S^V\iso D/S$ would introduce a sign \cite[Section 13.2]{Concise}, but that is irrelevant to the conclusion that $\tilde{\si}$ is a weak $G$-equivalence since $\hat{\nu}$ is a weak $G$-equivalence.
 \end{proof}

\section{General topological groups and compact Lie groups}\label{general}

As said in the introduction, we have so far focused on finite groups $G$.  This should seem strange since the
 basic foundations for equivariant stable homotopy theory, for example in \cite{LMS}, deal with compact Lie groups and do not greatly simplify when specialized to finite groups. However, that is not true of equivariant infinite loop space theory.   We explain here what is and is not true for infinite groups. 

We consider general topological groups and classical $G$-spectra in \autoref{classical}.  To avoid pathology, we assume throughout  that the identity element of $G$ is a nondegenerate basepoint.   Subgroups of $G$ are understood to be closed, and homomorphisms are understood to be continuous. We consider the Segal and operadic machines for compact Lie groups and genuine $G$-spectra in \autoref{Segcom} and \autoref{opercom} and their comparison in \autoref{CONCLUDE}.

In brief,  the group $G$ could be any topological group in all of our formal theory about classical $G$-spectra.  When considering genuine $G$-spectra, our formal theory makes sense for general compact Lie groups and everything we do remains valid provided that we restrict attention to $G$-representations whose isotropy groups are of finite index in $G$.  A quick explanation is that orbit $G$-spectra of compact Lie groups are self-dual only for such subgroups.  As examples, it is often sensible in applications to remember that maximal tori have finite index in their normalizers.  For example, multiplicative norms were first defined for  subgroups of finite index in compact Lie groups, in \cite[Definition 3.6]{GMMU}.  They were used there to give a partial analogue in  equivariant complex cobordism of the Atiyah-Segal completion theorem in equivariant $K$-theory.
The special role played by finite index subgroups suggests generalization to profinite groups, but we have not pursued that idea.

\subsection{Classical $G$-spectra for topological groups $G$}\label{classical}

It is straightforward to generalize infinite loop space theory so as to accept equivariant input and to deliver classical equivariant spectra as output.   The input can be $\sF$-$G$-spaces or it can be $G$-spaces with actions by classical $E_{\infty}$ $G$-operads, as defined below. Moreover, the nonequivariant comparison of infinite loop space machines of \cite{MT}, recalled in \autoref{MT}, generalizes effortlessly to this context.  Any infinite loop space machine landing in classical $G$-spectra is equivalent to the classical Segal machine.  Modifying what needs to be modified, our new explicit comparison of the Segal and operadic machines also works for general topological groups and classical $G$-spectra.

Although this all works for any $G$, caveats appear immediately.  We started work by defining families $\bF_n$ in \autoref{famFn} and considering finite $G$-sets in \autoref{finiteGset}.  The definitions and conventions there make sense in general, but they force focus on the subgroups $H$ of finite index in $G$.  

\begin{rem}\label{HomoGSi} Since our homomorphisms are continuous, any $\al\colon G\rtarr F$, where $F$ is a finite group, 
factors  through a homomorphism $\pi_0(G)\rtarr F$. In particular, there are no non-trivial $\al$ if $G$ is connected.  
There are also infinite discrete groups $G$ that admit no non-trivial $\al$ and have no non-trivial finite dimensional representations.  Thus although the families $\bF_n$ can be defined in general, they are mainly of interest when $G$ is finite. 
Note that the kernel of any homomorphism $\al\colon G\rtarr F$ must have finite index in $G$ and, if $G$ acts on a finite set $S$, then the isotropy group of any element of $S$ must have finite index in $G$.  \end{rem}

\begin{rem}  Since the objects of $\sF_G$ are finite $G$-sets, the category  $\sF_G$ sees only those orbits $G/H$ such that $H$ has finite index in $G$. Provided we restrict attention to such subgroups, we can use $\sF_G$ for general topological groups $G$, but there seems little point in doing so.  \end{rem}

 For the classical Segal machine for general topological groups,  we ignore $\sF_G$ and restrict attention to $\sF$ and the notions of specialness and of level $G$-equivalence that are defined in (ii) and (iv) of  \autoref{weakFn}.   When there are no non-trivial homomorphisms $G\rtarr \SI_n$, for example when $G$ is connected,  a $\PI$-$G$-space is \gen-special if and only if it is special.    The notion of an \gen-special $\PI$-$G$-space is only of substantial interest when  $G$ is finite, although \autoref{sectionFFG} gives some motivation that applies in general.

In \autoref{prelim}, the discussion of orthogonal $G$-spectra applies verbatim to compact Lie groups $G$ and everything else applies verbatim to general $G$. 
In \autoref{SegSec} and much of \autoref{SegHom}, everything applies verbatim to general $G$, except that everything referring to the families $\bF_n$ and to $\sF_G$ is best ignored.  The classical Segal machine appears in \autoref{notone} and in Propositions \ref{consist} and \ref{itsago}, as discussed in \autoref{bar2}.  We reiterate that this works precisely as written for any topological group $G$.  We reformulate  \autoref{Omnibus} to emphasize its perhaps surprising generality. 

\begin{thm}\label{Omnibus3}  Let $G$ be a topological group and let $X$ be a special
$\sF$-$G$-space.  Then the $\sW_G$-$G$-space that sends $A$ to $B(A^{\bullet},\sF,X)$
is positive linear and preserves connectivity. 
\end{thm} 

Although unrestricted, this result is the starting point towards understanding the genuine Segal machine.  Subsections \ref{SecSegDet} through \ref{WEDGE} apply to general topological groups and are largely devoted to the proof of \autoref{Omnibus3}. 
Focusing on categories of operators over $\sF$ rather than over $\sF_G$, the generalization of the Segal machine in \autoref{SegGen}, the generalization of the operadic machine in \autoref{OperadGen}, and the comparison of machines in \autoref{Equivalence}  work  just as written in the generality of classical infinite loop space machines for general topological groups.  

\subsection{The Segal machine for compact Lie groups}\label{Segcom}

We summarize our conclusions for the Segal machine.  For any topological group $G$, we have a functor $\bS_G^C$ that
takes  $\sF$-$G$-spaces to $G$-prespectra. It has four variants. The first uses 
Segal's original simplicially defined inductive machine.  The second is the conceptual machine,
and the third and fourth  are composites that first take $\sF$-$G$-spaces to $\sW_G$-$G$-spaces by a choice of bar constructions and then take $\sW_G$-$G$-spaces to $G$-prespectra, exactly as explained in \autoref{itsago}.
Restricting attention to special $\sF$-$G$-spaces $X$, the four choices are equivalent and the functor $\bS_G^C$ assigns a positive classical $\OM$-$G$-spectrum to $X$, together with a group completion $\et\colon X_1 \htp  (\bS_G^C X)_0 \rtarr \OM(\bS_G^C X)_1$.

If $H$ is a subgroup of $G$, with inclusion denoted $\io\colon H\rtarr G$, then we have various functors $\io^*$ that
restrict given $G$ actions to $H$ actions. By inspection, these functors commute  with all constructions in sight. For example,  $\io^*\bS_G^C X\iso \bS_H^C \io^*X$. These functors also commute with the group completion maps $\et$, so that  
$$\et\io^* = \io^*\et\colon \io^* X_1\rtarr \OM(\bS_H^C \io^*X)_1.$$ 

For finite groups $G$, if $X$ is \gen-special rather than just special, we have a genuine positive $\OM$-$G$-spectrum $\bS_G X$ with underlying classical positive $\OM$-$G$-spectrum $\bS_G^C X$.   Formally we are using a forgetful functor from genuine orthogonal $G$-spectra indexed on a complete $G$-universe to classical $G$-spectra indexed on the trivial 
universe. Restriction to subgroups works the same way on the level of genuine $G$-spectra as it does on the level of classical $G$-spectra.  Note that if $H$ is a finite subgroup of a topological group $G$ and $X$ is an $\sF$-$G$-space which is \gen-special as an $\sF$-$H$-space, then $\io^*\bS^C_G X\iso \bS^C_H\io^*X$ is the underlying classical $\OM$-$H$-spectrum of a genuine $\OM$-$H$-spectrum. 

Now let $G$ be a compact Lie group.  One might hope to apply the Segal machine to somehow construct genuine $\OM$ $G$-spactra from functors defined on 
$\sF$-$G$-spaces.  However, Blumberg shows in \cite[Appendix C]{Blum} that the Segal machine does not give such an infinite loop space machine even when $G=S^1$, and in fact there is no  condition on an $\sF$-$S^1$-space that can imply that it gives a positive  $\OM$-$S^1$-spectrum.   It does not help to switch from $\sF$ to functors defined on $\sF_G$, and this illustrates the problem.  We cannot expect functors defined on the category of finite $G$-sets to give homotopical information about representations that have isotropy groups of infinite index in $G$.  Nevertheless, as said before, when we restrict to  $G$-representations with isotropy groups of finite index, everything we say about finite groups applies equally well for compact Lie groups.   We explain in outline how this works.   In effect, the following  startng point describes restriction to those based finite $G$-CW complexes that are built up from finite $G$-sets.

\begin{defn}\label{WGf}  Define $\sW^f_G$ to be the full $G\sU_*$-subcategory of $\sW_G$ whose objects are the based finite $G$-CW complexes whose isotropy groups have finite index in $G$.   Note that $\sF_G \subset \sW_G$.  Define a $\sW^f_G$-$G$-space to be a $G\sU_{\ast}$-functor 
$\sW^f_G\rtarr \UG$.  By results of  Illman \cite{Illman2},  $S^V$ is in $\sW_G$ for all $G$-representations $V$.   Define $\sI_G^f$ to be the full subcategory of $\sI_G$ that is obtained by restricting to those $G$-representations such that $S^V$ is in $\sW_G^f$, and define $G\sS^f$ to be the analogue of the category of orthogonal $G$-spectra $G\sS$ (see  \autoref{OrthSpec}) defined by replacing $\sI_G$ with $\sI_G^f$.  Define $\bU_G^f\colon \Fun(\sW^f_G,\sT_G)\rtarr G\sS^f$ to be the evident forgetful functor.  It fits into the commutative diagram of forgetful functors
\[   \xymatrix{
\Fun(\sW_G,\UG) \ar[d]_{\bU_G}   \ar[r]  &  \Fun(\sW^f_G,\UG) \ar[d]^{\bU^f_G} \\
G\sS \ar[r] & G\sS^f. } \]
We say that a $G$-spectrum $E$ is a {\em positive finite index $\OM$-$G$-spectrum} if the maps $\tilde{\si}\colon E(V)\rtarr \OM^W E(V\oplus W)$ are weak $G$-equivalences when the isotropy groups of $V$ and $W$ have finite index in $G$ and  $V^G\neq 0$. 
\end{defn}

Our constructions work to obtain an orthogonal $G$-spectrum  $\bS_GX$ from a \gen-special $\sF$-$G$-space  $X$.  Equivalently, we can construct $\bS_G Y$ from a special $\sF_G$-$G$-space  $Y$ and start from there. We think of $Y = \bP X$ or, equivalently, $X = \bU Y$.  This works without restriction to give $\OM$-$G$-spectra when $G$ is finite.  In that case, $\bS_G X$ is a (genuine) positive $\OM$-$G$-spectrum whose bottom structural $G$-map is a group completion of $X_1$.  Prolongation of  $B(\sF,\sF, X)$ to a $\sW_G$-$G$-space does not give a positive $\OM$-$G$-spectrum; prolongation of $B(\sF_G,\sF_G, Y)$ to a $\sW_G$-$G$-space does.  This remains true for compact Lie groups if we restrict attention to $\sW^f_G$ and $\bU_G^f$.  Said differently, prolongation of $B(\sF_G,\sF_G, Y)$ gives rise to a positive finite index $\OM$-$G$-spectrum. 

Therefore, we now modify the definition, \autoref{Gmachine}, of an infinite loop space machine by replacing $G\bf{Sp}_{cp}$ with the category $G\bf{Sp}_{cp}^{f}$ of connective orthogonal $G$-spectra $\{E_V\}$ which restrict to positive $\OM$-$G$-spectra on the subuniverse obtained by allowing only those representations $V$ whose isotropy subgroups $H$ have finite index in $G$.  

Modifying everything by restricting isotropy  subgroups of representations to have finite index in $G$, everything in subsections  \ref{barFG} and  \ref{heredity} works with no changes other than notation to prove the following generalization of the fundamental theorem, \autoref{bigSegal},  about the Segal machine.

\begin{thm}\label{bigSegal2}  Let $G$ be a compact Lie  group and let $X$ be an \gen-special $\sF$-$G$-space. 
Then $\bS_G X$ is a positive finite index $\OM$-$G$-spectrum. Moreover,  if $S^V\in \sW^f_G$ and $V^G\neq 0$, then the composite 
$$ X_1 \rtarr B(\sF_G,\sF_G,\bP X)_1 = (\bS_G X)(S^0) \rtarr \OM^V(\bS_GX)(S^V)$$ 
of $\et_1$ and the structure $G$-map is a group completion. 
\end{thm}

The restriction to finite index isotropy groups is illuminated by consideration of duality and the Wirthm\"uller isomorphism.  The suspension $G$-spectrum 
$\SI^{\infty}_+ G/H$ is self dual if and only if $H$ has finite index in $G$.  In general, its dual is $\SI^{-L}\SI^{\infty}_+ G/H$. 
where $L$ is the tangent representation of $G/H$ at the identity coset.  Similarly, the Wirthm\"uller isomorphism says that, for an $H$-spectrum $D$, the left adjoint $G_+\sma_H D$ of the forgetful functor from $G$-spectra to $H$-spectra is equivalent to its right adjoint $F_H(G_+,D)$ when $H$ has finite index  in $G$, but is equivalent to   $F_H(G_+,\SI^L D)$ in general. A quick proof is given in \cite{MayWirth}. 	It is the appearance of $L$ in these results that prevents $\bS_G$ from giving positive
$\OM$-$G$-spectra and thus prevents us from having a fully satisfactory Segal machine for compact Lie groups.

Blumberg's paper \cite{Blum} focuses on finding necessary and sufficient conditions for a $\sW_G$-$G$-space to restrict to a positive $\OM$-$G$-spectrum for a compact Lie group $G$.  He finds that two conditions are necessary, one essentially being positive linearity and the other being closely related to the Wirthm\"uller isomorphism.  

\subsection{The operadic machine for compact Lie  groups}\label{opercom}

We must now distinguish between classical and genuine $E_{\infty}$ $G$-operads. A nonequivariant operad $\sC$ is said to be an $E_{\infty}$ operad if $\sC_G(n)$ is a universal principal $\SI_n$-bundle for each $n$.  Thus each $\sC_G(n)$ is $\SI_n$-free and contractible.

\begin{defn}  Let $G$ be a topological group.  A{\em classical} $E_{\infty}$ $G$-operad is a nonequivariant $E_{\infty}$ operad regarded as a $G$-trivial $G$-operad.  A $G$-operad $\sC_G$ is a {\em genuine} $E_{\infty}$ $G$-operad if
$\sC_G(n)$ is a universal principal $(G,\SI_n)$-bundle for each $n$.
\end{defn}

Thus the notion of a genuine $E_{\infty}$ $G$-operad is precisely the same as our earlier notion of an $E_{\infty}$ $G$-operad given in \autoref{Einfop}.  We have simply replaced the finite group implicit there by a general topological group here.  We emphasize  that genuine $E_{\infty}$ $G$-operads for compact Lie groups $G$ are useful for other purposes, even though their algebras do not give input to a fully general equivariant infinite loop space machine.  In particular, the linear isometries $G$-operad plays a central role in the construction of a symmetric monoidal category of $G$-spectra in which every object is an $\OM$-$G$-spectrum \cite{EKMM, MM}.   Such a construction is not possible for symmetric or orthogonal $G$-spectra.

For the classical operadic machine, we start with classical $E_{\infty}$ $G$-operads.   Here the nonequivariant recognition principle \cite{MayGeo} generalizes directly.  While \cite{GM3}  focused on finite groups, the basic theory of $G$-operads and  the recognition principle for classical $G$-spectra were developed for general topological groups in \cite[Section 2]{GM3}.   As long as we restrict to those representations $V$ with trivial $G$-action and to classical $G$-spectra, everything works exactly the same way as it does nonequivariantly.  The approximation theorem, \autoref{approx}, generalizes classically since a check of definitions shows that passage to $H$-fixed points gives an application of the nonequivariant approximation theorem for each $n$ (that is, for each $\bR^n$), and each closed subgroup $H$ of $G$.  

Now let $G$ again be a compact Lie group. Then  the operadic machine, like the Segal machine, generalizes to $G$ provided that we restrict attention to those $G$-representations  and $G$-spaces whose isotropy groups have finite index  in $G$.  Everything in the references \cite{GM3, Haus, RS} works exactly as stated for finite groups except for the approximation theorem, which addresses the homotopical behavior of $\al\colon \bC_G X \rtarr \OM^V\SI^V X$.  Even if $G = S^1$, the group completion property of $\al$ fails if $V$ is a general non-trivial representation of $G$, as was first noticed by Segal \cite{Seg2}. His counterexample is presented in  \cite[Appendix B]{Blum}.  

The only further discussion of the compact Lie case of which we are aware is that of Caruso and Waner \cite{CarWan}.   They address the approximation theorem in  \cite[Theorem 1.18]{CarWan}.  Although they restrict to finite groups in the {\em statement} of that result and have written things in a partially stable context, a minor reworking of their proof\footnote{We are indebted to Stefan Waner for discussion of this.} makes clear that the only use of the finiteness of $G$ is to ensure that the $G$-sets $G/H$ that they encounter are finite.   This gives the following conclusion.

\begin{thm}\label{classic2}   \autoref{classic}  remains true if we start with a compact Lie group $G$ and a genuine $E_{\infty}$ operad $\sC_G$ of $G$-spaces, provided that we require $V$ to have finite isotropy groups.  We do  not require the $\sC_G$-space $X$ to have finite isotropy groups.
\end{thm}

\subsection{The comparison theorem for compact Lie groups}\label{CONCLUDE}

From here,  everything that we have said in this paper applies verbatim to compact Lie groups, provided that we restrict attention to  finite index isotropy subgroups.  We focused on finite groups only for clarity of exposition.   We restate the comparison theorem, \autoref{WOW}, in this generality.  

\begin{thm}\label{WOWW}  Let $G$ be a compact Lie group and let $\sD_G$ be the category of operators over $\sF_G$ constructed from a genuine $E_{\infty}$ operad $\sC_G$ of $G$-spaces. For special $\sD_G$-$G$-spaces $Y$, there is a natural zigzag of equivalences of orthogonal $G$-spectra in $G\sS^f$  between $\bS_GY$  and $\bE_GY$.
\end{thm} 

Here we define the homotopy groups $\pi_*^H(X)$ of  $X\in G\sS^f$ for all subgroups $H$ of $G$ exactly as in \autoref{pistar}, but with the 
colimit running over those $V$ whose isotropy groups have finite index in $G$.  We are restricting to the finite index subuniverse of the complete universe used in \autoref{pistar}.  Just as in the complete universe case, we define an equivalence in  $G\sS^f$ to be a $\pi_*$-isomorphism.   With these equivalences, we then have the stable model structure on $G\sS^f$  exactly as in \cite[Theorem III.4.2]{MM}.  The discussion in the following epilogue applies verbatim in this more general context.

\section{Epilogue: model categorical interpretations}\label{epilogue}

Our work can be contextualized in terms of model categories,\footnote{As the referee very helpfully pointed out.} and we do so now. The essential point is that equivariant infinite loop space theory establishes what in 
\cite{MMSS} is called a ``connective Quillen equivalence" between the model category of grouplike input objects of either the Segalic or operadic sort to the connective part of the stable model category of orthogonal $G$-spectra. For the Segal machine, this is documented nonequivariantly in  \cite{BF, MMSS, Schwede}  and equivariantly in \cite{Oster, Sant}; we shall explain that result with different details here. The result for the operadic machine is documented in \cite{Units}, although the result was understood much earlier.

In fact, all of the categories in sight have standard and well-known model structures, so we shall just say what those are rather than reproving things from scratch. There are many adjoint pairs of functors in sight.  Most are Quillen adjoint pairs and some are (connective) Quillen equivalences.   It is reasonable to upgrade the definition of an infinite loop space machine, \autoref{Gmachine}, model theoretically to at least include model structures on the source and target categories and to require a characterization principle in the sense of \autoref{char}.  That is, one can insist on a connective Quillen equivalence between grouplike input and connective output.  Ignoring model structures, it is already apparent that we do have equivalences of homotopical categories in the sense of \cite{DHKS}.

On the input data side,  our categories have several relevant model structures, and a comparison among the roles they play is illuminating. We explain that in \autoref{inmodel}, after recalling some of the relevant structures on underlying ground categories in \autoref{ground}. On the output side, model structures on the relevant categories of $G$-spectra are also well understood and will be summarized as we go. The main point that needs investigation is the relationship between these various model structures on the one hand and the infinite loop space machines and their comparison on the other.  We discuss that for the (generalized) Segal machine in \autoref{modelSegal}.  Aside from understanding the difference between the classical and genuine dichotomy model theoretically, there is little that is new and we shall be brief. The answer for the operadic machine is easier, as in \cite{Units}.  
We discuss the model categorical relevance of the bar construction in \autoref{bar1}.

Much of what  we say in this section applies to general topological groups $G$ and classical $G$-spectra, but our focus is on finite groups $G$ and genuine $G$-spectra.

\subsection{Model structures on $G$-spaces and diagram $G$-spaces}\label{ground}

We have been working throughout in topologically enriched $G$-categories, as we discussed in some detail in \autoref{catpre}.   All have standard model structures. The three ground categories in our work are $G\sU_*$,  $\Fun(\PI,\UG)$ and $\Fun(\PI_G,\UG)$. These are all complete and are all enriched in $G\sU_*$.   Recall that we write  $G\sT$ for the full subcategory of $G\sU_*$ consisting of nondegenerately based  spaces, and we write $\PI[G\sT]=\PI[\sT_G]$ and $\PI_G[\sT_G]$ for the full subcategories of $\Fun(\PI,\UG)$ and $\Fun(\PI_G,\UG)$ satisfying the extra cofibration conditions spelled out in \autoref{Fspace}.

We start with the model structures on $G\sU$; they restrict to give model structures on $G\sU_{\ast}$ 
\cite[Theorem 15.3.6]{morecon}.  The standard $q$-model structure ($q$ for Quillen) on $G\sU$ has the weak $G$-equivalences as its weak equivalences and the Serre $G$-fibrations as its fibrations; these are the $G$-maps whose $H$-fixed point maps are weak equivalences or Serre fibrations for all (closed) subgroups $H$.  The cofibrations are the retracts of $G$-cell complexes, whose attaching maps are defined on $G$-spheres  $G/H\times S^n$.  The cofibrations of $G\sU_{\ast}$ are the based maps whose underlying unbased maps are $G$-cofibrations,  we refer to these as $q$-equivalences, $q$-fibrations, and $q$-cofibrations.  A recent discussion of these model structures on $G$-spaces, with their enrichments, is given in \cite[Section 1]{GMR}.   In $G\sU_*$, the $q$-cofibrations satisfy the  based HEP (homotopy extension property). 

The $h$-model structure ($h$ for homotopy or Hurewicz) on $G\sU$ has the $G$-homotopy equivalences as its weak equivalences and the Hurewicz $G$-fibrations (CHP) and Hurewicz $G$-cofibrations (HEP) as its fibrations and cofibrations.\footnote{Even nonequivariantly, the factorization axioms are subtle. Incorrect arguments in \cite{Cole1, MaySig, morecon} are corrected and generalized by Barthel and Riehl in \cite{BR, BMR};  the generalization is in a categorical context that includes $G\sU$.}  We refer to these as $h$-equivalences, $h$-fibrations, and $h$-cofibrations.  Since $h$-fibrations are $q$-fibrations and $h$-equivalences are $q$-equivalences, we have  a mixed model structure with $m$-fibrations = $h$-fibrations and $m$-equivalences = $q$-equivalences; the $m$-cofibrant objects are the spaces of the homotopy type of CW-complexes \cite{Cole}, \cite[\S17.4]{morecon}.   As discussed in \cite[\S4.1]{MaySig} and \cite{morecon}, classical algebraic topology implicitly focused on the mixed model structure on spaces, and the same is true equivariantly.  We shall return to the $m$-model structure in \autoref{bar1}, but we focus on $q$-model structures until then.

Since $\Fun(\PI,\UG)$ and $\Fun(\PI_G,\UG)$ are diagram categories, there is a large body of work that applies to them; see for example \cite{GM1, GMR, MM, Sant, Oster}.  Consider the following diagram of adjunctions 
\[\xymatrix@1{
G\sU_* \ar@<-.5ex>[r]_-{\bR}  & \ar@<-.5ex>[l]_-{\bL} \Fun(\PI,\UG)  \ar@<-.5ex>[r]_-{\bP}   
& \Fun(\PI_G,\UG) \ar@<-.5ex>[l]_-{\bU} \\} \]
Since $(\bP,\bU)$ is an adjoint equivalence, $(\bU,\bP)$ is also an adjunction, and we write $(\bL_G,\bR_G)$ for the displayed composite adjunction
$(\bL\bU,\bP\bR)$.   

\begin{defn} We define the level $q$-model structures on  $\Fun(\PI,\UG)$ and $\Fun(\PI_G,\UG)$ by letting a map  $f\colon X\rtarr Y$ be a level $q$-equivalence or level $q$-fibration if each $f_n$ or each $f_{\al}$ is a $q$-equivalence of $G$-spaces. We call these the {\em classical} model structure on $\Fun(\PI,\UG)$ and the {\em genuine} model structure on $\Fun(\PI_G,\UG)$.   These are often called projective model structures in the literature. 
\end{defn}

\begin{rem}  We have given the definition in terms of the $q$-model structure, but we could instead use the less familiar $h$ or $m$-model structure, with preference for the last.  Everything works the same way, using the foundational details in \cite{BR, BMR}.  
\end{rem}

These model structures are discussed nonequivariantly in \cite[Sections 6 and 17]{MMSS} and in a more general categorical context in \cite[Section 4.4]{GM1}.  Equivariant details are given in \cite[Sections 4 and 5]{Sant} and \cite[Section 4]{Oster}.   The following central remark uses \autoref{compFFG} to describe the adjoint equivalence  $(\bP,\bU)$ model theoretically. 

\begin{rem}\label{Both} We regard the level model structure on  $\Fun(\PI,\UG)$ as its classical model structure.  Under the adjoint equivalence $(\bP,\bU)$, it
transports to a new  {\em{classical}} model structure on $\Fun(\PI_G,\UG)$ whose weak equivalences are those maps that are level $G$-equivalences {\em{when restricted to $G$-trivial finite $G$-sets.}}  Similarly, we view the level model structure on  $\Fun(\PI_G,\UG)$ as its genuine model structure.  Using \autoref{compFFG},  we see that, under $(\bP, \bU)$, it transports to a new {\em{genuine}} model structure on $\Fun(\PI,\UG)$ whose weak equivalences are the \gen-level equivalences.    

With these definitions, the equivalence $(\bP,\bU)$ is a Quillen equivalence of model categories with respect to {\em either} the classical {\em or} the genuine level model structures on our two categories.   The identity functor on  $\Fun(\PI,\UG)$ or on $\Fun(\PI_G,\UG)$ is the left and right adjoint of a Quillen adjunction from the classical to the genuine model structure.  The latter is a localization of the former.
\end{rem}

\begin{rem}\label{Both2}  Pragmatically, we are only interested in {\em special} input.  Model theoretically, it therefore makes sense to localize the level model structures at Segal maps to obtain a new model structure on $\Fun(\PI,\UG)$ in which the fibrant objects are the (classically) special objects and a new model structure on  
$\Fun(\PI_G,\UG)$ in which the fibrant objects are the (genuinely) special objects. These localizations transfer along the equivalence $(\bP, \bU)$ to give a localization of the  classical model structure on $\Fun(\PI_G,\UG)$ in which the fibrant objects are the $\bP X$, where $X$ is classically special, and a localization of  the genuine model structure on $\Fun(\PI,\UG)$ in which the fibrant objects are the $\bU Y$ such that $Y$ is genuinely special; $(\bP,\bU)$ gives Quillen equivalences of classical and genuine model structures after localization. 
\end{rem}

\subsection{Quillen equivalences among input model categories}\label{inmodel}

We combine and expand the diagrams \autoref{InputDiag} and \autoref{InputDiag2} to discuss the input for our machines model categorically.  Assume given an $E_{\infty}$ operad $\sC_G$.  It has associated categories of operators $\sD$ and $\sD_G$, the first over $\sF$ and under $\PI$ and the second over $\sF_G$ and under $\PI_G$.  Consider the following schematic diagramm where all pairs of arrows represent adjoint functors. 

\begin{equation}\label{InputDiag3} 
\xymatrix{
\Fun(\sF,\UG) \ar@<.5ex>[dd]^{\xi^*}   \ar@<.5ex>[rr]^{\bP} & & \ar@<.5ex>[ll]^-{\bU}  \Fun(\sF_G,\UG) \ar@<.5ex>[dd]^{\xi^*_G}  \ar@<.5ex>[rr]^{\bP}
& & \Fun(\sW_G,\UG)  \ar@<.5ex>[ll]^{\bU}   \ar@<.5ex>[dd]^{\bU_{G\sS}} \\
& & & & \\
\Fun(\sD,\UG) \ar@<.5ex>[uu]^{\xi_*}  \ar@<.5ex>[rr]^{\bP} \ar@<-.5ex>[dd]_{\bL} & & \Fun(\sD_G,\UG)  \ar@<-.5ex>[ddll]_(.6){\bL_G}\ar@<.5ex>[ll]^-{\bU}  
 \ar@<.5ex>[uu]^{{\xi_G}_*} \ar[rr]^-{\bS_G} \ar[ddrr]_{\bE_G}   & &G\sS  \ar@<.5ex>[uu]^{\bP_{G\sS}}\\
 & & & & \\
\sC_G[G\sU_*]  \ar@<-.5ex>[uu]_{\bR} \ar@<-.5ex>[uurr]_(.45){\bR_G}  & & & &  G\sS \\} \end{equation}
Here we write $\xi_*$ and ${\xi_G}_*$ for the {\em categorical} rather than the derived left adjoints to $\xi^*$ and $\xi_G^*$.\footnote{We apologize to the reader for the conflict of notation.  However, when one interprets the bar construction as a cofibrant approximation, the conflict becomes a reasonably standard one.}  
Since $\xi^*\bU =\bU\xi_G^*$, $\bP\xi_*\iso {\xi_G}_*\bP$. By definition, $\bL_G = \bL\bU$ and $\bR_G = \bP\bR$.   Of course, $\bS_G$ and $\bE_G$ are the infinite loop space machines compared in \autoref{Equivalence}, while the composite  $\bU_{G\sS}\bP\bP$ is the conceptual Segal machine we started with. 

We ignore $\sW_G$ and $G\sS$ in this subsection and consider the six adjoint pairs displayed in the left square and triangle. We focus on finite groups and genuine model structures.

\begin{thm}\label{modelin}  The categories in the left square and the triangle have $q$-model structures, and the six adjoint pairs between them are all Quillen equivalences.
\end{thm}
We call all of these \emph{genuine} model structures.
\begin{proof} As noted in  Sections \ref{opermach} and \ref{sectionmonad}, the five categories on the lefthand side can be described as the categories of algebras over certain monads. As such, they 
 all have right adjoint forgetful functors to their ground categories, which are $G\sU_*$, 
 $\Fun(\PI,\UG)$, or $\Fun(\PI_G,\UG)$.  The genuine model structures on the ground categories are described in the previous section.    We say that a map  $f$ in any of our five categories is a  weak  equivalence or fibration if its underlying map in the ground category is a weak equivalence or fibration.
     Then
 \cite[Theorem 16.2.5]{morecon} applies to show that these specifications create the required $q$-model structures in our five categories.   It is then clear that the right adjoints in the left square and triangle preserve fibrations and acyclic fibrations and are thus Quillen right adjoints by \cite[Definition 16.2.1]{morecon}.    The discussion in \autoref{Both} applies verbatim to both pairs  $(\bP,\bU)$, showing that they are Quillen equivalences. That $(\xi_*,\xi^*)$ and $({\xi_G}_*, {\xi^G}^*)$ are Quillen equivalences follows from  \autoref{Reedyprop1} and Theorems \ref{Segalin2} and \ref{Segalin4}, which show that $\xi^*$ and $\xi_G^*$ become equivalences of homotopy categories with inverses the derived versions of  $\xi_{*}$ and ${\xi_{G}}_{*}$.  That  $(\bL,\bR)$ and $(\bL_G, \bR_G)$ are Quillen equivalences follows from
\autoref{Mayin}.
\end{proof}

\begin{rem}  Using \autoref{Both} , we see by the same argument that the previous result is true as stated for classical rather than genuine input to infinite loop space machines for arbitrary topological groups $G$.   In both the classical case and the genuine case,  we can localize at Segal maps as in  \autoref{Both2} and obtain the same conclusions as in \autoref{modelin}.
\end{rem}

The Segal machine focuses on properties of input data, whereas the operadic machine focuses on structure, as the following remark illuminates.  

\begin{rem}  An operad is $\SI$-free if $\SI_j$ acts freely on $\sC(j)$ for all $j$.  A $\SI$-free operad $\sC_G$ of $G$-spaces is {\em {classically}} 
$E_{\infty}$  if each $\sC_G(j)$ is contractible.  This makes sense for any topological group $G$, and classical nonequivariant operadic infinite loop space theory generalizes  effortlessly to give infinite loop space theory with classical (or naive) $G$-spectra as output.  In the body of the paper, we restricted attention to  genuine $E_{\infty}$ operads as defined in \autoref{Einfop}.  Of course, such operads are also classically $E_{\infty}$, giving forgetful passage from genuine to classical operadic infinite loop space theory.   
\end{rem}

\subsection{$\sF_G$-$G$-spaces, $\sW_G$-$G$-spaces, and orthogonal $G$-spectra}\label{modelSegal}

We now turn to the right square of \autoref{InputDiag3}.   We shall be brief since Blumberg has given a good summary  \cite[Appendix A]{Blum}, which itself is largely an equivariantization of material in \cite{MMSS} and a summary of material in \cite{MM}.  We first recall some of the model structures on (orthogonal) $G$-spectra from \cite{MM}.  Here $G$ is a compact Lie group.

\begin{defn} The category of orthogonal $G$-spectra has a level model structure \cite[Section III.2]{MM}.  Using representations as usual to define stable homotopy groups and stable (weak) equivalences \cite[Section  III.3]{MM}, we  define the stable model structure on $G\sS$ to have weak equivalences the $\pi_*$-isomorphisms and  cofibrations the $q$-cofibrations of the level model structure \cite[Section III.4]{MM}. The fibrant objects are the $\OM$-$G$-spectra. 
\end{defn}

We view $\Fun(\sW_G,\UG)$ as an intermediary between the input data of the variant Segal machines and the $G$-spectra output.   This is awkward model theoretically since the relevant comparisons are both given by {\em right adjoint} forgetful functors defined on $\Fun(\sW_G,\UG)$.  A drawback of model category theory in general is that its comparisons are designed to deal with Quillen adjunctions, which are intrinsically directional, and the comparison here is the composite of a left adjoint to $\Fun(\sW_G,\UG)$ and a right adjoint from it to $G\sS$.  We start with the absolute model structure, which makes sense for any topological group $G$.

\begin{defn}   In the {\em absolute} level model structure on  $\Fun(\sW_G,\UG)$, a map $Y\rtarr Z$ is a fibration or weak equivalence if  $Y(A)\rtarr Z(A)$ is a Serre $G$-fibration or a  weak $G$-equivalence for all $A\in \sW_G$.  
\end{defn}

The following observation about  the top right adjunction in \autoref{InputDiag3}  is immediate and applies to any topological group $G$.

\begin{prop}  $(\bP,\bU)$ is a Quillen adjunction from the genuine
model structure on  $\Fun(\sF_G,\UG)$  to the absolute level model structure on $\Fun(\sW_G,\UG)$.  
\end{prop}

Henceforward in this subsection,  take $G$ to be a compact Lie group.

\begin{defn}  In the {\em relative} level model structure on  $\Fun(\sW_G,\UG)$, a map $Y\rtarr Z$ is a fibration or weak equivalence if $Y(S^V)\rtarr Z(S^V)$ is a Serre $G$-fibration or a  weak $G$-equivalence for each $G$-representation $V$.  
\end{defn}

By construction,  the identity functor is the left adjoint of a Quillen adjunction from the relative to the absolute level model structure on $\Fun(\sW_G,\UG)$.
The forgetful functor  $\bU_{G\sS}$ in  \autoref{InputDiag3}  has a left adjoint prolongation functor  $\bP_{G\sS}$, constructed as in \cite[Definition 23.1]{MMSS}.   To interpret it model theoretically,  we use stable model structures on $\Fun(\sW_G,\UG)$.  These structures are discussed in \cite[A.2, A.4]{Blum}.  

\begin{defn} Define the relative and absolute stable model structures on $\Fun(\sW_G,\UG)$ to have weak equivalences the maps $f$ such that 
$\bU_{G\sS} f$  is a $\pi_*$-isomorphism.  The cofibrations are the maps that are cofibrations in the relative or absolute level model structure.  Then a $\sW_G$-$G$-space $Y$ is fibrant in the absolute stable model structure if and only if the structure map
$$  Y(A) \rtarr \OM^{W}Y(\SI^W A)$$
is a weak $G$-equivalence for all $A\in \sW_G$ and all $G$-representations $W$.   
\end{defn}

\begin{thm}\label{MSegal2}\cite[Theorems 1.3,  A.13 and A.19]{Blum}, \cite[Theorem 4.16]{MM}.  Let $G$ be a  compact Lie group.
\begin{enumerate}[(i)]
\item  The identity functor is the left adjoint of a Quillen equivalence from the relative to the absolute stable model structure on $\Fun(\sW_G,\UG)$.
\item  The pair $(\bP_{G\sS}, \bU_{G\sS})$ is a Quillen equivalence from the stable model structure on $G\sS$ to the relative stable model structure on $\Fun(\sW_G,\UG)$.
\end{enumerate}
\end{thm}

\begin{rem}\label{quibble1}  Just as we defined the stable model structures on  $\Fun(\sW_G,\UG)$ starting from level model structures, we can define stable model structures on the four categories in the left square in \autoref{InputDiag3}.  In all of them we say that a map is a stable weak equivalence if it becomes one after passing to 
$\Fun(\sF _G,\UG)$ and applying $\bP$ to land in  $\Fun(\sW_G,\UG)$.   Nonequivariantly, this idea goes back to \cite{BF} and later \cite{Schwede} and \cite[Theorem 18.3]{MMSS}.  It is pursued equivariantly in \cite{Oster, Sant1, Sant}.   It is usual to say that grouplike special objects are ``very special''.  The fibrant objects in these stable model structures are very special.   More precisely, the fibrant $\sF$-$G$-spaces are those that are grouplike and \gen-special and the fibrant $\sF_G$-$G$-spaces are those that are grouplike and special.  With the language of \autoref{char},  this model theoretic approach is studying a characterization principle rather than a recognition principle: it is blind to group completion.\end{rem}

Say that a $\sW_G$-$G$-space $Y$ is connective if the $G$-spectrum  $\bU_{G\sS}Y$ is connective.  The following is an equivariant version of \cite[Theorem 0.10]{MMSS}, and the proof is the same as there.  It remains true if we restrict attention throughout to  subgroups of finite index in a compact Lie group.  

\begin{thm}\label{MSegal1}  Let $G$ be a finite group.   Then $(\bP,\bU)$ is a connective Quillen equivalence from  $\Fun(\sF_G,\UG)$ with its stable model structure to $\Fun(\sW_G,\UG)$ with its absolute stable model structure.  
\end{thm}  
\begin{proof}
The unit $\Id \rtarr \bU\bP$ of the adjunction is the identity, hence a stable equivalence on all $X$.    The counit $\bP\bU \rtarr \Id$ is a stable equivalence when applied to connective stably bifibrant $\sW_G$-$G$-spaces,  just as in \cite[Lemma 18.10]{MMSS}.  
\end{proof}

The following is a quibble about the comparison between what we have proven and model categorical language.

\begin{rem}\label{quibble2}  Recall that a Quillen adjunction $(F,U)\colon \sC\rtarr \sD$ in which $U$  creates the weak equivalences in $\sD$ is a Quillen equivalence if and only the unit $\et\colon Y \rtarr UF Y$ is a weak equivalence for all cofibrant $Y$ \cite[Proposition 16.2.3]{morecon}.  In the introduction, we required of a model infinite loop space machine that it induce an equivalence of homotopy categories from grouplike input to connective output, without a requirement about the point-set level transformations that induce the equivalence.  The Segal machine $\bS_G$ is the composite of the left adjoint in \autoref{MSegal1} and the right adjoint in \autoref{MSegal2}.   It does induce such a connective equivalence of homotopy categories, but  the transformations that exhibit the equivalence are not the unit and counit of a single point-set level Quillen adjunction.
\end{rem}  

As already mentioned, the operadic machine with target category Lewis-May $G$-spectra admits a cleaner model theoretic interpretation in which  the defects of Remarks \ref{quibble1} and \ref{quibble2} are circumvented \cite{MayTwist}.

\subsection{The bar construction model theoretically}\label{bar1}  For background, we recall that, classically, many foundational questions were sensibly swept under the rug with the blanket assumption that all given spaces are of the homotopy types of CW complexes.   In more modern language, that means that all given spaces are cofibrant in the mixed model structure \cite{Cole}, \cite[Section 17.3]{morecon}, and that works just as well for $G$-spaces for any topological group $G$. 

The classical theorems about transfer of $q$-model structures along adjunctions (e.g. \cite[Theorem 16.2.5]{morecon})  have analogs for transfer of $h$-model structures along adjunctions \cite[Theorems 6.12-6.14] {BMR}.  From that starting point, we claim that  all of the $q$-model categories discussed above, including $\sW_G$-$G$-spaces and $G$-spectra,  can also be given $h$-model structures and thus also $m$-model structures.    We shall not prolong matters by giving proofs, but the idea should be clear.   The main point is to check that the hypotheses of the results of \cite{BMR} just cited are satisfied by all adjunctions in sight.
We leave details to the interested reader.

Therefore, assuming that all of our input $G$-spaces, including the $G$-spaces $\sC_G(n)$ of operads, have the homotopy types of $G$-CW complexes, our bar constructions are cofibrant in the relevant mixed model structures, so that our categorical and monadic bar constructions take (levelwise) $G$-CW homotopy types to $G$-CW homotopy types.  We note here that the categorical bar constructions used in the Segal machine can be viewed as  special cases of the monadic bar constructions used in the operadic machine \cite[Section 7]{MZZ}.  These observations lead to the following folklore theorem.  It is labelled as such only because there is as yet no fully worked out proof in the literature, but we believe that arguments  in \cite{BMR, BR} can be adapted to fill the gap. 

\begin{fthm}  The bar construction approximations that we start from in both the Segal and operadic machines can be interpreted as highly structured cofibrant approximations in $m$-model structures.
\end{fthm}

\begin{rem}\label{Shul} While our focus is on our topological context, it has been understood since the 1970's that bar constructions are ``like'' cofibrant approximations.  One precise formulation of ``like'' is given in \cite{Shul} in the context of the homotopical categories of \cite{DHKS}.\footnote{Bar constructions are not mentioned in \cite{DHKS}, but the definition of left deformable functors \cite[Section 40]{DHKS} seems designed to accommodate them.}  There are many papers with results saying implicitly that bar constructions ``are" cofibrant approximations in certain contexts, without actually mentioning bar constructions explicitly.  See for example \cite[Corollary 14.8.8]{Hirsc},  \cite{Gamb}, and \cite[Propositions 6.9 and 6.10]{Horel}. A general categorical discussion of when bar constructions can be interpreted as cofibrant approximations is given by Shulman in \cite[Sections 7 and 8]{Shul1}.\footnote{Those sections are independent of the type theory focus of that paper.}  In particular, focusing on simplicial model categories, \cite[Corollary 8.16]{Shul1}  gives an explicit general context in which bar constructions are cofibrant replacements, and that can be reworked topologically.  
\end{rem}

\section{Appendix:  Bearding functors $\sD \rtarr G\sU_*$}\label{beard} 

As in \cite{MT}, but working equivariantly, we first generalize \autoref{MTbeard} from $\sF$  to a general $G$-category of  operators $\sD$ associated to a $G$-operad $\sC_G$ and then prove the generalization.

\begin{prop}\label{MTbeard2}
There is a functor $W\colon \Fun(\sD, G\sU_*)\rtarr   \Fun(\sD, G\sU_*)$  together with a natural transformation $\pi\colon W\rtarr \Id$ such that, for any functor $X\colon \sD\rtarr G\sU_*$,  $WX$ is in $\DdashG$ and $\pi\colon WX\rtarr X$ is a levelwise $G$-homotopy equivalence. 
\end{prop}

We break the proof of \autoref{MTbeard2} into  steps, following \cite[Appendix B]{MT}.

\begin{defn}[The definition of $W_nX$ and of $\pi\colon W_nX \rtarr X_n$]  Let $I = [0, 1]$ have basepoint $1$ and the binary operation 
specified by $(s, t) \mapsto \min \{s, t\}$;\footnote{We are following \cite{MT} here, but the product $st$ would work just as well.} thus $I$ is a topological abelian monoid with unit $1$ and so determines an $\sF$-space $I\colon \sF \rtarr \sT$ with $n$th  space $I^n$. Explicitly, the map 
$\ph_\ast\colon I^m\rtarr I^n$ associated to a map $\ph\colon \bm\rtarr \bn$ in  $\sF$ is specified by 
\begin{equation}\label{sillyph}
\ph(s_1, \dots, s_m) = (t_1, \dots, t_n)  \, \, \text{where} \, \,  t_j= \min_{\ph(i) = j}{s_i}.
\end{equation}
Here the minimum of the empty set is interpreted as the basepoint $1$, so that $t_j=1$ if $j\notin \im(\phi)$.   Note that $I$ is an $\sF$-space; for an injection $\ph\colon \bm \rtarr \bn$, the map $\phi_\ast \colon I^m \rtarr I^n$ is indeed a $\SI_\ph$-cofibration. We give $I$ the trivial $G$-action and regard $I$ as a functor $\sD\rtarr \TG$ by pullback along $\xi\colon \sD \rtarr \sF$, thus obtaining a $\sD$-$G$-space.

For a functor $X\colon \sD\rtarr G\sU_*$,  we have a product functor  $I\times X\colon \sD \rtarr G\sU_*$ with $n$th based $G$-space $I^n\times X_n$, and the projection $I \times X \rtarr X$ is natural with respect to $\sD$.  Define $W_nX$ to be the subspace of $I^n\times X_n$ which consists of those points $(t,x)$, 
$t = (t_1, \dots, t_n)$, such that $x\in (\io_t)_*(X_m)$, where $\io_t\colon \bm\rtarr \bn$ is that ordered injection whose image is the set of $j$, $1\leq j\leq n$, such that  $t_j=0$.  Note that 
$\io_t =\id\colon \bn\rtarr \bn$ if all $t_j=0$ and that $m=0$ if no $t_j=0$.  The latter implies that the basepoint $((1,\dots, 1), *)$ of $I^n \times X_n$ is in $W_n X$.

Define a based $G$-map $\pi\colon W_nX \rtarr  X_n$ by $\pi(t,x) = x$. Observe that $X_n$ embeds in $W_nX$ as the set of points $((0,\dots, 0), x)$ and is an unbased $G$-deformation retract of $W_nX$ via the deformation $h$ given by $h((t, x), s) = (t-st, x)$ for $0\leq s\leq 1$, where $st = (st_1,\dots, st_n)$.  Thus $\pi$ is a levelwise $G$-homotopy equivalence, although the deformations are not based homotopies.
\end{defn}

Quoting \cite[Remarks B.2]{MT},  $W_nX$ should be thought of as obtained from $X_n$  by growing a beard consisting of a thick whisker $I^{n-m}$ attached to each point of $\io(X_m)$ for each ordered injection $\io\colon \bm\rtarr \bn$, but with (implicit) identifications of partial whiskers corresponding to compositions of injections. When  the underlying $\PI$-$G$-space of $X$ is $\bR Y$ for a based $G$-space $Y$, $W_nX$ is $(I \wed Y)^n$, where $I$ is given the basepoint $0$ when forming the wedge. 

\begin{lem} As $n$-varies,  the $W_n X$ define a subfunctor  $WX$ of $I\times X$.
\end{lem}
\begin{proof}
The claim is that the action map  $\sD(\bm,\bn) \sma  (I^m \times X_m) \rtarr I^n\times X_n$ restricts to a map  
$\sD(\bm,\bn)\sma W_m X  \rtarr W_n X$.  Let  
\[(\ph;c) = (\ph;c_1, \cdots, c_n)\in \sD(\bm,\bn),\]
 where $\ph\colon \bm\rtarr \bn$ is a map in $\sF$ and $c_j$ is an element of  $\sC(\ph_j)$, and let $(s, x)$ be an element of $W_mX$. We must show that the element 
 $(\ph;c)(s,x) = (t, (\ph;c)(x))$ of $I^n\times X_n$ is in $W_nX$.  Let  $\io_s\colon \bp \rtarr \bm$ be that ordered injection  whose image is the set of $i$, $1\leq i\leq m$, such that  $s_i=0$ and let  $\io_t\colon \bq\rtarr \bn$
be that ordered injection  whose image is the set of $j$, $1\leq j\leq n$, such that  $t_j=0$. We are given that  $x = \io_s(x')$ for some $x'\in X_p$, and we must show that 
$(\ph;c)(x)$ is in $\io_t(X_q)$. It suffices to show that there exists $(\ze,d)\in \sD(\bp,\bq)$ such that 
\[(\ph;c)\com \io_s =  \io_t \com (\ze;d).\]
Indeed, if such $(\ze,d)$ exists, then 
\[(\phi;c)(x) = (\phi;c)(\io_s(x'))=\io_t((\ze;d)(x')),\]
as needed.

Since $t_j = \min_{\ph(i)=j}s_i$, we have that $i\in \im(\io_s)$ implies that $\ph(i) \in \im (\io_t)$.   There is thus a unique map $\ze\colon \bp\rtarr \bq$ in $\sF$ such the $\io_t \com \ze = \ph\com \io_s$.  When $\sD =\sF$, we are done.

For  $\sD = \sD(\sC_G)$, the definition of $\sD$ gives that 
\[(\ph, c)\com \io_s = (\ph\com \io_s;\bar{c})\]
 where $\bar{c} = (\bar c_1,\cdots, \bar c_m)$ for certain $\bar c_j$.  Since $\ph\com \io_s = \io_t\com \ze$, we have $\bar c_j = \ast\in \sC_G(0)$ if $j\notin \im(\io_t)$. Let $d$ be the $q$-tuple obtained by deleting from $\bar{c}$ the entries corresponding to those $j\notin \im(\io_t)$. Using the definition of composition in $\sD$ one can check that
 $$  (\ph\com \io_s;\bar{c}) = \io_t \com (\ze; d).$$ \qedhere

\end{proof}

The following result completes the proof of  \autoref{MTbeard2}.

\begin{lem} $WX$ is in $\DdashG$.
\end{lem}
\begin{proof}  We must show that if $\ph\colon \bm \rtarr \bn$ is an injection in $\PI$ and $\SI_{\ph}$ is the group of permutations $\si\colon \bn\rtarr \bn$ such that $\si \ph = \ph$, then 
$\ph\colon  W_mX\rtarr W_nX$ is a $(G\times \SI_{\ph})$-cofibration.  We claim first that
$$ \ph(W_mX) = \{(t,x) | t_j = 1  \ \ \text{for all} \ \ j\notin \im(\ph) \} \subset W_n X.$$
Since $\ph$ is an injection and $WX\subset I\times X$, we certainly have $t_j=1$ if $j\notin \im(\ph)$ for any point $(t,x)\in  \ph(W_mX)$.  
Conversely, consider $(t,x)\in  W_nX$ such that $t_j = 1$ for all $j\notin \im(\ph)$.  Let $\io_t\colon \bp \rtarr \bn$ be that ordered injection whose image is the set of $j$, $1\leq j\leq n$, such that  $t_j=0$.  Our conditions imply that $\im(\io_t) \subset \im(\ph)$, hence there is a unique injection $\ze\colon \bp\rtarr \bm$ in $\PI$ such that   $\io_t = \ph \com \ze$.  We have that $x\in \im(\io_t)$ and thus $x \in \im(\ph)$ and $(t,x) \in \ph(W_m  X)$.  This proves the claim.  

It follows readily from this description that $\ph\colon  W_m X \rtarr W_nX$ is a $G\times \SI_{\ph}$-cofibration.   To see this, notice that $\ph$ is induced from $\ph\colon I^m \rtarr I^n$ of \autoref{sillyph}, which is a $\SI_{\ph}$-cofibration as mentioned above, and the isomorphism  $\ph\colon W_m X \rtarr \ph(W_mX)$. 
\end{proof}

\section{Appendix: Realization of levelwise $G$-cofibrations and $G$-equivalences}\label{COF}

We here prove the following results, which generalize Theorems \ref{PlenzWeak} and \ref{PlenzCof} from finite groups to topological groups  $G$.
As before, we assume that the unit of $G$ is a nondegenerate basepoint and we understand subgroups of $G$ to be closed.

\begin{thm}\label{PlenzWeak2} 
Let $f_{*}\colon X_{*}\rtarr Y_{*}$ be a map of Reedy cofibrant simplicial $G$-spaces such that each $f_n$ is a weak $G$-equivalence. 
Then the realization $|f_{*}|\colon |X_{*}|\rtarr  |Y_{*}|$ is a weak $G$-equivalence.
\end{thm}

\begin{thm}\label{PlenzCof2} 
Let $f_{*}\colon X_{*}\rtarr Y_{*}$ be a map of Reedy cofibrant simplicial $G$-spaces such that each $f_n$ is a $G$-cofibration. 
Then the realization $|f_{*}|\colon |X_{*}|\rtarr |Y_{*}|$ is a $G$-cofibration.
\end{thm} 

\autoref{PlenzWeak2} is an immediate consequence of the nonequivariant version and the following result, which is often taken for granted. Jonathan Rubin noticed this gap in the literature and provided the proof that follows.

\begin{lem}\label{Rubin}  Geometric realization of Reedy cofibrant simplicial $G$-spaces commutes with passage to $H$-fixed points for  subgroups $H$ of $G$.
\end{lem}
\begin{proof}
For a $G$-space  $X$,  $X^H$ is the equalizer exhibited in the diagram
$$\xymatrix@1{ X^H \ar[r]^-{i} &  X   \ar@<.5ex>[r]^-{\DE}   \ar@<-.5ex>[r]_-{(h)}  & \prod_{h\in H} X, \\}   $$
where $(h)$ denotes multiplication by $h$ on the $h$th component.
When $G$ is not finite, this is an infinite limit, and realization does not generally commute with infinite limits, not even with infinite products.   Consider the following diagram, where $X_{*}$ is a simplicial $G$-space, $K$ is a space, and $f\colon K \rtarr |X_*|$ is a map such that   $\DE \com f = (h) \com f$.

$$\xymatrix{   |X_{*}^H| \ar[r]^-{|i_*|} &  |X_*|   \ar@<.5ex>[rr]^-(.7){|\DE_*|}   \ar@<-.5ex>[rr]_-(.7){|(h)_*|}   \ar@<.5ex>[drr]^-(.7){\DE}   \ar@<-.5ex>[drr]_-(.7){(h)} 
&  & |\prod_{h\in H} X_*| \ar[d]^{\pi} \\
&  K\ar[u]_{f}   \ar@{-->}[ul]_{\tilde{f}} & & \prod_{h\in H} |X_*|  \\}   $$
The top row is an equalizer and $\pi$ is the natural map.  The maps to products have components given by multiplication by $h$, and the two evident triangles clearly commute.  Moreover,  $|\DE_*|\com f = |(h)_*|\com f$ by inspection of the unique representation of points $f(k) \in |X_*|$ in nondegenerate form \cite[Lemma 11.3]{MayGeo}.  Since the top row is an equalizer, there is a unique map  $\tilde{f}\colon K\rtarr |X_*^H|$ such that $|i_*|\com \tilde{f}= f$.   Therefore $|i_*|$ is the equalizer of $\DE$ and $(h)$, giving the conclusion.
\end{proof}

To prove \autoref{PlenzCof2}, we will need the following two standard results about $G$-cofibrations. Recall that we are working in the category of compactly generated weak Hausdorff spaces, so that  all cofibrations are closed inclusions. We use the convention of identifying the domain of a cofibration with its image.

 \begin{prop}\label{glueingcof} Consider the following diagram in $G$-spaces. 
 \[\xymatrix{
 A \ar[rr] \ar[dd]_i \ar[rd]^{\al} && C \ar '[d] [dd] \ar[rd]^{\ga}\\
 & A'\ar[rr] \ar[dd]_<(.3){i'} && C'\ar[dd]\\
 B \ar '[r] [rr] \ar[rd]_{\be} && D \ar[rd]\\
 & B'\ar[rr]  && D'\\
 }\]
 Assume that the front and back faces are pushouts, $\al$, $\be$, $\ga$, $i$ and $i'$ are $G$-cofibrations, and $A'\cap B = A$ (as subsets of $B'$). Then the map $D \rtarr D'$ is also a $G$-cofibration.
\end{prop}

\begin{proof} 
The nonequivariant result is \cite[Proposition 2.5]{LewisOMSIX} and, as pointed out there, the equivariant proof goes through the same way.
\end{proof}

The following result is  given in \cite[Proposition A.4.9]{BVbook} and \cite[Lemma 3.2.(a)]{LewisOMSIX}.

\begin{prop}\label{sequentialcolim}
Let $A_0 \rtarr A_1 \rtarr A_2 \rtarr \cdots$ and $B_0 \rtarr B_1 \rtarr B_2 \rtarr \cdots$ be diagrams of $G$-cofibrations and let $f_i \colon A_i \rtarr B_i$ be a map of diagrams such that each $f_i$ is a $G$-cofibration. Assume moreover that for every $i\geq1$, $A_{i-1}=A_i\cap B_{i-1}$. Then the induced map $\colim_i A_i \rtarr \colim_i B_i$ is a $G$-cofibration.
\end{prop}

The following observation is the key to using these results to prove \autoref{PlenzCof}.

\begin{lem}\label{intersectionsi}
Let $f_{*} \colon X_{*} \rtarr Y_{*}$ be a map of simplicial $G$-spaces that is levelwise injective. Then for all $n\geq 1$ and all $0\leq i \leq n-1$,
\[s_i(Y_{n-1}) \cap f_n(X_n) = f_n s_i(X_{n-1})\]
and
\[L_n Y \cap f_n(X_n) = f_n(L_n X).\]
\end{lem}

\begin{proof}
 For the first equality, one of the inclusions is obvious. For the other inclusion, take $y=s_i(y')=f_n(x)$. Then $y'=d_is_i(y')=d_if_n(x)=f_{n-1}d_i(x)$. Thus, $y=s_if_{n-1}d_i(x)=f_ns_id_i(x)\in f_ns_i(X_{n-1})$, as wanted. The second statement is obtained from the first by taking the union over all $i$.
\end{proof} 

\begin{prop}
Let $X_{*}$ and $Y_{*}$ be Reedy cofibrant simplicial $G$-spaces, and let $f_{*} \colon X_{*} \rtarr Y_{*}$ be a level $G$-cofibration. Then for all $n\geq 1$, $L_nX \rtarr L_nY$ is a $G$-cofibration.
\end{prop}

\begin{proof}
 We proceed by induction on $n$. Note that when $n=1$, $L_1X = s_0 X_0$ which is homeomorphic to $X_0$, so there is nothing to prove. For $n>1$, let $k\leq n-1$. Just as nonequivariantly, we have a pushout square 
\begin{equation}\label{Lnpushout}\xymatrix{
s_k(\bigcup_{i=0}^{k-1} s_i(X_{n-2}))  \ar[d] \ar[r]  & \bigcup_{i=0}^{k-1} s_i(X_{n-1}) \ar[d] \\
s_k(X_{n-1})  \ar[r] & \bigcup_{i=0}^{k} s_i(X_{n-1}).
}\end{equation}
We show by induction on $n$ and $k$ that the right vertical map is a $G$-cofibration. By the inductive hypothesis for $n-1$ and the fact that $X_{*}$ is Reedy cofibrant, we have that the composite
\[\bigcup_{i=0}^{k-1} s_i(X_{n-2})\rtarr \bigcup_{i=0}^{k} s_i(X_{n-2})\rtarr \dots \rtarr \bigcup_{i=0}^{n-2} s_i(X_{n-2}) =L_{n-1}X\rtarr X_{n-1}\]
is a $G$-cofibration. Since $s_k$ is a $G$-homeomorphism onto its image, the left hand map in the square is a $G$-cofibration. Therefore, the righthand map in \autoref{Lnpushout} is also a $G$-cofibration by cobase change. 

We apply \autoref{glueingcof} to the pushout in \autoref{Lnpushout} to deduce that the maps
$$ \bigcup_{i=0}^{k} s_i(X_{n-1}) \rtarr  \bigcup_{i=0}^{k} s_i(Y_{n-1}) $$ are $G$-cofibrations. 
In particular, by induction on $n$ and $k$, $L_nX \rtarr L_nY$ is a $G$-cofibration. The required intersection condition follows from \autoref{intersectionsi}.
\end{proof}

\begin{proof}[Proof of \autoref{PlenzCof2}]
 Recall from \autoref{SecSegDet} that $|X_{*}|$ is the colimit of its filtration pieces $F_p |X|$, and that these can be built by the iterated pushout squares \autoref{grpushout}. Note that all of the vertical maps in diagram  \autoref{grpushout} are $G$-cofibrations. 
There is a similar diagram for $|Y_{*}|$ and a map of diagrams induced by $f_{*}$.  The maps between the three corners of the upper squares are $G$-cofibrations. By \autoref{glueingcof}  and \autoref{intersectionsi}, the map between the pushouts of the upper squares is a $G$-cofibration. 

We assume by induction that the map $F_{p-1} |X| \rtarr F_{p-1}|Y|$ is a $G$-cofibration, so again we have that all three maps between the corners of the lower pushout square for $X$ and the one for $Y$ are $G$-cofibrations. \autoref{intersectionsi} implies again that the intersection condition necessary to apply \autoref{glueingcof} holds, and we can deduce by induction that the map $F_{p} |X| \rtarr F_{p}|Y|$ is a $G$-cofibration.

Lastly, we check the intersection condition of \autoref{sequentialcolim}, namely that $f(F_{p}|X|) \cap F_{p-1}|Y|=f(F_{p-1}|X|)$.  One inclusion is  obvious. To see the other inclusion, take an element $(u,f_px)$ in $f(F_p |X|) \setminus f(F_{p-1}|X|) =  (\DE_p \setminus \pa\DE_p) \times f_p(X_p \setminus L_pX)$. Since $f_p$ is injective, we can again use \autoref{intersectionsi} to see that $(u,f_p x) \in F_p Y \setminus  F_{p-1}Y$. Thus  an element in the intersection must be in $F_{p-1}|X|$. By  \autoref{sequentialcolim}, we conclude that the map
$$|f_{*}|\colon |X_{*}|\rtarr |Y_{*}|$$ is a $G$-cofibration.
\end{proof}

\bibliographystyle{plain}
\bibliography{references}

\end{document}